\documentclass[reqno]{aomart}

\usepackage[english]{babel}
\usepackage[utf8]{inputenc}
\usepackage{amsmath,amssymb}
\usepackage{mathrsfs} 
\usepackage{yhmath}
\usepackage{dsfont}
\usepackage{titlesec}
\usepackage[T1]{fontenc}
\usepackage{enumitem}
\usepackage{stmaryrd}
\usepackage{tikz}
\usepackage{tikz-cd}
\usepackage[none]{hyphenat}
\usepackage{microtype}

\usepackage{titletoc}

\allowdisplaybreaks

\title{}\date{}\author{}

\DeclareMathOperator{\D}{\mathbb D}
\DeclareMathOperator{\Grad}{grad}
\DeclareMathOperator{\Div}{div}
\DeclareMathOperator{\divoperator}{div}
\DeclareMathOperator{\diffoperator}{d}
\DeclareMathOperator{\dom}{Dom}
\DeclareMathOperator{\Dop }{D}
\DeclareMathOperator{\Real}{Re}
\DeclareMathOperator{\Imag}{Im}
\DeclareMathOperator{\Sgn}{sgn}
\DeclareMathOperator{\Loperator}{L}
\DeclareMathOperator{\range}{Range}

\newcommand{\diff}{\diffoperator \!}
\newcommand{\itodiff}{\raisebox{0.3ex}{.} \diff}
\newcommand{\stratdiff}{\circ \diff}
\newcommand{\norm}[1]{\left\lVert #1 \right\rVert}
\newcommand{\normb}[2][\null]{\left\lVert #2 \right\rVert_{#1}}
\newcommand{\abs}[1]{\left\lvert #1 \right\rvert}

\newcommand{\proc}[1]{\mathcal{#1}}

\newcommand{\thmref}[1]{Theorem~\ref{#1}}

\newcommand{\lemref}[1]{Lemma~\ref{#1}}

\newcommand{\R}{{\mathbb R}}
\newcommand{\N}{{\mathbb N}}
\newcommand{\E}{{\mathbb E}}

\newcommand{\C}{{\mathbb C}}
\newcommand{\sko}{{\mathbb D}}
\newcommand{\F}{\mathcal{F}}
\newcommand{\W}{\mathcal{W}}
\newcommand{\U}{\mathcal{U}}
\newcommand{\V}{{\mathcal V}}
\newcommand{\M}{{\mathcal M}}
\newcommand{\T}{{\mathcal T}}

\newcommand{\wienmeas}{d\gamma^{\N}}
\newcommand{\dinf}{\sko^\infty}
\newcommand{\dat}{\widetilde{\delta}_{\widecheck{A}}}
\newcommand{\pf}{{\it Proof}}
\newcommand{\drp}{\sko_r^p}
\newcommand{\dqs}{\sko_s^q}
\newcommand{\conx}[1]{\overset{(#1)}{\nabla}}
\newcommand{\up}{\uparrow}

\newtheorem{cor}{Corollary}[section]
\newtheorem{crit}{Criterion}[section]
\newtheorem{defn}{\sc Definition}[section]
\newtheorem{lem}{\sc Lemma}[section]
\newtheorem*{notation}{Notation}
\newtheorem{prop}{Proposition}[section]
\newtheorem{rem}{Remark}[section]
\newtheorem*{unnumbered remark}{Remark}

\newtheorem{thm}{Theorem}[section]
\newtheorem*{example}{Example}
\newtheorem{counterexample}{Counter Example}[section]

\renewcommand{\P}{{\mathbb P}}
\renewcommand{\L}[1]{\Loperator^{#1}}
\renewcommand{\div}{\divoperator}
\renewcommand{\Re}{\mathfrak{Re}}
\renewcommand{\mid}{\mathop{/}}

\newcommand*\circled[1]{%
  \tikz[baseline=(C.base)]\node[draw,circle,inner sep=0.5pt](C) {#1};
}

\DeclareFontFamily{U}{mathx}{\hyphenchar\font45}
\DeclareFontShape{U}{mathx}{m}{n}{
      <5> <6> <7> <8> <9> <10>
      <10.95> <12> <14.4> <17.28> <20.74> <24.88>
      mathx10
      }{}
\DeclareSymbolFont{mathx}{U}{mathx}{m}{n}
\DeclareFontSubstitution{U}{mathx}{m}{n}
\DeclareMathAccent{\widecheck}{0}{mathx}{"71}
\DeclareMathAccent{\wideparen}{0}{mathx}{"75}

\newcommand*{\justifyheading}{\raggedright}

\renewcommand{\thesection}{\arabic{section}.}
\renewcommand{\thesubsection}{\arabic{section}. \arabic{subsection}}

\titleformat{\section}{\normalfont\LARGE\bfseries\justifyheading}{\thesection}{1em}{}
\titleformat{\subsection}{\normalfont\large\bfseries\justifyheading}{\thesubsection}{1em}{}

\begin{document}

Anatole KHELIF \\
Institut de Mathématique de Jussieu \\
khelif@math.univ-paris-diderot.fr \\

Alain TARICA \\
8 rue Chausse Coq \\
Genève 1204 \\
alaintarica@bluewin.ch \\

\renewcommand{\contentsname}{\Huge SOMMAIRE}
\tableofcontents
\titlecontents{section}[0em]{\addvspace{0.5em plus 0pt}\bfseries}{\thecontentslabel~~}{}{\hfill\contentspage}[\addvspace{0pt}]
\titlecontents{subsection}[3em]{}{\thecontentslabel~~}{}{\dotfill\contentspage}[\addvspace{0pt}]

\section*{Abstract}
\section*{\large <<Stochastic Manifolds>>}
Malliavin Calculus can be seen as a differential calculus on Wiener spaces.
We present the notion of stochastic manifold for which the Malliavin Calculus plays the same role as the classical differential calculus for the $\C^\infty$ differential manifolds.
The set of the paths in a Riemmanian compact manifold is then seen as a particular case of the above structure.

\section*{Abreviations Index}

\begin{enumerate}[label=-]
\item a.s. : almost surely
\item A.M. : antisymmetrical matrix
\item C.M. : Cameron Martin space
\item l.h.s, r.h.s : left-hand side, right-hand side
\item N.C.M. : new Cameron Martin space
\item N.S.C. : necessary and sufficient condition
\item OTHN : orthonormal
\item O.U. : Ornstein-Uhlenbeck
\item S.D.E. : stochastic differential equation
\item S.M. : semimartingale
\item S.T.P. : stochastic parallel transport
\end{enumerate}

\section*{Conventions Index}

\begin{enumerate}[label=-]
\item $\D^\infty$-derivation : a derivation on $\D^\infty$ that is
  $\D^\infty$-continuous.
\item Einstein summation, unless the contrary is specified.
\item $\Grad(f {\Grad g}) = {\Grad f} \otimes {\Grad g} + f {\Grad {\Grad g}}$.
\end{enumerate}

\section*{Notations Index}

\begin{enumerate}[label=-]
\item $B_{p,q}^\lambda(H)$ : Besov space built on an Hilbert $H$ with
  indexes $\lambda,p,q$.
\item $B_{p,q}^\lambda = B_{p,q}^\lambda(\mathbb R)$.
\item $\complement_E A$ or $\complement A$ : complementary of the set
  $A$.
\item $\mathcal C_n(\Omega)$ : chaos of order $n$ in $\mathrm{L}^2(\Omega)$.
\item $\delta_j^i$ : Kronecker symbol.
\item $(\ ,\ )$ : duality bracket between a space and its dual, or
  between a distribution and some test function.
\item $\mathcal F_{|A}$ : $\mathcal F$ being a $\sigma$-algebra,
  $\mathcal F_{|A}$ is the $\sigma$-algebra : $\{A \cap F \mathrel{/} F
  \in \mathcal F \}$.
\item $f_{|A}$ : $f$ being a map, $f_{|A}$ is the restriction of $f$
  to $A \subset \dom f$.
\item $\Grad$ : the Malliavin derivative unless otherwise specified.
\item $\overrightarrow{\Grad} f$ : the classic gradient of a $\mathcal
  C^\infty$-function $f$ on an $n$-dimensional manifold.
\item $\Gamma(V_n)$ : $\mathcal C^\infty$-vector fields on the
  $n$-dimensional manifold $V_n$.
\item $\int_a^b{f \itodiff B}$ : Ito integral.
\item $\int_a^b{f \stratdiff B}$ : Stratonovich integral.
\item $h \in H$ : $h$ is a vector, element of the Cameron Martin space
  $H$ defined by $t \mapsto \int_0^t{\dot h(s) \diff s}$.
\item $h(\omega)$ : $h(\omega)$ is the vector field defined by $t
  \mapsto \int_0^t \dot h(s,\omega) \diff s$.
\item $\mathcal L(H_1,H_2)$ : vector space of the bounded linear maps
  between the Hilbert spaces $H_1, H_2$.
\item $L^{p+0}(\Omega)$ : $\bigcup_{q>p} L^q(\Omega)$.
\item $L^{p-0}(\Omega)$ : $\bigcap_{q<p} L^q(\Omega)$.
\item $\mathbb N_\ast = \mathbb N \setminus \{0\}$.
\item $(\psi_i)_{i \in I} \xrightarrow{L^p, \D^\infty} \psi$ : the net
  $(\psi_i)_{i \in I}$ converges toward $\psi$ in $L^p(\Omega),
  \D^\infty(\Omega)$.
\item $\langle\ ,\ \rangle_H$ : scalar product on the Hilbert $H$.
\item $\llbracket \tau_1, \tau_2 \rrbracket$ : stochastic interval
  delimited by the stopping times $\tau_1,\tau_2$.
\item ${}^tV,{}^th$ : transposes of the vectors $V, h$.
\item $\mathcal W$ : Wiener space.
\item $W(h)$ : Gaussian variable, centered on $0$, with law :
  \begin{displaymath}
    \frac 1 {\sqrt{2\pi {\normb[H] h}}} \, {\mathrm e}^{-\frac {x^2} {2 \normb[H] h}} \diff x.
  \end{displaymath}
\end{enumerate}

\addcontentsline{toc}{section}{\large 0. \huge  Introduction}
\section*{\large 0. \huge  Introduction}

"To do a geometry you do not need a space, you only need an algeba of functions on this would-be space." 

A. Grothendieck

The Malliavin Calculus can be seen as a differential calculus on Wiener spaces. It is then possible to establish a new dimensionless differential geometry, for which the Malliavin calculus plays the same role as the one played by the classical differential calculus in the theory of $n$-dimensional manifolds.

Moreover, is it also possible to obtain a Variational Calculus on a random structure, which is built by constraints subjected to infinitesimal variation? This sort of problem is recurrent in Physics and Econometry.

Such a Variational Calculus imposes a reasonnable space of "measurable and regular" functions, with a compatibility between associated differentiation and integration processes, from which a generalized divergence operator.

As such a problem requires an infinite dimensional space, it becomes needed to have an infinite dimensional differential calculus with a related good notion of a divergence.

The Malliavin Calculus provides such a tool. More precisely: in $\R^n$, there is compatibility between differentiation and integration because the Lebesgue measure is translation invariant. Unfortunately, in the case of an infinite dimensional topological vector space $E$, such a non-trivial translation measure does not exist. But there can be quasi-invariant measures $\mu$, that is : there is a dense subspace $H$ of $E$, such that the image measure of $\mu$ by a translation with the vector $h \in H$, admits a density relatively to $\mu$.

A natural is $E$=Wiener space $\mathcal W$, with $\mu$=Gaussian measure, and $H$ being the Cameron-Martin space which then is an Hilbert space.

More precisely, given a basis manifold $V$ and a fiber space $F$ on $V$, it is possible to endow the space of the random sections of $F$, witch a reasonable measure so that there is a Variational calculus.

Two particular cases which are extreme case have already been studied : random Brownian fields (maps from $V$ in a Gaussian space), and the set of continuous paths in a Compact Riemannian manifold.

In the first case, there has been the Wiemann (Wiener + Riemann) manifold [9, 10, 11, 14].

But it brought several very strong limitations :

\begin{enumerate}
\item a Wiemann manifold is a triple $(W, \tau, g)$ Banach $C^j$-manifold, modelled on an abstract Wiener space $(H, B)$ with $j \geq 1$. And: $\forall x \in W$, $\tau_x$ is a norm on $T_xW$ and $g(x)$ is a densely defined inner product on $T_xW$

\item The chart change maps must be of the form $I_B + K$, $K$ having to fulfill several conditions [10, 14].

\end{enumerate}

Moreover the set of continuous paths on a compact Riemannian manifold $V_n$, starting from $m_0$ (denoted in this paper $\P_{m_0}(V_n, g)$ ) cannot be naturally described as a Wiemann structure, while we will prove that $\P_{m_0}(V_n, g)$ is a $\dinf$-stochastic manifold [4].

A slightly different definition of a Wiemann manifold, $\W$, is given by 

G. Peters [2], which does not impose that $\W$ be a Riemannian manifold, but instead, that $\W$ be a measure space, its $\sigma$-algebra being generated by a locally-finite countable family of subsets of $\W$, $(\U_\alpha)_{\alpha \in \N_\star}$, each $\U_\alpha$ being a $H-C^k$ set and the family $(\U_\alpha)_{\alpha \in \N_\star}$ must admit a subordinate $\dinf$-unity partition.

Moreover the chart change maps must admit similar conditions as in the previous definition above.

The other extremal case has been studied by P. Malliavin and A. B. Cruzeiro, [4], and it also brought major constraints:

\begin{enumerate}
\item $C_0 \left( [0,1], \R^2 \right)$ and $C_0 \left( [0,1], \R^3 \right)$ are not diffeomorphic, although, as Wiener spaces, they are isomorphic.

\item $\P(m_0, V_n)$ is seen by the authors as the domain of a single chart, with the It\={o} map. But the It\={o} map does not admit a natural linear tangent map. So the authors had to enlarge the tangent space with particular processes, called tangent spaces, which are semi-martingales. So the time filtration becomes invariant. This invariance has important consequences, among them the impossibility to include the Brownian fields in this framework. And if a manifold structure could be endowed on $P_{m_0}(V_n, g)$, this structure would strictly depend on the time filtration.
\end{enumerate}

We offer here a new mathematical structure named: the $\dinf$-stochastic manifold, which overcomes the Wiemann structure, and its limitations, for which $P_{m_0}(V_n, g)$ is a particular case, $C_0 \left( [0,1], \R^n \right)$ and $C_0 \left( [0,1], \R^m \right)$, $n \neq m$ being $\dinf$-diffeomorphic.

Moreover, with such a structure:

\begin{enumerate}
\item the notion of time (filtration) does not play any role anymore

\item in the case of $\P_{m_0}(V_n, g)$, dim $V_n$ is not anymore relevant.

\item given a metric on $V_n$, the various connections compatible with the metric, induce canonical associated It\={o} maps, and $\dinf$-diffeomorphisms on $\P_{m_0}(V_n, g)$.

\end{enumerate}

To build the general theory of $\dinf$-stochastic manifolds, a source will be the Grothendieck identification of an $n$-dimensional manifold with a sheaf of $C^\infty$-functions; here, $C^\infty$ will be replaced by $\dinf(\Omega)$, and a diffeomorphism will be a map between two Gaussian spaces that will keep the $\dinf$ property through right-composition and this diffeomorphism will have a canonical "cotangent" linear map.

A generalisation of the notion of metric will be established, which will live on the "cotangent spaces" (and not on the "tangent spaces"). And $\P_{m_0} (V_n, g)$ will be a particular case of this $\dinf$-structure.

Moreover, for the general $\dinf$-structure, it is possible to define the notions of curvature and torsion, but they can become infinite. But nevertheless, a variational calculus of the curvature, function of the metric, can be realized.

Among the notable differences between a $\dinf$-stochastic structure on a set, and a $C^\infty$-$n$-dimensional manifold, we have:

\begin{enumerate}
\item in a $C^\infty$-$n$-dimensional manifold, vector fields and derivations coincide; such is not the case for a $\dinf$-stochastic manifold.

\item on a Riemannian $n$-dimensional manifold, $C^\infty$ functions can be defined either through $C^\infty$-charts reading, or by iteration of the Laplacian; both definitions coincide.

\end{enumerate}

On a $\dinf$-stochastic manifold we can define a $\dinf$-function either through charts reading or, if there is a metric and probabilistic measure, by iteration of the Ornstein-Uhlenbeck operator: these two definitions, in general, do not coincide but for $\P_{m_0}(V_n, g)$ there is an inclusion.

More diffeomorphisms give more changes of variables thus more opportunities to compute integrals.
In all the following, we suppose that all Cameron-Martin spaces have countable Hilbertian bases, but this is just for simplification for the reader and is not a loss of generality.

\section{\huge $\mathbb{D}^\infty_r$-Stochastic Manifold, $r\in \mathbb{N}_*$}
Here we will study the $\mathbb{D}^\infty_r$-stochastic manifold structure. In this particular case, any map, change of charts, admits a tangent linear map between the respective Cameron-Martin spaces; such a tangent linear map does not exist anymore in the case of the $\mathbb{D}^\infty$ structure.\\

Moreover this $\ {D}^\infty_r$ type of structure is not satisfying because it does not include as a particular case the set of the continuous paths in a compact manifold $V_n$, starting from $m_0$ (denoted $\mathbb{P}(m_0,V_n)$).\\

Reminder: 1) all $\sigma$-fields are  complete,\\
2) a Gaussian probability space $[ 13 ]$ is given by the following elements: \\
i) $(\Omega, \mathcal{F}, \mathbb{P})$ a probability space,\\
ii) a closed subspace $H$ of $L^2 (\Omega, \mathcal{F}, \mathbb{P})$ such that all the random variables belonging to $H$ have a centered Gaussian law,\\
iii) the $\sigma$-field generated by these variables is $\mathcal{F}$.\\
3) $\mathbb{D}^\infty_r(\Omega)$ is a Frechet space, and its distance is denoted by $d$: \[d(\varphi,\psi)=d(\varphi-\psi,0)= \sum_{\substack{k\geqslant 1 \\ j\in \lbrace 0,\ldots,r \rbrace }} \frac{1}{2^{k}}\cdot 1\wedge \|\varphi-\psi\|_{\mathbb{D}^k_j(\Omega)}\].

\subsection{Definition and charts exchange maps}

\begin{defn} Let $\mathcal{S}$ be a set; a stochastic chart on $\mathcal{S}$ is given by a subset of $\mathcal{S}$ denoted $\mathcal{U}$, named: domain of the chart, a Gaussian space  $(\Omega, \mathcal{F}, \mathbb{P}, H)$ and a bijection $b$ from $\mathcal{U}$ onto $\Omega$.\\

This chart is denoted $(\mathcal{U}, b, \Omega, \mathcal{F}, H)$ or in short: $(\mathcal{U}, b, \Omega)$. 
\end{defn}

\begin{defn} Two stochastic maps $(\mathcal{U}_i, b_i, \Omega_i, \mathcal{F}_i,\mathbb{P}_i,H_i), i=1,2$, will be said to be  $\mathbb{D}^\infty_r$-compatibles if and only if:\\
i) $b_1 \circ b_2^{-1}=b_{21}$ and $b_2 \circ b_1^{-1}=b_{12}$ are measurable maps between \\
$(b_2 (\mathcal{U}_2\cap \mathcal{U}_1), \mathcal{F}_2 \vert_{b_2(\mathcal{U}_1\cap \mathcal{U}_2)})$ and $(b_1 (\mathcal{U}_1\cap \mathcal{U}_2), \mathcal{F}_1 \vert_{b_1(\mathcal{U}_1\cap \mathcal{U}_2)})$ \\
ii) $b_{12}$ and $b_{21}$ exchange the $\mathbb{P}_i$-null sets of  $\mathcal{F}_i \vert_{b_i(\mathcal{U}_1\cap \mathcal{U}_2)}$, $i=1,2$\\
iii) $\forall A \subset b_1(\mathcal{U}_1\cap \mathcal{U}_2), A \in \mathcal{F}_1, \mathbb{P}_1(A)>0, \exists A_1 \subset A, A_1 \in  \mathcal{F}_1 \vert_{b_1(\mathcal{U}_1\cap \mathcal{U}_2)}$, \\
$\mathbb{P}_1(A_1)>0$ such that $\forall \varphi \in \mathbb{D}^\infty_r(\Omega_2)$, $\varphi \circ b_{12} \vert_{A_1}$ admits an extension map, denoted $\widetilde{\varphi \circ b_{12} \vert_{A_1}}$ and $\widetilde{\varphi \circ b_{12} \vert_{A_1}} \in \mathbb{D}^\infty_r(\Omega_1)$ and conversely, the same extension property is valid for $B \subset b_2(\mathcal{U}_1\cap \mathcal{U}_2), B\in \mathcal{F}_2, \mathbb{P}_2(B)>0, \varphi\in \mathbb{D}^\infty_r(\Omega_1)$. $b_{12}$ and $b_{21}$ are called charts changes, or charts maps. 
\end{defn}

\begin{defn}
A $\mathbb{D}^\infty_r$-stochastic manifold is a set $\mathcal{S}$ and a family $\mathcal{A}$ of stochastic charts, which are $\mathbb{D}^\infty_r$-compatibles, and such that the union of the domains of the charts covers $\mathcal{S}$. It is denoted: $(\mathcal{S}, (\mathcal{U}_i, b_i, \Omega_i)_{i \in I})$. \\
$(\mathcal{U}_i, b_i, \Omega_i)_{i \in I}$ being the family of charts.\\

Such a family is called an $\mathbb{D}^\infty_r$-atlas of $\mathcal{S}$.
\end{defn}

\begin{defn}
Let $(\mathcal{S}, (\mathcal{U}_i, b_i, \Omega_i)_{i \in I})$ a $\mathbb{D}^\infty_r$-stochastic manifold, and $(A, \mathcal{F}, \mathbb{P})$ a probability space with $A \subset \mathcal{S}$. The atlas $(\mathcal{U}_i, b_i, \Omega_i)_{i \in I}$ will be said to cover $A$ if and only if: $\forall B \subset A, \exists i \in I$ and $B_1 \in \mathcal{F}$, such that $B_1 \subset \mathcal{U}_i \cap B$ and $\mathbb{P}(B_1)>0$.
\end{defn}

\begin{lem}
With the notations of definition 1.4, there exist a countable family of charts from the atlas $(\mathcal{U}_i, b_i, \Omega_i)_{i \in I}$, $(\mathcal{U}_j, b_j, \Omega_j)_{j \in \mathbb{N}_*}$ such that:\[A \subset \bigcup_{j \in \mathbb{N}_*} \mathcal{U}_j.\]
\end{lem}
The notions of a compatible chart to an atlas, or of equivalent atlases, will be given in the case of the $\mathbb{D}^\infty$-stochastic manifold. \\

\subsection{Existence of a tangent linear maps}

\hspace{-8mm}{\bf Morphisms associated to chart changes}\\

Let $\mathcal{S}$ be a set, and $(\mathcal{U}_i, b_i, \Omega_i, \mathcal{F}_i,\mathbb{P}_i,H_i),i=1,2$ two charts on $\mathcal{S}$, $\mathbb{D}^\infty_r$-compatibles. We denote again by $\mathbb{P}_1$ and $\mathbb{P}_2$ the probability measures restricted to $\mathcal{F}_1 \vert_{b_1(\mathcal{U}_1\cap \mathcal{U}_2)}$ and $\mathcal{F}_2 \vert_{b_2(\mathcal{U}_1\cap \mathcal{U}_2)}$. Let $B\in \mathcal{F}_2 \vert_{b_2(\mathcal{U}_1\cap \mathcal{U}_2)}$ such that $\mathbb{P}(B_2>0)$, and denote $\hat{B}=\lbrace \beta \in \mathbb{D}^\infty_r (\Omega_2) / \beta\vert_B =0\rbrace$ and $\stackrel{B}{\sim}$ the equivalence relation: $\varphi_1, \varphi_2 \in \mathbb{D}_r^\infty (\Omega_2): (\varphi_1-\varphi_2)\vert_B=0$, denoted: $\varphi_1 {\stackrel{B}{\sim}} \varphi_2$. $[\varphi]_B$ will be the class of $\varphi$, according to ${\stackrel{B}{\sim}}$; then $\mathbb{D}^\infty_r(\Omega_2)/_{\stackrel{B}{\sim}}$ is a Frechet space , the distance $d_B$ being built with the semi-norms:
\[\| [\varphi]_B \|_{\mathbb{D}^p_r}= \inf_{\beta\in \hat{B}} \| \varphi+\beta \|_{\mathbb{D}^p_r(\Omega_2)}\]
This definition of $\| [\varphi]_B \|_{\mathbb{D}^p_r}$ is legitimate.\\

In a same way, denoting $A=b_{12}^{-1}(B), \mathbb{P}_1(A)>0$, one can define $\mathbb{D}^\infty_r(\Omega_1)/_{\stackrel{A}{\sim}}$, which is also a Frechet space.\\

Let $F_B$ the map: $\mathbb{D}^\infty_r(\Omega_2)/_{\stackrel{B}{\sim}} \rightarrow \mathbb{D}^\infty_r(\Omega_1)/_{\stackrel{A}{\sim}}$, define by: $[\varphi]_B \rightarrow [\widetilde{\varphi\circ b_{12}\vert_A}]_A$. This definition of $F_B$ is legitimate with regard to the equivalence $\stackrel{B}{\sim}$.\\

Conversely, we can define
\[ F_A: \mathbb{D}^\infty(\Omega_1)/_{\stackrel{A}{\sim}} \rightarrow \mathbb{D}^\infty(\Omega_2)/_{\stackrel{B}{\sim}}\]
by: $F_A([\varphi]_A) \rightarrow [\widetilde{\varphi_A\circ b_{21}\vert_B}]_B$.

\begin{lem}\hfill\\
i) $F_B$ and $F_A$ are continuous\\
ii) \[(b_{12})_* \mathbb{P}_1 \ll\mathbb{P}_2\quad on \quad(b_2 (\mathcal{U}_1\cap \mathcal{U}_2), \mathcal{F}_2 \vert_{b_2(\mathcal{U}_1\cap \mathcal{U}_2)})\] 
and \[(b_{21})_* \mathbb{P}_2 \ll\mathbb{P}_1 \quad on \quad (b_1 (\mathcal{U}_1\cap \mathcal{U}_2), \mathcal{F}_1 \vert_{b_1(\mathcal{U}_1\cap \mathcal{U}_2)})\]\\
iii) If $ \lambda$ and $\mu$ are the densities: $(b_{12})_* \mathbb{P}_1/\mathbb{P}_2$ and $(b_{21})_* \mathbb{P}_2/\mathbb{P}_1$, one as: \\
$\lambda\times (\mu\circ b_{21})>0, \mathbb{P}_2$-a.s and $\mu\times (\lambda\circ b_{12})>0, \mathbb{P}_1$-a.s.
\end{lem}

\pf: i) Let $\varphi_n \in \mathbb{D}_r^\infty(\Omega_2)$ and $\psi \in \mathbb{D}_r^\infty(\Omega_1)$ such that $d_B([\varphi_n]_B,0)\rightarrow 0$ and $d_A([\widetilde{\varphi_n\circ b_{21}\vert_A}]_A,[\psi]_A)\rightarrow 0$. This implies: \[\inf_{\beta\in\hat{B}} \| \varphi_n+\beta \|_{L^1(\Omega_2)}\rightarrow 0\] 
and: 
\[\inf_{\alpha\in\hat{A}} \| \widetilde{\varphi_n\circ b_{21}\vert_A}+\alpha-\psi\|_{L^1(\Omega_1)}\rightarrow 0.\] 
Then, by sequences extractions, we get $\psi \vert_A =0, \mathbb{P}_1$-a.s.\\
ii) If $Z\in \mathcal{F}_2 \vert_{b_2(\mathcal{U}_1\cap \mathcal{U}_2)}$ such that: $\mathbb{P}_2(Z)=0 \Rightarrow b_{12}^{-1}(Z)\in \mathcal{F}_1\vert_{b_1(\mathcal{U}_1\cap \mathcal{U}_2)}$ and $\mathbb{P}_1(b_{12}^{-1}(Z))=0$; Definition 1.2, ii, implies $(b_{12})_* \mathbb{P}_1 \ll \mathbb{P}_2$.\\
iii) On $b_1(\mathcal{U}_1\cap \mathcal{U}_2)$ and $b_2(\mathcal{U}_1\cap \mathcal{U}_2)$: 
\[b_{12}\circ b_{21}=\textup{Id}_{b_2(\mathcal{U}_1\cap \mathcal{U}_2)} \textup{ and } b_{21}\circ b_{12}=\textup{Id}_{b_1(\mathcal{U}_1\cap \mathcal{U}_2)}\]

\hspace{-8mm}{\bf Tangent linear map to a chart change}\\

The notations are the same as in the beginning of this chapter.
\begin{thm}
Given two $\mathbb{D}_r^\infty$-compatible charts, $(\mathcal{U}_i, b_i, \Omega_i, \mathcal{F}_i,\mathbb{P}_i,H_i)$, \\$i=1,2$, there exists $\mathbb{P}_1$-a.s on $b_1(\mathcal{U}_1 \cap \mathcal{U}_2)$, a bounded linear map, denoted $T_{12}(\omega)$, with $\mathbb{P}_1$-a.s: $T_{12}(\omega)\in \mathcal{L}(H_1,H_2)$, which verifies: $\forall A\in \mathcal{F}_1 \vert_{b_1(\mathcal{U}_1 \cap \mathcal{U}_2)}$,\\$\exists A_1\in \mathcal{F}_1 \vert_{b_1(\mathcal{U}_1 \cap \mathcal{U}_2)}, A_1 \subset A$, and $\mathbb{P}_1(A_1)>0$ such that $\forall u \in H_1, \forall v\in H_2$, $\mathbb{P}_1$-a.s on $A_1$:
\begin{equation}
<T_{12}(\omega)u,v>_{H_2}\circ b_{12}=<u,\textup{grad }[\widetilde{W(v)\circ b_{12}\vert_{A_1}}]>_{H_1}
\end{equation}
\end{thm}

Formula (1) dose not have any ambiguity: if there is $B\in\mathcal{F}_1\vert_{b_1(\mathcal{U}_1 \cap \mathcal{U}_2)}$ with $\mathbb{P}_1(A_1 \cap B)>0$ such that (1) is verified with $B$ instead of $A_1$, then: $(W(v)\circ b_{12}\vert_{A_1})\vert_{B}=W(v)\circ b_{12}\vert_{A_1 \cap B}$ which implies:
\[ \textup{grad }[\widetilde{W(v)\circ b_{12}\vert_{A_1}}]\vert_{A_1 \cap B}=\textup{grad }[\widetilde{W(v)\circ b_{12}\vert_{B}}]\vert_{A_1 \cap B} \]

For the demonstration of Theorem 1.1, we need the following lemmas:
\begin{lem} 
Let $(\Omega,\mathcal{F}, \mathbb{P}, H)$ a Gaussian space, $u\in H$
\[ I_a=[a,+\infty[ \subset \mathbb{R} \textup{ and } A=\lbrace  \omega\in\Omega/W(u)(\omega)\in I_a\rbrace\]
Then these exist sequences of functions $\varphi_n, \psi_n\in \mathbb{D}^\infty(\Omega)$ such that: \[\|\varphi_n\|_\infty \leqslant 1, \|\psi_n\|_\infty \leqslant 1, \varphi_n.\psi_n=0\] and 
\[\lim_n \varphi_n=\mathds{1}_A, \lim_n \psi_n=\mathds{1}_{CA}\]
\end{lem}
\pf: obvious.

\begin{lem} 
The set of $A\in \mathcal{F}$ such that there exist sequences of functions $\varphi_n\in \mathbb{D}^\infty(\Omega)$ and $\psi_n\in\mathbb{D}^\infty(\Omega)$ with $\| \varphi_n \|_\infty \leqslant 1, \| \psi_n \|_\infty \leqslant 1, \varphi_n\times\psi_n=0$
and $\lim_n \varphi_n=\mathds{1}_A, \lim_n \psi_n=\mathds{1}_{CA}$, is a $\sigma$-field, equals to $\mathcal{F}$.
\end{lem}
\pf: We use the monotone class theorem; let $\tilde{\mathcal{F}}$ be the set of $A\in \mathcal{F}$ verifying the above properties. Then:\\
-$\phi$ and $\Omega\in \tilde{\mathcal{F}}$\\
-$\tilde{\mathcal{F}}$ is stable by complementation\\
-$\tilde{\mathcal{F}}$ is stable for finite intersections: $A_1, A_2\in \tilde{\mathcal{F}}$. \\ 
If $\varphi_n^{(1)}\rightarrow\mathds{1}_{A_1},\varphi_n^{(2)}\rightarrow\mathds{1}_{A_2}, \psi_n^{(1)}\rightarrow\mathds{1}_{CA_1},\psi_n^{(2)} \rightarrow \mathds{1}_{CA_2}$ with the above properties:
\[\theta_n=\varphi_n^{(1)}\psi_n^{(2)}+\varphi_n^{(2)}\psi_n^{(1)}+\psi_n^{(1)}\psi_n^{(2)}\in \mathbb{D}^\infty \quad(\mathbb{D}^\infty \textup{algebra})\]
\[ \theta_n\rightarrow \mathds{1}_{C(A_1\cap A_2)},\quad \varphi_n^{(1)}\varphi_n^{(2)}\rightarrow\mathds{1}_{A_1\cap A_2} \]
and $\| \theta_n \|_\infty \leqslant 1$ (check on the supports). \\

What is left to show is the stability of $\tilde{\mathcal{F}}$ for increasing sequences of items in $\tilde{\mathcal{F}}$.\\
Let $A_0\subset A_1\subset \ldots\subset A_k\subset\ldots$ an increasing sequence of items in $\tilde{\mathcal{F}}$; $A=\cup_{\substack{k}} A_k$.\\
$\forall k\in\mathbb{N}$ there exist sequences of functions $\varphi_n^{(k)}$ and $\psi_n^{(k)}$ with the related properties because $A_k\in \tilde{\mathcal{F}}$. Then with the dominated convergence theorem: 
\[\varphi_n^{(k)} \xrightarrow{L^p} \mathds{1}_{A_k} \textup{ and } \psi_n^{(k)} \xrightarrow{L^p} \mathds{1}_{CA_k} \quad (p\geqslant 1)\]
Then by extracting diagonal sequences, we get a sequence $\tilde{\varphi}_m \xrightarrow{L^p} \mathds{1}_{A}$ and a sequence $\tilde{\psi}_m \xrightarrow{L^p} \mathds{1}_{CA}$. Then we extract from $(\tilde{\varphi}_m)_m$ and $(\tilde{\psi}_m)_m$ two sequences which converges $\mathbb{P}$-a.s towards $\mathds{1}_A$ and $\mathds{1}_{CA}$.

\begin{lem}
Let $Z_1,\ldots,Z_k$ a finite partition of $\Omega$, $Z_l\in\mathcal{F}$. Then there exists, $\forall l\in \lbrace 1,\ldots,k \rbrace$, a sequence of functions $\varphi_n^{(l)}$ such that: 
\[ \varphi_n^{(l)} \in \mathbb{D}^\infty (\Omega), \| \varphi_n^{(l)} \|_\infty \leqslant 1, \forall l\neq l^\prime: \varphi_n^{(l)}\cdot \varphi_n^{(l^\prime)}=1, \lim_n \varphi_n^{(l)}=\mathds{1}_{Z_l}\]
\end{lem}
\pf: $\forall l$: there exist $\varphi_n^{(l)}$ and $\psi_n^{(l)}$ as in Lemma 1.2.  $u_n^{(l)}=\varphi_n^{(l)}\cdot \prod^k_{j\neq l} \psi_n^{(j)}$ satisfies the Lemma 1.5.\\

Proof of Theorem 1.1:
Let $A\in \mathcal{F}_1 \vert_{b_1(\mathcal{U}_1\cap \mathcal{U}_2)}$, and $\mathbb{P}_1(A)>0$. There exists $A_1\subset A, A_1\in \mathcal{F}_1 \vert_{b_1(\mathcal{U}_1\cap \mathcal{U}_2)}$ and $\mathbb{P}_1(A_1)>0$ such that: $\forall \varphi\in \mathbb{D}^\infty_r(\Omega_2), \varphi\circ b_{12} \vert_{A_1}$ admits an extension; $\widetilde{\varphi\circ b_{12}\vert_{A_1}} \in \mathbb{D}_r^\infty(\Omega_1)$. Let $B_1=b_{12}(A_1)$ and $F_{B_1}$ the continuous morphism:
\[ \mathbb{D}^\infty_r(\Omega_2)/_{\stackrel{B_1}{\sim}} \rightarrow \mathbb{D}^\infty_r(\Omega_1)/_{\stackrel{A_1}{\sim}} \]
Then: $\forall p>1 \quad\exists q>1$ and: $\exists C(p,q)>0$ such that: $\forall \varphi \in \mathbb{D}^\infty_r(\Omega_2)$:
\begin{equation}
\| [\widetilde{\varphi\circ b_{12}\vert_{A_1}}]_{A_1}\|_{\mathbb{D}^p_r(\Omega_1)/_{\stackrel{A_1}{\sim}}}\leqslant  C(p,q) \|[\varphi]_{B_1}\|_{\mathbb{D}^q_r(\Omega_2)/_{\stackrel{B_1}{\sim}}}
\end{equation}
Remind: $\hat{A_1}=\lbrace  \alpha\in\mathbb{D}^\infty_r(\Omega_1)/\alpha\vert_{A_1}=0\rbrace$.\\
(2) implies:
\begin{equation*}
\hspace{25.5mm} \inf_{\alpha\in \hat{A}_1} \| \widetilde{\varphi\circ b_{12}\vert_{A_1}} +\alpha \|_{\mathbb{D}^p_r(\Omega_1)}\leqslant C(p,q)\|\varphi\|_{\mathbb{D}^q_r(\Omega_2)} \hspace{25.5mm} (2')
\end{equation*}
Now $i=\sqrt{-1}$ as usual. Let $\varphi^{(l)}\in \mathbb{D}^\infty(\Omega_2), l=1,\ldots,k$ and denote 
\[\varphi=\sum^k_{l=1}\varphi^{(l)}\times \frac{e^{imW(e_l)}}{m^r},\] 
$(e_l)_{l\in\mathbb{N}}$ being a base of $H_2$; then ($2^\prime$) becomes:
\begin{align}
\inf_{\alpha\in\hat{A}_1} &\left\Arrowvert \sum^k_{l=1} \frac{\widetilde{\varphi^{(l)}\circ b_{12}\vert_{A_1}}\cdot e^{im\widetilde{W(e_l)\circ b_{12}\vert_{A_1}}}}{m^r}+\frac{\alpha}{m^r}\right\Arrowvert_{\mathbb{D}_r^p(\Omega_1)} \nonumber \\
&\leqslant C(p,q)\left\Arrowvert \sum^k_{l=1} \varphi^{(l)} \frac{e^{imW(e_l)}}{m^r}\right\Arrowvert_{\mathbb{D}_r^q(\Omega_2)} 
\end{align}
After having computed all derivations in (3), and letting $m\rightarrow \infty$, we get:
\begin{align}
&\left\Arrowvert \sum^k_{l=1} \widetilde{\varphi^{(l)}\circ b_{12}\vert_{A_1}} (\stackrel{r}{\otimes} \textup{grad }[\widetilde{W(e_l)\circ b_{12}\vert_{A_1}}])\right\Arrowvert_{L^p(\Omega_1, \stackrel{r}{\otimes}H_1)} \nonumber \\
&\leqslant C(p,q)  \left\Arrowvert \sum^k_{l=1} \varphi^{(l)}(\omega)(\stackrel{r}{\otimes} e_l)\right\Arrowvert_{{L^q(\Omega_2, \stackrel{r}{\otimes}H_2)}}
\end{align}
Let $Z_1,\ldots,Z_k$ be a partition of $B_1=b_{12}(A_1)$; then $b_{12}^{-1}(Z_1),\ldots,b_{12}^{-1}(Z_k)$ is a partition of $A_1=b_{12}^{-1}(B_1)$. We choose for $\varphi_n^{(l)}$ is sequence of functions converging punctually towards $\mathds{1}_{Z_l}$ as in Lemma 1.3:\\
then: (4) becomes
\begin{align}
&\left\|\sum^k_{l=1} \mathds{1}_{Z_l} \circ b_{12}(\cdot)\stackrel{r}{\otimes} \textup{grad }[\widetilde{W(e_l)\circ b_{12}\vert_{A_1}}])\right\| _{L^p(\Omega_1, \stackrel{r}{\otimes}H_1)} \nonumber \\
&\leqslant C(p,q)  \left\| \sum^k_{l=1} \mathds{1}_{Z_l}(\omega)(\stackrel{r}{\otimes} e_l)\right\| _{L^q(\Omega_2, \stackrel{r}{\otimes}H_2)}\leqslant C(p,q)
\end{align}
As $Z_l, l=1,\ldots,k$ is a partition of $B_1=b_{12}(A_1)$ \\
(5) becomes:
\begin{equation}
\int\sum^k_{l=1} \mathds{1}_{Z_l}\circ b_{12} \cdot \| \textup{grad }[\widetilde{W(e_l)\circ b_{12}\vert_{A_1}}] \|_{H_1}^{pr} \mathbb{P}(d\omega)\leqslant C(p,q)^p
\end{equation}
As partition of $B_1=b_{12}(A_1)$,\\
we choose:
\begin{align*}
Z_l=b_{12} \lbrace \omega\in\Omega_1 / \forall j<l: &\| \textup{grad }[\widetilde{W(e_l)\circ b_{12}\vert_{A_1}}] \|_{H_1} > \|\textup{grad }[\widetilde{W(e_j)\circ b_{12}\vert_{A_1}}] \|_{H_1} \\ \textup{and } \forall j<l: &\| \textup{grad }[\widetilde{W(e_l)\circ b_{12}\vert_{A_1}}] \|_{H_1} \geqslant \|\textup{grad }[\widetilde{W(e_j)\circ b_{12}\vert_{A_1}}] \|_{H_1}\rbrace 
\end{align*}
Then (6) becomes:
\begin{equation}
\left\| \sup_{l\in \lbrace 1,\ldots,k \rbrace} \| \textup{grad }[\widetilde{W(e_l)\circ b_{12}\vert_{A_1}}] \|^r_{H_1} \right\|_{L^p(A_1)}\leqslant C(p,q)
\end{equation}
As (7) is valid for each subset of the Hilbertian basis of $H$, $(e_l)_{l\in \mathbb{N}_*}$, we deduce that the map: 
\[H_1 \ni e_l \rightarrow \|\textup{grad }[\widetilde{W(e_l)\circ b_{12}\vert_{A_1}}] \|_{H_1}\] 
is $L^p$-bounded, $\mathbb{P}_1$-a.s on $A_1$, uniformly relatively to $l\in\mathbb{N}_*$. \\
So the linear map $T_{A_1}$ defined by: $\forall u\in H_1, v\in H_2$:
\[<T_{A_1}(\omega)u,v>_{H_2}\circ b_{12}=<u, \textup{grad }[\widetilde{W(v)\circ b_{12}\vert_{A_1}}]>_{H_1},\]
is a bounded linear map from $H_1$ in $H_2$, $\mathbb{P}_1$-a.s on $A_1$. We can, by exhaustion, find a countable sequence of subsets of $b_1(\mathcal{U}_1\cap \mathcal{U}_2)$, denoted $A_i, i\in \mathbb{N}_*$ with $\forall i: A_i\in \mathcal{F}_1\vert_{b_1(\mathcal{U}_1\cap \mathcal{U}_2)}, \mathbb{P}_1(A_i)>0$, such that $\cup_{i\in\mathbb{N}_*}A_i=b_1(\mathcal{U}_1\cap \mathcal{U}_2)$.\\

On such each $A_i$, there exists $\mathbb{P}_1$-a.s a linear bounded operator denoted $T_{A_i}$ such that:$\forall u\in H_1, \forall v\in H_2$:
\[ <T_{A_i}(u),v>_{H_2}=<u, \textup{grad }[\widetilde{W(v)\circ b_{12}\vert_{A_i}}]>_{H_1}\circ b_{21} \]
This last equation shows that on $A_i \cap A_j, T_{A_i}=T_{A_j},\mathbb{P}_1$-a.s; so there exists a linear bounded operator from $H_1$ in $H_2$, such that: $\mathbb{P}_1$-a.s on  $b_1(\mathcal{U}_1\cap \mathcal{U}_2)$: $\forall \varphi\in \mathbb{D}^\infty_r(\Omega_2)$ and $ \forall u\in H_1$, denoted $T_{12}$:
\[ <T_{12}(\omega)\cdot u,\textup{grad }\varphi>_{H_2}=<u, \textup{grad }[\widetilde{\varphi\circ b_{12}\vert_{A_1}}]>_{H_1}\circ b_{21} \]
It is easy to show the transitivity of these $T_{ij}$: if we have three compatible charts $(\mathcal{U}_i, b_i, \Omega_i, \mathcal{F}_i,\mathbb{P}_i,H_i), i=1,2,3$ such that $\mathbb{P}_i[b_i(\mathcal{U}_1\cap \mathcal{U}_2\cap \mathcal{U}_3)]>0$, then $T_{ij}\circ T_{jk}=T_{ik}$, a.s. on $b_i(\mathcal{U}_1\cap \mathcal{U}_2\cap \mathcal{U}_3)$.\\

Last, this bounded linear map $T_{12}$, is measurable from $(b_1(\mathcal{U}_1\cap \mathcal{U}_2),\mathcal{F}_1\vert_{b_1(\mathcal{U}_1\cap \mathcal{U}_2)})$ in $\mathcal{L}(H_1,H_2)$ and is $L^{\infty-0}(b_1(\mathcal{U}_1\cap \mathcal{U}_2))$: the constant $C(p,q)$ in (7) might be dependent of $A_1$. 

\section{\huge Preliminaries}

$(\Omega, \F, \P, H)$ being a Gaussian space:
Here we state some definitions of mathematical tools that will be needed, and some of their properties, and prove:

\begin{enumerate}
\renewcommand{\labelenumi}{\roman{enumi})}
\item some theorems about the existence of unique continuous linear extensions of continuous linear operators from $\dinf(\Omega)$ in $\dinf(\Omega)$
\item $\sko_{\infty}^2(\Omega) \cap L^{\infty - 0}(\Omega) = \dinf(\Omega)$
\item Any continuous derivation on $\dinf(\Omega)$ is a strong limit of $\dinf$-vector fields, and some properties of these continuous derivations.
\end{enumerate}

Whenever no particular setting is specified, it is assumed that we deal with a Gaussian space $(\Omega, \F, \P, H)$.

The theorems proved here will be needed for the development, but the reader can also go directly to Section 3 and use the results of this section, which will be referred to when they appear.

On $\sko_r^p(\Omega)$, the two following norms are equivalent:
\begin{align*}
f \in \drp(\Omega): \| f \|_{\drp}^{(1)} &= \left[ \sum_{j=0}^{r} \int \|\Grad ^j f \|_{\overset{j}{\otimes}H}^p \P(d\omega) \right]^{\frac{1}{p}} \\
\text{and } \| f \|_{\drp}^{(2)} &= \sum_{j=0}^{r} \left( \int \|\Grad ^j f \|_{\overset{j}{\otimes}H}^p \P(d\omega) \right)^{\frac{1}{p}} \text{ (Malliavin)}
\end{align*}

We have: $\| f \|_{\drp(\Omega)}^{(1)} \leq \| f \|_{\drp(\Omega)}^{(2)} \leq r^{1-\frac{1}{p}} \| f \|_{\drp(\Omega)}^{(1)}$

\begin{notation} If $(\Omega_i, \F_i, \P_i, H_i)$ are two Gaussian spaces $(i=1,2)$, we denote by $K^{\infty}(\Omega_1 \times \Omega_2)$:
\begin{align*}
K^{\infty}(\Omega_1 \times \Omega_2) = \left\{ \sum_{j \in J} \alpha_j(\omega_1) \beta_j(\omega_2) | J \text{ finite }, \alpha_j(\omega_1) \in \dinf(\Omega_1), \beta_j(\omega_2) \in \dinf(\Omega_2) \right\}
\end{align*}

Then $K^{\infty}(\Omega_1 \times \Omega_2)$ is a $\dinf$-dense subset of $\dinf(\Omega_1, \Omega_2)$.
\end{notation}

\subsection{Some extensions of continuous linear maps}

\begin{defn}\label{def2_1}
\begin{enumerate}
\renewcommand{\labelenumi}{\roman{enumi})}
\item A subset $D \subset \dinf(\Omega)$ will be said to be 

$\dinf$-bounded iff: 

$\forall (p, r), p>1, r \in \N$, $\exists$ a constant $C(p,r)$ such that:
\begin{align*}
\sup_{f \in D} \| f \|_{\drp} \leq C(p, r)
\end{align*}

\item a process $\varphi(t, \omega): [0,1] \times \Omega \rightarrow \R$ will be said to be $\dinf$-bounded iff: 

$\forall (p, r), p>1, r \in \N_\star$, $\exists$ a constant $C(p,r)$ such that:
\begin{align*}
\sup_{t \in [0,1]} \| \varphi \|_{\drp(\Omega)} \leq C(p, r)
\end{align*}
\end{enumerate}

\end{defn}

\begin{thm}\label{thm2_1}
i) Let $(\Omega, \F, \P)$ and $(\Omega_1, \F_1, \P_1)$ be two probability spaces, and $T$ a continuous linear operator from $L^q(\Omega, H_1)$ in $L^p(\Omega, H_2)$ with $q \geq p$, $H_1$ and $H_2$ being two abstract Hilbert spaces.

Denote $\tilde{T}$ the linear operator defined on:
\begin{align*}
K_{q,p}(\Omega \times \Omega_1) = \left\{ \sum_{j \in J} \alpha_j(\omega)\beta_j(\omega_1) | J \text{ finite}, \alpha_j \in L^q(\Omega, H_1), \beta_j \in L^q(\Omega_1, \R) \right\}
\end{align*}

by $\tilde{T} \left( \sum_{j \in J} \alpha_j(.) \beta_j(.) \right) = \sum_{j \in J} \beta_j(\omega_2)(T \alpha_j)(\omega_1)$.

Then there is a unique linear continuous extension of $\tilde{T}$, denoted $\tilde{T}$, from $L^q(\Omega \times \Omega_1, H_1)$ to $L^p(\Omega \times \Omega_1, H_2)$.

ii) If $T_k$ is a sequence of continuous linear operators from $L^q(\Omega, H_1)$ to $L^p(\Omega, H_2)$, $k$-uniformly continuous, then the sequence $\tilde{T}^k$ is $k$-uniformly continuous from $L^q(\Omega \times \Omega_1, H_1)$ to $L^p(\Omega \times \Omega_1, H_2)$.
\end{thm}

\begin{proof}
i) Let $\omega \in \Omega$, $\omega_1 \in \Omega_1$:

\begin{align*}
\left\| \tilde{T} \left( \sum_{j \in J} \alpha_j(.) \beta_j(.) \right) \right\|_{L^p(\Omega \times \Omega_1, H_2)}^p
&= \int \P(d\omega_1) \int \P(d\omega) \left\| T \left( \sum_{j \in J} \alpha_j(.) \beta_j(\omega_1) \right) \right\|_{H_2}^p \\
&= \int \P(d\omega_1) \left\| T \left( \sum_{j \in J} \alpha_j(.) \beta_j(\omega_1) \right) \right\|_{L^p(\Omega, H_2)}^p \\
&\leq \| T \|^p \int \P(d\omega_1) \left\| \sum_{j \in J} \alpha_j(.) \beta_j(\omega_1)  \right\|_{L^q(\Omega, H_1)}^p \\
&= \| T \|^p \left\| \left\| \sum_{j \in J} \alpha_j(.) \beta_j(\omega_1) \right\|_{L^q(\Omega, H_1)}^p \right\|_{L^1(\Omega_1, \R)}\\ 
&\leq \| T \|^p \left\| \left\| \sum_{j \in J} \alpha_j(.) \beta_j(\omega_1) \right\|_{L^q(\Omega, H_1)}^p \right\|_{L^{q/p}(\Omega_1, \R)} \text{(by Hölder inequality)} \\
&=\| T \|^p \left\| \sum_{j \in J} \alpha_j(.) \beta_j(.) \right\|_{L^{q}(\Omega \times \Omega_1, H_1)}^p
\end{align*}

ii) immediate from above.
\end{proof}

\begin{cor}\label{cor2_1}
If $T$ is a continuous linear operator of $L^{\infty - 0}(\Omega, H_1)$ in $L^{\infty - 0}(\Omega, H_2)$, then $\tilde{T}$ is a continuous linear operator from $L^{\infty - 0}(\Omega \times \Omega_1, H_1)$ to $L^{\infty - 0}(\Omega \times \Omega_1, H_2)$.
\end{cor}

\begin{thm}\label{thm2_2}
Let $(\Omega_i, \F_i, \P_i, H_i)_{i = 1,2}$ be two Gaussian spaces and $H_3$ and $H_4$ two abstract Hilbert spaces.
\begin{enumerate}
\renewcommand{\labelenumi}{\roman{enumi})}
\item If $T$ is a linear continuous operator from $\dinf(\Omega_1, H_3)$ in $\dinf(\Omega_1, H_4)$, there exists a unique linear extension $\tilde{T}$ of $T$, from $\dinf(\Omega_1 \times \Omega_2, H_3)$ in $\dinf(\Omega_1 \times \Omega_2, H_4)$.
\item If $T_k$ is a $k$-uniformly convergent sequence from $\dinf(\Omega_1, H_3)$ in $\dinf(\Omega_1, H_4)$, then the sequence $\tilde{T}_k$ is $k$-uniformly convergent from $\dinf(\Omega_1 \times \Omega_2, H_3)$ in $\dinf(\Omega_1 \times \Omega_2, H_4)$.
\end{enumerate}

\end{thm}

\begin{proof}

\begin{enumerate}
\renewcommand{\labelenumi}{\roman{enumi})}
\item to simplify, we suppose $H_3 = H_4 = \R$. Let $L_1$ and $L_2$ be the O.U. operators on $\dinf(\Omega_1)$ and $\dinf(\Omega_2)$, and $L$ the O.U. operator on $\dinf(\Omega_1 \times \Omega_2)$. $\tilde{L}_1$ is $\dinf$-continuous on $\dinf(\Omega_1 \times \Omega_2)$: if $r \in \N_\star$ and $\varphi$ in $K^{\infty}(\Omega_1 \times \Omega_2)$, we have $\tilde{L}_1 + \tilde{L}_2 = L$, so 

\begin{align*}
\| \tilde{L}_1 \varphi \|_{\drp(\Omega_1 \times \Omega_2)} 
&= \|(1-L)^{r/2} \tilde{L}_1 \varphi \|_{L^p(\Omega_1 \times \Omega_2)} \\
&= \| \tilde{L}_1 (1-L)^{r/2} \varphi \|_{L^p(\Omega_1 \times \Omega_2)} \\
&\leq C \| (1-L)^{r/2} \varphi \|_{L^q(\Omega_1 \times \Omega_2)} \\
&\leq C \| \varphi \|_{\sko_r^q(\Omega_1 \times \Omega_2)}
\end{align*}
 $C$ being a constant, $q > 1$.

Therefore, $\tilde{L}_2$ is $\dinf(\Omega_1 \times \Omega_2)$-continuous.

We will prove by induction that $\tilde{T}$ is $\dinf(\Omega_1 \times \Omega_2)$-continuous.

Corollary \ref{cor2_1} shows that $\tilde{T}$ is continuous from $\dinf(\Omega_1 \times \Omega_2)$ 

to $L^{\infty - 0}(\Omega_1 \times \Omega_2)$.

We suppose that $\tilde{T}$ is continuous from $\dinf(\Omega_1 \times \Omega_2)$ to $\dinf_r(\Omega_1 \times \Omega_2)$.

On $K^{\infty}(\Omega_1 \times \Omega_2)$, we have $\tilde{L}_2 \circ \tilde{T} = \tilde{T} \circ \tilde{L_2}$.

So using the induction hypothesis: $\exists (q, s), q>1, s \in \N_\star$ and a constant $C_1$ such that $\forall \varphi \in K^{\infty}(\Omega_1 \times \Omega_2)$:
\begin{align*}
\| ( \tilde{L}_2 \circ \tilde{T} ) \varphi \|_{\drp} = \| ( \tilde{T} \circ \tilde{L_2} ) \varphi \|_{\drp} \leq C_1 \| \varphi \|_{\sko_s^q}
\end{align*}

Again with the induction hypothesis on $r$, there exists
 
$(q', s'), q>1, s' \in \N_\star$ and a constant $C_2$ such that:

\begin{align*}
\| ((1-\tilde{L}_1) \circ \tilde{T}) \varphi \|_{\drp} \leq C_2 \| \varphi \|_{\sko_{s'}^{q'}}
\end{align*}

Then: $\| ((1-L) \circ \tilde{T}) \varphi \|_{\drp} \leq (C_1 \vee C_2) \| \varphi \|_{s \vee s'}^{q \vee q'} $.

So $(1-L) \circ \tilde{T}$ is continuous from $K^{\infty}(\Omega_1 \times \Omega_2)$ in $\drp(\Omega_1 \times \Omega_2)$. So $\tilde{T}$ is continuous from $K^{\infty}(\Omega_1 \times \Omega_2)$ to $\sko_{r+2}^p(\Omega_1 \times \Omega_2)$ and then from $\dinf(\Omega_1 \times \Omega_2)$ to $\dinf(\Omega_1 \times \Omega_2)$.
\end{enumerate}

\end{proof}

\begin{thm}\label{thm2_3}
Let $(\Omega, \F, \P, H)$ be a Gaussian space, $\tilde{H}$ an abstract Hilbert space, $(e_i)_{i \in \N_{\star}}$ an Hilbertian basis of $\tilde{H}$, and $T$ a linear continuous operator from $L^p(\Omega)$ to $L^q(\Omega)$.

Denote $\tilde{T}$ the linear operator defined by the serie: 

if $X(\omega) = \sum_{i=1}^{\infty} f_i(\omega)e_i, X \in \dinf(\Omega, \tilde{H}), f_i(\omega) \in \dinf(\Omega)$ 

then $\tilde{T}(X) = \sum_{i=1}^{\infty}(T f_i)(\omega)e_i$.

This definition of $\tilde{T}$ is meaningful, and $\tilde{T}$ is a continuous linear operator from $L^p({\Omega, \tilde{H}})$ to $L^q({\Omega, \tilde{H}})$.
\end{thm}

\begin{proof}
Let $(\Omega_1, \F_1, \P_1, H_1)$ be a Gaussian space, independent of $(\Omega, \F, \P, H)$, and $(e_j)_{j \in \N_{\star}}$ an Hilbertian basis of $H_1$.

Denote $Y_j = W(e_j)$; then the random variable $\sum_{i=1}^{\infty} f_i(\omega) Y_i(\omega_1)$ is correctly defined because its $L^p(\Omega \times \Omega_1)$ norm is equivalent to the $L^p(\Omega, \tilde{H})$ norm of X:

\begin{align*}
\| \sum_{i=1}^{\infty} f_i(\omega) Y_i(\omega_1) \|_{L^p(\Omega \times \Omega_1)}^p = \int \P (d\omega) \int \P (d\omega_1) | \sum_{i=1}^{\infty} f_i(\omega) Y_i(\omega_1) |^p
\end{align*}

and $\sum_{i=1}^{\infty} f_i(\omega) Y_i(\omega_1)$ is a Gaussian variable with law $N(0, \sqrt{ \sum_{i=1}^{\infty} |f_i(\omega)|^2})$, for almost all $\omega \in \Omega$.

Similarly, we have: $\| \tilde{T}(X) \|_{L^q({\Omega, \tilde{H}})} \sim \| \sum_{i=1}^{\infty} (T f_i)(\omega) Y_i(\omega_1) \|_{L^q{(\Omega \times \Omega_1)}}$.

To prove that $\tilde{T}$ is continuous, it is enough to show the inequality; $C$ being a constant:

\begin{align*}
\| \sum_{i=1}^{\infty} (T f_i)(.) Y_i(.) \|_{L^q{(\Omega \times \Omega_1)}} \leq C \| \sum_{i=1}^{\infty} f_i(.) Y_i(.) \|_{L^q{(\Omega \times \Omega_1)}}
\end{align*}

As $\sum_{i=1}^{\infty} f_i(\omega) Y_i(\omega_1) \in L^p(\Omega \times \Omega_1)$ we have $\P(d\omega_1)$-a.s.: 

$\sum_{i=1}^{\infty} f_i(\omega) Y_i(\omega_1) \in L^p(\Omega)$.

There are two cases to study:

a) $1 < q \leq p < + \infty$;
$T$ being continuous from $L^p(\Omega)$ in $L^q(\Omega)$, there exists a constant $C_0$ such that $\P(d\omega_1)$-a.s.:

\begin{align*}
\int \P(d\omega) | \sum_{i=1}^{\infty} (T f_i)(\omega) Y_i(\omega_1) |^q \leq C_0 \left[ \int | \sum_{i=1}^{\infty} f_i(\omega) Y_i(\omega_1) |^p \P(d\omega) \right]^{q/p}
\end{align*}

which implies the following inequalities:
\begin{align*}
\int \P(d\omega) \otimes \P(d\omega_1) | \sum_{i=1}^{\infty} (T f_i)(\omega) Y_i(\omega_1) |^q
&\leq C_0 \int \P(d\omega_1) \left[ \int | \sum_{i=1}^{\infty} f_i(\omega) Y_i(\omega_1) |^p \P(d\omega) \right]^{q/p} \\
&\leq C_0 \left[ \int \P(d\omega) \otimes \P(d\omega_1) | \sum_{i=1}^{\infty} f_i(\omega) Y_i(\omega_1) |^p \right]^{q/p} \text{ (H\"{o}lder)} 
\end{align*}

b) $1 < p < q < + \infty$
Let $r$ be such that $\frac{1}{q} + \frac{1}{r} = \frac{1}{p}$ and $g \in L^r(\Omega)$. We define $T_g f = g . Tf$. Then $T_g: L^p(\Omega) \rightarrow L^p(\Omega)$, and $\| T_g \|_{L^p} \leq \| T \| . \| g \|_{L^r}$. We apply case a) above to $T_g$, with $C_1 = \| T \|$:

\begin{align*}
\| \tilde{T}_g(X) \|_{L^p({\Omega, \tilde{H}})} \leq C_1 \|g\|_{L^r} \| X \|_{L^p({\Omega, \tilde{H}})}
\end{align*}

which implies:
\begin{align*}
\| g . \tilde{T}(X) \|_{L^p({\Omega, \tilde{H}})} \leq C_1 \|g\|_{L^r} \| X \|_{L^p({\Omega, \tilde{H}})}
\end{align*}

so: 
\setcounter{equation}{0}
\begin{align}
\left\| g . \left[ \sum_{i=1}^{\infty} (T f_i)(.)^2 \right]^{\frac{1}{2}}\right\|_{L^p{(\Omega)}} \leq C_1 \|g\|_{L^r} \| X \|_{L^p({\Omega, \tilde{H}})} \label{eq2_1}
\end{align} 

And (\ref{eq2_1}) is valid for all $g \in L^r(\Omega)$.

Let $p'$ and $q'$ be such that $\frac{1}{p} + \frac{1}{p'} = 1$ and $\frac{1}{q} + \frac{1}{q'} = 1$; $\forall h \in L^{q'}, h^{q'/r} \in L^r(\Omega)$, and $\| h^{q'/r} \|_{L^r} = \| h \|_{L^{q'}}^{q'/r}$.

With (\ref{eq2_1}) we have: 

\begin{align*}
\left\| h^{q'/r} \left[ \sum_{i=1}^{\infty} |T f_i|^2 \right]^{\frac{1}{2}} \right\|_{L^p} 
\leq C_1 \| h \|_{L^{q'}}^{q'/r} . \| X \|_{L^p({\Omega, \tilde{H}})}
\end{align*}

which implies:

\begin{align*}
\left\| h^{(q'/r + q'/p')} \left[ \sum_{i=1}^{\infty} |T f_i|^2 \right]^{\frac{1}{2}} \right\|_{L^1} 
\leq C_1 \| h \|_{L^{q'}}^{(q'/r + q'/p')} . \| X \|_{L^p({\Omega, \tilde{H}})}
\end{align*}
 
But $\frac{q'}{r} + \frac{q'}{p'} = q'(1 - \frac{1}{q}) = 1$. So:
\begin{align*}
\| h \| \tilde{T} X \|_H \|_{L^1(\Omega)} 
\leq C_1 \| h \|_{L^{q'}} . \| X \|_{L^p({\Omega, \tilde{H}})}
\end{align*}

which implies that $\tilde{T}$ is continuous from $L^p({\Omega, \tilde{H}})$ to $L^q({\Omega, \tilde{H}})$.
\end{proof}

\begin{cor}\label{cor2_2}
Let $T$ be a continuous linear operator from $\drp(\Omega)$ in $\sko_s^q(\Omega)$, and $\tilde{H}$ be an abstract Hilbert space. Then $\tilde{T}$ is a continuous linear operator from $\drp({\Omega, \tilde{H}})$ to $\sko_s^q({\Omega, \tilde{H}})$.
\end{cor}

\begin{proof}

We denote by $T'$ the continuous linear operator, with which the following diagram is commutative: 

\[
\begin{tikzcd}
\drp (\Omega) \arrow{r}{T} & \dqs (\Omega) \arrow{d}{(Id - L)^\frac{s}{2}} \\
L^p(\Omega) \arrow{u}{(Id - L)^{-\frac{r}{2}}} \arrow{r}{T'} & L^q (\Omega)
\end{tikzcd}
\]

Then we denote by $(Id - L)_{\tilde{H}}$ the O.U. operator on $\dinf(\Omega, \tilde{H})$. Then we use Theorem 2.3 to define $\tilde{T}: L^p({\Omega, \tilde{H}}) \rightarrow L^q({\Omega, \tilde{H}})$ and 

$\tilde{T} = (Id - L)_{\tilde{H}}^{-s/2} \circ \tilde{T}' \circ (Id - L)_{\tilde{H}}^{r/2}$.

\end{proof}

Now another extension theorem:
$H_1, H_2, H'$ being three abstract Hilbert spaces and $(e_i)_{i \in \N_{\star}}$ an Hilbertian basis of $H'$. Let $T$ be a continuous linear operator from $L^p(\Omega, H_1)$ in $L^q(\Omega, H_2)$. On the subset of $L^p(\Omega, H_1 \otimes H')$, with $J$ finite, $J \subset \N_{\star}$ defined by $\left\{ \sum_{j \in J} X_j \otimes e_j \big/ X_j \in L^p(\Omega, H_1) \right\}$, we define an operator $\tilde{T}$ by: $\tilde{T} \left( \sum_{j \in J} X_j \otimes e_j \right) = \sum_{j \in J} T X_j \otimes e_j$; and we have:

\begin{thm}\label{thm2_4}
If $p \geq q$, there exists an unique extension of $\tilde{T}$, which is continuous linear, from $L^p(\Omega, H_1 \otimes H')$ in $L^q(\Omega, H_2 \otimes H')$.
\end{thm}

\begin{proof}
We will first prove that :
\begin{align*}
\left\| \left( \sum_{j \in J} X_j \otimes e_j \right) \right\|_{L^p(\Omega, H_1 \otimes H')}
\simeq \left\| \left( \sum_{j \in J} W(e_j) X_j \right) \right\|_{L^p(\Omega \times \Omega_1, H_1)}
\end{align*}
$(\Omega_1, \F_1, \P_1, H')$ being a Gaussian space, independent of $(\Omega, \F, \P, H_1)$, but with its chaos $\mathcal{C}_1(\Omega_1)$ being generated by the $(W(e_i))_{i \in \N_{\star}}$.

Denote $U = \sum_{j \in J} X_j \otimes e_j$. Then with $\omega \in \Omega, \omega_1 \in \Omega_1$:

\begin{align*}
\| U \|_{L^p(\Omega, H_1 \otimes H')}
= \left[ \int \left( \sum_{j \in J} \| X_j \|_{H_1}^2 \right)^{\frac{p}{2}} \P(d\omega) \right]^{\frac{1}{p}}
= \left\| \left\| \sum_{j \in J} W(e_j)(\omega_1) X_j(\omega) \right\|_{L^2(\Omega_1, H_1)} \right\|_{L^p(\Omega)}
\end{align*}

Let $\omega$ be fixed and $(\Omega_2, \F_2, \P_2, H_1)$ be another Gaussian space, independent of the first two others, and whose chaos $\mathcal{C}_1(\Omega_2)$ includes the 

$W[X_j(\omega)](\omega_2), j \in J$.
Then:
\begin{align*}
\| U \|_{L^p(\Omega, H_1 \otimes H')}
= \left\| \left\| \sum_{j \in J} W(e_j)(\omega_1) W[X_j(\omega)](\omega_2) \right\|_{L^p(\Omega_1 \times \Omega_2)} \right\|_{L^p(\Omega)}
\end{align*}

$\omega$ being fixed, $\sum_{j \in J} W(e_j)(\omega_1) W[X_j(\omega)](\omega_2)$ is in $\mathcal{C}_2(\Omega_1 \times \Omega_2)$.

All $L^p$ norms being equivalent on this chaos $\mathcal{C}_2(\Omega_1 \times \Omega_2)$, we get:

\begin{align*} 
\| U \|_{L^p(\Omega, H_1 \otimes H')}
&\simeq \left\| \left\| \sum_{j \in J} W(e_j)(\omega_1) W[X_j(\omega)](\omega_2) \right\|_{L^2(\Omega_1 \times \Omega_2)} \right\|_{L^p(\Omega)} \\
&= \left\{ \int \P(d\omega) \otimes \P(d\omega_1) \otimes \P(d\omega_2) | \sum_{j \in J} W(e_j)(\omega_1) W[X_j(\omega)](\omega_2) |^p \right\}^{\frac{1}{p}}
\end{align*}

Now we fix $\omega$ and $\omega_1$: 

$\sum_{j \in J} W(e_j)(\omega_1) W[X_j(\omega)](\omega_2)$ is a Gaussian $\in \mathcal{C}_1(\Omega_2)$; as all $L^p$ norms are equivalent on $\mathcal{C}_1(\Omega_2)$,

we have: 

\begin{align} 
\| U \|_{L^p(\Omega, H_1 \otimes H')}
&\simeq \left\{ \int \P(d\omega) \otimes \P(d\omega_1) | \sum_{i,j \in J} W(e_i)(\omega_1) W(e_j)(\omega_1) \langle X_i, X_j \rangle_{H_1}(\omega) |^{\frac{p}{2}} \right\}^{\frac{1}{p}} \notag\\
&= \left\| \sum_{j \in J} W(e_j)(\omega_1) X_j(\omega) \right\|_{L^p(\Omega \times \Omega_1, H_1)} \label{eq2_2}
\end{align}

A similar computation proves that:
\begin{align}
\left\| \tilde{T} \left( \sum_{j \in J} X_j \otimes e_j \right) \right\|_{L^q(\Omega, H_1 \otimes H')}
\simeq \left\| \sum_{j \in J} W(e_j)(\omega_1)(T X_j)(\omega) \right\|_{L^q(\Omega \otimes \Omega_1, H_1)} \label{eq2_3}
\end{align}

Now we consider the operator $S$ defined on the subset of 

$L^p(\Omega \times \Omega_1, H_1)$: $\left\{ \sum_{j \in J} W(e_j) X_j | J \text{ finite } \subset \N_{\star} \right\}$, by:

\begin{align*}
S \left( \sum_{j \in J} W(e_j) X_j \right) = \sum_{j \in J} W(e_j)(\omega_1).(T X_j)(\omega)
\end{align*}

The same proof as in Theorem \ref{thm2_1}.i shows that there exists a constant $C>0$ such that:

\begin{align}
\left\| S \left( \sum_{j \in J} W(e_j) X_j \right) \right\|_{L^q(\Omega \otimes \Omega_1, H_2)} \leq C \left\|  \sum_{j \in J} W(e_j) X_j \right\|_{L^p(\Omega \otimes \Omega_1, H_1)} \label{eq2_4}
\end{align}

So (\ref{eq2_2}), (\ref{eq2_3}), (\ref{eq2_4}) imply:

\begin{align*}
\left\| \tilde{T} \left( \sum_{j \in J} X_j \otimes e_j \right) \right\|_{L^q(\Omega, H_2 \otimes H')}
&= \left\| \sum_{j \in J} T X_j \otimes e_j \right\|_{L^q(\Omega, H_2 \otimes H')} \\
\simeq \left\| \sum_{j \in J} W(e_j) T X_j \right\|_{L^q(\Omega \otimes \Omega_1, H_2)} &= \left\| S \left( \sum_{j \in J} W(e_j) X_j \right) \right\|_{L^q(\Omega \otimes \Omega_1, H_2)} \\
\leq C \left\|  \sum_{j \in J} W(e_j) X_j \right\|_{L^p(\Omega \otimes \Omega_1, H_2)} 
&\simeq C \left\| \sum_{j \in J} X_j \otimes e_j \right\|_{L^p(\Omega, H_1 \otimes H')} 
\end{align*}

Then $\tilde{T}\left(\sum_{j \in J} X_j \otimes Y_j \right)$ can be defined, using the decomposition of $Y_j$ on the basis $(e_i)_{i \in \N_{\star}}$ of $H'$.

Finally, we have an extension of $\tilde{T}$, continuous linear operator from 

$L^p(\Omega, H \otimes H')$ to $L^q(\Omega, H \otimes H')$.

This extension, denoted again $\tilde{T}$, does not depend on the Hilbertian basis of $H'$: if $B$ and $B'$ are two Hilbertian bases of $H'$, we have with obvious notations:

\begin{align*}
\tilde{T}_{(B)} \left( X_j \otimes Y_j \right) = \sum_{j \in J} T X_j \otimes Y_j = \tilde{T}_{(B')} \left( \sum_{j \in J} X_j \otimes Y_j \right)
\end{align*}

\end{proof}

\begin{cor}\label{cor2_3}
In the same setting as in Theorem \ref{thm2_4}, if $T$ is a continuous linear operator from $L^{\infty - 0}(\Omega, H_1)$ in $L^{\infty - 0}(\Omega, H_2)$, then $\tilde{T}$ is continuous linear from $L^{\infty - 0}(\Omega, H_1 \otimes H')$ in $L^{\infty - 0}(\Omega, H_2 \otimes H')$
\end{cor}

\begin{cor}\label{cor2_4}
In the same setting as in Theorem \ref{thm2_4}, if $T$ is a continuous linear operator from $\dinf(\Omega, H_1)$ in $\dinf(\Omega, H_2)$, then $\tilde{T}$ is continuous linear from $\dinf(\Omega, H_1 \otimes H')$ in $\dinf(\Omega, H_2 \otimes H')$.
\end{cor}

\begin{proof}
same proof as in Corollary \ref{cor2_2}.
\end{proof}

\subsection{$\sko_{\infty}^2 \cap L^{\infty - 0} = \dinf$}

\begin{thm}\label{thm2_5}
\begin{align*}
\sko_{\infty}^2 \cap L^{\infty - 0} = \dinf
\end{align*}
\end{thm}

\begin{proof}
$\dinf \subset \sko_{\infty}^2 \cap L^{\infty - 0}$.
For the reverse inclusion, we will need the Phragmen-Lindelof method [16].

$\forall z \in \C$ with $\Real z \geq 0$ and $\forall f \in \sko_{\infty}^2 \cap L^{\infty - 0}$, we define $(Id - L)^{-z} f$ by: 

$f_n$ being the component of $f$ in the chaos $\mathcal{C}_n$: $(Id - L)^{-z} f = \sum_{n=1}^{\infty} \frac{1}{(1 + n)^{z}} f_n$.
This definition is meaningful and coincides with the classic definition when $\Real z > 0$:
\begin{align*}
(Id - L)^{-z} f = \frac{1}{\Gamma(z)} \int_0^{\infty} e^{-t} t^{z-1} P_t f dt
\end{align*}

If $r > 0, 1 > \epsilon > 0$ fixed, we denote by $g: g = (1-L)^r f$, then: $\forall t \geq 0: (1-L)^{-(r + \epsilon + it)} g \in L^{\infty - 0}$, because $\forall \psi \in L^{q_0}, q_0 > 1$, and $\| \psi \|_{L^{q_0}} \leq 1$, we have: 

\begin{align}
\left| \int \P(d \omega) \psi . (Id - L)^{-(r + \epsilon + it)} g \right| 
&= \left| \int \P(d \omega) \psi . (Id - L)^{-(\epsilon + it)} f \right| \notag\\
&= \frac{1}{ | \Gamma(\epsilon + it) | } \left| \int \P(d \omega) \psi \int_0^{\infty} e^{-s} s^{\epsilon + it - 1} (P_s f) ds \right| \notag\\
&\leq \frac{1}{ | \Gamma(\epsilon + it) | } \int_0^{\infty} ds ~ e^{-s} s^{\epsilon - 1} \left| \int \P(d\omega) \psi (P_s f) \right| \notag\\
&\leq \frac{\Gamma(\epsilon)}{ | \Gamma(\epsilon + it) | } \| f \|_{L^{p_0}} \| \psi \|_{L^{q_0}} \text{ with } \frac{1}{p_0} + \frac{1}{q_0} = 1 \notag\\
&\leq \frac{\Gamma(\epsilon)}{ | \Gamma(\epsilon + it) | } \| f \|_{L^{q_0}} \label{eq2_5}
\end{align}

Then (\ref{eq2_5}) implies:
\begin{align}
\| (Id - L)^{-(r + \epsilon + it)} g \|_{L^{p_0}} \leq \frac{\Gamma(\epsilon)}{ | \Gamma(\epsilon + it) | } \| f \|_{L^{p_0}} \label{eq2_6}
\end{align}

And from [3, p. 213], $| \Gamma(\epsilon + it) |$ is asymptotically equivalent to $t^{m + \frac{1}{2}} e^{-\frac{\pi}{2} t}$ when $t \uparrow \infty$, $m$ being the largest integer $< \epsilon$; if $\epsilon < 1$, $m=0$.

Let $1 < q_0 < 2$, $1 < q_0 < q' < 2$ and $\varphi \in L^{\infty - 0}$.
We denote, with $i = \sqrt{-1}$:

\begin{align*} 
\theta(z) = (iz)^{\frac{1}{2}} e^{i \pi \frac{z}{2}} \int \P(d\omega) (Id - L)^{-z} g . \Sgn \varphi . |\varphi|^{\frac{q'}{2} + \frac{z}{r+\epsilon} q'(\frac{1}{q_0} - \frac{1}{2})}
\end{align*}

with: $z \in \Delta = \left\{ z \big/ 0 \leq \Real z \leq r + \epsilon, \Imag z > 0 \right\}$

Then $|\theta(z)|$ is bounded on $\Delta$ and is continuous on $\bar{\Delta}$, because:

\begin{align*}
|\theta(z)| \leq \sqrt{|z|} e^{-\frac{\pi}{2}(\Imag z)} \frac{\Gamma(\epsilon)}{ | \Gamma(\epsilon + it) |}  \| f \|_{L_2} \| \varphi^{\frac{q'}{q_0}} \|_{L_2}
\end{align*}

Computing $|\theta(0 + it)|$ and $|\theta(r + \epsilon + it)|$, we get, $C$ being a constant:

\begin{align*} 
|\theta(it)| \leq t^{\frac{1}{2}} e^{ -\frac{\pi}{2} t } \| (Id - L)^{ r - it} f \|_{L^2} \| \varphi^{\frac{q'}{2}} \|_{L^2} \leq C \| f \|_{\sko_{2r}^2} \| \varphi \|_{L^q}^{\frac{q'}{2}} \\
|\theta(r + \epsilon + it)| \leq C t^{\frac{1}{2}} e^{ -\frac{\pi}{2} t } \| (Id - L)^{ -(r + \epsilon + it} g \|_{L^{p_0}} \| \varphi^{\frac{q'}{q_0}} \|_{L^{q_0}}
\end{align*}

Using (\ref{eq2_6}), we get:

\begin{align*}
|\theta(r + \epsilon + it)| \leq C \frac{ t^{\frac{1}{2}} e^{ -\frac{\pi}{2} t } }{| \Gamma(\epsilon + it) |} \Gamma(\epsilon) \|f\|_{L^{p_0}} \|\varphi\|_{L^{q'}}^{\frac{q'}{q_0}} 
\end{align*}

We choose $\varphi \in L^{\infty - 0}$ such that $\| \varphi \|_{L^{q'}} \leq 1$. Now, when 

$t \uparrow \infty$, $\frac{\Gamma(\epsilon)}{ | \Gamma(\epsilon + it) |} \sim e^{\frac{\pi}{2} t} t^{-\frac{1}{2}} (0 < \epsilon < 1)$, so we have:

\begin{align}
\max ( |\theta(it)|, |\theta(r + \epsilon + it)| ) \leq C \max ( \|f\|_{\sko_{2r}^2}, \|f\|_{L^{p_0}}) \label{eq2_7}
\end{align}

Then the Phragmen-Lindelof method tells us that $|\theta(z)|$ is bounded on $\bar{\Delta}$ by the r.h.s. of (\ref{eq2_7}), for all $\varphi \in L^{\infty-0}$ and $\| \varphi \|_{L^{q'}} \leq 1$.

Then: $\forall 1 < q' < 2, \exists a \in ]0, r[$ such that: $ \frac{r-a}{r+\epsilon} . q' . \left( \frac{1}{q_0} - \frac{1}{2} \right) + \frac{q'}{2} = 1$: 

$a = r - (r + \epsilon) \left( \frac{ \frac{1}{q'} - \frac{1}{2} }{ \frac{1}{q_0} - \frac{1}{2} }\right)$ and we have $a \in ]0, r[$ if $\epsilon < r \left( \frac{ \frac{1}{q_0} - \frac{1}{q'} }{ \frac{1}{q'} - \frac{1}{2} }\right)$.

Then $\theta(r-a) = \sqrt{i(r-a)} e^{i \frac{\pi}{2} (r-a)} \int \P(d\omega) \varphi (Id - L)^{-(r-a)} g$, so

 $| \int \P(d\omega) \varphi (Id - L)^{-(r-a)} g  | \leq C$

which implies: $(Id - L)^{-(r-a)} g \in L^{p'}(\Omega)$ with $\frac{1}{p'} + \frac{1}{q'} = 1$.

Finally: $\forall a \in ]0, r[$ and $\forall p' > 2$: $(Id - L)^a f \in L^{p'}(\Omega)$ then $f \in \dinf(\Omega)$.
\end{proof}

\begin{prop}\label{pr2_1}
The O.U. operator commutes with the conditional expectation.
\end{prop}

\begin{proof}
Let $(\Omega, \F, \P, H)$ be a Gaussian space and $\F_t, t \in [0, 1]$ be a filtration; $\forall \varphi \in \F_t$, $\Grad$ being the Malliavin derivative and $\Div $ its adjoint: $f \in \dinf(\Omega)$:

\begin{align*}
\int \varphi \Div \Grad \E \left[ f | \F_t \right] \P(d\omega) 
&= - \int \langle \Grad \varphi, \Grad \E \left[ f | \F_t \right] \rangle_H \P(d\omega) \\
&= - \int \langle \Grad \varphi, \E \left[ \Grad f | \F_t \right] \mathds{1}_{[0, t]}(.) \rangle_H \P(d\omega) \\
&= - \int \langle \E \left[ \Grad \varphi | \F_t \right] \mathds{1}_{[0, t]}(.), \Grad f \rangle_H \P(d\omega) \\
&= - \int \langle \Grad \E \left[ \varphi | \F_t \right], \Grad f \rangle_H \P(d\omega) \\
&= + \int f . \Div \Grad \E \left[ \varphi | \F_t \right] \P(d\omega) \\
&= + \int \E \left[ \varphi | \F_t \right] \Div \Grad f \P(d\omega) \\
&= + \int \varphi \E \left[ \Div \Grad f | \F_t \right] \P(d\omega)
\end{align*}

\end{proof}

\subsection{Existence of a sequence of $\mathbb{D}^\infty$-vector fields converging $\mathbb{D}^\infty$-strongly to wards a derivation}

\begin{thm}\label{thm2_6}
Let $(\Omega, \F, \P, H)$ be a Gaussian space and $(e_i)_{i \in \N_{\star}}$ an Hilbertian basis of $H$; we denote by $\F_n$ the $\sigma$-algebra $\F_n = \sigma( W( e_i ) \big/ i \leq n)$.
Let $\delta$ be a continuous derivation from $\dinf$ to $L^{\infty - 0}$. Then the sequence of vector fields $X_N = \sum_{i=1}^N \E \left[ \delta( W(e_i) | \F_N \right] e_i$ strongly converges towards $\delta$ in $L^{\infty - 0}(\Omega)$.
\end{thm}

\begin{proof}
Let $f_{N,k} \left[ W(e_1), \dots, W(e_N), W(e_{N+1}), \dots, W(e_{N+k}) \right]$ be a cylindrical function; $\xi_1, \dots, \xi_k$ being parameters.
We denote $f_{N, k, \xi} = f_{N, k} \left[ W(e_1), \dots, W(e_N), \xi_1, \dots, \xi_k \right]$. Then: 

\begin{align*}
\| f_{N, k, \xi} \|_{\sko_s^q} = \sum_{j=0}^s \left( \int \| \Grad ^j f_{N, k, \xi} \|_{ \overset{j}{\otimes} H }^q \P(d\omega)  \right)^{\frac{1}{q}}
\end{align*}

so:

\begin{align}
\frac{1}{(\sqrt{2\pi})^k} \int_{\R^k} d\xi_1 \dots d\xi_k e^{-\frac{1}{2} \sum_{i=1}^k \xi_i^2 }  \| f_{N, k, \xi} \|_{\sko_s^q}^q \leq \| f_{N, k} \|_{\sko_s^q}^q \label{eq2_8} 
\end{align}

Let $\varphi \in \dinf(\Omega)$; there exists a sequence of cylindrical functions denoted $\varphi_{N, k} \left[ W(e_1), \dots, W(e_N), W(e_{N+1}), \dots, W(e_{N+k}) \right]$ which $\dinf$-converges towards $\varphi$.

Then direct computation shows that $\forall q > 1, \exists (p, s), p > 1, s \in \N_\star$ such that there exists a constant $C(p, q, s)$ with $\| X_N . \varphi_{N, k, \xi} \|_{L^q(\Omega)} \leq C(p, q, s) \| \varphi_{N, k, \xi} \|_{\sko_s^p(\Omega)}$

We choose $p>q$. Then, using (\ref{eq2_8}), we have:

\begin{align*}
\frac{1}{(\sqrt{2\pi})^k} \int \| X_N . \varphi_{N, k, \xi} \|_{L^q(\Omega)}^p e^{-\frac{1}{2} \sum_{i=1}^k \xi_i^2 } d\xi_1 \dots d\xi_k \leq C^p(p, q, s) \| \varphi_{N, k} \|_{\sko_s^p(\Omega)}^p 
\end{align*}

The l.h.s. of the above inequality is bigger than:

\begin{align*}
\left[ \frac{1}{(\sqrt{2\pi})^k} \int \| X_N . \varphi_{N, k, \xi} \|_{L^q(\Omega)}^q e^{-\frac{1}{2} \sum_{i=1}^k \xi_i^2 } d\xi_1 \dots d\xi_k \right]^{\frac{p}{q}} = \| X_N . \varphi_{N, k} \|_{L^q}^p 
\end{align*}

So we have $\| X_N . \varphi_{N, k} \|_{L^q(\Omega)} \leq C(p, q, s) \| \varphi_{N, k} \|_{\sko_s^p}$

which implies:
\begin{align}
\| X_N . \varphi \|_{L^q(\Omega)} \leq C(p, q, s) \| \varphi \|_{\sko_s^p} \label{eq2_9}
\end{align}

Then the triangle inequality
\begin{align*}
\| \delta \varphi - X_{N, k} . \varphi \|_{L^q}
&\leq \| \delta \varphi - \delta \varphi_{N,k} \|_{L^q} + \| \delta \varphi_{N, k} - \E \left[ \delta \varphi_{N,k} | \F_{N+k} \right] \|_{L^q} + \\ 
& + \| \E \left[ \delta \varphi_{N,k} | \F_{N+k} \right] - X_{N+k} . \varphi_{N,k} \|_{L^q} + \| X_{N+k} . (\varphi_{N,k} - \varphi) \|_{L^q}
\end{align*}

with $\E \left[ \delta \varphi_{N,k} | \F_{N+k} \right] = X_{N+k} . \varphi_{N,k}$; and (\ref{eq2_9}) shows that $X_N \rightarrow \delta$ and the convergence is $L^{\infty-0}$-strong from $\dinf$ in $L^{\infty-0}$.
\end{proof}

\begin{cor}\label{cor2_5}
If $\delta$ is a continuous derivation from $\dinf(\Omega)$ to $L^{\infty-0}(\Omega)$, for each $f \in \dinf(\Omega)$, $A$ being a measurable subset of $\Omega$, then $\mathds{1}_A \Grad f = 0$ implies $\mathds{1}_A \delta f = 0$.
\end{cor}

\begin{proof}
Let $X_N$ be the sequence of vector fields as in Theorem \ref{thm2_6}. Then from the hypothesis, we have: $\mathds{1}_A(X_N . f) = 0$; and at the limit when $N \uparrow \infty$, we have $\mathds{1}_A \delta f = 0$.
\end{proof}

\begin{thm}\label{thm2_7}
If $\delta$ is a continuous derivation of $\dinf(\Omega)$, there exists a sequence of vector fields $\tilde{X}_N$, which converges $\dinf$-strongly towards $\delta$.
\end{thm}

\begin{proof}
Let $N$ be fixed and $X_N = \sum_{i=1}^N \E \left[ \delta(W(e_i)) | \F_N \right] e_i$.
If $f \left[ W(e_1), \dots, W(e_N) \right]$ is a cylindrical function, direct calculus shows that:

\begin{align*}
X_N . f \left[ W(e_1), \dots, W(e_N) \right] = \E \left[ \delta f | \F_N \right]
\end{align*}

So $X_N$ is $\dinf$-continuous from $\dinf(\Omega, \F_N, \P)$ in itself.
We extend $X_N$ to $\dinf(\Omega, \F_N, \P) \times \dinf(\Omega, \F_N^{\perp}, \P) \simeq \dinf(\Omega, \F, \P)$ denoted $\tilde{X}_N$ 

as in Theorem 2, 2, i); $\tilde{X}_N$ is again $\dinf$-continuous and is a vector field. With Theorem 2, 2, ii), the sequence $(\tilde{X}_N)$ is $N$-uniformly bounded when the $\tilde{X}_N$ are considered as operators from $\dinf(\Omega)$ in $\dinf(\Omega)$. 

The convergence of $\tilde{X}_N$ towards $\delta$, strong convergence 
as operators, is obtained by:
$\forall f \in \dinf, \forall (p, r), p>1, r \in \N_{\star}:$

\begin{align*}
\| \delta f - \tilde{X}_N . f \|_{\sko_r^p}
&\leq \| \delta f - \E \left[ \delta f | \F_N \right] \|_{\drp}
+ \| \E \left[ \delta f | \F_N \right] - \E \left[ \delta f_N | \F_N \right] \|_{\drp}
+ \| \E \left[ \delta f_N | \F_N \right] - X_N . f_N \|_{\drp} \\
&+ \| X_N . f_N - \tilde{X}_N . f_N \|_{\drp}
+ \| \tilde{X}_N . \left( f_N - f \right) \|_{\drp}
\end{align*}

$f_N$ being a sequence of cylindrical functions, with $f_N(\omega) = f_N \left[ W(e_1), \dots, W(e_N) \right]$, converging $\dinf$ towards $f$.
\end{proof}

\begin{defn}\label{def2_2}
A process $X: [0,1] \times \Omega \rightarrow \R$ is said to be 

completely $\dinf$ iff: $\forall r \in \R$:
\begin{align*}
(1-L)^{r/2} X \in L^{\infty - 0} \left( [0,1] \times \Omega, dt \otimes \P(d\omega) \right)
\end{align*}
\end{defn}

\begin{lem}\label{lem2_1} 
i) The space of the completely $\dinf$-processes $\mathscr{S}$ is a Frechet space.

ii) If $F$ is a continuous linear map from $\mathscr{S}$ in $\mathscr{S}$, there is a unique continuous linear extension of $F$, denoted $\tilde{F}$, from the space of completely $\dinf$-processes with values in an Hilbert $\tilde{H}$, in itself. 
\end{lem}

\begin{prop}\label{prop2_2}
Some properties of the convolution and the fractionnal derivation:

$f$ being a function: $[0,1] \rightarrow \R$ with $f(0) = 0$, we extend $f$ in $\tilde{f}$ by 

$\left.\tilde{f}\right|_{\R_-} = \left.\tilde{f}\right|_{\big[2, +\infty\big[} = 0$ and by an afine function on $[1,2]$ such that 

$\tilde{f}(1) = f(1)$ and $\tilde{f}(2) = 0$. Then, we denote by $\beta_s$ the function: $\R \rightarrow \R$ defined by: $\beta_s(x) = 0, x \leq 0$ and $\beta_s(x) = \frac{1}{\Gamma(1-s)} \frac{1}{x^s}$ for $x > 0$. Then:

\begin{enumerate}
\renewcommand{\labelenumi}{\roman{enumi})}
\item If $f$ is $\alpha$-H\"{o}lderian, $\forall s$ with $0 < s < \alpha < 1$, $\tilde{f} \star \beta_s \in C^1(\R)$
\item $f$ being $\alpha$-H\"{o}lderian, we have: $(\tilde{f} \star \beta_s)' \star \beta_{1-s} = f$.
\item If $f$ is a $\drp$-bounded process, $\tilde{f} \star \beta_s$ is $(1-s)$-$\drp$-H\"{o}lderian (see Definition \ref{def2_3}).
\item The convolution of $\beta_s$ with an adapted process is again an adapted process.
\end{enumerate}

\end{prop}

\subsection{$\mathbb{D}^\infty$-Hölderian process and divergence of a derivation}

\begin{defn}\label{def2_3}
A real-valued process $\Phi(t, \omega)$ will be said to be 

$\dinf$-$\alpha$-H\"{o}lderian iff:
$\forall t \in [0, 1], \forall t' \in [0, 1]: \forall (p, r), p>1, r \in \N_{\star}, \exists \text{ constant } C(p, r)$ such that:
\begin{align*}
\sup_{t, t'} \| \Phi(t', \omega) - \Phi(t, \omega) \|_{\drp} \leq C(p, r) |t'-t|^{\alpha}
\end{align*}
\end{defn}

There is an analogous definition for a matrix-valued process.

\begin{thm}\label{thm2_8}
Let $X: [0,1] \times \Omega \rightarrow \R$ be an $\dinf$-$\alpha$-H\"{o}lderian process. Then $\exists s, 0 < s < 1$ such that if $Y = \frac{d}{dt} \left[ X \star \beta_{1-s} \right]$, $X = Y \star \beta_s$, and $Y$ being a completely $\dinf$-process.
\end{thm}

\begin{defn}\label{def2_4new}
A $\dinf$-process defined on $[0,1] \times \Omega$ with values in the $n \times n$ matrices, will be said to be $\dinf$-bounded iff:

$\forall (p, r) ~ p>1, r \in \N_\star, ~\exists C(p, r) > 0$ such that: 

$\sup_{t \in [0,1]} \| A(t, .) \|_{\drp} \leq C(p, r)$, where $\|A(t, .)\|_{\drp}$ denotes the ${\drp}$-norm of any $n \times n$ matrix norm, which are all equivalent.

\end{defn}

\begin{lem}\label{lem2_2}
Let $A$ be a $\dinf$-bounded matrix process. Then the process: $\Phi(t, \omega) = \int_0^t A dB$, $B$ being a $n$-valued Brownian motion, is $\frac{1}{2}$-$\dinf$-H\"{o}lderian.
\end{lem}

\begin{proof}
$A$ being a $n \times n$ matrix, $\int_t^{t+h} A dB$ being a Skorokhod integral, we have: $\int_t^{t+h} A dB = \sqrt{h} \Div (A . X_h)$ where:

$X_h$ is the vector $X_h = s \rightarrow \frac{1}{\sqrt{h}} \int_0^s \sum_{i=1}^n \mathds{1}_{[t, t+h]}(u) e_i du$, $(e_i)_{i=1,\dots,n}$ being the canonical basis of $\R^n$, because $A X_h$ is $\dinf(\Omega, H)$-bounded.

The operator $\Div $, being continuous: 
\begin{align*}
\exists C(p,r): \| \Phi(t', \omega) - \Phi(t, \omega) \|_{\drp} \leq C(p, r) \sqrt{h}
\end{align*}
\end{proof}

\begin{defn}\label{def2_4}
Let $\delta$ be a continuous derivation of $\dinf(\Omega)$.

\begin{enumerate}
\renewcommand{\labelenumi}{\roman{enumi})} 
\item an element $T$ in $\sko^{- \infty}$ will be called divergence of $\delta$, denoted $\Div \delta$, iff: $\forall \varphi \in \dinf(\Omega): - \int \delta \varphi \P(d \omega) = + \left( \Div \delta, \varphi \right)$
( $\left( \Div \delta, \varphi \right)$ is the duality bracket).

\item a continuous derivation $\delta$ of $\dinf(\Omega)$ is said to be an adapted derivation iff: $\forall$ adapted process $\Phi(t, \omega)$, $(\delta \Phi)(t, \omega)$ is an adapted process.
\end{enumerate}
\end{defn}

\begin{rem}\label{rem2_1}
If $U$ is a vector field, $\Div U$ as in Definition \ref{def2_4} coincides with the classical definition of divergence of a $\dinf$-vector field.
\end{rem}

\begin{rem}\label{rem2_2}
\begin{enumerate}
\renewcommand{\labelenumi}{\roman{enumi})}
\item If $\Div \delta = 0$, then $\delta$ is $L^2$-antisymmetrical.
\item $(\Omega, \F, \P, H)$ being a Gaussian space, if $A$ is a $n \times n$-A.M. matrix such that $\forall \varphi \in \dinf(\Omega)$, $A \Grad \phi \in \dinf(\Omega, H)$, then $\Div A \Grad$ is a $\dinf$-derivation, and $\Div (\Div A \Grad ) = 0$. Such an $A$ is called a multiplicator and they will be studied in Section 4.
\end{enumerate}
\end{rem}

\begin{thm}\label{thm2_9}
Let $(\Omega, \F, \P, H)$ be a Gaussian space.

If $V(s, \omega): t \rightarrow \int_0^t \dot{h}(s, \omega) ds$ is a $\dinf$-vector field, then $t \rightarrow \int_0^t \E \left[ \dot{h}(s, \omega) | \F_s \right] ds$ is a $\dinf$-vector field.
\end{thm}

\begin{proof}
The O.U. operator commutes with the conditional expectation (Proposition 2.1) so:

\begin{align*}
t \rightarrow \int_0^t \E \left[ \dot{h}(s, \omega) | \F_s \right] ds \in \sko_{\infty}^2(\Omega, H)
\end{align*}

Thanks to the Theorem 2.5, we only have to prove:

\begin{align*}
t \rightarrow \int_0^t \E \left[ \dot{h}(s, \omega) | \F_s \right] ds \in L^{\infty-0}(\Omega, H)
\end{align*}

We denote by $L_{ad}^{1+0}(\Omega, H)$ the set of $\dinf$-vector fields $Z(t, \omega) = \int_0^t \dot{Z}(s, \omega) ds$ such that $\dot{Z}(s, \omega)$ is an adapted process.

We have: 

$\E \left[ \langle V(., \omega), Z(., \omega) \rangle_H \right] < + \infty$.
so:
\begin{align*}
\E \left[ \int_0^1  \langle \dot{h}(s, \omega), \dot{Z}(s, \omega) \rangle_{\R^n} ds  \right] = \E \left[ \int_0^1 \langle \E \left[ \dot{h}(s, \omega) | \F_s \right] , \dot{Z}(s, \omega) \rangle_{\R^n} ds  \right] < + \infty
\end{align*}

which implies that: 

\begin{align}
\E \left[ \dot{h}(s, \omega) | \F_s \right] \in \left( L_{ad}^{1+0}(\Omega, H) \right)^{\star} \label{eq2_10} 
\end{align}

Let $g \in L^q(\Omega)$ with $\frac{1}{2} + \frac{1}{q} < 1$; then there exists $p'$ with $1<p'<2$ such that $u \rightarrow \int_0^u \E \left[ g | \F_s \right] . \E \left[ \dot{h}(s, \omega) | \F_s \right] ds \in L^{p'}(\Omega, H)$, 
because: 
\begin{align*}
\int \P(d\omega) \left[ \int_0^1 ds \E \left[ g | \F_s \right]^2 . \E \left[ \dot{h}(s,\omega) | \F_s \right]^2 \right]^{\frac{p'}{2}} 
\leq \int \P(d\omega) \left\{ \sup_{s \in [0, 1]} \E \left[ g | \F_s \right]^2 \int_0^1 ds \E \left[ \dot{h}(s,\omega) | \F_s \right]^2 \right\}^{\frac{p'}{2}} 
\end{align*}

From the Doob inequality:
\begin{align*}
    \sup_{s \in [0, 1]} | \E \left[ g | \F_s \right] |^2 \in L^{\frac{q}{2}}(\Omega) 
\end{align*}

and: $ \left[ \int_0^1 ds \E \left[ \dot{h}(s, \omega) | \F_s \right]^2 \right]^{\frac{1}{2}} \in L^2(\Omega)$, we have: 
\begin{align*}
\left[ \int_0^1  \E \left[ \dot{h}(s, \omega) | \F_s \right]^2 ds\right]^{\frac{p'}{2}} \in L^{\frac{2}{p'}}(\Omega)
\end{align*}

And from: $\frac{1}{q} + \frac{1}{2} < 1, \exists p'$ with $\frac{1}{q/p'} + \frac{1}{2/p'} = 1, 1 < p' < 2$ and so for this $p'$ we have:

\begin{align*}
\int \P(d\omega) \left[ \int_0^1 ds \E \left[ g | \F_s \right]^2 \E \left[ \dot{h}(s,\omega) | \F_s \right]^2 \right]^{\frac{p'}{2}} < +\infty
\end{align*}

which implies: 
\begin{align}
t \rightarrow \int_0^t \E \left[ g | \F_s \right] \E \left[ \dot{h}(s,\omega) | \F_s \right] ds \in L_{ad}^{1+0}(\Omega, H) \label{eq2_10}
\end{align}

To prove that the vector field $t \rightarrow \int_0^t \E \left[ \dot{h}(s, \omega) | \F_s \right] ds$ belongs to $L^{\infty-0}(\Omega, H)$, we use an induction:

we have already $t \rightarrow \int_0^t \E \left[ \dot{h}(s, \omega) | \F_s \right] ds \in L^2(\Omega, H)$.
Let $g \in L^{q'}(\Omega)$ with $\frac{1}{p} + \frac{1}{q'} < 1$; 
Let $t \rightarrow \int_0^t \E \left[ \dot{h}(s, \omega) | \F_s \right] ds \in L^p(\Omega, H)$; From (\ref{eq2_9}) and (\ref{eq2_10}), we get:

\begin{align*}
\int \P(d\omega) \int_0^1 ds \E \left[ g \E \left[ \dot{h}(s, \omega) | \F_s \right]^2 | \F_s \right] < + \infty
\end{align*}

which implies: 
\begin{align*}
\int \P(d\omega) g(\omega) \int_0^1 ds \E \left[ \dot{h}(s, \omega) | \F_s \right]^2 < + \infty
\end{align*}

then: $\int_0^1 \E \left[ \dot{h}(s, \omega) | \F_s \right]^2 ds \in L^{p-0}(\Omega)$
and: 

$t \rightarrow \int_0^t \E \left[ \dot{h}(s, \omega) | \F_s \right] ds \in L^{2p-0}(\Omega, H)$.

\end{proof}

\begin{thm}\label{thm2_10}
Let $V_n$ be a sequence of $\dinf$-vector fields such that the associated derivations $\delta_n$ are adapted and converge pointwise in $\dinf(\Omega)$ towards a derivation $\delta$ verifying $\Div \delta = 0$; then $\delta \equiv 0$.
\end{thm}

\begin{proof}
\begin{align*}
\forall \varphi \in \dinf(\Omega): 
&\int \delta_n \varphi \P(d\omega) 
= \int \langle V_n, \Grad \varphi \rangle_H \P(d\omega) = - \int \varphi \Div V_n \P(d\omega)
\end{align*}

which implies that $\int \varphi \Div V_n \P(d\omega)$ converges towards 

$-\int \varphi \Div \delta . \P(d\omega) = 0$.

So $\forall V$ vector field $\in \dinf(\Omega, H)$: $\int \langle \Grad (\Div V_n), V \rangle_H \P(d\omega)$ converges towards $0$.

With Theorem \ref{thm2_9}, $t \rightarrow \int_0^t \E \left[ \Grad \Div V_n | \F_s \right] ds$ is also a $\dinf$-vector field, and an adapted process:

we have: 

\begin{align*}
\int \langle \E \left[ \Grad \Div V_n | \F_t \right], V \rangle_H \P(d\omega)
= \int \langle \Grad \Div V_n, \E \left[ V | \F_t \right] \rangle_H \P(d\omega) \rightarrow 0
\end{align*}

But using the Clark-Ocone formula:
\begin{align*}
\int_0^1 \E \left[ \Grad \Div V_n | \F_t \right] dB = \Div V_n
\end{align*}

And $V_n$ being an adapted process, using a result on the Skorokhod integral: $\Div V_n = \int_0^1 V_n . dB$, It\={o} Integral.

Then: $ \E \left[ \Grad \Div V_n | \F_t \right] = V_n$ and so $V_n$ is a sequence of $\sko^{-\infty}$ which converges towards $0$.
Then let $\varphi, \psi \in \dinf(\Omega)$: $\varphi \Grad \psi \in \dinf(\Omega, H)$
and:

\begin{align*}
(V_n, \varphi \Grad \psi) &= \int \varphi(\omega) \langle V_n, \Grad \psi \rangle_H \P(d\omega) \\
& = \int \varphi(\omega) (\delta_n \psi)(\omega) \P(d\omega) 
\rightarrow \int \varphi(\omega)(\delta \psi)(\omega) \P(d\omega)
\end{align*}

So $\int \varphi(\omega)(\delta \psi)(\omega) \P(d\omega) = 0$ which implies $\delta = 0$.
\end{proof}

\begin{rem}\label{rem2_3}
A consequence of Theorem \ref{thm2_10} is that a $\dinf$ adapted derivation is not in general a limit of a sequence of $\dinf$-adapted vector fields.
\end{rem}

\begin{defn}\label{def2_5}
Let $\theta$ be a continuous map from $\dinf(\Omega)$ to $\dinf(\Omega)$. Then a linear map $\delta$ from $\dinf(\Omega)$ to $\dinf(\Omega)$ is said to be a $\theta$-derivation iff:

\begin{align*}
\forall f, g \in \dinf(\Omega): \delta(f g) = \theta(f) \delta(g) + \theta(g) \delta(f)
\end{align*}

\end{defn}

Now we state a version of the Dini-Lipschitz theorem for an H\"{o}lderian functon $f$, piecewise continuous, on a closed interval, with values in a Frechet space $F$.

\subsection{A generalisation of the Dimi-Lipschitz theorem and interpolation between $\mathbb{D}^P_r$ spaces}

\begin{thm}\label{thm2_11} (Dini-Lipschitz)
Let $f$ be a function as above. Then its Fourier series is uniformly convergent piecewise, that is on each closed interval on which it is continuous, and converges towards the half-value of the jump on each discontinuity point.

Moreover, the convergence of its Fourier series is uniformly bounded, relatively to the partial sums of the Fourier series.
\end{thm}

Last, an interpolation theorem which will often be used:

\begin{thm}\label{thm2_12}
Let $T$ be a linear continuous operator from $\drp(\Omega)$ to $\sko_s^q(\Omega)$, and $\sko_{r'}^{p'}(\Omega)$ to $\sko_{s'}^{q'}(\Omega)$. $\forall \alpha \in [0,1]$, $T$ is continuous from $\sko_{\alpha r + (1-\alpha)r'}^{(\frac{\alpha}{p} + \frac{(1-\alpha)}{p'})^{-1}}$ to $\sko_{\alpha s + (1-\alpha)s'}^{(\frac{\alpha}{q} + \frac{(1-\alpha)}{q'})^{-1}}$.
\end{thm}

\section{\huge $\dinf$-stochastic manifolds}

In this section we study the general notion of a $\dinf$-stochastic manifold, and the following themes:

\begin{enumerate}
\renewcommand{\labelenumi}{\roman{enumi})}

\item we will examine the notion of $\dinf$-equivalent atlases, which is more complex than in the case of $n$-dimensional differential manifolds

\item we will exhibit a $\dinf$-chart change which does not admit a linear tangent map in $\mathcal{L}(H)$

\item we will study the space of $\dinf$-continuous derivations on $\dinf(\Omega)$, denoted $Der(\Omega)$ and its dual $Der(\Omega)^\star$ (denoted also $(Der\Omega^\star)$)

\item we will study the notion of a derivation field on a $\dinf$-stochastic manifold.

\item we study then the notion of metric (this time on $Der(\Omega)^\star$), and the fundamental metric; then an important subspace of $Der(\Omega)$, denoted $\mathcal{D}_0(\Omega)$, the Levi-Civita connection, the curvature and the torsion.

\end{enumerate}

When no particular setting is specified, we assume that the context is a Gaussian space $(\Omega, \F, \P, H)$.

\subsection{Definition}

Let $\mathscr S$ be a set. The definitions of $\mathbb D^\infty$-charts, 
of two $\mathbb D^\infty$-compatible charts and of a $\mathbb D^\infty$-atlas, on $\mathscr S$
are direct generalisations of the $\mathbb D_r^\infty$ case.
We first define the notion of canonical tribe 
on a set $\mathscr S$, endowed with a $\mathbb D^\infty$-atlas:

\subsection{Canonical $\sigma$-algebra associated to a $\mathbb{D}^\infty$-stochastic manifold}

\begin{prop}
  Let $\mathscr S$ a $\mathbb D^\infty$-stochastic manifold
  with the atlas: ${\mathscr A = \left(U_i, b_i, \Omega_i\right)_{i\in I}}$.
  Then if
  \begin{equation*}
    \mathscr C_1 = \left\{A\subset\mathscr S\,/\,\forall i\in I, b_i(A\cap U_i)\in\mathcal F_i\right\}
  \end{equation*}
  and
  \begin{equation*}
    \mathscr C_2 = \left\{A\subset\mathscr S\,/\,\forall A_i\in\mathcal F_i, b_i(A\cap b_i^{-1}(A_i))\in\mathcal F_i\right\}
  \end{equation*}
  we have $\mathscr C_1 = \mathscr C_2$ and $\mathscr C_1$ is a $\sigma$-algebra.
\end{prop}

\begin{proof} ~
  \begin{enumerate}
    \item[i)] $\mathscr C_2 \subset \mathscr C_1$: let $A_i=\Omega$
    \item[ii)] $\mathscr C_1 \subset \mathscr C_2$: let $A \in \mathscr C_1$,
      \begin{align*}
         b_i\left[A\cap b_i^{-1}(A_i)\right] 
          & = b_i\left[A\cap\left(U_i\cap b_i^{-1}(A_i)\right)\right] \\
          & = b_i\left[(A\cap U_i)\cap b_i^{-1}(A_i)\right] \\
          & = b_i(A\cap U_i)\cap A_i \in \mathcal F_i
      \end{align*}
    \item[iii)] $\mathscr C_1$ is a $\sigma$-algebra: obvious.
  \end{enumerate}
\end{proof}

We denote by $\mathscr C(\mathscr S, \mathscr A)$ this $\sigma$-algebra, or in short $\mathscr C(\mathscr S)$.
We also denote by: $\mathscr N(\mathscr S,\mathscr A)$ or by $\mathscr N(\mathscr S)$:
\begin{equation*}
  \mathscr N(\mathscr S, \mathscr A) 
    = \left\{N\subset\mathscr S\,/\,\forall i\in I, \mathbb P_i\left[b_i(N\cap U_i)\right]=0 \right\}
\end{equation*}
The definition of $\mathscr N(\mathscr S,\mathscr A)$ is meaningful thanks to the Lemma 1.2,iii: we know that
if there is $i\in I$ such that ${\mathbb P_i\left[b_i(A\cap U_i)\right]>0}$, 
then $\forall j\in I$: ${\mathbb P_j\left[b_j(A\cap U_j)\right]>0}$.

\subsection{$\mathbb{D}^\infty$-morphismes between $\mathbb{D}^\infty$-stochastic manifolds}

\begin{defn}
  $\mathscr S_1$ and $\mathscr S_2$ being two $\mathbb D^\infty$-stochastic manifolds, 
  the map $\varphi: \mathscr S_1 \to \mathscr S_2$ will be said to be measurable
  iff it is measurable relatively to the $\sigma$-algebra
  $\mathscr C(\mathscr S_1)$ and $\mathscr C(\mathscr S_2)$ and if $\varphi^{-1}\left[\mathscr N(\mathscr S_2)\right]\subset\mathscr N(\mathscr S_1)$.
\end{defn}

\begin{defn}
  Let $\mathscr A=\left(U_i, b_i, \Omega_i, \mathcal F_i, \mathbb P_i, H_i\right)_{i\in I}$, a $\mathbb D^\infty$-atlas
  on the set $\mathscr S_1$, and $(V_\ell,\widetilde b_\ell, \widetilde \Omega_\ell)_{\ell\in L}=\mathscr B$ a $\mathbb D^\infty$-atlas 
  on the set $\mathscr S_2$; let $\varphi$ be a measurable map from $\mathscr S_1$
  in $\mathscr S_2$.
  
  The subset $A\subset U_i$ will be said to be
  $(\varphi,\mathscr A,\mathscr B)$-balanced iff $A\in\mathcal F_i$, $\mathbb P_i(A)>0$
  and $\exists\ell_0\in L$ such that:
  $\varphi(A)\subset V_{\ell_0}$ and $\forall f\in\mathbb D^\infty(\widetilde \Omega_{\ell_0})$, then
  $\left.f\circ \widetilde b_{\ell_0}\circ\varphi\circ b_i^{-1}\right|_{b_i(A)}$ admits an extension, denoted $\widetilde f_{i,\ell_0,A}$,
  or $\widetilde f_A$ and ${\widetilde f_{i,\ell_0,A}\in\mathbb D^\infty(\Omega_i)}$.
  $A$ will also be said to be $(\varphi, U_i,V_{\ell_0})$-balanced.
\end{defn}
\begin{rem}
  Using the $\mathbb D^\infty$-structure of the atlas $\mathscr A$, 
  it is immediate to see that if $A$ is $(\varphi, U_i, V_{\ell_0})$-balanced
  and if $A\subset U_j$, $j\neq i$, then $A$ is $(\varphi, U_j,V_{\ell_0})$-balanced.
  So this definition of the balanced set does not
  depend on the chart domain in $\mathscr A$, where $A$ lies.
\end{rem}

\begin{rem}
  If $\mathscr S$ is a $\mathbb D^\infty$-stochastic manifold, with the
  atlas ${\mathscr A = (U_i,b_i,\Omega_i)_{i\in I}}$, and denoting $\mathrm{Id}_{\mathscr S}$
  the identity on $\mathscr S$, we have that: for every $i,j\in I$
  with $U_i\cap U_j\notin\mathscr N(\mathscr S,\mathscr A)$: $U_i\cap U_j$ is a 
  $(\mathrm{Id}_{\mathscr S},U_i,U_j)$-balanced set of $\mathscr S$.
\end{rem}
 
  Now to simplify the notations, we identify the
  domain $U_i$ of a chart with its image on the
  Gaussian space $(\Omega_i,\mathcal F_i,\mathbb P_i,H_i)$ through the bijection $b_i$.
  So now $U_i$ is endowed with the $\sigma$-algebra
  $b_i^{-1}(\mathcal F_i)$ which is also the restriction to $U_i$ of
  $\mathscr C(\mathscr S, \mathscr A)$, and with a probability measure $(b_i^{-1})_*\mathbb P_i$.
  So the property of a balanced subset of $U_i$, $A$, can be 
  restated: $\forall f\in\mathbb D^\infty(V_{\ell_0}), \quad \left.f\circ\varphi\right|_A$ admits an
  extension denoted $\widetilde f_A$, with $\widetilde f_A\in\mathbb D^\infty(U_i)$.
  If $A\subset U_i$, 
  we denote by $L^0(A)$ the $\mathbb R$-valued functions 
  on $A$, measurable relatively to the $\sigma$-algebra $b_i^{-1}(\mathcal F_i)$

\begin{defn}
  Let $\mathscr S_1$ and $\mathscr S_2$ two $\mathbb D^\infty$-stochastic manifolds,
  $\mathscr S_1$ being endowed with a $\mathbb D^\infty$-atlas $\mathscr A=(U_i,b_i,\Omega_i,\mathcal F_i,\mathbb P_i,H_i)_{i\in I}$
  and $\mathscr S_2$ being endowed with the $\mathbb D^\infty$-atlas $\mathscr B=(V_\ell,\widetilde b_\ell,\widetilde \Omega_\ell)_{\ell\in L}$.
  Let $\varphi$ be a map from $\mathscr S_1$ to $\mathscr S_2$, $\varphi$ will be said to be
  a $\mathbb D^\infty$-morphism from $\mathscr S_1$ to $\mathscr S_2$, iff:
  \begin{enumerate}
    \item[i)] $\varphi$ is measureable with respect to the canonical $\sigma$-algebras 
              on $\mathscr S_1$ and $\mathscr S_2$;
    \item[ii)] $\forall i\in I$, there is a countable set of indices, denoted
               $L_i$, $L_i\subset L$, such that $U_i\in\bigcup_{\ell\in L_i}\varphi^{-1}(V_\ell)$;
    \item[iii)] $\forall i\in I, \forall \ell\in L$ with $\mathbb P_i\left[\varphi^{-1}(V_\ell)\cap U_i\right]>0$
                and $\forall A\in\mathcal F_i$, with ${A\subset\varphi^{-1}(V_\ell)\cap U_i}$ and $\mathbb P_i(A)>0$:
                $\exists A'\subset A, A'\in\mathcal F_i, \mathbb P_i(A') >0$ such that 
                $A'$ is $(\varphi,U_i,V_\ell)$-balanced;
    \item[iv)] $\forall A\subset U_i, A\in\mathcal F_i,\mathbb P_i(A)>0$ and $\forall g\in L^0(A)$
               if $\forall A_\ell$, $(\varphi,U_i,V_\ell)$-balanced, there exists $f_\ell\in\mathbb D^\infty(V_\ell)$
               such that $\left.f_\ell\circ\varphi\right|_{A\cap A_\ell}=\left.g\right|_{A\cap A_\ell}$,
               then $g$ admits an extension $\widetilde g$ such that $\widetilde g\in\mathbb D^\infty(U_i)$.
  \end{enumerate}
\end{defn}
\begin{rem}
  From iii) we see that $\forall i\in I, \exists\ell(i)\in L$ 
  such that ${\mathbb P_i[\varphi^{-1}(V_{\ell(i)})\cap U_i]>0}$
  otherwise $\mathbb P_i\left[\bigcup_{\ell\in L_i}\varphi^{-1}(V_\ell)\cap U_i\right] =\mathbb P_i(U_i) = 0$.
  And we also have that $\forall A\in\mathcal F_i, A\subset\varphi^{-1}(V_{\ell(i)})\cap U_i$
  and $\mathbb P_i(A)>0$, there exists a countable subset of $L$
  denoted $L_{\ell(i)}$ and a family of $(\varphi,U_i,V_{j_i})_{j_i\in L_{\ell(i)}}$-balanced
  sets $\varepsilon_{j_i}$, such that they form a partition 
  of $A$. And there exists a countable subset of $L$,
  $L_0$, such that:
  \begin{equation*}
    U_i = \bigcup_{i\in L_0}\left(\bigcup_{j_i\in L_{\ell(i)}}\varepsilon_{j_i}\right)
  \end{equation*}
  with $\mathbb P_i(\varepsilon_{j_i}\cap\varepsilon_{j_k})=0$, $\forall j_i\in L_{\ell(i)}$ and $\forall j_k\in L_{\ell(k)}$.
\end{rem}
\begin{prop}
  The composition of $\mathbb D^\infty$-morphisms is a
  $\mathbb D^\infty$-morphism.
\end{prop}
\begin{proof}
  Let $\mathscr S_1, \mathscr S_2, \mathscr S_3$ be three $\mathbb D^\infty$-stochastic manifolds
  with respectively the $\mathbb D^\infty$-atlases
  \begin{equation*}
    \mathscr A=(U_i)_{i\in I},\quad\mathscr B=(V_j)_{j\in J},\quad\mathscr C(W_k)_{k\in K}
  \end{equation*}
  and $\varphi_1$ a $\mathbb D"\infty$-morphism from $\mathscr S_1$ to $\mathscr S_2$, and $\varphi_2$ a $\mathbb D^\infty$-morphism
  from $\mathscr S_2$ to $\mathscr S_3$.
  \begin{enumerate}
    \item[i) and ii)] are trivially verified by $\varphi_2\circ\varphi_1$
                       ($\varphi_2\circ\varphi_1$ is measurable);

    \item[iii)] Let $i_0\in I$ and $k\in K$. We have to prove that 
                for ${A\subset\varphi_1^{-1}\circ\varphi_2^{-1}(W_k)\cap U_{i_0}}$, ${A \in \mathcal F_{i_0}}$, ${\mathbb P_{i_0}(A_{i_0})>0}$ 
                $\exists A'\subset A,A'\in\mathcal F_{i_0},\mathbb P_{i_0}(A')>0$ such that
                $A'$ is $(\varphi_2\circ\varphi_1,U_{i_0},W_k)$ balanced.
                Without loss of generality, we can suppose $\exists j_0\in J$ 
                such that $\varphi_1(A)\subset V_{j_0}$.
                Then, $\varphi_2$ being a $\mathbb D^\infty$-morphism, there exists $B\subset\varphi_1(A)$,
                $B$ being $(\varphi_2,V_{j_0},W_k)$-balanced (and $\mathbb P_{j_0}(B)>0$).
                So for $g\in\mathbb D^\infty(W_k)$ there exists an extension 
                of $\left.g\circ\varphi_2\right|_B$, denoted $\widetilde g$, such that $\widetilde g\in\mathbb D^\infty(V_{j_0})$.
                But $\varphi_1$ being a $\mathbb D^\infty$-morphism, there exists a subset 
                $A'$ of $\varphi_1^{-1}(B)$ (remind $\mathbb P_{i_0}(\varphi_1^{-1}(B))>0$), which is
                $(\varphi_1,U_{i_0},V_{j_0})$-balanced.
                So $\left.\widetilde g\circ\varphi_1\right|_{A'}=\left.g\circ\varphi_2\circ\varphi_1\right|_{A'}$, and
                there exists an extension of $\left.\widetilde g\circ\varphi_1\right|_{A'}$ ($A'$ is a
                $(\varphi_1,U_{i_0},V_{j_0})$-balanced subset), denoted $\widehat g$, {which is
                $\in\mathbb D^\infty(U_{i_0})$}; and $\widehat g$ is an extension of $\left.g\circ\varphi_2\circ\varphi_1\right|_{A'}$
    \item[iv)] We have to show that if $i_0\in I$
               and $A\subset U_{i_0}$, $A\in\mathcal F_{i_0}$, $\mathbb P_{i_0}(A)>0$ and $g\in L^0(A)$;
               We suppose that $\forall A_{i_0k}$, $(\varphi_2\circ\varphi_1,U_{i_0},W_k)$-balanced set
               there exists $f\in\mathbb D^\infty(W_k)$ such that
               \begin{equation*}
                 \left.f\circ\varphi_2\circ\varphi_1\right|_{A\cap A_{i_0k}} = \left.g\right|_{A\cap A_{i_0k}}
               \end{equation*}
               Then we must show that $g$ admits an extension $\widetilde g$ 
               such that $\widetilde g\in\mathbb D^\infty(U_{i_0})$ and $\left.\widetilde g\right|_A=g$.
               As in iii) we can suppose without loss of generality
               that there exists $j_0\in J$ such that $\varphi_1(A)\subset V_{j_0}$.
               We know, from Remark 3.3, that there exists 
               a countable partition of $V_{j_0}$ made with {$(\varphi_2,V_{j_0},W_k)$-balanced}
               sets, denoted here $(\varepsilon_\alpha)_{\alpha\in\mathbb N_*}$, such that
               \begin{equation*}
                 \mathbb P_{j_0}\left[V_{j_0}\setminus\bigcup_{\alpha\in\mathbb N_*}\varepsilon_\alpha\right]=0
               \end{equation*}
               We define then $\left.h_\alpha\right|_{\varepsilon_\alpha} = \left.f\circ\varphi_2\right|_{\varepsilon_\alpha}$ and $0$ on $\complement \varepsilon_\alpha$.
               Then $h_\alpha\in L^0(V_{j_0})$; and $h=\sum_{\alpha=1}^\infty \mathbf{1}_{\varepsilon_\alpha}h_\alpha\in J^0(V_{j_0})$.
               $\forall B_{j_0k}$, $(\varphi_2,V_{j_0},W_k)$-balanced set, there exists
               a countable extracted partition of the $(\varepsilon_\alpha)_{\alpha\in\mathbb N_*}$,
               denoted $(\varepsilon_\beta)_{\beta\in\mathbb N_*}$, such that $B_{j_0k}=\bigcup_{\beta\in\mathbb N_*}\varepsilon_\beta$, $\mathbb P_{j_0}$-a.s.
               Then: 
               \begin{align*}
                 \left.h\right|_{V_{j_0}\cap B_{j_0k}}
                 & = \left.h\right|_{B_{j_0k}} \\
                 & = \left.h\right|_{\bigcup_{\beta\in\mathbb N_*}\varepsilon_\beta} \\
                 & = \sum_{\beta\in\mathbb N_*}\mathbf{1}_{\beta}h_\beta \\
                 & = \sum_{\beta\in\mathbb N_*}\left.f\circ\varphi_2\right|_{\varepsilon_\beta} \\
                 & = \left.f\circ\varphi_2\right|_{B_{j_0k}}
               \end{align*}
               $\varphi_2$ being a $\mathbb D^\infty$-morphism, there exists an extension
               $\widetilde h$ of $h$ {which is $\in \mathbb D^\infty(V_{j_0})$}.
               Then:
               \begin{align*}
                 \left.g\right|_{A\cap A_{i_0k}}
                   & = \left.f\circ\varphi_2\circ\varphi_1\right|_{A\cap A_{i_0k}} \\
                   & = \left.f\circ\varphi_2\right|_{\varphi_1(A\cap A_{i_0k})} \\
                   & = \left.\widetilde h\right|_{\varphi_1(A\cap A_{i_0k})}
               \end{align*}
               This last equality being $\mathbb P_{j_0}$-a.s.
               As $A\cap A_{i_0k}\subset\varphi^{-1}\left[\varphi(A\cap A_{i_0k})\right]$, we have
               $\left.g\right|_{A\cap A_{i_0k}}=\left.\widetilde h\circ\varphi_1\right|_{A\cap A_{i_0k}}$, $\mathbb P_{i_0}$-a.s. and $\widetilde h\in\mathbb D^\infty(V_{j_0})$.

               $\varphi_1$ being a $\mathbb D^\infty$-morphism, with iv) we have that
               there exists an extension $\widetilde g\in\mathbb D^\infty(U_{i_0})$ such that
               $\left.\widetilde g\right|_A=\left.g\right|_A$. So $\varphi_2\circ\varphi_1$ is a $\mathbb D^\infty$-morphism.
  \end{enumerate}
\end{proof}

\begin{defn} ~
  \begin{enumerate}
    \item[i)] Two $\mathbb D^\infty$-atlases on the set $\mathscr S$ are said to be
              $\mathbb D^\infty$ equivalent iff the identity $\mathrm{Id}_{\mathscr S}$ is a $\mathbb D^\infty$-isomorphism;
    \item[ii)] Let $\mathscr S$ be a $\mathbb D^\infty$-stochastic manifold with the $\mathbb D^\infty$-atlas
               $\mathscr A=(U_i,b_i,\Omega_i)_{i\in I}$; the chart $(U,b,\Omega,\mathcal F,\mathbb P,H)$, with $U\subset\mathscr S$
               is said to be $\mathbb D^\infty$-compatible with $\mathscr A$, if the atlases
               $\mathscr A$ and $\mathscr A \cup \left\{(U,b,\Omega)\right\}$ are equivalent.
               
  \end{enumerate}
\end{defn}
\begin{rem}
  We will here give the Definition 3.3.iv. 
  but without identifying the domain $U$ of a chart $(U,b,\Omega)$
  with $\Omega$:
  $\forall A\in\Omega,A\in\mathcal F_i,\mathbb P_i(A)>0$ and $\forall g\in L^0(A)$:
  If $\forall A_\ell$, $(\varphi,U_i,V_\ell)$-balanced set, there exists
  $f_\ell\in\mathbb D^\infty(\widetilde \Omega_\ell)$ such that $\left.f\circ\widetilde b_\ell\circ\varphi\circ b_i^{-1}\right|_{A\cap b_i(A_\ell)} = \left.g\right|_{A\cap b_i(A_\ell)}$
  then $g$ admits an extension $\widetilde g\in\mathbb D^\infty(\Omega_i)$.
\end{rem}
\begin{rem}
  If the $\mathbb D^\infty$-atlas $\mathscr A$ verifies the condition iv) of the 
  Definition 3.3, then $\mathbb D^\infty$-compatibility between two charts of $A$ is
  equivalent to: ${\forall\varphi\in\mathbb D^\infty(\Omega_i)}$, $\exists$ an extension $\widetilde\varphi\in\mathbb D^\infty(\Omega_j)$
  of $\left.\varphi\circ b_{ij}\right|_{b_i(U_i\cap U_j)}$, and moreover this extension is unique.
\end{rem} 

  Now we will show that there exists a derivation 
  from $\mathbb D^\infty(\Omega)$ to $\mathbb D^\infty(\Omega)$ which is not a vector field
  so to prove after this, that there are $\mathbb D^\infty$-charts,
  $\mathbb D^\infty$-compatible, which do have linear tangent maps.

\subsection{Existence of a $\mathbb{D}^\infty$-derivation which is not a vector field}

  Let $\mathscr S$ be a $\mathbb D^\infty$-stochastic manifold, with an atlas
  having only one chart $(\Omega, \mathcal F,\mathbb P,H)$.
  Let $\left\{(e_i)_{i\in\mathbb N_*}, (\varepsilon_j)_{j\in\mathbb N_*}\right\}$ be a Hilbert basis of $H$ and let $A$
  be the bounded operator on $H$ defined by:
  \begin{equation*}
    A(e_i)=\varepsilon_i\quad\text{and}\quad A(\varepsilon_j)=-e_j
  \end{equation*}
  We denote by $\delta$ the operator on $\mathbb D^\infty(\Omega)$ defined by:
  \begin{equation*}
    \forall\varphi\in\mathbb D^\infty(\Omega),\quad \delta\varphi = \mathrm{div}\,A\,\mathrm{grad}\,\varphi
  \end{equation*}
  Then direct computation shows that $\delta$ is a derivation.
  If there existed a $\mathbb D^\infty$-vector field $X$ such that
  \begin{equation*}
    \forall\varphi\in\mathbb D^\infty(\Omega), \quad X\cdot\varphi = \mathrm{div}\,A\,\mathrm{grad}\,\varphi=\delta\varphi
  \end{equation*}
  then we write:
  \begin{equation*}
    X = \sum_{i\geq 1}X^ie_i + \sum_{j\geq 1}Y^j\varepsilon_j
  \end{equation*}
  and we would have
  \begin{align*}
    X_i & = X\cdot W(e_i)=\delta(W(e_i)) = W(\varepsilon_i) \\
    \text{and } Y_j &= X\cdot W(\varepsilon_j) = \delta(W(\varepsilon_j))=-W(e_j)
  \end{align*}
  But 
  \begin{align*}
    \sum_{i,j}\left(\int_\Omega X_i^2\mathbb P(\mathrm d\omega)+\int_\Omega Y_j^2\mathbb P(\mathrm d\omega)\right) 
    & = \sum_{i,j}\left(\int_\Omega W(\varepsilon_j)^2\mathbb P(\mathrm d\omega) + \int_\Omega W(e_i)^2\mathbb P(\mathrm d\omega)\right) \\
    & = +\infty
  \end{align*}
  Now we can prove that there exists compatible
  $\mathbb D^\infty$-charts for which, the change maps do not
  admit a linear tangent map.
  Let $\mathscr S_1$ be a $\mathbb D^\infty$-stochastic manifold with an atlas
  reduced to one chart $(\Omega,\mathcal F,\mathbb P,H)$
  and $\left\{(e_i)_{i\in\mathbb N_*}, (\varepsilon_j)_{j\in\mathbb N_*}\right. h\}$ a Hilbert basis of $H$.
  Denote $X_i=W(e_i)$, $Y_j=W(\varepsilon_j)$ and $Z=h$.
  We define an inversible map $\psi$ by:
  \begin{align*}
    \overline X_i & = X_i\cos Z + Y_i\sin Z & X_i & = \overline X_i\cos\overline Z - \overline Y_i\sin\overline Z \\
    \overline Y_i & =-X_i\sin Z + Y_i\cos Z & Y_i & = \overline X_i\sin\overline Z + \overline Y_i\cos\overline Z \\
    \overline Z & = Z & Z & = \overline Z
  \end{align*}
  The system $(\overline X_i,\overline Y_j, \overline Z)$ is a Gaussian system
  and has the same laws as the system $(X_i,Y_j,Z)$.
  We define an isometric morphism again denoted $\psi$,
  from a dense domain of $L^{\infty-0}$ to $L^{\infty-0}$ by:
  \begin{equation*}
    \forall n\in\mathbb N_*, \text{if } f\in\mathscr S(\mathbb R^{2n+1}), 
    \quad (\psi f)(\overline X,\overline Y,\overline Z) = f(\psi(X),\psi(Y),\psi(Z))
  \end{equation*}
  $\mathscr S(\mathbb R^{2n+1})$ being the set of fast decreasing functions.

  This morphism preverses laws, so it can be extended 
  to a bijective and isometric map from $L^{\infty-0}(\Omega)$  into
  $L^{\infty-0}(\Omega)$. $L$ and $\overline L$ being the O.U. operator respectively
  in the charts $(\Omega,\mathcal F,\mathbb P,H)$ and $(\Omega,\overline{\mathcal F},\overline{\mathbb P},\overline H)$, we have:
  \begin{align*}
    {\psi}^{-1}\circ\overline L(\psi f)
      & = Lf + \psi^{-1}\circ\frac{\partial^2(\psi f)}{\partial\overline Z^2}-\psi^{-1}\circ\left(\overline Z\frac{\partial(\psi f)}{\partial\overline Z}\right) \\
      & = Lf + \left(\psi^{-1}\circ\frac{\partial}{\partial\overline Z}\circ\psi\right)\circ\left(\psi^{-1}\circ\frac{\partial}{\partial\overline Z}\circ\psi f\right)-\psi^{-1}\left(\overline Z\frac{\partial(\psi f)}{\partial\overline Z}\right)
      \tag{1}\label{l1}
  \end{align*}
  And
  \begin{align*}
    \psi^{-1}\circ\frac{\partial}{\partial\overline Z}(\psi f)
      & = \sum_{i=1}^nX_i\frac{\partial f}{\partial Y_i}-Y_i\frac{\partial f}{\partial X^i}+\frac{\partial f}{\partial Z} \\
      & = \mathrm{div}\,A_n\,\mathrm{grad}\,f+\frac{\partial f}{\partial Z}
  \end{align*}
  where $A_n$ is the determinist
  operator defined by $A_n(e_i)=\varepsilon_i$, $A_n(\varepsilon_i)=-e_i$, $1,\dotsc,n$.
  If we denote by $\mathrm{div}\,A_n\,\mathrm{grad}$ by $\delta_n$ we have
  \begin{align*}
    \psi^{-1}\circ\frac{\partial}{\partial\overline Z}(\psi f)
      & = \delta_nf + \frac{\partial f}{\partial Z} \\
      & = \mathrm{div}\,A\,\mathrm{grad}\,f + \frac{\partial f}{\partial Z}
  \end{align*}
  where $A$ is the bounded operator on $H$ such that
  $A(e_i)=\varepsilon_i$, $A(\varepsilon_j)=-e_j$, and $A(h)=h$.
  Now we show by induction that $\psi$ sends $\mathbb D^\infty$ in $\mathbb D^\infty$.
  We know already that $\psi$ sends $\mathbb D^\infty$ in $L^{\infty-0}$.
  Suppose that $\psi:\mathbb D^\infty\to\mathbb D_r^\infty$, let $f\in\mathbb D^\infty(\Omega)$.
  From \eqref{l1}, we see that, as $\mathrm{div}\,A\,\mathrm{grad}\,f\in\mathbb D^\infty$
  and $\left(\psi^{-1}\frac{\partial}{\partial\overline Z}\psi\right)(f)$ is also $\in\mathbb D^\infty$:
  $\left(\psi^{-1}\circ\overline L\circ\psi\right)f\in\mathbb D^\infty$ 
  so $(\overline L\circ\psi)(f)\in\mathbb D_r^\infty$ which implies
  that $\psi\in\mathbb D_{r+2}^\infty$.
  And we have seen previously that
  $\mathrm{div}\,A\,\mathrm{grad}$ is a derivation which cannot be a
  vector field. But if the linear tangent map
  of $\psi$ existed we would have
  \begin{equation*}
    T_{psi}\frac{\partial}{\partial\overline Z} = \psi^{-1}\circ\frac{\partial}{\partial\overline Z}\psi
  \end{equation*}
  Now we study the sets $\mathrm{Der}(\Omega)$ and $\mathrm{Der}(\Omega)^*$.
  \begin{defn}
    Given a Gaussian space $(\Omega,\mathcal F,\mathbb P,H)$, we
    denote $\mathrm{Der}(\Omega)$ the set of $\mathbb D^\infty$-continuous derivations
    from $\mathbb D^\infty(\Omega)$ to $\mathbb D^\infty(\Omega)$.
    
    $\mathrm{Der}(\Omega)$ is then a non-metrisable topological space
    when endowed with the single point convergence.
  \end{defn}
  \begin{defn}
    A subset $A$ of $\mathrm{Der}(\Omega)$ is said to be bounded 
    or {$\mathrm{Der}$-bounded} iff $\forall (p,q),\exists (q,s),\,(p,q>1; r,s\in\mathbb N_*)$,
    $\forall \mathbb D^\infty$-bounded subset ${D\subset\mathbb D^\infty(\Omega)}$, \\
	$\exists C(p,q,s,r,s,D)$ a 
    constant such that
    \begin{equation*}
      \forall f\in D, \quad \sup_{\delta\in A}\left\|\delta f\right\|_{\mathbb D_r^p}
      \leq C(p,q,r,s,D)\left\|f\right\|_{\mathbb D_s^q}
    \end{equation*}
  \end{defn}
  \begin{defn}
    We denote by $\mathrm{Der}(\Omega)^*$ the set of
    $\mathbb D^\infty$-linear maps on $\mathrm{Der}(\Omega)$ which are verifying this
    continuity property: let $A$ be a bounded set in $\mathrm{Der}(\Omega)$,
    and $(\delta_i)_{i\in I}$ a net in $A$, converging towards $\delta\in\mathrm{Der}$.
    Then for $u\in\mathrm{Der}(\Omega)^*$, the net $(u(\delta_i))_{i\in I}$ converges
    $\mathbb D^\infty$ towards $u(\delta)$.
  \end{defn}
  \begin{defn}
    A subset $(\alpha_i)_{i\in I}$ of $\mathrm{Der}(\Omega)^*$ is said to be 
    bounded iff for each bounded subset $A\subset\mathrm{Der}(\Omega)$,
    $\sup_{i\in I}|\alpha_i(A)|<+\infty$.
  \end{defn}
  \begin{lem}
    Let $u\in\mathrm{Der}(\Omega)^*$. Then $\forall$ bounded subset $A$ 
    of $\mathrm{Der}(\Omega)$, the set $\{u(\delta)\,/\,\delta\in A\}$ is $\mathbb D^\infty$-bounded in $\mathbb D^\infty(\Omega)$.
  \end{lem}
  \begin{proof}
    Suppose $\exists A$ subset bounded in $\mathrm{Der}(\Omega)$, $\exists(p,r),p>1,r\in\mathbb N_*$ 
    and $\exists\delta_n\in A$ such that $\|u(\delta_n)\|_{\mathbb D_r^p} > n$.
    Let $\alpha_n$ a sequence of numbers $>0$, which converges towards $0$. 
    Then $\{\alpha_n\delta_n\,/\,n\in\mathbb N_*\}$ is bounded in $\mathrm{Der}(\Omega)$ and $\alpha_n\delta_n\to 0$ in $\mathrm{Der}$
    so $u(\alpha_n\delta_n)\to 0$ which is contradictory.
  \end{proof}
  \begin{rem}
  If $(f_i)_{i\in I}$ is a $\mathbb D^\infty$-bounded family of $\mathbb D^\infty(\Omega)$,
  then $(\mathrm{grad}\,f_i)_{i\in I}$ is a bounded family in $\mathrm{Der}(\Omega)$.
  \end{rem}
  
\subsection{Derivation field on a $\mathbb{D}^\infty$-stochastic manifold}
  
  Given a $\mathbb D^\infty$-stochastic manifold, endowed with the
  atlas $\mathscr A=(U_i,b_i,\Omega_i)_{i\in I}$, a family of $\mathbb D^\infty$-continuous derivations
  $\delta_i\in\mathrm{Der}(\Omega_i)$, $i\in I$, is said to be a {$\mathbb D^\infty$-derivation} field
  on $\mathscr S$ iff $\forall(i,j)\in I^2$ and $\forall f\in\mathbb D^\infty(\Omega_j)$:
  \begin{equation*}
    \left.\delta_jf\right|_{b_j(U_i\cap U_j)} = \left.\delta_i\left(\widetilde{f\circ b_{ij}}\right)\right|_{b_i(U_i\cap U_j)}\circ b_{ji}
  \end{equation*}
  $b_{ij}$ and $b_{ji}$ being the $\mathbb D^\infty$-chart changes between the
  charts $(U_i,b_i,\Omega_i)$ and $(U_j,b_j,\Omega_j)$, $\widetilde{f\circ b_{ij}}$ being a
  $\mathbb D^\infty$-extension of $\left.f\circ b_{ij}\right|_{b_i(U_i\cap U_j)}$.

  This definition is legitimate:
  If $\widetilde{f\circ b_{ij}}^{(1)}$ and $\widetilde{f\circ b_{ij}}^{(2)}$ are two extensions on
  $\mathbb D^\infty(\Omega_i)$, of $\left.f\circ b_{ij}\right|_{b_i(U_i\cap U_j)}$, we have:
  \begin{equation*}
    \left.\left(\widetilde{f\circ b_{ij}}^{(1)}-\widetilde{f\circ b_{ij}}^{(2)}\right)\right|_{b_i(U_i\cap U_j)}=0
  \end{equation*}
  So with Corollary 2.5, we have:
  \begin{equation*}
    \delta_i\left.\left(\widetilde{f\circ b_{ij}}^{(1)}\right)\right|_{b_i(U_i\cap U_j)} = \delta_i\left.\left(\widetilde{f\circ b_{ij}}^{(2)}\right)\right|_{b_i(U_i\cap U_j)}
  \end{equation*}
  Now, using the definition of an admissible $\mathbb D^\infty$-chart
  to a $\mathbb D^\infty$-atlas, we prove that the definition of a
  derivation field is consistent, that is: we can build
  on this $\mathbb D^\infty$-admissible chart a derivation such that
  the new derivation field (the initial one $+$ this new
  derivation) has the same action as the first derivation field.
  
  From this we can deduce that if we have a derivation 
  field associated to a $\mathbb D^\infty$-atlas, on another equivalent
  $\mathbb D^\infty$-atlas can be built a derivation field which
  has the same action as the initial one.
  
  Let $(U,b,\Omega,\mathcal F,\mathbb P)$ be a $\mathbb D^\infty$-admissible chart
  to the $\mathbb D^\infty$-atlas $\mathscr A = (U_i,b_i,\Omega_i,\mathcal F_i,\mathbb P_i)_{i\in I}$,
  then the identity is a $\mathbb D^\infty$-isomorphism between $\mathscr A$. and the atlas \\
  $\mathscr A \cup \{(U, b, \Omega,\mathcal F, \mathbb P)\}$.

  $\Omega$   can be covered by a countable collection of sets $b(U\cap U_j)$, $j\in J$. Denote by $\varphi_j$ the map change
  of charts between $\Omega_j$ and $\Omega$, $\varphi_j=b\circ b_j^{-1}$.
  
  $\forall j\in J$, we define $\left.\delta f\right|_{b(U\cap U_j)}$, $\forall f\in\mathbb D^\infty(\Omega)$, by:
  \begin{equation*}
    \left.\delta f\right|_{b(U\cap U_j)} = \left.\widetilde{\delta_j(f\circ\varphi_j)\circ\varphi_j^{-1}}\right|_{b(U\cap U_j)}
  \end{equation*}
  The symbol $\widetilde{~}$ is for the $\mathbb D^\infty(\Omega)$ extension of
  $\delta_j(f\circ\varphi_j)\circ\varphi_j^{-1}$ which exists because the chart
  $(U,b,\Omega)$ is $\mathbb D^\infty$-admissible to the $\mathbb D^\infty$-atlas $\mathscr A$.
  
  Then 
  \begin{equation*}
    \left.\delta f\right|_{b(U\cap U_j)} = \left.\widetilde{\delta_j(f\circ\varphi_j)\circ\varphi_j^{-1}}\right|_{b(U\cap U_j)}\circ\mathrm{Id}_{\mathscr S}
  \end{equation*}
  Then we define:
  \begin{equation*}
    \delta f = \sum_{j\in J}\mathbf{1}_{b(U\cap U_j)}\cdot\left.\delta f\right|_{b(U\cap U_j)}
  \end{equation*}
  The definition is legitimate because if $\omega\in b(U\cap U_{j_1})\cap b(U\cap U_{j_2})$
  the map change of charts shows:
  \begin{equation*}
    \left.\delta f\right|_{b(U\cap U_{j_1})}(\omega) = \left.\delta f\right|_{b(U\cap U_{j_2})}(\omega)
  \end{equation*}

\subsection{Metric and fundamental bilinear form on a $\mathbb{D}^\infty$-stochastic manifold}

  \begin{defn}
    Let $(\Omega,\mathcal F,\mathbb P,H)$ be a Gaussian space.
    A $\mathbb D^\infty$-valued bilinear form on $(\mathrm{Der\,\Omega})^*$ is
    a $\mathbb D^\infty$-bilinear form on $\mathrm{Der}\,\Omega$ 
    denoted $q$, which is continuous relatively to each of
    its arguments. $q$ is said to be positive definite if $\alpha\in\left(\mathrm{Der}\,\Omega\right)^*$ is such that if $q(\alpha,\alpha)=0$ then $\alpha=0$.
  \end{defn}

  \begin{rem}
    The continuity of $q$ means that if a 
    net $(\alpha_i)_{i\in I}$, included in a bounded part of $(\mathrm{Der}\,\Omega)^*$ 
    converges towards $\alpha\in(\mathrm{Der}\,\Omega)^*$, then  $\forall\beta\in(\mathrm{Der}\,\Omega)^*$, 
    $q(\alpha_i,\beta)$ converges in $\mathbb D^\infty$ towards $q(\alpha,\beta)$.
  \end{rem}
  
  \begin{notation}
    Let $q$ be a bilinear form on $(\mathrm{Der}\,\Omega)^*$,
    \begin{enumerate}
      \item[i)] $\forall f\in\mathbb D^\infty(\Omega)$, we denote $\lambda_f\in(\mathrm{Der}\,\Omega)^*$ defined by:
        \begin{equation*}
          \forall\delta\in\mathrm{Der}, \lambda_f(\delta) = \delta(f) \qquad (\in\mathbb D^\infty(\Omega))
        \end{equation*}
      \item[ii)] If $u\in(\mathrm{Der}\,\Omega)^*$, we denote $\delta_u\in\mathrm{Der}\,\Omega$ defined by:
        \begin{equation*}
          \forall f\in\mathbb D^\infty(\Omega), \delta_uf=q(\lambda_f,u)
        \end{equation*}
    \end{enumerate}
  \end{notation}
  
  \begin{defn}
    The fundamental bilinear form on $(\mathrm{Der}\,\Omega)^*$,
    also named the fundamental metric, denoted $q_0$, is defined
    by: if $(e_i)_{i\in\mathbb N_*}$ is a Hilbert basis of $H$, and $\alpha,\beta\in(\mathrm{Der}\,\Omega)^*$,
    \begin{equation*}
      q_0(\alpha,\beta) = \sum_{i=1}^\infty\alpha(e_i)\beta(e_i)
    \end{equation*}
  \end{defn}

    We have to show that this series is $\mathbb D^\infty$-convergent
    and that this definition does not depend on the
    choice of basis $(e_i)_{i\in\mathbb N_*}$.
  
  \begin{rem}
    $q_0$ being the fundamental metric, we have:
    \begin{equation*}
      \forall u\in(\mathrm{Der}\,\Omega)^*, \delta_u(f)=u(\mathrm{grad}\,f)
    \end{equation*}
  \end{rem}
  
  \begin{thm}
    The fundamental form $q_0$ is well defined
    on $(\mathrm{Der}\,\Omega)^*$, and if $\alpha, \beta\in(\mathrm{Der}\,\Omega)^*$, we have
    $q_0(\alpha,\beta)=\alpha(\delta_\beta)=\beta(\delta_\alpha)$; and $q_0$ is non degenerate.
  \end{thm}
  
  \begin{proof} ~ 
    \begin{enumerate}
      \item[i)]
        We know that (Theorem 2.7) given $\alpha\in(\mathrm{Der}\,\Omega)^*$,
        the sequence 
        \begin{equation*}
          {X_N=\sum_{k=1}^N\mathrm{E}\left[\alpha(e_k)\,|\,\mathcal F_n\right]e_k}
        \end{equation*}
        converges
        towards $\delta_\alpha$ ($(e_k)_{k\in\mathbb N_*}$ is a Hilbert basis of $H$, as usual).
        Then $\alpha(X_N)\xrightarrow{\mathbb D^\infty}\alpha(\delta_\alpha)$ 
        implies
        \begin{equation*}
          \mathrm{E}\left[\alpha(X_N)\,|\,\mathcal F_N\right]\xrightarrow{\mathbb D^\infty}\alpha(\delta_\alpha)
        \end{equation*}
        with $\delta_\alpha[W(e_k)]=\alpha(e_k)$, we have
        \begin{equation*}
          \sum_{k=1}^N\mathrm{E}\left[\alpha(e_j)\,|\,\mathcal F_N\right]^2\xrightarrow{\mathbb D^\infty}\alpha(\delta_\alpha)
          \tag{2}\label{l2}
        \end{equation*}
        which implies
        \begin {equation*}
          \sum_{k=1}^N\alpha(e_j)^2 < +\infty
          \tag{3}\label{l3}
        \end{equation*}
        So the definition of $q_0$ is legitimate.
      \item[ii)]
        From $X_N\xrightarrow{\mathbb D^\infty}\delta_\alpha$, we deduce: $\forall p>1$,
        $\beta(X_N)\xrightarrow{L^p}\beta(\delta_\alpha)$,
        then 
        \begin{equation*}
        \forall\varepsilon>0,\exists N_0>0,\forall k\in\mathbb N_*, \forall \ell\in\mathbb N_*,
        \left\|\beta(\delta_\alpha)-\beta\left(X_{N_0+k+\ell}\right)\right\|_{L^p} \leq \varepsilon
        \end{equation*}
        So
        \begin{align*}
          \Bigg\| \beta(\delta_\alpha) & -\sum_{j=1}^{N_0+k}\beta(e_j)\mathrm{E}\left[\alpha(e_j)\,|\,\mathcal F_{N_0+k+\ell}\right] \\
          & -\sum_{i=1}^\ell \beta\left(e_{N_0+k+i}\right)\mathrm{E}\left[\alpha(e_{N_0+k+i})\,|\,\mathcal{F}_{N_0+k+\ell}\right] \Bigg\|_{L^p} \leq \varepsilon
        \end{align*}
        Then
        \begin{align*}
          \Bigg\| \beta(\delta_\alpha) & -\sum_{j=1}^{N_0+k}\beta(e_j)\mathrm{E}\left[\alpha(e_j)\,|\,\mathcal F_{N_0+k+\ell}\right]\Bigg\|_{L^p} \\
          & \leq \varepsilon + \left\|\sum_{i=1}^\ell \beta\left(e_{N_0+k+i}\right)\mathrm{E}\left[\alpha(e_{N_0+k+i})\,|\,\mathcal{F}_{N_0+k+\ell}\right] \right\|_{L^p}
        \end{align*}
        The r.h.s. is lower or equal to
        \begin{align*}
          \varepsilon 
           &  
             + \left[
               \int \left(\sum_{i=1}^\ell \beta(e_{N_0+k+i})^2\right)^{\frac{p}{2}}
               \left(\sum_{i=1}^\ell\mathrm{E}\left[\alpha(e_{N_0+k+i})\,|\,\mathcal F_{N_0+k+\ell}\right]^2\right)^{\frac{p}{2}}
             \right]^{\frac{1}{p}} \\
           & \leq \varepsilon +
             \left[ \int\left(\sum_{i=1}^\ell\beta(e_{N_0+k+i})^2\right)^p\right]^{\frac{1}{2p}}
             \times
             \left[\int\left(\sum_{i=1}^\ell\mathrm{E}\left[\alpha(e_{N_0+k+i})\,|\,\mathcal F_{N_0+k+\ell}\right]^2\right)^p\right]^{\frac{1}{2p}}
        \end{align*}
        From \eqref{l2} we know that the series $\sum_{k=1}^\infty\mathrm{E}[\alpha(e_j)|\mathcal F_N]^2$ is
        convergent so we can find $\ell_0$ such that for all $\ell \geq \ell_0$,
        \begin{equation*}
          \left[\int\left(\sum_{i=1}^\ell\mathrm{E}\left[\alpha(e_{N_0+k+i})\,|\,\mathcal F_{N_0+k+\ell}\right]^2\right)^p\right]^{\frac{1}{2p}} \leq \varepsilon^{\frac12}
        \end{equation*}
        Now if we write $Y_N=\sum_{i=1}^N\mathrm{E}[\beta(e_i)|\mathcal F_N]e_i$
        repeating the same calculus than in $i)$, we get
        that the series $\sum_{j=1}^\infty\beta(e_j)^2$ is convergent; so we can find $\ell_1$
        such that for all $\ell > \ell_1$,
        \begin{equation*}
          \left[\int\left(\sum_{i=1}^\ell\beta(e_{N_0+k+i})^2\right)^p\right]^{\frac{1}{p}}\leq\sqrt{\varepsilon}
        \end{equation*}
        so
        \begin{equation*}
          \lim_{\ell\uparrow\infty} \left\|\beta(\delta_\alpha) - \sum_{j=1}^{N_0+k}\beta(e_j)\mathrm{E}\left[\alpha(e_j)\,|\,\mathcal F_{N_0+k+\ell}\right]\right\|\leq 2\varepsilon
        \end{equation*}
        So $q_0(\alpha,\beta) = \beta(\delta_\alpha)$ which proves that $q_0(\alpha,\beta)\in\mathbb D^\infty(\Omega)$,
        that $q_0$ is continuous for each of its arguments,
        and $q_0(\alpha,\alpha)=0$ implies $\alpha=0$.
    \end{enumerate}
  \end{proof}
  \begin{cor}
    The map $\left(\mathrm{Der}\,\Omega\right)^*\ni\alpha \to \delta_\alpha\in\mathrm{Der}\,\Omega$ is injective
  \end{cor}
  \begin{proof}
    If $\delta_\alpha = \delta_\beta$, $\forall f\in\mathbb D^\infty(\Omega)$, 
    $q_0(\alpha-\beta,\lambda_f) = \lambda_f(\delta_\alpha-\delta_\beta) = 0$.
    So with $f = W(e_k)$, $q_0(\alpha-\beta,\lambda_{W(e_k)}) = \sum_{j=1}^\infty[\alpha(e_j)-\beta(e_j)]\delta_{jk} = 0$
  \end{proof}
  
  \begin{rem}
    $\forall f,g\in\mathbb D^\infty(\Omega),$
    \begin{align*}
      q_0(\lambda_f,\lambda_g) 
        & = \sum_{j=1}^\infty \lambda_f(e_j)\lambda_g(e_j) \\
        & = \sum_{j=1}^\infty \langle e_j,\mathrm{grad}\,f\rangle_H\langle e_j,\mathrm{grad}\,g\rangle_H \\
        & = \langle \mathrm{grad}\,{f},\mathrm{grad}\,g\rangle_H
    \end{align*}
  \end{rem}
  
  \begin{defn}
    Let $(\Omega,\mathcal F,\mathbb P,H)$ be a Gaussian space and $q$
    a $\mathbb D^\infty$-bilinear form on $\left(\mathrm{Der}\,\Omega\right)^*$, $\mathbb D^\infty$-valued, continuous
    for each of its arguments.
    We define a map $T_q$ from $\left(\mathrm{Der}\,\Omega\right)^*$ to $\mathrm{Der}\,\Omega$ by
    \begin{equation*}
      \forall u\in\left(\mathrm{Der}\,\Omega\right)^*,\forall f\in\mathbb D^\infty(\Omega), \quad
      (T_qu)\cdot f = q(u,\lambda_f)
    \end{equation*}
    and we denote $\mathscr D_q = \mathrm{range}\,T_q = T_q(\left(\mathrm{Der}\,\Omega\right)^*)$.
  \end{defn}
  
  \begin{lem} ~ 
    \begin{enumerate}
      \item[i)] $T$ is continuous from $\left(\mathrm{Der}\,\Omega\right)^*$ to $\mathrm{Der}\,\Omega$;
      \item[ii)] $T_q(\lambda_f)=\delta_f$, ($\delta_f(g) = \langle\mathrm{grad}\,f,\mathrm{grad}\,g\rangle_H$)
    \end{enumerate}
  \end{lem}
  
  \begin{proof} ~
    \begin{enumerate}
      \item[i)] If $(u_i)_{i\in I}$ is a net in a bounded part of $\left(\mathrm{Der}\,\Omega\right)^*$ 
        converging towards $u\in\left(\mathrm{Der}\,\Omega\right)^*$, we have $\forall g\in\mathbb D^\infty(\Omega)$:
        \begin{equation*}
          \left\|[T_q(u_i)-T_q(u)]\cdot g\right\|_{\mathbb D_r^p} = \left\|q(u_i-u,\lambda_f)\right\|_{\mathbb D_r^p}
        \end{equation*}
      \item[ii)] Straightforward calculus.
    \end{enumerate}
  \end{proof} 
  Now there is a result, difficult to prove:
  
  \begin{thm}
    If $u\in\left(\mathrm{Der}\,\Omega\right)^*$, then there is a bounded net $(u_i)_{i\in\mathbb N_*}$,
    $u_i\in\left(\mathrm{Der}\,\Omega\right)^*$, such that 
    \begin{enumerate}
      \item[i)] $(u_i)_{i\in I}$ is bounded in $\left(\mathrm{Der}\,\Omega\right)^*$
      \item[ii)]$\forall i\in I, u_i=\sum_{j_i\in A_i}f_{j_i}\lambda_{g_{j_i}}$, $A_i$ being a finite subset
        of $\mathbb N_*$, ${f_{j_i}, g_{j_i}\in\mathbb D^\infty(\Omega)}$ 
      \item[iii)] The net $(u_i)_{i\in I}$ converges towards $u$ in $\left(\mathrm{Der}\,\Omega\right)^*$.
    \end{enumerate}
  \end{thm}
  
  To prove this result we need a lemma:
  \begin{lem}
    Let $(\Omega,\mathcal F,\mathbb P,H)$ be a Gaussian space, $(e_i)_{i\in\mathbb N_*}$
    a Hilbert basis of $H$; denote $\mathcal F_N = \sigma[W(e_1),\dotsc,W(e_N)]$ 
    the  $\sigma$-algebra generated by $\sigma(W(e_i))$, $i\in\{1,\dotsc,N\}$
    and by $\mathcal F_N^\bot$ the $\sigma$-algebra $\sigma[W(e_{N+1}),\dotsc]$.
    
    Let $\delta\in\mathrm{Der}\,\Omega$; with 
    \begin{equation*}
      X_M=\sum_{j=1}^M\mathrm{E}[\delta(W(e_j))|\mathcal F_M]e_j
    \end{equation*}
    and 
    \begin{equation*}
      V_N = \sum_{i=1}^N\mathrm{E}[u(e_i)|\mathcal F_N]\lambda_{W(e_i)} \quad (V_N\in\left(\mathrm{Der}\,\Omega\right)^*)
    \end{equation*}
    we have $\sup_{N,M}\left\|V_N(X_M)\right\|_{\mathbb D_r^p} < +\infty$ and $\sup_N\|V_N(\delta)\|_{L^2}<+\infty$.
    We remind that we denote by $(e_i)_{i\in\mathbb N_*}$ a Hilbert basis of $H$.
  \end{lem}
  \begin{proof} ~
    \begin{enumerate}
      \item[a)]
        We denote by $\theta_N$ the $L^2$-isometric map:
        \begin{align*}
          \theta_N[W(e_i)] & = W(e_i), \quad i \leq N \\
          \theta_N[W(e_{N+i})] & = W(e_{N+i+1})
        \end{align*}
        then we extend $\theta_N$ 
        on the set of polynomials in Gaussian variables,
        by considering this extension of $\theta_N$ as a morphism, and
        then extend this $\theta_N$ again to $L^2(\Omega)$, thanks that $\theta_N$
        leaves laws invariants; and $\theta_N$ commutes with the O.U. operator.
        
        We will show in this section that
        \begin{equation*}
          \forall f\in L^2(\Omega),\quad \lim_{k\uparrow\infty}\left[\frac{1}{k+1}\sum_{j=0}^k\theta_N^jf\right] = \mathrm E[f|\mathcal F_N]
          \tag{4}\label{l4}
        \end{equation*}
        the limit being $L^2(\Omega)$.
        
        It is enough to prove that this is true for a dense
        subset of $L^2$, and we choose the subset composed
        by finite linear combinations of products of
        Hermite polynomials $P$ in Gaussian variables, $P_1$ and $P_2$,
        $P_1$ and $P_2$ being polynomials on Gaussian variables
        respectively in $\mathcal F_N$ and $\mathcal F_N^\bot$.
        
        If $P = P_1\times P_2$, $P_2$ being a constant, the result is
        trivial and is independent of $N$.
        
        If $P_2$ is not a constant, let $\alpha=\max |r_1-r_2|$,
        $r_1$ and $r_2$ being rhe indices of the Gaussian variables
        appearing in $P_1$ and $P_2$, then 
        \begin{equation*}
          \mathrm{E}[P_1\times P_2|\mathcal F_N] = 0
        \end{equation*}
        Let $m_0\in\mathbb N, m_0 > \alpha +1$.        
        Then $\forall(b,d)\in\mathbb N_*^2, b\neq d$:
        \begin{equation*}
          \left\langle \theta_N^{m_0b}(P_1\times P_2),\theta_N^{m_0d}(P_1\times P_2)\right\rangle_H=0
          \tag{5}\label{l5}
        \end{equation*}
        We have, with $k>m_0$, $k=m_0\alpha+r$, $0 \leq r\leq m_0-1$. 
        
        Then:
        \begin{align*}
          & \left\|
            \frac{1}{k+1} \sum_{j=0}^k\theta_N^j(P_1\times P_2)
          \right\|_{L^2} \\ 
          & \qquad \leq \frac{1}{k+1} \left\{\left\|\sum_{\beta=0}^{m_0-1}\sum_{\gamma=0}^{\alpha-1}\theta_N^\beta\theta_N^{m_0 \gamma}(P_1\times P_2) \right\|_{L^2}  +  \left\| \theta_N^{m_0\alpha}\left(\sum_{\ell=1}^r\theta_N^\ell(P_1\times P_2)\right) \right\|_{L^2} \right\} \\
          & \qquad \leq \frac{1}{k+1}\left\{ m_0 \left\|\sum_{\gamma=0}^{\alpha-1}\theta_N^{m_0\gamma}\theta_N^{m_0\gamma}(P_1\times P_2)\right\|_{L^2} + r \left| P_1 \times P_2 \right\|_{L^2}\right\}
        \end{align*}
        So:
        \begin{equation*}
          \left\|\frac{1}{k+1}\sum_{j=0}^k\theta_N^j(P_1\times P_2)\right\|_{L_2}
          \leq \frac{1}{k+1}[m_0\sqrt{\alpha} + r] \left\|P_1 \times P_2\right\|_{L^2}
        \end{equation*}
        which converges towards 0.
      \item[b)] 
        Let $\widetilde\theta_N: \mathrm{Der}\,\Omega\to\mathrm{Der}\,\Omega$ by: 
        \begin{equation*}
          \forall\delta\in\mathrm{Der}\,\Omega, \forall f\in\mathbb D^\infty(\Omega):\quad
          \widetilde\theta_N(\delta)\cdot f = \theta_N^{-1}[\delta(\theta_Nf)]
        \end{equation*}
        Then direct calculus shows that:
        \begin{equation*}
          \left(\widetilde\theta_N\right)^n = \widetilde{\left(\theta_N^n\right)}
        \end{equation*}
        Now let $\widehat\theta_N: \left(\mathrm{Der}\,\Omega\right)^*\to \left(\mathrm{Der}\,\Omega\right)^*$ defined by:
        \begin{equation*}
          \forall u\in\left(\mathrm{Der}\,\Omega\right)^*, \forall \delta\in\mathrm{Der}\,\Omega:\quad
          \left(\widehat\theta_Nu\right)\cdot\delta = \theta_N[u(\widetilde\theta_N\delta)]
        \end{equation*}
        One can check that $\widehat\theta_N$ is $\mathbb D^\infty$-linear and that
        \begin{equation*}
          \widehat\theta_N : \left(\mathrm{Der}\,\Omega\right)^*\to\left(\mathrm{Der}\,\Omega\right)^*
        \end{equation*}       
      \item[c)]
        If $A \subset \mathrm{Der\,\Omega}$ is a bounded subset of $\mathrm{Der}\,\Omega$, 
        then the set ${\{(\widehat\theta_Nu^n)\cdot\delta\,/\,n\in\mathbb N,\delta\in A\}}$ is a
        $\mathbb D^\infty$-bounded subset of $\mathbb D^\infty(\Omega)$. For $\forall (p,r)$
        \begin{align*}
          \sup_n\sup_{\delta\in A} \left\|\left(\widehat\theta_Nu^n\right)\cdot\delta\right\|_{\mathbb D_r^p} 
          & = \sup_n\sup_{\delta\in A} \left\|\theta_N^n\left(u(\widetilde\theta^n\delta)\right)\right\|_{\mathbb D_r^p} \\
          & = \sup_n\sup_{\delta\in A} \left\|u\cdot(\left(\widetilde\theta_N^n(\delta)\right)\right\|_{\mathbb D_r^p}
        \end{align*}
        and it is easy to check that $\{\widetilde\theta_N^n(\delta)\,/\,n\in\mathbb N, \delta\in A\}$
        is a bounded subset in $\mathrm{Der}\,\Omega$
        
      \item[d)]
        Now we will show that $i\leq N$ and $W(e_i)\in\mathcal F_N$ implies
        \begin{equation*}
          \lim_{k}\left[\frac{1}{k+1}\sum_{j=0}^k\left(\widehat\theta^ju\right)(e_i)\right]
           = \mathrm{E}[u(e_i)|\mathcal F_N]
           \tag{6}\label{l6}
        \end{equation*}
        the limit being $L^2(\Omega)$, and that for $e_{N+i}$, $W(e_{N+i})\in\mathcal F_N^\bot$ 
        ($i\in\mathbb N_*)$ then the limit of the above sum is $0$.
        \begin{enumerate}
          \item[d.1)] First we prove that
            \begin{equation*}
              \forall\ell\leq N: \widetilde\theta_N^j(e_\ell)=e_\ell\quad(\forall j\in\mathbb N_*)
            \end{equation*}
            We denote by $e_a$ a basis vector with $a\leq N$ 
            and by $e_b$ a basis vector with $b>N$.
            $P[W(e_a),W(e_b)]$ being a polynomial built on
            Gaussian variables belonging to $\mathcal F_N$ (the $W(e_a)$) and
            belonging to $\mathcal F_N^\bot$ (the $W(e_b)$), we have:
            \begin{align*}
              \widetilde{\theta_N}^j(e_\ell)\cdot\left(P[W(e_a),W(e_b)]\right)
              & = \theta_N^{-j}\left\{e_\ell\cdot\theta_N^j\left(P[W(e_a),W(e_b)]\right)\right\} \\
              & = \theta_N^{-j}\left\{e_\ell\cdot P[W(e_a),W(e_{b+j}]\right\} \\
              & = e_\ell\cdot P[W(e_a),W(e_b)]
            \end{align*}
            which proves $\widetilde\theta_N^j(e_\ell) = e_\ell$.

            Then \eqref{l6} becomes:
            \begin{equation*}
              \lim_k\left[\frac{1}{k+1}\sum_{j=0}^k(\widehat\theta^ju)(e_i)\right]
               = \lim_k \left[\frac{1}{k+1}\sum_{j=0}^k\theta_N^j(u(e_i))\right]
            \end{equation*}
            And with \eqref{l4} (from a)) we get:
            \begin{equation*}
              \lim_k\left[\frac{1}{k+1}\sum_{j=0}^k(\widehat\theta^ju)(e_i)\right]
               = \mathrm E[u(e_i)|\mathcal F_N]
            \end{equation*}
            
          \item[d.2)] For this case, where we consider $e_{N+i}$,
            (then $W(e_{N+i})\in\mathcal F_N^\bot$), we first use a bijection
            between $\{N+1, N+2, \dotsc\}$ and $\mathbb Z$; then $\theta$ is rewritten
            \begin{align*}
              \theta(f) & = f \qquad\text{if $f \in\mathcal F_N$} \\
              \theta(W(e_b)) & =W(e_{b+1}) \qquad\text{where now $b\in\mathbb Z$} \\
              \mathcal F_N^\bot & = \sigma\left[W(e_b)\,/\,b\in\mathbb Z\right]
            \end{align*}
            Then $\theta$ is again extended as a morphism from $L^2(\Omega)$
            to $L^2(\Omega)$, which is unitary on $L^2(\Omega)$, commutes
            with the O.U. operator, and leaves invariant laws and
            the chaos $\mathscr C_n$; and $X_N$ can be rewritten:
            \begin{equation*}
              X_M = \sum_{i=1}^N \mathrm E[\delta(W(e_i))|\mathcal F_M]e_i
              + \sum_{b\in B}\mathrm E[\delta(W(e_b))|\mathcal F_M]e_b
            \end{equation*}
            where $B$ is a finite subset of $\mathbb Z$.
          
            As in d.1) we show first that
            \begin{equation*}
              \widetilde\theta_N^j (e_\ell) = e_{\ell-j}\quad \forall j\in\mathbb N
            \end{equation*}
            with $e_\ell\in\mathcal F_N^\bot$, and with the new definition of $\mathcal F_N^\bot$
            we have $e_{\ell-j}\in\mathcal F_N^\bot$. Direct calculus proves:
            \begin{align*}
              \widetilde\theta_N^j(e_\ell)\cdot W(e_a)
               & = 0 \\
               & = e_{\ell-j}\cdot W(e_a) \\
               & = \langle e_{\ell-j}, e_a\rangle_H
            \end{align*}
            with $W(e_a)\in\mathcal F_N$.

            With $W(e_b)\in\mathcal F_N^\bot$ we have
            \begin{equation*}
              \widetilde\theta_N^j(e_\ell)(W(e_b))
               = \theta_N^{-j}[e_\ell\cdot W(e_{b+j})]
               = \theta_N^{-j}(\delta_{b+j}^\ell)
               = \delta_{b+j}^\ell 
            \end{equation*}
            $\delta_{b+j}^\ell$ being the Kronecker symbol.
            
            So $\widetilde\theta_N^j(e_\ell)=e_{\ell-j}$.
            
            Then \eqref{l6} becomes, with $W(e_\ell)\in\mathcal F_N^\bot$:
            \begin{equation*}
              \lim_k\left[\frac{1}{k+1}\sum_{j=0}^k(\widehat\theta_N^ju)(e_\ell)\right]
              = \lim_k\left[\frac{1}{k+1}\sum_{j=0}^k\theta_N^j(u(e_{\ell-j}))\right]
            \end{equation*}
			so :
            \begin{equation*}
              \left\| \lim_k\left[ \frac{1}{k+1}\sum_{j=0}^k(\theta_N^ju)(e_\ell)\right] \right\|_{L^2(\Omega)}
			  \leq \lim_k \left[\frac{1}{k+1} \left\|\sum_{j=0}^ku(e_{\ell-j})\right\|_{L^2} \right]
            \end{equation*}
            But we know with Theorem 3.1 that 
            \begin{align*}
              q_0(u,u) & < \infty \\
              q_0(u,u) & = \sum_{i\leq N} u(e_i)^2 + \sum_{\ell \in\mathbb Z}u(e_\ell)^2
            \end{align*}
            Using this result, we deduce that \eqref{l6} in the case
            of $e_\ell$, with ${W(e_\ell)\in\mathcal F_N^\bot}$, converges $L^2$ towards 0.
          \end{enumerate}
          
          We recapitulate:
            \begin{equation*}
              \lim_k\left[\frac{1}{k+1}\sum_{j=0}^k\left(\widehat\theta_N^ju\right)(e_i)\right]
               = \mathrm E[u(e_i)|\mathcal F_N]\qquad\text{for $i \leq N$}
               \tag{7}\label{l7}
            \end{equation*}
          and
            \begin{equation*}
              \lim_k\left[\frac{1}{k+1}\sum_{j=0}^k\left(\widehat\theta_N^ju\right)(e_\ell)\right] = 0
              \tag{8}\label{l8}
            \end{equation*}
          for $e_\ell$ such that $W(e_\ell)\in\mathcal F_N^\bot$, 
          the convergenge being $L^2$ and being independent
          of $N$.
        
      \item[e)] We compute $V_N(X_M)$:
        \begin{equation*}
          V_N(X_M) = \sum_{i=1}^N\mathrm E[u(e_i)|\mathcal F_N]\lambda_{W(e_i)}(X_M)
          \tag{9}\label{l9}
        \end{equation*}
        From \eqref{l7} and \eqref{l9}, we get:
        \begin{equation*}
          V_N(X_M) = \sum_{i=1}^N \lim_k\left[\frac{1}{k+1}\sum_{j=1}^k(\widehat\theta_N^ju)(e_i)\right]\cdot\lambda_{W(e_i)}(X_M)
        \end{equation*}

     As each $\widehat\theta_N^ju\in\left(\mathrm{Der}\,\Omega\right)^*$ and is $\mathbb D^\infty$-linear, we have
     \begin{equation*}
       V_N(X_M) 
        = \lim_k\left[\frac{1}{k+1}\sum_{j=1}^k\left(\widehat\theta_N^ju\right)\left(\sum_{i=1}^N\lambda_{W(e_i)}(X_M)\cdot e_i\right) \right]
       \tag{9}
       \label{l9}
     \end{equation*}
     From \eqref{l8}, we know that
     \begin{equation*}
       \lim_k\left[\frac{1}{k+1}\sum_{j=1}^k(\widehat\theta_N^ju)(e_b)\right]=0
     \end{equation*}
     $X_M$ can be decomposed, with respect to $\mathcal F_N$ and $\mathcal F_N^\bot$:
     \begin{equation*}
       X_M = \sum_{i=1}^N\mathrm E[\delta(W(e_i))|\mathcal F_M]e_i 
           + \sum_{b\in B}\mathrm E[\delta(W(e_k))|\mathcal F_M]e_b
     \end{equation*}
     with $\mathrm{Card}\,B$ finite. Using \eqref{l8}, we can rewrite \eqref{l9} as:
     \begin{equation*}
       V_N(X_M)
        = \lim_k\left[\frac{1}{k+1}\sum_{j=1}^k\left(\widehat\theta_N^ju\right)\left(\sum_{i=1}^N\lambda_{W(e_i)}(X_M)e_i + \sum_{b\in B}\lambda_{W(e_b)}(X_M)e_b\right)\right]
     \end{equation*}
     But
     \begin{equation*}
       \sum_{i=1}^N\lambda_{W(e_i)}(X_M)e_i + \sum_{b\in B}\lambda_{W(e_b)} = X_M
     \end{equation*}
     So :
     \begin{equation*}
       V_N(X_M) = \lim_k\left[\frac{1}{k+1}\sum_{j=1}^k\left(\widehat\theta_N^ju\right)(X_M)\right]
     \end{equation*}
     Then we prove that $\left\|(\widehat\theta_N^ju)X_M\right\|_{\mathbb D_r^p}$ is uniformly,
     in $j$ and $M$, $\mathbb D^\infty$-bounded.
     \begin{align*}
       \|(\widehat\theta_N^ju)(X_M)\|_{\mathbb D_r^p}
        & = \|\theta_N^j(1-L)^{\frac{r}{2}}(u\cdot(\widetilde\theta_N^j(X_M)))\|_{L^p} \\
        & = \|(1-L)^{\frac{r}{2}}\{u(\widetilde\theta_N^j(X_M))\}\|_{L^p} \\
        & = \| u\cdot(\widetilde\theta_N^j(X_M))\|_{\mathbb D_r^p}
     \end{align*}
     But $\{\widetilde\theta_N^j(X_M)\,/\,j\in\mathbb N_*; N,M\in\mathbb N_*\}$ is a bounded subset of $\mathrm{Der}$:
     let $f\in D$, $D$ being a $\mathbb D^\infty$-bounded set of $\mathbb D^\infty(\Omega)$:
     \begin{equation*}
       \|\widetilde\theta_N^j(X_M)\|_{\mathbb D_r^p}
        = \|\theta_N^{-j}(X_M(\theta_N^jf))\|_{\mathbb D_r^p} 
        = \|X_M(\theta_N^jf)\|_{\mathrm D_r^p}
     \end{equation*}
      From $\{\theta_N^jf\,/\,j\in\mathbb N_*; N\in\mathbb N_*; f\in D\}$ is a $\mathbb D^\infty$-bounded set
      and that $X_M\to\delta$ in $\mathrm{Der}\,\Omega$, we get that 
      the set $\{\widetilde\theta_N^j(X_M)\,/\,{j\in\mathbb N_*}; {N\in\mathbb N_*}; {M \in\mathbb N_*}\}$ is bounded in $\mathrm{Der}$,
      uniformly in $j$, $N$, $M$. And so is the set
      \begin{equation*}
        \{(\widehat\theta_N^ju)(X_M)\,/\,{j\in\mathbb N_*};{N\in\mathbb N_*}; {M\in\mathbb N_*}\}
      \end{equation*} 
      uniformly in $j$, $N$, $M$. 
      Which proves that the set ${\{ V_N(X_M)\,/\,N\in\mathbb N_*; M\in\mathbb N_*\}}$ is 
      uniformly in $N$, $M$, $\mathbb D^\infty$-bounded.
    \end{enumerate}
  \end{proof}
  
  Now we go back to the proof of Theorem 3.2: 
  we have by direct computation:
  \begin{equation*}
    \mathrm E[V_N(\delta)|\mathcal F_N] = \mathrm E[u(X_N)|\mathcal F_N]
  \end{equation*}
  and 
  \begin{equation*}
    \lim_{N\uparrow\infty}\mathrm E[u(X_N)|\mathcal F_N] = u(\delta)
  \end{equation*}
  this limit being a $L^2$-limit. 
  But from e) we also have $\sup_N\|V_N(\delta)\|_{L^2} < +\infty$.
  
  These two properties of the sequence $V_N(\delta)$
  imply $V_N(\delta)\rightharpoonup u(\delta)$, $(V_N(\delta))_N$ converges weakly
  towards $u(\delta)$.
  
  As the $(V_N(\delta))$ converges weakly towards $u(\delta)$,
  there is a net of barycenters, built on the $V_N(\delta)$,
  which will strongly converge towards $u(\delta)$. 
  Each item of this net has the form:
  \begin{equation*}
    B(A_j) = \sum_{j=1}^n\alpha_j\left(\sum_{i_j\in A_j}\mathrm E[u(e_{i_j})|\mathcal F_{A_j}]\lambda_{W(e_{ij})}\right)
  \end{equation*}
  where $\alpha_j\geq 0$ and $A_j$ being
  a finite subset of $\mathbb N_*$. 
  Then $\|B(A_j)\|_{\mathbb D_r^p}<+\infty$, independently of the $(A_j)$

  Then the net $\{B(A_j)(\delta)\,/\,A_j\in\mathbb N_*\}$ converges towards $u(\delta)$
  and we also have $\sup_{A_j}\|B(A_j)(\delta)\|_{\mathbb D_r^p}<+\infty$.
  
  From these two properties, using the interpolation (Theorem 2.12),
  we deduce that the net $B(A_j)(\delta)$ converges $D^\infty$
  towards $u(\delta)$.

  \begin{cor}
    Given a bilinear positive form $q$ on $\mathrm{Der}^*$,
    the map $\left(\mathrm{Der}\,\Omega\right)^*\ni u\to T_u\in\mathrm{Der}\,\Omega$ is injective
  \end{cor}
  
  \begin{proof}
    Let $u\in\left(\mathrm{Der}\,\Omega\right)^*$ such that $T_u=0$.
    There exists a net \\
	$u_{F_i}=\sum_{f\in F_i}a_f\lambda_{b_f}$, $a_f$ and $b_f \in \mathbb D^\infty(\Omega)$
    and $i\in I$, which converges towards $u$ in $\left(\mathrm{Der}\,\Omega\right)^*$.
    Then $q\left(u,\sum_{f\in F_i}a_f\lambda_{b_f}\right) = \sum_{f\in F_i}T_u(b_f)$ converges
    towards $0$; so $q(u,u) = 0$.
  \end{proof}
  
\subsection{Metric, Levi-Civita connection, curvature}
  
  We now study the notions of Levi-Civita connection, the curvature, and the torsion.

  Let $q$ be a bilinear form, positive an non degenerate
  on $\left(\mathrm{Der}\,\Omega\right)^*$; $q$ induces a map $T_q$:
  \begin{equation*}
    \forall\alpha\in\left(\mathrm{Der}\,\Omega\right)^*, (T_q\alpha)\cdot f=q(\alpha,\lambda_f)=\delta_\alpha(f)
  \end{equation*}
  We denote by $\mathscr D_q = \left\{T_q\alpha\,/\,\alpha\in\left(\mathrm{Der}\,\Omega\right)^*\right\}\subset\mathrm{Der}\,\Omega$
  
  \begin{defn}[of the bilinear form $q_0$ on $\mathscr D_q$]
    On $\mathscr D_q$ we define $\widehat q(\delta_\alpha,\delta_\beta) = q(\alpha,\beta)$.
    This definition is legitimate because $T_q$ is injective.
  \end{defn}
  
  \begin{defn}[of the Levi-Civita connection associated to $\widehat q$]
    Let $\delta_\alpha,\delta_\beta,\delta_\gamma\in\mathscr D_q$, we define $\nabla_{\delta_\alpha}\delta_\beta \in\mathrm{Der}\,\Omega$ by
    \begin{align*}
      2\gamma(\nabla_{\delta_\alpha}\delta_\beta) 
          = & \delta_\alpha\cdot\widehat q(\delta_\beta,\delta_\gamma) 
          + \delta_\beta\cdot\widehat q(\delta_\alpha,\delta_\gamma)
          - \delta_\gamma\cdot\widehat q(\delta_\alpha,\delta_\beta) \\
          & - T^{-1}(\delta_\alpha)([\delta_\beta,\delta_\gamma])
          + T^{-1}(\delta_\beta)([\delta_\gamma,\delta_\alpha]) 
          + T^{-1}(\delta_\gamma)([\delta_\alpha,\delta_\beta])
    \end{align*}
  \end{defn}

    Each term of the r.h.s. of the above equation is
    meaningful because ${T^{-1}:\mathscr D_q\to\left(\mathrm{Der}\,\Omega\right)^*}$ and
    $[\delta_\alpha,\delta_\beta]\in\mathrm{Der}\,\Omega$. We now denote $\nabla_{\delta_\alpha}\delta_\beta$ by $\nabla_\alpha\beta$.
    Then we define ${(\nabla_\alpha\beta)\cdot f=\lambda_f(\nabla_\alpha\beta)}$, where $f\in\mathbb D^\infty(\Omega)$
    and $\lambda_f\in\left(\mathrm{Der}\,\Omega\right)^*$.

    Now we write improperly $\widehat q(\nabla_\alpha\beta,\delta_\gamma)=\gamma(\nabla_\alpha\beta)$.
    With this notation, it is easy to show that:
    \begin{itemize}
      \item $\nabla_\alpha\beta$ verifies the Leibniz formula,
      \item $\delta_\alpha\cdot\widehat q(\delta_\beta,\delta_\gamma)=\widehat q(\nabla_\alpha\beta,\delta_\gamma)+\widehat q(\delta_\beta,\nabla_\alpha\gamma)$ (compatibility with the metric)
      \item $\widehat q(\nabla_\alpha\beta,\delta_\gamma)-\widehat q(\nabla_\beta\alpha,\delta_\gamma)=(T^{-1}\delta_\gamma)([\delta_\alpha,\delta_\beta])$ (the torsion is zero)
    \end{itemize}
    which implies $\forall\gamma\in\left(\mathrm{Der}\,\Omega\right)^*$, 
    \begin{equation*}
      \gamma(\nabla_\alpha\beta-\nabla_\beta\alpha-[\delta_\alpha,\delta_\beta])=0
    \end{equation*}
    or $\nabla_\alpha\beta-\nabla_\beta\alpha-[\delta_\alpha,\delta_\beta]=0$.

  \begin{lem}
    $\delta_\alpha\in\mathscr D_q \Rightarrow \nabla_\alpha\alpha\in\mathscr D_q$.
  \end{lem}
  
  \begin{proof}
    Using the definition of the Levi-Civita connection, we have
    \begin{equation*}
      \widehat q(\nabla_\alpha\alpha,\delta_\beta)
       = \delta_\alpha\cdot\widehat q(\delta_\alpha,\delta_\beta)
       + \alpha([\delta_\beta,\delta_\alpha])
       - \frac12\delta_\beta\cdot\widehat q(\delta_\alpha,\delta_\alpha)
    \end{equation*}
    For all $\delta\in\mathrm{Der}\,\Omega$, the map which associates $\delta$ with
    \begin{equation*}
      \delta_\alpha\cdot\alpha(\delta)-\frac12\delta\cdot\widehat q(\delta_\alpha,\delta_\alpha)+\alpha([\delta,\delta_\alpha])
    \end{equation*}
    is an element of $\left(\mathrm{Der}\,\Omega\right)^*$; we denote it by $\rho$, then 
    we have:
    \begin{equation*}
      \widehat q(\nabla_\alpha\alpha,\delta_\beta) = q(\rho,T_q^{-1}(\delta_\beta)) = \widehat q(\delta_\rho,\delta_\beta)
    \end{equation*}
    Since $\widehat q$ is non degenerate, $\nabla_\alpha\alpha=\delta_\rho\in\mathscr D_q$.
  \end{proof}

  \begin{cor}  
    $\delta_\alpha,\delta_\beta\in\mathscr D_q \Rightarrow \nabla_\alpha\beta+\nabla_\beta\alpha\in\mathscr D_q$.
  \end{cor}
  
  \begin{proof}
    $\nabla_{\alpha+\beta}(\alpha+\beta)-\nabla_\alpha\alpha-\nabla_\beta\beta\in\mathscr D_q$.
  \end{proof}
  
  \begin{defn}[formal definition of the curvature]
    Let $\widehat q$ be the positive non degenerate bilinear form
    on $\mathscr D_q$, for $\delta_\alpha,\delta_\beta\in\mathscr D_q$, we define only
    formally first:
    \begin{align*}
      R(\delta_\alpha,\delta_\beta,\delta_\alpha,\delta_\beta) 
        =  \widehat q(\nabla_\alpha(\nabla_\beta\alpha),\delta_\beta)
        - \widehat q(\nabla_\beta(\nabla_\alpha\alpha),\delta_\beta)
        - \widehat q(\nabla_{[\delta_\alpha,\delta_\beta]}\delta_\alpha,\delta_\beta)
    \end{align*}
  \end{defn}

    Using the compatibility with the Levi-Civita 
    connection and the torsion being null, we get,
    still formally:
    \begin{align*}
      \delta_\alpha\cdot\widehat q(\nabla_\beta\alpha,\delta_\beta)
        & = \widehat q(\nabla_\alpha(\nabla_\beta\alpha),\delta_\beta) + \widehat q(\nabla_\beta\alpha,\nabla_\alpha\beta) \\
      \widehat q(\nabla_{[\delta_\alpha,\delta_\beta]},\delta_\alpha)
        & = \widehat q([[\delta_\alpha,\delta_\beta],\delta_\alpha],\delta_\beta) + \widehat q(\nabla_\alpha [\delta_\alpha,\delta_\beta],\delta_\beta)
    \end{align*}
    So, still formally,
    \begin{align*}
      R(\delta_\alpha,\delta_\beta,\delta_\alpha,\delta_\beta)
        =\, & \delta_\alpha\cdot\widehat q(\nabla_\beta\alpha,\delta_\beta)
        - \widehat q(\nabla_\beta\alpha,\nabla_\alpha\beta)
        - \delta_\beta\cdot\widehat q(\nabla_\alpha\alpha,\delta_\beta) \\
        & + \widehat q(\nabla_\alpha\alpha,\nabla_\beta\beta)
        - \widehat q(\nabla_{[\delta_\alpha,\delta_\beta]}\delta_\alpha,\delta_\beta)
        \tag{10}\label{l10}
    \end{align*}
    The zero torsion implies:
    \begin{align*}
      \widehat q (\nabla_{[\delta_\alpha,\delta_\beta]}\delta_\alpha,\delta_\beta)
        = \widehat q([[\delta_\alpha,\delta_\beta],\delta_\alpha],\delta_\beta)
        + \delta_\alpha\cdot\widehat q([\delta_\alpha,\delta_\beta],\delta_\beta)
        - \widehat q([\delta_\alpha,\delta_\beta],\nabla_\alpha\beta)
        \tag{11}\label{l11}
    \end{align*}
    Using \eqref{l11} in \eqref{l10}, and denoting $\{\delta_\alpha,\delta_\beta\} = \frac12(\nabla_\alpha\beta+\nabla_\beta\alpha)\in\mathscr D_q$:
    \begin{align*}
      R(\delta_\alpha,\delta_\beta,\delta_\alpha,\delta_\beta)
      =~ & \delta_\alpha\cdot\widehat q(\nabla_\beta\alpha,\delta_\beta)
      - \widehat q(\{\delta_\alpha,\delta_\beta\},\{\delta_\alpha,\delta_\beta\}) \\
      & + \frac34\widehat q([\delta_\alpha,\delta_\beta],[\delta_\alpha,\delta_\beta])
      - \delta_\beta\cdot\widehat q(\nabla_\alpha\alpha,\delta_\beta)\\
      & + \widehat q(\nabla_\alpha\alpha,\nabla_\beta\beta)
      - \widehat q([[\delta_\alpha,\delta_\beta],\delta_\alpha],\delta_\beta)\\
      & - \delta_\alpha\cdot\widehat q([\delta_\alpha,\delta\beta],\delta_\beta)
      + \widehat q([\delta_\alpha,\delta_\beta],\{\delta_\alpha,\delta_\beta\})
      \tag{12}\label{l12}
    \end{align*}
    As $\widehat q([[\delta_\alpha,\delta_\beta],\delta_\alpha],\delta_\beta) = \beta\cdot([[\delta_\alpha,\delta_\beta],\delta_\alpha])$, 
    \begin{equation*}
      \widehat q([\delta_\alpha,\delta_\beta],\delta_\beta) = \beta\cdot([\delta_\alpha,\delta_\beta])
    \end{equation*}
    and
    \begin{equation*}
      \widehat q([\delta_\alpha,\delta_\beta],\{\delta_\alpha,\delta_\beta\})
        = T_q^{-1}(\{\delta_\alpha,\delta_\beta\})\cdot([\delta_\alpha,\delta_\beta])
    \end{equation*}
    The only element in \eqref{l12} which does not have 
    any meaning is $\widehat q([\delta_\alpha,\delta_\beta],[\delta_\alpha,\delta_\beta])$.
    
    Then if we have on $\left(\mathrm{Der}\,\Omega\right)^*$ two non degenerate
    positive bilinear forms $q_1$ and $q_2$, and if 
    the difference of $\widehat{q_1-q_2}$ is defined on all $\mathrm{Der}\,\Omega$,
    the difference of the associated curvatures, $R_1-R_2$,
    is meaningful.

\section{\huge Multiplicators, derivations}
	
	Here we will characterize $\dinf$-continuous derivations which are also adapted and with zero-divergence. They bijectively correspond to some particularly important operators, named multiplicators, which we will first study. The general setting unless otherwise specified is a Wiener space $(\Omega, \F, \P, H)$ where $H$, the C-M. space is the set of functions $[0, 1] \rightarrow \R^n$ verifying the usual conditions of the C-M. space.

\subsection{Definition of $\mathbb{D}^\infty$-bounded processes and of multiplicators}
	
\begin{defn}
A $\dinf$ process $A(t, \omega)$, defined on $[0,1] \times \R$ with values in the $n \times n$ antisymmetrical matrices (denoted in short: $n \times n$-A.M.), will be said to be $\dinf$-bounded iff:

$\forall (p, r), p>1, r \in \N_{\star}, \exists C(p,r)>0$ such that: 

$\sup_{t \in [0, 1]} \| |A(t,\omega)| \|_{ \sko^{p}_{r} } \leq C(p, r)$, where $|A(t, \omega)|$ denotes any $n \times n$ matrix norm, which are all equivalent.
\end{defn}

\begin{notation}
A $\dinf$-vector field $X(\omega)$ is a map from $\Omega$ in $H$; this vector field generates a process on $[0,1] \times \Omega$, with values in $\R^n$, and then is denoted: $X(t, \omega)$.
\end{notation}

\begin{defn}
An adapted vector field $X$, is a vector field $X(t, \omega)$ which, when read as a process $X(t, \omega)$ is an adapted process.
\end{defn}

\begin{defn}
A $\dinf$-process $A(t, \omega): [0,1] \times \Omega \rightarrow n \times n - A.M.$ will be said to be a multiplicator iff its image of a $\dinf$-vector field $V(\omega)$ is again a $\dinf$-vector field. That means: if $t \rightarrow \int_{0}^{t} \dot{V}(s, \omega) ds$ is a $\dinf$-vector field, then $t \rightarrow \int_{0}^{t} A(s, \omega) \dot{V}(s, \omega) ds$ is again a $\dinf$-vector field.
\end{defn}

\begin{lem}\label{lem4_1}

Let $A = (a_{ij})$ be a $n \times n$-A.M. matricial process on 

$[0,1] \times \Omega$:

\begin{enumerate}
\renewcommand{\labelenumi}{\roman{enumi})}
\item $A$ is a multiplicator $\Leftrightarrow$ $\forall i, j: a_{ij}$ is a multiplicator.
\item $A$ is a multiplicator implies: $A$ is a linear continuous operator.
\item $A$ multiplicator implies $A$ is $\dinf$-bounded.
\end{enumerate}

\end{lem}

\begin{proof}

\begin{enumerate}
\renewcommand{\labelenumi}{\roman{enumi})}
\item trivial.
\item direct application of the closed graph theorem.
\item let $(e_i)_{i \in \{1, \dots, n\} }$ be the canonical basis of $\R^n$ and 

$X_k(t, \epsilon, \omega) : u \rightarrow \frac{1}{\sqrt{\epsilon}} \int_0^u \mathds{1}_{[t, t+\epsilon[} (s) ds . e_k$ a $\dinf$-vector field.
\end{enumerate}

\begin{align*}
\left( (1-L)^{r/2} A X_k \right) (t, \epsilon, \omega): u \rightarrow \sum_{j=1}^{n} \frac{1}{\sqrt{\epsilon}} \int_{0}^{u} \left( (1-L)^{r/2} a_{kj} \right)(s, \omega) \mathds{1}_{[t, t+\epsilon[} (s) ds . e_j
\end{align*}

which implies:


\begin{align*}
\left\| (1-L)^{r/2} A . X_k (t, \epsilon, \omega) \right\|_{H}^{2} &= \sum_{j=1}^{n} \frac{1}{\epsilon} \int_{0}^{1} du |(1-L)^{r/2} a_{kj}(u, \omega)|^{2} \mathds{1}_{[t, t+\epsilon[}(u) \\
&= \sum_{j=1}^{n} \frac{1}{\epsilon} \int_{t}^{t+\epsilon} |(1-L)^{r/2} a_{kj}(u, \omega)|^2 du \\
&\geq \frac{1}{\epsilon} \int_{t}^{t+\epsilon} du |(1-L)^{r/2} a_{kj}|^{2}, \forall k,j
\end{align*}

By the continuity of $A$, we get:

\begin{align*}
\forall (p, r) ~ \exists (q, s), \exists \text{ constant } C&: \\
\| A .X_k (t, \epsilon) \|_{\sko_{r}^{p}(\Omega, H)} &\leq C \| X_k (t, \epsilon) \|_{\sko_{s}^{q}(\Omega, H)} \leq \text{C}
\end{align*}

Combining these two inequalities, we have:

\begin{align*}
\left\| \left[ \int_{t}^{t+\epsilon} \frac{1}{\epsilon} du |(1-L)^{r/2} a_{kj}|^2 \right]^{\frac{1}{2}} \right\|_{L^p(\Omega)} \leq \text{C}
\end{align*}

\end{proof}

\section*{\Large Examples of multiplicators}

\begin{crit}
Let $A(t, \omega): [0,1] \times \Omega \rightarrow n \times n$-A.M. be a 

$\dinf$-matricial process such that: $\forall f \in \sko_{\infty}^{2}(\Omega, \R^n)$: $f \rightarrow Af$ is continuous from $\sko_{\infty}^{2}$ to $\sko_{\infty}^{2}$, $t$-uniformly. Then $A$ is a multiplicator. 

The proof will be given later (Lemma \ref{lem4_7}).
\end{crit}

\begin{crit}
Let $A(t, \omega): [0,1] \times \Omega \rightarrow n \times n$-A.M. be a 

$\dinf$-matricial process such that with $A = (a_{ij}), \forall i, j, r$ such that: 

$\sup_{t \in [0,1]} (1-L)^{r/2} a_{ij}$ is bounded by a function $\in L^{\infty - 0}(\Omega)$; then $\sup_{t \in [0,1]} \| D^{r} a_{ij} \|_{\overset{r}{\otimes} H}$ is also bounded by a function $\in L^{\infty - 0}(\Omega)$ and A is a multiplicator.
\end{crit}

\begin{proof}
Let $f \in \dinf(\Omega)$; we denote $D_i f = \langle \Grad f, e_i \rangle_{H}$, 
$(e_i)_{i \in \N_{\star}}$ being an Hilbertian basis of $H$.

With the Mehler formula:

\begin{align*} 
\left[ D_i \left((1-L)^{-1} f \right) \right](x) = D_i \left[ \int_{0}^{\infty} e^{-t} dt \int_{\R^n} f(x e^{-t} + y \sqrt{1 - e^{-2t}}) \wienmeas(y) \right]
\end{align*}

$\wienmeas(y)$ being the Wiener measure; so:

\begin{align*}
\left[ D_i \left((1-L)^{-1} f \right) \right](x) = \int_{0}^{\infty} \frac{ e^{-2t} dt }{ (1 - e^{-2t})^{ \frac{1}{2} }} \int_{\R^n} y^{i} f(x e^{-t} + y \sqrt{1 - e^{-2t}} ) \wienmeas(y)
\end{align*}

The $y^i$ are i.i.d. Gaussian variables.
So Bessel-Perseval implies:

\begin{align*}
\left\| \Grad [(1-L)^{-1} f] \right\|_{H}^{2}(x) &= \sum_{i=1}^{\infty} \left| D_i (1-L)^{-1} f \right|^2 (x) \\
&= \sum_{i} \left[ \int_{\R^n} \wienmeas(y) y^{i} \left\{ \int_{0}^{\infty} \frac{ e^{-2t} }{ \sqrt{ 1 - e^{-2t} } } f(x e^{-t} + y \sqrt{1 - e^{-2t}}) dt \right\} \right]^{2} \\
&\leq \left[ \int_{\R^n} \wienmeas(y) \int_{0}^{\infty} \frac{ e^{-2t} dt }{ (1 - e^{-2t})^{ \frac{1}{2} }} f(x e^{-t} + y \sqrt{1 - e^{-2t}}) dt \right]^{2}
\end{align*}
which implies:
\setcounter{equation}{0}
\begin{align} 
\| \Grad (1-L)^{-1} f \|_{H} (\omega) &\leq \int_{0}^{\infty} \frac{ e^{-2t} dt }{ (1 - e^{-2t})^{ \frac{1}{2} }} \left( P_t f \right)(\omega) \label{eq4_1a}
\end{align}

where $P_t$ is the generator of the O.U. semi-group. 

Now, we have by hypothesis: $(1-L) a_{ij} \leq \alpha_{ij}, \alpha_{ij} \in L^{\infty - 0}(\Omega)$, so:

$\| \Grad a_{ij} \|_{H}(\omega) = \| \Grad (1-L)^{-1} (1-L) a_{ij} \|_{H} (\omega)$;
Using (\ref{eq4_1a}):

\begin{align*}
\| \Grad a_{ij} \|_{H}(\omega) &\leq \int_{0}^{\infty} \frac{ e^{-2t} dt }{ (1 - e^{-2t})^{ \frac{1}{2} }} P_t \left[ (1-L) a_{ij} \right] \leq \int_{0}^{\infty} \frac{ e^{-2t} dt }{ (1 - e^{-2t})^{ \frac{1}{2} }} \left(P_t \alpha_{ij}\right) dt \\
\text{so:} &\int \P( d\omega ) \left( \int_{0}^{1} ds \| \Grad a_{ij} \|_{H}^2 \right)^{p/2} \leq \int \P( d\omega ) \int_{0}^{1} ds \| \Grad a_{ij} \|_{H}^{p} \\
&\leq \int \P( d\omega ) \left| \int_{0}^{\infty} \frac{ e^{-2t} dt }{ (1 - e^{-2t})^{ \frac{1}{2} }} dt . P_t \alpha_{ij} \right|^{p} \leq \int \P( d\omega ) \int_{0}^{\infty} \frac{ e^{-2pt} dt }{ (1 - e^{-2t})^{ \frac{p}{2} }} | P_t \alpha_{ij} |^p \\
&\leq \int_{0}^{\infty} \frac{ e^{-2pt} dt }{ (1 - e^{-2t})^{ p }} \|\alpha_{ij} \|_{L^{p}(\Omega)}^{p}
\end{align*}
\end{proof}

\begin{rem}
If the $\overset{r}{\otimes} H$-norms of the iterated Malliavin derivation of $f \in \dinf(\Omega)$ are bounded by elements of $L^{\infty - 0}$, this does not imply that the iterated O.U. of $f$ are bounded by $L^{\infty - 0}$ functions.

Example: $f = \cos \frac{ B_t }{ \sqrt{t} }$, $B_t$ being Brownian.

\end{rem}

\begin{rem}
If $A$ is a $\dinf$ process from $[0,1] \times \Omega$, valued in the 

$n \times n$-A.M. but with the items being vectors of $H$, $A = (h_{ij})$, 

if $\sup_{t \in [0,1]} \| (1-L)^{r/2} h_{ij} \|_H$ is bounded by a function in $L^{\infty - 0}(\Omega)$ then $\Grad A = (\Grad h_{ij})$ is a multiplicator from $\dinf (\Omega, H)$ to $\dinf(\Omega, H \otimes H)$
\end{rem}

\subsection{Example of $\mathbb{D}^\infty$-bounded processes which is not a multiplicator}

\begin{prop}
The set of $\dinf$-multiplicators is strictly included in the set of $\dinf$-bounded processes.
\end{prop} 

\begin{proof}
Let 
\begin{align*}
f(x) &= 0 \text{ if } x \in ]-\infty, 0] \cup [1, +\infty[ \\
f(x) &= e^{ - \frac{1}{ x(1-x) } } \text{ if } x \in ]0, 1[
\end{align*}

and $\varphi(x) = \int_{x}^{\infty} f(t) dt / \int_{-\infty}^{+\infty} f(t) dt$; then $\varphi(x) = 1$ if $x < 0$, $\varphi(x) = 0$ if $x > 1$, and $\varphi$ is strictly decreasing on $]0, 1[$.

Let $\varphi_n = \varphi( x - \sqrt{2 \log n} )$.

An Hilbertian basis of $L^2([0,1])$ is: $k \in \N_{\star}$:
\begin{align*}
e_k: t \rightarrow \int_{0}^{t} 2^{ \frac{ k+1 }{ 2 } } \mathds{1}_{ ] 1-\frac{1}{2^k}, 1 - \frac{1}{2^{k+1}} ] }(s) ds
\end{align*}

Then we define the vector field $V$ as:

\begin{align*}
V = \sum_{n=1}^{\infty} \left( \prod_{k=1}^{n-1} \varphi_k(X_k) \right) \left( 1 - \varphi_n(X_n) \right) e^{ - \frac{1}{2} \sqrt{1 + X_{n}^{2}} } e_n
\end{align*}

$X_1, \dots , X_n, \dots$ being independent, centered, identically distributed 

Gaussian variables.

The candidate multiplicator process is:

\begin{align*}
C = \sum_{k=0}^{\infty} e^{ \sqrt{1 + X_{k}^{2}} } \mathds{1}_{ ] 1 - \frac{1}{2^k}, 1 - \frac{1}{2^{k+1}} ] } (t)
\end{align*}

\begin{enumerate}
\renewcommand{\labelenumi}{\roman{enumi})}
\item $C$ is $\dinf$-bounded: let $t = t_0$ and $k_0$ be such that $t \in ] 1 - \frac{1}{2^{k_0}}, 1 - \frac{1}{2^{{k_0}+1}} ]$; we have:
\begin{align*}
\left\| C(t_0, \omega) \right\|_{\sko_r^p}^{p} &= \left\| (1-L)^{r/2} C(t_0, \omega) \right\|_{L^p}^p
&= \int \left|(1-L)^{r/2} e^{ \sqrt{1 + X_{k_0}^{2}} } \right|^p \P(d\omega) < +\infty,
\end{align*}
uniformly relatively to $t$.

\item Now we show that $V \in \dinf(\Omega, H)$. 
It is enough to show that $V \in L^{\infty - 0} \cap \sko_{\infty}^{2}$ (Theorem 2, 5). 
We have $\| V \|_{H} \leq 1$. We will show that $V \in \sko_{\infty}^{2}$ by bounding $\| L^r V \|_{L^2 (\Omega, H)}, r \in \N_{\star}$, with a convergent serie:

Let $\psi_k(X_k) = \varphi_k(X_k)$ if $k \leq n-1$, $\psi_n(X_n) = \left(1 - \varphi_n(X_n) \right) e^{ - \frac{1}{2} \sqrt{1 + X_n^2} }$ if $k=n$.

Then: 
\begin{align*}
\int \P(d\omega) \left| L^r \left( \prod_{k=1}^{n} \psi_k(X_k) \right) \right| ^ 2 = \int \left[ \prod_{k=1}^{n} \psi_k(X_k) \right] . \left[ L^{2r} \left\{ \prod_{j=1}^{n} \psi_j(X_j)  \right\} \right] \P(d\omega) \\
= \sum_{\alpha_1 + \dots + \alpha_n = 2r} \int \left( \prod_{k=1}^{n} \psi_k(X_k) \right) \frac{ (2r)! }{ \alpha_1!\dots \alpha_n! } \left( \prod_{i=1}^{n} L^{\alpha_i} ( \psi_i( X_i ) ) \right) \P(d\omega) \\
= \sum_{\alpha_1 + \dots + \alpha_n = 2r} \frac{ (2r)! }{ \alpha_1!\dots \alpha_n! } \prod_{k=1}^{n} \int \P(d\omega) \psi_k(X_k) L^{\alpha_k} \left( \psi_k( X_k ) \right)
\end{align*}

For each factor $I_k = \int \P(d\omega) \psi_k(X_k) L^{\alpha_k} . \psi_k$ we have four possibilities:

a) $1 \leq k \leq n-1$ and $\alpha_k \neq 0$:
Then: 
\begin{align*}
I_k^{(1)} = \frac{1}{\sqrt{2 \pi}} \int_{\sqrt{2 \log k}}^{\sqrt{2 \log k}+1} \varphi( x - \sqrt{ 2 \log k } ) L^{\alpha_k} [ \varphi(x - \sqrt{2 \log k }) ] e^{ -\frac{x^2}{2} } dx
\end{align*}
then:

\begin{align*}
| I_k^{(1)} | &\leq \frac{1}{\sqrt{2 \pi}} \frac{1}{k} \int_{\sqrt{2 \log k}}^{\sqrt{2 \log k}+1} | L^{\alpha_k} \varphi(x - \sqrt{2 \log k }) | dx \\
&= \frac{1}{\sqrt{2 \pi}} \frac{1}{k} \int_{0}^{1} | L^{\alpha_k} \varphi(u) | du \leq \frac{1}{k \sqrt{2\pi}} M_1
\end{align*}
with $M_1 = \sup_{1 \leq j \leq r} \int_{0}^{1} | L^{j} \varphi(u) | du$.

b) $ 1 \leq k \leq n-1$ and $\alpha_k = 0$:

\begin{align*}
| I_k^{(2)} | = \frac{1}{\sqrt{2 \pi}} \int_{-\infty}^{\sqrt{2 \log k}+1} \varphi( x - \sqrt{ 2 \log k } )^2 e^{ -\frac{x^2}{2} } dx \leq \frac{1}{\sqrt{2 \pi}} \int_{-\infty}^{\sqrt{2 \log k}+1} e^{ -\frac{x^2}{2} } dx \leq 1
\end{align*}

c) $k = n, \alpha_n = 0$
\begin{align*}
I_n^{(1)} &= \frac{1}{\sqrt{2 \pi}} \int_{\sqrt{2 \log n}}^{+\infty} \left( 1 -  \varphi( x - \sqrt{ 2 \log k } ) \right)^2 e^{-\sqrt{1 + x^2}} e^{ -\frac{x^2}{2} } dx \\
&\leq \frac{1}{\sqrt{2 \pi}} \int_{\sqrt{2 \log n}}^{+\infty} e^{-\sqrt{1 + x^2}} e^{ -\frac{x^2}{2} } dx \\
&\leq \frac{1}{\sqrt{2 \pi}} \cdot \frac{1}{e^{\sqrt{1 + 2 \log n}}} \int_{\sqrt{2 \log n}}^{\infty} e^{ -\frac{x^2}{2} } dx \\
&\leq \frac{1}{\sqrt{2 \pi}} \cdot \frac{1}{e^{\sqrt{1 + 2 \log n}}} \cdot \frac{2}{n} \cdot \frac{1}{ \sqrt{2 \log n} }
\end{align*}

d) $k = n, \alpha_n \neq 0$:
\begin{align*}
I_n^{(2)} = \int_{\sqrt{2 \log n}}^{+\infty} \left( 1 -  \varphi( x - \sqrt{ 2 \log n } ) \right) e^{-\frac{1}{2} \sqrt{1 + x^2}} L^{\alpha_n} \left( \left( 1 -  \varphi( x - \sqrt{ 2 \log k } ) \right) e^{-\frac{1}{2} \sqrt{1 + x^2}} \right) e^{ -\frac{x^2}{2} } dx 
\end{align*}

and $L^{\alpha_n} = P_{\alpha_n} \left( \frac{\partial}{\partial x} \right)$ where $P_{\alpha_n}$ is a polynomial such that: 

$1 \leq deg( P_{\alpha_n} ) \leq 4r$.

So:

\begin{align*}
I_n^{(2)} &\leq \frac{1}{\sqrt{2 \pi}} \int_{\sqrt{2 \log n}}^{\infty} e^{-\frac{1}{2} \sqrt{1 + x^2}} P_{\alpha_n} \left( \frac{\partial}{\partial x} \right) \left( \left( 1 -  \varphi( x - \sqrt{ 2 \log n } )\right) e^{-\frac{1}{2} \sqrt{1 + x^2}} \right) e^{ -\frac{x^2}{2} } dx \\
\text{and: } I_n^{(2)} &\leq \frac{1}{ \sqrt{2 \pi} e^{\frac{1}{2} \sqrt{1 + 2\log n}} } \int_{\sqrt{2 \log n}}^{\infty} P_{\alpha_n} \left( \frac{\partial}{\partial x} \right) \left( \left( 1 -  \varphi( x - \sqrt{ 2 \log n } )\right) e^{-\frac{1}{2} \sqrt{1 + x^2}} \right) e^{ -\frac{x^2}{2} } dx
\end{align*}

But $P_{\alpha_n} \left( \frac{\partial}{\partial x} \right) \left\{ \left( 1 -  \varphi( x - \sqrt{ 2 \log n } )\right) e^{-\frac{1}{2} \sqrt{1 + x^2}} \right\}$ can be rewritten as:

\begin{align*}
P_{\alpha_n} \left( \frac{\partial}{\partial x} \right) \left\{ \left( 1 -  \varphi( x - \sqrt{ 2 \log n } )\right) e^{-\frac{1}{2} \sqrt{1 + x^2}} \right\} = e^{-\frac{1}{2} \sqrt{1 + x^2}} \left[ \sum_{i=1}^{\alpha_n + 1} \frac{ P_i(x, \varphi^{(j)} )}{ (1 + x^2)^{i/2} } \right]
\end{align*}

where $P_i$ is a polynomial in $x$, and $\varphi^{(j)}$ being the $j^{\text{th}}$ derivative of $\varphi$ $(j=1, \dots, \alpha_n)$.

Using the same type of development for $P_{2r} \left( \frac{\partial}{\partial x} \right)$, $| \varphi^{(j)} | \leq 1$, and substituting each coefficient with its module, we get:
$\forall \alpha_n: 1 \leq \alpha_n \leq 2r:$

\begin{align*}
\left| P_{\alpha_n} \left( \frac{\partial}{\partial x} \right) \left\{ \left( 1 -  \varphi( x - \sqrt{ 2 \log n } )\right) e^{-\frac{1}{2} \sqrt{1 + x^2}} \right\} \right| \leq e^{-\frac{1}{2} \sqrt{1 + x^2}} \left[ \sum_{i=1}^{2n+1} \frac{ | P_i(x, 1) |}{ (1 + x^2)^{1/2} } \right] = R(x)
\end{align*}

So there exists $N_0 \in \N$ such that: $\forall x \geq \sqrt{2 \log N_0}$, we have $R(x) \leq 1$;

Which implies:

\begin{align*}
I_n^{(2)} \leq \frac{1}{ \sqrt{2 \pi} } \cdot \frac{1}{e^{\frac{1}{2} \sqrt{1 + 2\log n}} } \int_{\sqrt{2 \log n}}^{\infty} e^{ -\frac{x^2}{2} } dx = \frac{1}{ \sqrt{2 \pi} } \cdot \frac{1}{e^{\frac{1}{2} \sqrt{1 + 2\log n}} } \cdot \frac{1}{n \sqrt{2 \log n}}
\end{align*}

Now, the entry of rank $n$ of the serie defining $\|V\|_{\sko_{r}^{2}}$ can be rewritten as:

\begin{align} 
\sum_{\alpha_1 + \dots + \alpha_n = 2r} \frac{ (2r)! }{ \alpha_1! \dots \alpha_n! } \prod_{k=1}^{n-1} I_k I_n + \sum_{\alpha_n = 1}^{2r} \left( \sum_{\alpha_1 + \dots + \alpha_{n-1} = 2r - \alpha_n} \frac{ 2r! }{ \alpha_1!\dots \alpha_n! } \prod_{k=1}^{n-1} I_k I_n \right) \tag{1'}\label{eq4_1b}
\end{align}

Let $J$ be a finite subset of $\{2, \dots, n-1\}$, of size $\lambda$ with: $0 \leq \lambda \leq 2r$. We denote by $P(J) = \prod_{k \in J} I_k \leq \text{C} \prod_{k \in J} \frac{1}{k}$, $C$ being a constant.

Then we know that there is a set $\widetilde{J}$ of $\alpha_k, k \in J$, such that for each $\alpha_k \in \widetilde{J}$: $\alpha_k \neq 0$ and $\alpha_k \leq 2r$ and $\sum_{k \in J} \alpha_k = 2r$.

From the above decomposition of $2r$ we can construct another decomposition of $2r$ such that $\forall j \in J$: $\alpha_j = 1$, and 

$\alpha_1 = 2r - \lambda (= 2r - \sum_{j \in J} j )$.

Then: $P(J) I_{\alpha_{1}} \leq \text{C} \prod_{k \in J} \frac{1}{k}$.

Now this item $\prod_{k \in J} \frac{1}{k}$ appear in the development of:
\begin{align*}
\left( 1 + \dots + \frac{1}{n-1} \right)^{\alpha_1} \underbrace{\left( 1 + \dots + \frac{1}{n-1} \right) \dots \left( 1 + \dots + \frac{1}{n-1} \right) }_{\lambda \text{ times}} =  \left( 1 + \dots + \frac{1}{n-1} \right)^{2r}
\end{align*}

The number of sums that can be used with $\lambda$ integers to obtain $2r$ (each of the integers beging less or equal to $2r$) is bounded by $C(r)$, a constant which depends only on $r$, and not on $n$.

So each $\sum_{\alpha_1 + \dots + \alpha_\lambda = 2r} \frac{ (2r)! }{ \alpha_1! \dots \alpha_\lambda! } \prod_{k \in J} I_k$ can be bounded by $(2r)! C(r) \left( 1 + \dots + \frac{1}{n-1} \right)^{2r}$. So (\ref{eq4_1b}) can be bounded by: $C(r)$, being an "absorbing" constant:

\begin{align*}
\frac{ C(r) \left( 1 + \dots + \frac{1}{n-1} \right)^{2r} }{ e^{\sqrt{1 + 2\log n}} n \sqrt{2 \log n} } + 
\sum_{\alpha_n = 1}^{2r} \frac{ C(r) \left( 1 + \dots + \frac{1}{n-1} \right)^{2r - \alpha_n} }{ e^{\sqrt{1 + 2\log n}} n \sqrt{2 \log n} } \leq \frac{ C(r) \left( \log n \right)^{2r} }{ e^{\sqrt{1 + 2\log n}} n \sqrt{2 \log n} }
\end{align*}

This last inequality shows that the serie defining $\|V\|_{\sko_{r}^{2}}$ is bounded by a convergent serie, so $V \in \dinf(\Omega, H)$.

\item Now we prove that $C.V \notin L^2(\Omega, H)$, which will prove that although $C$ is $\dinf$-bounded, it is not a multiplicator. We have:

\begin{align*}
C.V (t, \omega) = \sum_{n=1}^{\infty} \left( \prod_{k=1}^{n-1} \varphi_k(X_k) \right) \left(1 - \varphi_n(X_n)\right) e^{\frac{1}{2} \sqrt{1+X_n^2}} e_n(t)
\end{align*}

So:

\begin{align*}
\int \| C.V \|_{H}^{2} \P(d\omega) = \sum_{n=1}^{\infty} \int \prod_{k=1}^{n-1} (\varphi_k(X_k))^2 \left(1 - \varphi_n(X_n)\right)^2 e^{\sqrt{1+X_n^2}} \P(d\omega)
\end{align*}

But: 
\begin{align*}\int \varphi_k(X_k)^2 \P(d\omega) \geq 1 - \frac{1}{\sqrt{2\pi}} \int_{\sqrt{2 \log k}}^{+\infty} e^{-\frac{x^2}{2}} dx \geq 1 - \frac{1}{(\sqrt{4n \log k}) k}
\end{align*}

We fix $\epsilon, 0 < \epsilon < 1$:

\begin{align*}
\int (1 - \varphi_n(X_n))^2 e^{\sqrt{1+X_n^2}} \P(d\omega) &\geq \frac{1}{\sqrt{2\pi}} \int_{\sqrt{2 \log n} + \epsilon}^{\sqrt{2 \log n} + 1 + \epsilon} \left(1 - \varphi_n(x_n)\right)^2 e^{\sqrt{1+x_n^2}} e^{-\frac{x_n^2}{2}} dx_n \\
&\geq \left(1 - \varphi(\epsilon)\right)^2 e^{\sqrt{1 + 2\log n}} e^{- \frac{1}{2} (1 + \epsilon + \sqrt{2 \log n})^2}
\end{align*}

Now we will show that the serie with the general term:

\begin{align}
\prod_{k=1}^{n} \left( 1 - \frac{1}{k \sqrt{ 4\pi \log k}} \right) \cdot \frac{ (1 - \varphi(\epsilon))^2 }{\sqrt{2 \pi}} \int_{\sqrt{2 \log n} + \epsilon}^{\sqrt{2 \log n} + 1 + \epsilon} e^{\sqrt{1+x^2}} e^{-\frac{x^2}{2}}  dx \label{eq4_2}
\end{align}

is divergent.

We need the following lemma:

\begin{lem}\label{lem4_2}
Let $a_k$ be a sequence of numbers such that: 

$0 < a_k < \frac{1}{2}$, $\sum_{k} a_k = + \infty$, $\sum_{k} a_k^2 < + \infty$

Then there exists a constant $C_0 > 0$, with
$\forall n:$

$\prod_{k=1}{n} (1-a_k) \geq C_0  e^{- \sum_{k=1}^{n} a_k }$

\end{lem}

We have: $\left( \sqrt{ \frac{ \log x }{ \pi } } \right)' = \frac{1}{x \sqrt{ 4 \pi \log x }}$, so: $\frac{1}{k \sqrt{ 4 \pi \log k }} \leq \sqrt{ \frac{ \log k }{ \pi } } - \sqrt{ \frac{ \log (k-1) }{ \pi } }$. 

With Lemma \ref{lem4_2} we get: $\prod_{k=2}^{n-1} \left( 1 - \frac{1}{k \sqrt{ 4 \pi \log k }} \right) \geq C_0 e^{- \sqrt{ \frac{ \log (n-1) }{ \pi } }}$

Then the general term in (\ref{eq4_2}) is bigger than

\begin{align*} e^{- \sqrt{ \frac{ \log (n-1) }{ \pi } }} \cdot \frac{ (1 - \varphi(\epsilon))^2 }{\sqrt{2 \pi}} \cdot \int_{\sqrt{2 \log n} + \epsilon}^{\sqrt{2 \log n} + 1 + \epsilon} e^{\sqrt{1+x^2}} e^{-\frac{x^2}{2}} dx
\end{align*}

and 

\begin{align*}
\int_{\sqrt{2 \log n} + \epsilon}^{\sqrt{2 \log n} + \epsilon + 1} e^{\sqrt{1+x^2}} e^{-\frac{x^2}{2}} dx &= \int_{\sqrt{2 \log n} + \epsilon}^{\sqrt{2 \log n} + 1 + \epsilon} e^{\sqrt{1+x^2}} e^{-\frac{x^2}{2}} \frac{x dx}{x} \\
&\geq \frac{1}{2 (\sqrt{2 \log n} + \epsilon + 1)} \cdot \int_{(\sqrt{2 \log n} + \epsilon)^2}^{(\sqrt{2 \log n} + \epsilon + 1)^2} e^{\sqrt{1+u}} e^{-\frac{u}{2}} du \\
&\geq \frac{ e^{\sqrt{1 + (\sqrt{2 \log n} + \epsilon)^2}} }{ 2(\sqrt{2 \log n} + \epsilon + 1) } \int_{(\sqrt{2 \log n} + \epsilon)^2}^{(\sqrt{2 \log n} + \epsilon + 1)^2} e^{-\frac{u}{2}} du \\
&= \frac{1}{\sqrt{2 \log n} + \epsilon + 1} \cdot e^{\sqrt{1 + (\sqrt{2 \log n} + \epsilon)^2}} \cdot \left[ e^{-\frac{(\sqrt{2 \log n} + \epsilon)^2}{2}} - e^{-\frac{(\sqrt{2 \log n} + \epsilon + 1)^2}{2}} \right]
\end{align*}

Now: $e^{-\frac{a^2}{2}} - e^{-\frac{(a+1)^2}{2}} \sim e^{-\frac{a^2}{2}}$ when $a \uparrow \infty$ and $e^{\sqrt{1 + (\sqrt{2 \log n} + \epsilon)^2}} \sim e^{\sqrt{2 \log n} + \epsilon}$ when $n \uparrow \infty$.

So the general term above is equivalent to, $C_1$ being a constant:

\begin{align*}
C_1 \cdot (1 - \varphi(\epsilon))^2  e^{-\sqrt{\frac{\log(n-1)}{\pi}}}  \frac{ e^{\sqrt{2 \log n} + \epsilon} }{ \sqrt{2 \log n} + 1 + \epsilon }  e^{-\log n}  e^{-\epsilon \sqrt{2 \log n} }  e^{-\frac{\epsilon^2}{2} }
\end{align*}

which is equivalent to:

\begin{align}
C_1 \cdot \frac{ (1 - \varphi(\epsilon))^2 e^{-\frac{\epsilon^2}{2} + \epsilon }}{ \sqrt{2 \log n} + 1 + \epsilon }  \frac{1}{n}  e^{ \left( - \frac{1}{\sqrt{\pi}} - \epsilon \sqrt{2} + \sqrt{2} \right) \sqrt{\log n} } \label{eq4_3}
\end{align}

So choosing $\epsilon$ such that: $0 < \epsilon < 1 - \frac{1}{\sqrt{2 \pi}}$, (\ref{eq4_3}) is the general of a divergent serie which is less than: $\int \| C.V \|_H^2 \P(d\omega)$.

\end{enumerate}

\end{proof}

\subsection{Identity of $\mathbb{D}^\infty$-derivations which null-divergence adapted multiplicators}

Now we will study the relation between adapted multiplicators and adapted derivations with null divergence.

Let $A = (a_j^i)$ be an adapted multiplicator process taking its values in $n \times n$-A.M. We define a map $D_A$ from the Gaussian variables in $\dinf(\Omega)$ by: 

let $B$ be an $\R^n$-valued Brownian on $W(h_1, \dots, h_n)$: 

$D_A[ W(h_1, \dots, h_n) ] = - \int_{0}^{1} (\dot{h_1}, \dots, \dot{h_n}) A dB$

$D_A$ can be extented on the set of polynomials in Gaussian variables with the Leibniz rule.

$A$ being an antisymmetrical matrix, $\Div A \Grad$ is a derivation; $P$ being a polynomial of Gaussian variables, and $A$ being adapted, $\Div A \Grad P$ can be written with an It\={o} integral which coincides with $D_A P$. So $D_A$ can be extended in a $\dinf$-continuous derivation from $\dinf$ to $\dinf$.

\begin{thm}\label{thm4_1}

Conversely, each $\dinf$-continuous derivation $\delta$ of $\dinf$ in $\dinf$, which is adapted and with a null divergence, has the form $D_A$, $A$ being an adapted multiplicator process with values in the $n \times n$-A.M. .
\end{thm}
The proof of Theorem \ref{thm4_1} needs several lemmas. We first define the coefficients $a_j^i$ of the candidate $A$: let $(B_t^1, \dots, B_t^n)$ be an n-uple of independent Brownian processes. Then by the Clark-Ocone formula we have: 

$\delta B_1^i = \int_0^1 {}^t \E \left[ \Grad (\delta B_1^i)_j(s) | \F_s \right] dB_s^j$, $\Grad (\delta B_1^i)_j$ being the $j^{\text{th}}$ component of the vector $\Grad (\delta B_1^i)$ which is a process, for which the value in $s$ is taken. We define $A = {}^t(a_j^i)(t, \omega) = {}^t \E \left[ \Grad (\delta B_1^i)_j(t) | \F_t \right]$ and will show that $A$ is an adapted multiplicator with values in $n \times n$-A.M., and that $D_A = \Div A \Grad = \delta$. 

It is obvious from the definition of $A$ that $A$ is adapted. We first prove that $A$ is antisymmetric.

\begin{lem}\label{lem4_3}
$\delta$ being an $\dinf$-continuous derivation, adapted and with null divergence, then $B_t$ being a Brownian process:

i) $\delta B_t, t \in [0, 1]$ is a martingale.

ii) $\forall W(h_1, \dots, h_n): \delta \left( \E \left[ W(h_1, \dots, h_n) | \F_t \right] \right)$ is a $\F_t$-martingale.

iii) $\delta \left( \E \left[ W(h_1, \dots, h_n) | \F_t \right] \right) \E \left[ W(h_1, \dots, h_n) | \F_t \right]$ is a $\F_t$-martingale.

iv) ${}^t a_j^i = \E \left[ \Grad ( \delta B_1^i)_j(s) | \F_s \right]$ is antisymmetrical.
\end{lem}

\begin{proof}


i) straightforward: $\forall \varphi \in \F_s:$
$\int \varphi \left[ \delta B_s - \E \left[ \delta B_t | \F_s \right] \right] = \int \varphi \delta \left( B_s - B_t \right) = 0$,

$\delta$ being adapted.

ii) $\forall \varphi \in \F_s$:
\begin{align*}
\int \varphi \E \left[ \delta W(h_1, \dots, h_n) | \F_s \right] &= - \int \delta \left( \E \left[ \varphi | \F_s \right] \right) W(h_1, \dots, h_n) \\
&= - \int \delta \varphi . \E \left[ W(h_1, \dots, h_n) | \F_s \right] \\
&= + \int \varphi . \delta \E \left[ W(h_1, \dots, h_n) | \F_s \right]
\end{align*}

implies $\E \left[ \delta W(h_1, \dots, h_n) | \F_s \right] = \delta \E \left[ W(h_1, \dots, h_n) | \F_s \right]$

iii) Denote $M_s = \E \left[ W(h_1, \dots, h_n) | \F_s \right]$. Let: $0 \leq u \leq s$.

\begin{align*}
\forall \varphi \in \F_u: \int \frac{1}{2} \delta (M_s)^2 \varphi &= - \int \frac{1}{2} \delta \varphi . \E \left[ M_s^2 | \F_u \right] (\delta \text{ adapted}) \\
&= - \int \frac{1}{2} \delta \varphi . \E \left[ 2 \int_0^s M dM + \int_0^s [dM, dM] \bigg| \F_u \right] \\
&= + \int \frac{1}{2} \varphi \delta \left\{ \E \left[ 2 \int_0^s M dM + \int_0^s [dM, dM] \bigg| \F_u \right] \right\}
\end{align*}

But $dM = \sum_{i=1}^n \dot{h_i} dB_i$, so $\int_0^s [dM, dM] = \int_0^s \sum_{i=1}^n \dot{h_i^2} d\rho$ 

Then: 
\begin{align*}\int \frac{1}{2} \varphi . \delta( M_s )^2 = \int \frac{1}{2} \varphi . \delta \left( \E \left[ 2 \int_0^s M dM \right] \right)
&& \left( \delta \left[ \int_0^s \sum_{i=1}^n \dot{h_i^2} d\rho \right] = 0 \right)
\end{align*}

So:

\begin{align*}
\int \frac{1}{2} \varphi \delta( M_s )^2 &= \int \varphi . \delta \left( \left[ \frac{1}{2} M_u^2 - \frac{1}{2} \int_0^u [dM, dM] \right] \right) \text{ (It\={o} formula)} \\
&= \int \varphi . M_u \delta M_u 
\end{align*}

iv)
\begin{align*}
\delta \left( B_t^i B_t^j \right) &= B_t^j \delta B_t^i + B_t^i \delta B_t^j \\
&= \int_0^t \delta B_s^i  dB_s^j + \int_0^t B_s^j d(\delta B_s^i) + \int_0^t \delta B_s^j dBs^i + \int_0^t B_s^i d( \delta B_s^j) \\ 
&+ \int_0^t \left[ d( \delta B_s^i ), dB_s^j \right] + \int_0^t \left[ d( \delta B_s^j ), dB_s^i \right]
\end{align*}

iii) implies: 
\begin{align}
\int_0^t \left[ d( \delta B_s^i ), dB_s^j \right] + \int_0^t \left[ d( \delta B_s^j ), dB_s^i \right] = 0 \label{eq4_4}
\end{align}

Now: $\delta B_s^i = \E \left[ \delta B_1^i | \F_s \right] = \int_0^s {}^t a_j^i(\rho) dB_{\rho}^{j}$, $^t a_i^j$ being the transposed matrix of $a_i^j$.

Then: $d( \delta B_1^i) = {}^t a_k^i(s, \omega) dB_s^k$. (\ref{eq4_4}) becomes: 

\begin{align*}
\int_0^t \left[ a_j^i(s, \omega) + a_i^j(s, \omega) \right] ds = 0, \forall t \in [0, 1]
\end{align*}
\end{proof}

\begin{lem}\label{lem4_4}
$A = (a_j^i)$ is $\dinf$-bounded.
\end{lem}

\begin{proof}

\begin{align*}
\E \left[ \frac{ B_{t+\epsilon} - B_t^k}{\sqrt{\epsilon}} \cdot \frac{ \delta \left[ B_{t+\epsilon}^i - B_t^i\right] }{\sqrt{\epsilon}}\right] &= \E \left[ \frac{1}{\epsilon} \int_{t}^{t+\epsilon} {}^t a_j^i dB^d . \int_{t}^{t+\epsilon} dB^k \bigg| \F_t \right] \\
&= \E \left[ \frac{1}{\epsilon} \int_{t}^{t+\epsilon} M^{(1)} . {}^t a_j^i dB^j + \frac{1}{\epsilon} \int_{t}^{t+\epsilon} M^{(2)} dB^k + \frac{1}{\epsilon} \int_{t}^{t+\epsilon} {}^t a_k^i ds \bigg| \F_t \right].
\end{align*}

with It\={o} formula, $M^{(1)}$ and $M^{(2)}$ being the obvious martingales.

Let $\epsilon \rightarrow 0$: the r.h.s. of the above equation converges $L^1(\Omega)$ towards ${}^t a_k^i$, and is $\dinf$-bounded, so $a_k^i$ is $\dinf$-bounded. 

\end{proof}

Now $A$ being antisymmetric, $\Div A \Grad$ is a derivation on the domain of the polynomials in Gaussian variables, which coincides with $\delta$ on this domain.

The only property left to verify is that $A$ is a multiplicator and then: $D_A = \Div A \Grad$ will be a derivation on $\dinf(\Omega)$, which will coincide with $\delta$ on $\dinf(\Omega)$.

For this, the following lemmas are needed:

\begin{lem}\label{lem4_5}

i) If $P \in \mathcal{C}_n, Q \in \mathcal{C}_m, \exists k(n) \in \N$ so that: 

$\| PQ \|_{L^2} \leq k(n) (m+1)^n \| P \|_2 \| Q \|_2$.

ii) The multiplication by an element of $\mathcal{C}_n$ sends $\sko_{\infty}^2$ to $\sko_{\infty}^2$.

\end{lem}

\begin{proof}

i) let $A_{n,m} = \sup \left\{ \| PQ \|_{L^2} | P \in \mathcal{C}_n, \|P\|_2 \leq 1, Q \in \mathcal{C}_m, \|Q\|_2 \leq 1\right\}$
and define $P(t) = \E \left[ P | \F_t \right] = \int_0^t \E \left[ (\Grad P)(s) | \F_s \right] dB_s$
and 

$Q(t) = \E \left[ Q | \F_t \right] = \int_0^t \E \left[ (\Grad Q)(s) | \F_s \right] dB_s$

From the It\={o} formula, we get:

\begin{align}
P(1)Q(1) - P(0)Q(0) &= PQ \notag\\
&= \int_0^1 \left[ \int_0^t \E \left[ (\Grad Q)(s) | \F_s \right] dB_s \right] . \E \left[ (\Grad P)(t) | \F_t \right] dB_t \notag\\
& + \int_0^1 \left[ \int_0^t \E \left[ (\Grad P)(s) | \F_s \right] dB_s \right] . \E \left[ (\Grad Q)(t) | \F_t \right] dB_t \notag\\
& + \int_0^1 dt \E \left[ (\Grad P)(t) | \F_t \right] . \E \left[ (\Grad Q)(t) | \F_t \right] \label{eq4_5}
\end{align}

The square of the $L^2$-norm of the above first integral is bounded by:

\begin{align*}
\int_0^1 dt \E \left[ \E \left[ (\Grad P)(t) | \F_t \right]^2 \cdot \E \left[ Q(t) | \F_t \right]^2 \right]
\leq \int_0^1 dt A_{n-1, m}^2 \| Q \|_{L^2}^2 \cdot \E \left[ \E \left[ (\Grad P)(t) | \F_t \right]^2 \right]
\end{align*}

and:

\begin{align*}
\int_0^1 dt \E \left[ \E \left[ (\Grad P)(t) | \F_t \right]^2 \right] = \E \left[ \left( \int_0^1 \E \left[ (\Grad P)(t) | \F_t \right] dB_t \right)^2 \right] = \E \left[ P^2 \right]
\end{align*}

So the $L^2$-norm of the first above integral is bounded by $A_{n-1, m}$; similarly, the $L^2$-norm of the second integral is bounded by $A_{n, m-1}$. For the third and last integral in (\ref{eq4_5}), we have:

\begin{align*}
\left\| \int_0^1 dt \E \left[ (\Grad P)(t) | \F_t \right] . \E \left[ (\Grad Q)(t) | \F_t \right] \right\|_{L^2(\Omega)}
\leq \int_0^1 dt \left\| \E \left[ (\Grad P)(t) | \F_t \right] . \E \left[ (\Grad Q)(t) | \F_t \right] \right\|_{L^2(\Omega)} \\
 \leq \int_0^1 dt A_{n-1, m-1} \left\| \E \left[ (\Grad P)(t) | \F_t \right] \right\|_{L^2} . \left\| \E \left[ (\Grad Q)(t) | \F_t \right] \right\|_{L^2} \\
 \leq A_{n-1,m-1} \left[  \int_0^1 dt \| \E \left[ (\Grad P)(t) | \F_t \right] \|_{L^2} \right]^{\frac{1}{2}} \left[ \int_0^1 dt \| \E \left[ (\Grad Q)(t) | \F_t \right] \|_{L^2} \right]^{\frac{1}{2}}
\end{align*}

So at the end:

\begin{align*}
\| PQ \|_{L^2} \leq (A_{n-1, m} + A_{n, m-1} + A_{n-1, m-1}) \| P \|_{L^2} . \| Q \|_{L^2} 
\end{align*}

So we have: $ A_{n-1, m} + A_{n, m-1} + A_{n-1, m-1} \geq A_{n,m} $.

Let's write: $A_{n,m} = k(n)(m+1)^n$ 
then $k(n) = k(n-1)(2^{n-1} + 1)$ fits. 

ii) Let $\alpha \in \sko_{\infty}^2(\Omega)$; $\alpha = \sum_{n=1}^{\infty} \alpha_n, \alpha_n \in \mathcal{C}_n(\Omega)$.

Let $\beta \in \mathcal{C}_m$; the sequence of $L^2$-norms $\| \alpha_n \|_{L^2}$ is fast decreasing; each $\beta \alpha_n$ is then a polynomial with terms belonging to the chaos of orders 

$ |n-m|, \dots, n+m $; and the multiplication of a fast decreasing $L^2$-norm sequence, by a fixed polynomial is again a sequence of fast decreasing $L^2$-norms.

\end{proof}

\begin{lem}\label{lem4_6}

\begin{enumerate}
\renewcommand{\labelenumi}{\roman{enumi})}
\item Let $A$ be a $\dinf$-bounded process with values in 

$n \times n$-A.M., and $\delta_A$ defined on the finite linear combinations of polynomials in Gaussian variables, by $\Div A \Grad$.

Then: $\delta( \mathcal{C}_m ) \subset \dinf(\Omega)$, and if $P$ is a polynomial in Gaussian variables, $\delta(P) \in \dinf(\Omega)$.

\item Moreover if $\delta_A$ is defined on $\sko_{\infty}^2$ and sends continuously $\sko_{\infty}^2$ in $\sko_{\infty}^2$, then $\delta_A$ sends $\dinf$-continuously $\dinf(\Omega)$ in $\dinf(\Omega)$.

\end{enumerate}

\end{lem}

\begin{rem}
$\sko_{\infty}^{2}$ is not an algebra, so $\delta_A$ cannot be called a derivation on $\sko_{\infty}^2$.
\end{rem}

\begin{proof}

\begin{enumerate}
\renewcommand{\labelenumi}{\roman{enumi})}
\item Let $Q \in \mathcal{C}_m(\Omega)$. Then $\Grad Q \in \mathcal{C}_{m-1}(\Omega, H)$ which can be written $(m > 1)$:

\begin{align*}
(\Grad Q)_j(t, \omega) = \int_0^t Z_j(s, \omega) ds, Z_j \in \mathcal{C}_{m-1}(\Omega), j=1, \dots n
\end{align*}

Then $A (\Grad Q)$ can be written:

\begin{align}
(A . \Grad Q)_k(t, \omega) &= \int_0^t A_k^j Z_j(s, \omega) ds, k = 1, \dots, n \notag\\
&= \int_0^t A_k^j \frac{ Z_j(s, \omega) }{ \sqrt{1 + \|Z(s, .)\|_{L^2(\Omega, \R^n)}} } \sqrt{1 + \|Z(s, .)\|_{L^2(\Omega, \R^n)} } ds \label{eq4_6} 
\end{align}

In this last integral $\frac{ Z_j }{ \sqrt{1 + \|Z(s, .)\|_{L^2(\Omega, \R^n)}} }$ belongs to $\mathcal{C}_{m-1}(\Omega)$ on which all $L^p$ norms are equivalent.

So $A_k^j \frac{ Z_j (s, \omega)}{ \sqrt{1 + \|Z(s, .)\|_{L^2(\Omega, \R^n)}} }$ is $\dinf$-bounded; and the measure $\sqrt{1 + \|Z(s, .)\|_{L^2(\Omega, \R^n)}} . ds$ is $L^2([0,1])$-bounded, so the r.h.s. integral in (\ref{eq4_6}) is $\dinf$-bounded, which implies $\delta_A Q \in \dinf$.

\item $\delta_A^2(e^{iP})$, $P$ being a polynomial in Gaussian variables is meaningful because $P \in \sko_{\infty}^2$ implies $\delta_A(e^{iP}) \in \sko_{\infty}^2$ so $\delta_A (\delta_A . e^{iP}) \in \sko_{\infty}^2$. 
Then $\Grad (\delta_A e^{iP})$ formally equals to $ i e^{iP} \Grad (\delta_A P) - e^{iP} \delta_A P . \Grad P$ which is meaningful as a definition because $\dinf$ is an algebra.

Then $\Div \left[ A \Grad e^{iP} \right] = i e^{iP} \delta_A P = \delta_A( e^{iP} )$. 

And $\delta_A^2( e^{iP} ) = \Div A \Grad (\delta_A(e^{iP}))$ which formally equals to: 

\begin{align*}
i e^{iP} \Div [A \Grad (\delta_A P)] - ie^{iP} \delta_A P . \langle \Grad P, A (\Grad P) \rangle_H - e^{iP}(\delta_A P)^2 = ie^{iP} \delta_A^2 P - e^{iP}(\delta A P)^2 
\end{align*}
since $A$ is antisymmetric.

Each of the two terms in the above sum is meaningful because 

$\dinf \times \sko_{\infty}^2 = \sko_{\infty}^{2 - 0}$;
so: $\delta_A^2( e^{iP} ) = i e^{iP} \delta_A^2 P - e^{iP}(\delta_A P)^2$
and 
\begin{align}
(\delta_A P)^2 = i \delta_A^2 P - e^{-iP} \delta_A^2 ( e^{iP} ) \label{eq4_7}
\end{align}

So with a sequence of polynomials in Gaussian variable $\dinf$-converging towards $f \in \dinf(\Omega)$, with (\ref{eq4_7}) we see that $\delta_A f \in L^4(\Omega)$.

Then we have: $\delta_A: \sko_{\infty}^2 \rightarrow \sko_{\infty}^2$ and also: $\delta_A: \dinf \rightarrow \sko_{\infty}^2 \cap L^4(\Omega)$.

By interpolation, we get: $\delta_A: \sko_{\infty}^{2+} \rightarrow \sko_{\infty}^{2+}$.

Now let $f \in \dinf(\Omega)$ and $g \in \sko_{\infty}^{2+}$: then $fg \in \sko_{\infty}^{2+}$ and 

$\delta_A( fg ) = f \delta_a g + g \delta_A f$ which implies: $g \delta_A f \in \sko_{\infty}^{2+}(\Omega)$.

So the operator multiplication by $\delta_A f$ is such that: $\sko_{\infty}^{2+} \rightarrow \sko_{\infty}^{2+}$; so $\delta_A f \in \sko_{\infty}^{2+}$ (with $g=1$). 

Then all powers of $\delta_A f$ are in $\sko_{\infty}^{2+}$, which implies $\delta_A f \in \dinf$.

The $\dinf$-continuity of $\delta: \dinf \rightarrow \dinf$ is obtained by the closed graph theorem and by the continuity of $\delta_A: \sko_{\infty}^2 \rightarrow \sko_{\infty}^2$.
\end{enumerate}

\end{proof}

\begin{defn}
A $\sko_{\infty}^2$ multiplicative operator $A$, process from 

$[0,1] \times \Omega$ to the $n \times n$-A.M. is an operator which acts $\sko_{\infty}^2$-continuously by simple multiplication on functions: $\Omega \rightarrow \R^n$ that is: 

$\forall r > 1, \exists r' > 1, \exists C(r, r') > 0: \forall f = (f_1, \dots, f_n)$:

\begin{align*}
\sup_{t \in [0,1]} \| A . f \|_{ \sko_{r}^2 (\Omega) } \leq C(r, r') \| f \|_{ \sko_{r'}^2 (\Omega) } 
\end{align*}
\end{defn}

\begin{lem}\label{lem4_7}
Let $A(t, \omega)$ be a $\sko_{\infty}^2$-multiplicative operator, $n \times n$-A.M. valued. Then $A(t, \omega)$ is a multiplicator from $\dinf(\Omega, H)$ in $\dinf(\Omega, H)$.

\end{lem}

\begin{proof}
let $\theta$ be the morphism defined on $\mathcal{C}_1(\Omega)$ by:

\begin{align*}
\theta \left[ W \left( t \rightarrow \int_0^t \mathds{1}_{[0, a]} (s) ds \right) \right] = \sqrt{2} W \left[ \left( t \rightarrow \int_0^t \mathds{1}_{[0, a/2]} (s) ds \right)\right] 
\end{align*}

$\theta$ can be extended as a morphism from $\sko_{\infty}^2$ in $\sko_{\infty}^2$ because it leaves invariant the $L^2$ scalar product, each chaos $\mathcal{C}_n$, and it commutes with the O.U. operator. 

Then $\theta$ is bijective and isometric from $( \mathcal{C}_1, \F_1 )$ into $( \mathcal{C}_1, \F_{ \frac{1}{2} } )$, so it induces a bijection from $L^2(\Omega, \F_1, \P)$ into $L^2(\Omega, \F_{\frac{1}{2}}, \P)$, and from $\sko_r^2(\Omega, \F_1, \P)$ into $\sko_r^2(\Omega, \F_{ \frac{1}{2} }, \P)$. 

We define: $\widecheck{A}(s, \omega) = 0$ if $s \leq \frac{1}{2}$

$\widecheck{A}(s, \omega) = \theta A(2s - 1, \omega)$ if $\frac{1}{2} < s \leq 1$.

$\widecheck{A}$ is an adapted process $( \theta: \F_1 \rightarrow \F_{\frac{1}{2}})$. 
Direct computation shows that $\widecheck{A}$ is a $\sko_{\infty}^2 \cap \F_{\frac{1}{2}}$ multiplicative operator: 

$\forall \alpha \in ( \sko_{\infty}^2 \cap \F_{\frac{1}{2}})( \Omega )$: 
$\forall r \geq 1 ~ \exists r' \geq 1 ~ \exists C(r, r')$:

$ \| \widecheck{A}(s, .) \alpha(.) \|_{\sko_r^2(\Omega)} \leq C(r, r') \| \alpha(.) \|_{\sko_{r'}^2(\Omega)}$

Then $\widecheck{A}$ is a $\sko_{\infty}^2(\Omega, H) \cap \F_{\frac{1}{2}}$ multiplicator.

\begin{align*}
\forall X \in \sko_{\infty}^2(\Omega, H) \cap \F_{\frac{1}{2}}: \\
\| \widecheck{A}.X \|_{ \sko_{r}^2(\Omega, H) } &= \int \P(d\omega) \int_0^1 ds \| (1-L)^{r/2} \{ \widecheck{A}(s, \omega) X(s, \omega) \} \|_{\R^n}^2 \\
&= \int \P(d\omega) \int_{\frac{1}{2}}^1 ds \| (1-L)^{r/2} \{ \theta [ A(2s-1, \omega) ]X(s, \omega) \} \|_{\R^n}^2 \\
&= \int \P(d\omega) \int_{\frac{1}{2}}^1 ds \| (1-L)^{r/2} \{ \theta [ A(2s-1, \omega) . \theta^{-1} X(s, \omega) ] \} \|_{\R^n}^2 \\
&= \int_0^1 du \int \P(d\omega) \| (1-L)^{r/2} \{ \theta [ A(u, \omega) . \theta^{-1} X(\frac{u+1}{2}, \omega) ] \} \|_{\R^n}^2 \\
&= \int_0^1 du \left\| A(u, .) . \theta^{-1} X(\frac{u+1}{2}, .) \right\|_{\sko_{r}^2}^2 \text{ ($\theta$ commutes with O.U.) }\\ 
&\leq \int_0^1 du C(r, r')  \left\| \theta^{-1} X(\frac{u+1}{2}, .) \right\|_{\sko_{r'}^2}^2 \text{ ($C(r,r')$ is a constant) }\\ 
&\leq 2 C(r) \| X \|_{\sko_{r'}^2(\Omega, H)}^2
\end{align*}

Then $\delta_{\widecheck{A}} = \Div A \Grad$ is a $\sko_{\infty}^2$-continuous operator on $\sko_{\infty}^2 \cap \F_{\frac{1}{2}}$. 

With Theorem 2, 2, we get an extension $\widetilde{\delta}_{\widecheck{A}}$ on 

$\sko_{\infty}^2 \left[ (\Omega, \F_{ \frac{1}{2} }, \P) \times (\Omega, \F_{ \frac{1}{2} }^{\perp}, \P) \right] = \sko_{\infty}^2(\Omega)$. 

This extension is $\sko_{\infty}^2$-continuous. With Lemma \ref{lem4_6} ii) it is $\dinf$-continuous from $\dinf(\Omega)$ in $\dinf(\Omega)$; using this continuity and direct computation, $\widetilde{\delta}_{\widecheck{A}}$ is a derivation, which is adapted because $\widecheck{A}$ is adapted, so $\delta_{\widecheck{A}}$ is an adapted derivation.

Now let $\left( V_i(s, \omega) \right)_{i = 1, \dots, n}$ be a $\dinf$-adapted vector field. Then $\int_0^1 {}^t V_i dB^i $ is an It\={o} integral, and:

\begin{align*}
\dat \left[ \int_0^1 {}^t V_i dB^i \right] &= \int_0^t ( \dat {}^t V_i ) dB^i + \int_0^1 {}^t V_i(\dat dB^i) \\
&= \int_0^1 (\dat {}^t V_i) dB^i + \int_0^1 ({}^t V_i) {}^t \widecheck{A} dB^i
\end{align*}

Let $F = \dat ( \int_0^1 {}^t V_i dB^i )$; with Clark-Ocone we have:

\begin{align*}
F = \int_0^1 \E \left[ (\Grad F)_i(t) | \F_t \right] dB^i = \int_0^1 (\dat {}^t V_i) dB^i + \int_0^1 ({}^t V_i) {}^t \widecheck{A} dB^i
\end{align*}

In this last equation, all integrals are It\={o} integrals so:

\begin{align*}
{}^t V_i {}^t \widecheck{A} = \E \left[ (\Grad F)_i(t) | \F_t \right] - \dat( {}^t V_i )
\end{align*}

With Theorem 2, 8, $\E \left[ (\Grad F)_i(t) | \F_t \right]$ is a $\dinf$-vector field;
which proves that ${}^t V_i {}^t \widecheck{A}$ is a $\dinf$-vector field.

The map: $X \overset{\mu}{\rightarrow} \hat{X}(s, \omega) = X(\frac{s+1}{2}, \omega)$ is the left inverse of: $X \overset{\lambda}{\rightarrow} \widecheck{X}$ and $\mu \circ \lambda (X (s, \omega) ) = X(s, \omega)$.

Let $V(s, \omega)$ be a possibly non-adapted $\dinf$-vector field; then ${}^t \widecheck{V}(s, \omega)$ is an adapted $\dinf$-vector field, and so is ${}^t A . {}^t \widecheck{V}$; so using the left inverse map $\mu$ above, $A.V$ is a $\dinf$-vector field.

So $A$ is $\dinf(\Omega, H)$ multiplicator.
\end{proof}

\begin{cor}\label{cor4_1}

If $A(s, \omega)$ is a process with values in $n \times n$-A.M., such that $\exists m \in \N_{\star}: \forall s \in [0,1]: A(s, \omega) \in \mathcal{C}_m$ and $\|A(s, .)\|_{L^2}$ is uniformly bounded (in $s$), then $A$ is a multiplicator.
\end{cor}

\begin{proof}
Use Lemma \ref{lem4_5}, ii), and Lemma \ref{lem4_7}
\end{proof}

Now we go back to the end of the proof of Theorem \ref{thm4_1}:

We already know that $A$ being the $n \times n$ matrix with ${}^t a_j^i(t, \omega) = \E \left[ (\delta B_1^i)_j(t) | \F_t \right]$ is asymetrical, adapted. We now show that $A$ is a $\dinf(\Omega, H)$ multiplicator process.

The family $ t \rightarrow \frac{ B^i(t+\epsilon) - B_t^i}{\sqrt{\epsilon}}$ is in $\mathcal{C}_1(\Omega)$ and is $(t, \epsilon)$-uniformly $L^2$-bounded, and so is a family of multiplicators (Corollary 3, 1), uniformly in $(t, \epsilon)$.

And if $h(t, \omega) = \int_0^t \dot{h}(s, \omega) ds $ is a $\dinf$-vector field then $\forall i=1, \dots, n$: $t \rightarrow \int_0^t \frac{ B_{t+\epsilon}^i - B_t^i}{\sqrt{\epsilon}} \dot{h}(s, \omega) ds$ is a $\dinf$-vector field and:
\begin{align*}
\delta \left[ \int_0^t \left( \frac{ B_{s+\epsilon}^i - B_s^i}{\sqrt{\epsilon}} \right) \dot{h}(s, \omega) ds \right] = \int_0^t \delta \left( \frac{ B_{s+\epsilon}^i - B_s^i}{\sqrt{\epsilon}} \right) \dot{h}(s, \omega) ds + \int_0^t \left( \frac{ B_{s+\epsilon}^i - B_s^i}{\sqrt{\epsilon}} \right) \delta \dot{h}(s, \omega) ds
\end{align*}

In the above equation, the first and last items are vector fields (in the first integral, $\delta$ is the extension of Corollary 2, 4); so $t \rightarrow \int_0^t \delta \left( \frac{ B_{t+\epsilon}^i - B_t^i}{\sqrt{\epsilon}} \right) \dot{h}(s, \omega) ds$ is a $\dinf$-vector field, which proves that: $s \rightarrow \delta \frac{ \left[ B_{s+\epsilon}^i - B_s^i \right]}{\sqrt{\epsilon}} $ are multiplicators, $\epsilon$-uniformly $(0 < \epsilon < 1)$.

Then, with the It\={o} formula, and Lemma \ref{lem4_3} i):

\begin{align*}
F^{i,j}(\epsilon) &= \frac{ B^i(s+\epsilon) - B^i(s)}{\sqrt{\epsilon}} \cdot \delta \left[ \frac{ B^j(s+\epsilon) - B^j(s)}{\sqrt{\epsilon}} \right] \\
&= \frac{1}{\epsilon} \int_s^{s+\epsilon} {}^t a_i^j du + \frac{1}{\epsilon} \int_s^{s+\epsilon} \left( \int_s^u {}^t a_k^j dB^k \right)dB_u^i + \frac{1}{\epsilon} \int_s^{s+\epsilon} \left( \int_s^u dB^i \right) {}^t a_k^j dB_u^k 
\end{align*}

so:

\begin{align}
\left\| F^{i,j}(\epsilon) - \frac{1}{\epsilon} \int_s^{s+\epsilon} {}^t a_i^j du \right\|_{L^2(\Omega)} 
&\leq \frac{1}{\epsilon} \left[ \int_s^{s+\epsilon} du \left( \int_s^u d\alpha \sum_{k=1}^n \E \left[ {}^t a_k^i(\alpha, .)^2 \right] \right) \right]^{\frac{1}{2}} \notag\\
&+ \frac{1}{\epsilon} \left[ \int_s^{s+\epsilon} du \E \left[ \sum_{k=1}^n \left( \int_s^u dB^j \right)^{2} . \left( {}^t a_k^i \right)^2 \right] \right]^{\frac{1}{2}} \label{eq4_8}
\end{align}
as ${}^t a_k^i$ is $\dinf$-bounded (Lemma \ref{lem4_4}), we see that the r.h.s. of (\ref{eq4_8}) is $L^2$-bounded.

But $F^{i,j}(\epsilon) \in \F_s^{\perp} \cap \F_{s+\epsilon}$; the filtration being right-continuous:

$\lim_{\epsilon \downarrow 0} \F_s^{\perp} \cap \F_{s+\epsilon} = \{ 0 \}$. 

So $(F^{i,j}(\epsilon))$ admits an adherence value, which is $0$, to which $F^{i,j}(\epsilon)$ converges $L^2$-weakly.

And as from (\ref{eq4_8}) we see that $F^{i,j}(\epsilon) - \frac{1}{\epsilon} \int_s^{s+\epsilon} {}^t a_i^j du$ admits a $L^2$-weak limit, we see that:

\begin{align*}
\frac{ B^i(s+\epsilon) - B^i(s)}{\sqrt{\epsilon}} \cdot \delta  \frac{ \left[ B^j(s+\epsilon) - B^j(s) \right]}{\sqrt{\epsilon}}  \overset{L^2}{\rightharpoonup} {}^t a_i^j
\end{align*}

All that is left to prove is that the $L^2$-weak limit of multiplicators 

$\frac{ B^i(s+\epsilon) - B^i(s)}{\sqrt{\epsilon}} \cdot \delta  \frac{ \left[ B^j(s+\epsilon) - B^j(s) \right]}{\sqrt{\epsilon}}$ is again a multiplicator.

Then a net of barycenters $b_{ij}(\epsilon)$ constructed using
the items of the sequence $\frac{ B^i(s+\epsilon) - B^i(s)}{\sqrt{\epsilon}} \cdot \delta  \frac{ \left[ B^j(s+\epsilon) - B^j(s) \right]}{\sqrt{\epsilon}}$ is $L^2$-strongly convergent towards ${}^t a_i^j$.

But the $\frac{ B^i(s+\epsilon) - B^i(s)}{\sqrt{\epsilon}} \cdot \delta  \frac{ \left[ B^j(s+\epsilon) - B^j(s) \right]}{\sqrt{\epsilon}}$ are $\epsilon$-uniformly multiplicators; so are the $b_{ij}(\epsilon)$, as they are barycenters built on multiplicators. 

Then, for $X \in \dinf(\Omega, H)$, $b_{ij}(\epsilon).X$ converges in $L^2$ towards ${}^t a_i^j X$ and the $b_{ij}(\epsilon).X$ are uniformly $\dinf$-bounded; so by interpolation, the convergence of the $b_{ij}(\epsilon).X$ is $\dinf(\Omega, H)$, which proves that ${}^t a_j^i$ is a $\dinf$-multiplicator.

\begin{thm}\label{thm4_2}
Let $U_n$ be a sequence of adapted $\dinf$-vector fields such that the associated derivations converge pointwise towards a derivation $\delta$, with zero divergence. Then $\delta \equiv 0$.
\end{thm}

\begin{proof} 
We remind that for $\delta \in Der$, $\Div \delta$ is an operator such that: $\forall \varphi \in \dinf(\Omega): (\Div \delta).\varphi = \int \delta \varphi$.

Denote by $\delta_n$ the derivation associated with $U_n$: 

$\forall \varphi \in \dinf(\Omega): \delta_n \varphi = U_n . \varphi = \langle U_n, \Grad \varphi \rangle_H$

Then $\Div \delta_n \rightarrow \Div \delta$ in the "distribution" meaning, that is:

$ \int \langle U_n, \Grad \varphi \rangle_H \rightarrow \int \delta \varphi, \forall \varphi \in \dinf(\Omega)$

By hypothesis: $\int \langle U_n, \Grad \varphi \rangle_H d\P \rightarrow 0$,

Then: $\int (\Div U_n) . \varphi \P(d \omega) \rightarrow 0$, so $\Div U_n \rightarrow 0$ as distributions $\in (\dinf)^{\star}$. 

Then $\Grad ( \Div U_n) \rightarrow 0$ as distributions $\in (\dinf(\Omega, H))^{\star}$.

Then $\E \left[ \overset{\rightarrow}{\Grad } \Div U_n | \F_t \right]$ is also a sequence of adapted vector fields, which as distributions of $(\dinf(\Omega, H))^{\star}$ converges towards $0$ as the operator projection on $\F_t$ is continuous. 

But:
$ \int \E \left[ \overset{\rightarrow}{\Grad } \Div U_n | \F_t \right] dB = \Div U_n$ (Clark-Ocone) and as $U_n$ is adapted:

$\int U_n dB = \Div U_n$ (Skorokhod integral)

So $\E \left[ \overset{\rightarrow}{\Grad } \Div U_n | \F_t \right] = U_n$ (Fundamental isometry).

Then $U_n \rightarrow 0$ as distributions $\in (\dinf(\Omega, H))^{\star}$.

With: $\varphi \in \dinf(\Omega), \psi \in \dinf(\Omega)$, we have: $\varphi \Grad \psi \in \dinf(\Omega, H)$ which implies: 

\begin{align*}
\int \langle U_n, \varphi \Grad \psi \rangle_H ~ \P(d\omega) = \int \varphi . U_n(\psi) \P(d\omega) \rightarrow \int \varphi (\delta \psi) \P(d \omega) = 0 
\end{align*}
\end{proof}

\begin{cor}\label{cor4_2}
An adapted derivation is not generally a limit of a sequence of adapted vector fields.
\end{cor}

For the Corollary \ref{cor4_3} that follows, the notion of stochastic parallel transport of a vector $X$ in a Riemannian manifold is needed. It will be defined later in Section 5, and denoted $X_{//}(t, \omega)$.

\begin{cor}\label{cor4_3}
Given a compact Riemannian manifold $(V_n, g)$ and $\P_{m_0}(V_n, g)$ being the set of continuous paths: $[0,1] \rightarrow V_n$, starting from 

$m_0 \in V_n$, there cannot be a global chart from $\P_{m_0}(V_n, g)$ into the Wiener space such that:

\begin{enumerate}
\renewcommand{\labelenumi}{\roman{enumi})}
\item it leaves the measure invariant.
\item it has a continuous linear tangent map on the associated spaced of $\dinf$-continuous derivations.
\item this tangent map sends a dense subset $\mathfrak{X}$ of the adapted vector fields:
\begin{align*}
\mathfrak{X} =  \left\{ \sum_{i=1}^n a_i(t, \omega) e_{i//}(t, \omega) \bigg| a_i(t, \omega) \text{ adapted} \right\}
\end{align*}
in a dense subset of the adapted vector fields on the Wiener space; $(e_i)_{i=1,\dots,n}$ being an orthonormal basis of $T_{m_0}V_n$.
\end{enumerate}

\end{cor}

\begin{rem}
This result proves that even another chart map than the It\={o} map, satisfying reasonable conditions, does not have a linear tangent map.
\end{rem}

\begin{proof}
Suppose there exists such a global chart $\psi$ and that $T_{\psi}$ is the associated linear tangent map.
From Theorem 4.1, we know that there is a derivation $\delta$ adapted, with a null divergence, and three vector fields $u, v, w,$ such that: $\delta = [u, v] - w$. Then from iii) there exist three sequences of ${X}~$, $(u_n)_{n \in \N_{\star}}$, $(v_n)_{n \in \N_{\star}}$, $(w_n)_{n \in \N_{\star}}$, such that:
$u_n \overset{\text{Der}}{\rightarrow} u$, $v_n \overset{\text{Der}}{\rightarrow} v$, $w_n \overset{\text{Der}}{\rightarrow} w$; as $(u_n)_{n \in \N_{\star}}$ is bounded in Der, by Banach-Steinhaus we have:

$[u_n, v_n] \overset{\text{Der}}{\rightarrow} [u, v]$; so $[u_n, v_n] - w_n \overset{\text{Der}}{\rightarrow} \delta$

and $T_{\psi}( [u_n, v_n] - w_n ) = [T_{\psi} u_n, T_{\psi} v_n] - T_{\psi} w_n \overset{\text{Der}}{\rightarrow} T \delta$.

But by i): $\Div (T_{\psi} \delta) = 0)$, and $ [T_{\psi} u_n, T_{\psi} v_n] - T_{\psi} w_n $ is an adapted vector field.

Then Theorem \ref{thm4_2} proves that $T_{\psi} \delta = 0$, which implies $\delta = 0$.

\end{proof}

\begin{rem}
It is the property of adaptation that is at the root of this impossibility to find a regular enough global chart, which would admit a linear tangent map.
\end{rem}

\begin{rem}
Instead of $\mathfrak{X}$ being the set of the adapted vector fields, we could have chosen
any dense subset of adapted vector fields, because later we will see that the $\dinf$-module generated by such a subset is dense in the $\dinf$-continuous derivations.
\end{rem}

\begin{thm}\label{thm4_3}
Let $X: [0,1] \times \Omega \rightarrow \R$ be an $\alpha$-$\dinf$-Holderian process. Then X is a $\dinf$-multiplicator.
\end{thm}

\begin{proof}
From Theorem 2, 10, we know that $X = Y \star \beta_s$ with 

$Y = \frac{d}{dt}( X \star \beta_{1-s} )$ and Y is completely $\dinf$.

Then fixing $\gamma > 1$, and $q$ such that $0 < qs < 1$, there exists $p > 1$ with $1 + \frac{1}{\gamma} = \frac{1}{p} + \frac{1}{q}$. 

Y being completely $\dinf$, $\| Y(\omega, .) \|_{L^p([0,1], dt)} \in L^{\infty - 0}(\Omega)$. 
So:

\begin{align*}
\forall t: |X(t, \omega)| &= |Y(., \omega) \star \beta_s(.)(t)| \leq \| Y(\omega, .) \|_{L^p([0,1], dt)} \| \beta_s \|_{L^q([0,1], dt)}
\end{align*}

And $\| Y(\omega, .) \|_{L^p([0,1], dt)} \in L^{\infty - 0}(\Omega)$.

So with the Criterion 2, 2, we have that $X$ is a multiplicator.
\end{proof}

\begin{cor}
If $X: t \rightarrow \int_0^t \dot{h}(s, \omega) ds$ is a $\dinf$-vector field, then the process $X(t, \omega)$ is a $\dinf$-multiplicator.
\end{cor}

\begin{proof}
\begin{align*}
(1-L)^{r/2} (X(t+\epsilon, \omega) - X(t, \omega)) = \int_t^{t+\epsilon} (1-L)^{r/2} \dot{h}(s, \omega) ds
\end{align*}

Then:

\begin{align*}
\| X(t+\epsilon, .) - X(t, .) \|_{\sko_r^p(\omega)} &\leq \epsilon^{\frac{1}{2}} \left\| \left[ \int_0^1 | (1-L)^{r/2} \dot{h}(s, \omega)|^2 ds \right]^{\frac{1}{2}} \right\|_{L^p(\Omega)}
\end{align*}

so:

\begin{align*}
\| X(t+\epsilon, .) - X(t, .) \|_{\sko_r^p(\omega)} &\leq \epsilon^{\frac{1}{2}} \| X \|_{\sko_r^p(\Omega, H)}
\end{align*}

\end{proof}

\begin{thm}\label{thm4_4}
If $X_n(t, \omega)$ is a sequence of processes, converging $\dinf$-towards $X(t, \omega)$, $t$-uniformly, then the vector fields $Y_n : t \rightarrow \int_0^t X_n(s, \omega) ds$ converges $\dinf(\Omega, H)$ towards the vector field $Y: t \rightarrow \int_0^t X(s, \omega) ds$.
\end{thm}

\begin{proof}
Suppose $X(t, \omega) = 0$, $\P-a.s.$. Then:

\begin{align*}
\| Y_n \|_{\sko_r^p(\Omega, H)}^p &= \int \P(d\omega) \left( \int_0^1 ds |(1-L)^{r/2} X_n(s, \omega)|^2 \right)^{p/2} \\
&\leq \int \P(d\omega) \int_0^1 ds |(1-L)^{r/2} X_n(s, \omega)|^p \\
&= \int_0^1 ds \| X_n \|_{\sko_r^p(\Omega)}^p
\end{align*}

\end{proof}

\begin{thm}\label{thm4_5}
If $X$ is a $\dinf$-vector field, the associated process is 

$\dinf$-bounded.
\end{thm}

\begin{proof}
$X(t, \omega) = \int_0^t \dot{h}(s, \omega) ds$. Then: $X(t, \omega) = \langle h(s, \omega), u \rightarrow \int_0^u \mathds{1}_{[0, t]}(s) ds \rangle_H$.
\end{proof}

\begin{thm}\label{thm4_6}
If $A_n(t, \omega)$ is a $n$-uniform sequence of multiplicators, and if $A_n(t, \omega)$ converges $\P(d\omega) \otimes ds [0,1]$-a.s. towards $A$, then $A_n$ converges, in the multiplicator sense, towards $A$.
\end{thm}

\begin{rem}
Convergence in the multiplicator sense means that $\forall (p, r)$ $A_n$ converges strongly as operators towards $A$: 
\begin{align*}
\forall X \in H, \| A_n X - AX \|_{\sko_r^p(\Omega, H)} \rightarrow 0
\end{align*}
\end{rem}

\begin{proof}
Using the theorem on conditions of equivalence of $L^p$-convergence and a.s.-convergence, when the $L^p$ norms of the sequence are uniformly bounded and the measure is finite.
\end{proof}

\subsection{Any adapted multiplicator is a limit of a sequence of "step-multiplicators"}

\begin{thm}\label{thm4_7}
Any adapted multiplicator $A$ is a limit in the multiplicator sense of a sequence of adapted "step multiplicators", that is that have the form:
\begin{align*}
\sum_{i=0}^k \mathds{1}_{[t_i, t_{i+1}[}(t) A(t_i, \omega), 
\end{align*}
where $A(t_i, \omega) \in \F_{t_i}$ and $A(t_i, \omega)$ being a multiplicator. $(t_0 = 0, t_{k-1} = 1)$.
\end{thm}

\begin{proof}
a) We first prove than any adapted multiplicator is a limit of a sequence $A_n$ of continuous multiplicators, in a "multiplicator sense", that is $A_n$ converges towards $A$, strongly as operators; 

for each $\sko_r^p(\Omega, H)$: $\forall X \in H: A_n X \overset{\sko_r^p(\Omega, H)}{\rightarrow} AX$.

b) Let $u(t, \omega)$ be a $\dinf$-vector field and for $\lambda \in ]0, 1[$, we denote by $v_{\lambda}$ the vector field:

$v_{\lambda}(t, \omega) = \int_0^t \mathds{1}_{[0, \lambda]}(s) \dot{u}(\frac{s}{\lambda}, \omega) \frac{1}{\sqrt{\lambda}} ds$

Then straightforward computation shows that: $\| v_{\lambda} \|_{\sko_r^p(\Omega, H)} = \| u \|_{\sko_r^p(\Omega, H)}$. Now we denote by $A_{\lambda}(t, \omega) = A(\lambda t, \omega)$.

Straight computation gives:
\begin{align*}
\| A_{\lambda}(t, \omega) u(t, \omega) \|_{\sko_r^p(\Omega, H)} &= \| A(t, \omega) v_\lambda(t, \omega) \|_{\sko_r^p(\Omega, H)} \\
&\leq C(p, r, p', r') \| v_{\lambda}(r, \omega) \|_{\sko_{r'}^{p'}(\Omega, H)} \\
&\leq C(p, r, p', r') \| u(t, \omega) \|_{\sko_{r'}^{p'}(\Omega, H)}
\end{align*}

$C(p, r, p', r')$ being a constant.

So the family $A_{\lambda}$ is a $\lambda$-uniform family of multiplicators.

Then the $\widetilde{A}_n(t, \omega) = n \int_{1 - \frac{1}{n}}^{1} A_{\lambda} (t, \omega) d\lambda$ are $n$-uniformly multiplicators.
\begin{align*}
\| \widetilde{A}_n X \|_{\sko_r^p(\Omega, H)} \leq n \int_{1 - \frac{1}{n}}^{1} \| A_{\lambda} (t, \omega) X \|_{\sko_r^p} d\lambda \leq C(p, r, p', r') \| X \|_{\sko_{r'}^{p'}(\Omega, H)}
\end{align*}

As $\widetilde{A}_n \rightarrow A$, $L^2([0,1] \times \Omega)$-a.s., we have with Theorem \ref{thm4_6}, that $t^{\frac{1}{n}} \widetilde{A}_n \rightarrow A$ in the multiplicator sense, and each $\widetilde{A}_n$ is continuous and adapted.

b) Now we prove that any continuous, adapted multiplicator is a limit, in the multiplicator sense, of adapted step-multiplicators.

If $u(t, \omega) = \int_0^t \dot{u}(s, \omega) ds$, let $\varphi$ be an increasing bijection of $[0,1]$ on itself, such that $\varphi \in \mathcal{C}^1([0,1])$, and $\varphi'(t) > C_0$, $C_0$ constant $> 0$.

And we denote by: $u_{\varphi}(t, \omega) = \int_0^t \dot{u} ( \varphi^{-1}(s), \omega) \sqrt{ (\varphi^{-1})'(s) } ds$, $\varphi^{-1}$ being the inverse function of $\varphi$.

Straight computation shows that: $\| u_{\varphi} \|_{\sko_r^p(\Omega, H)} = \| u \|_{\sko_r^p(\Omega, H)}$

Then we define $(A_{\varphi} u)(t, \omega) = \int_0^t A \left[ \varphi(s), \omega \right] \dot{u}(\beta, \omega) ds$.

Straight computation shows: $\| A_{\varphi} u \|_{\sko_r^p(\Omega, H)} = \| A u_{\varphi} \|_{\sko_r^p(\Omega, H)}$

So the family $A_{\varphi}$ is $\varphi$-uniformly a family of adapted multiplicators.

c) If $\psi$ is a step function on $[0,1]$, there exists a sequence $\varphi_k$ of functions like in b), which converges towards $\psi$. 

Then $A_n \left( \varphi_k(.), \omega \right)$ converges towards $A_n \left( \varphi(.), \omega \right)$ in the multiplicator sense thanks to the continuity of $A_n$, and this convergence is uniform relatively to the step functions $\psi$. 

d) Now there exists a sequence of step functions converging towards $t$ on $[0,1]$ from below, denoted $\psi_l, l \in \N_{\star}$. Then $A_n \left( \psi_l(.), \omega \right)$ will converge in the multiplicator 
sense towards $A_n(t, \omega)$, which converges in the multiplicator sense towards $A(t, \omega)$.

And $A \left( \sum_{i=0}^k a_i \mathds{1}_{[t_i, t_{i+1}[} (t) , \omega \right) = \sum_{i=0}^k \mathds{1}_{[t_i, t_{i+1}[} A( a_i, \omega)$.

\end{proof}	

\section{\huge $\D^\infty$-morphisms, charts maps,\\ and inversibility}

\subsection{Theorems showing under which conditions a $\mathbb{D}^\infty$-morphism is a $\mathbb{D}^\infty$-diffeomorphism}

Unless otherwise specified, the setting is a Gaussian space
$(\Omega,\mathcal F,\mathbb P,H)$. $\proc U$ is an adapted process
with values in $n\times n$ unitary matrices, and a multiplicator;
the map $\theta_{\proc U}$ on $\mathcal C_1(\Omega)$:
\begin{displaymath}
  \theta_{\proc U}(W(h)) = \int_0^1{}^t\dot h\proc U^{-1}\itodiff B
\end{displaymath}
where B is a $n$-dimensional Brownian, and $h \in H$, can be extended
in an injective morphism on $\L{\infty-0}(\Omega)$, because it
preserves laws. In this chapter we will study some conditions under
which $\theta$ can be a $\D^\infty$-morphism, or a
$\D^\infty$-isomorphism.

\medskip

Let $(\Omega_i,\mathcal F_i,\mathbb P_i,H_i), i=1,2$, two Gaussian
spaces and denote by $\mathcal M$:
\begin{align*}
  \mathcal M = \{ &m \mid m \text{ map of $\Omega_1$ in unitary operators on $H_2$ such that $m$}\\
  & \text{ is a $\D^\infty(\Omega_1,H_2)$-multiplicator and there exists $m^{-1} \in \mathcal M$ }\\
  & \text { such that : $m^{-1}(\omega_1) = (m(\omega_1))^{-1}$}\}.
\end{align*}                    

\begin{rem}
  $m$, $\D^\infty(\Omega_1,H_2)$-multiplicator, means that:
  \begin{displaymath}
    \forall \alpha \in \D^\infty(\Omega_1,H_2), m(\omega_1)\alpha(\omega_1) \text{ is a $\D^\infty(\Omega_1,H_2)$ vector field}.
  \end{displaymath}
\end{rem}

\begin{rem}
  The existence's condition of $m^{-1}$ is useless if $m$ has the form
  $\proc U(t,\omega)$ where $\proc U$ is a unitary operator on a
  finite dimensional space ($\proc U^{-1} = \proc U^\ast$).
\end{rem}

\begin{notation}
  We denote by $W_i(h_i)$, $\Grad_i,\div_i$, Gaussian variables,
  Malliavin derivatives, and divergences built respectively on
  $(\Omega_i,\mathcal F_i,\mathbb P_i,H_i), i=1,2$ ; otherwise $W$,
  $\Grad$, and $\div$ relate to $(\Omega_1 \times \Omega_2, \mathcal
  F_1 \otimes \mathcal F_2,\mathbb P_1 \otimes \mathbb P_2, H_1 \oplus
  H_2)$.
\end{notation}

Let $m \in \mathcal M$, and denote $\theta_m \colon \mathcal
C_1(\Omega_1 \times \Omega_2) \to \L{\infty-0}(\Omega_1 \times
\Omega_2)$ an $\mathbb R$-linear map defined by:
\begin{align*}
  \theta_m[W_1(h_1)] &= W_1(h_1), \quad h_1 \in H_1, \\
  \theta_m[W_2(h_2)] &= (\omega_1 \to W_2(m(\omega_1)h_2)), \quad h_2
  \in H_2.
\end{align*}
Then we extend $\theta_m$ on Gaussian polynomials as an algebraic
morphism and denote it again by $\theta_m$.

\begin{thm} ~
  \label{thm:5.1}
  \begin{enumerate}[label=\roman*)]
  \item if $m \in \mathcal M$, $\range \theta_m \in
    \D^\infty(\Omega_1 \times \Omega_2)$.
  \item $\theta_m$ can be extended in a bicontinuous bijection of
    $\L{\infty-0}$ on itself.
  \end{enumerate}
\end{thm}
\begin{proof}
  i). Let $(\varepsilon_j)_{j \in \mathbb N_\ast}$ be an Hilbertian
  basis of $H_2$ and $h_2 \in H_2$:
  \begin{align*}
    \div [m(\omega_1)h_2] &= \sum_{j=1}^\infty \div[f^j(\omega_1)\varepsilon_j] \\
    &= \sum_{j=1}^\infty f^j(\omega_1)W_2(\varepsilon_j) + \langle \Grad_1 f^j(\omega_1), \varepsilon_j \rangle_{H_1 \oplus H_2} \\
    &= W_2[\sum_{j=1}^\infty f^j(\omega_1)\varepsilon_j] \\
    &= \theta_m[W_2(h_2)] \implies \theta_m[W_2(h_2)] \in \D^\infty(\Omega_1 \times \Omega_2).
  \end{align*}
  
  ii). $a_1$ and $a_2$ being numerical constants, direct calculation
  shows: \\ ${\theta_m[a_1W_1(h_1)+a_2 W_2(h_2)]}$ is a Gaussian
  variable which has the same law than $W[a_1h_1+a_2h_2]$. So
  $\theta_m$ can be extended by continuity to $\L{\infty-0}(\Omega_1
  \times \Omega_2)$ and is a map of $\L{\infty-0}(\Omega_1 \times
  \Omega_2)$ in itself, again denoted $\theta_m$. Then:
  \begin{displaymath}
    \theta_{m^{-1}} \circ \theta_m [W_1(h_1)] = W(h_1)
  \end{displaymath}
  and
  \begin{align*}
    \theta_{m^{-1}} \circ \theta_m [W_2(h_2)] &= \theta_{m^{-1}}[\div(m(\omega_1)h_2)] \\
    &= \theta_{m^{-1}}\left[\sum_{j=1}^\infty f^j(\omega_1)W_2(\varepsilon_j)\right] \\
    &= \sum_{j=1}^\infty f^j(\omega_1)\theta_{m^{-1}}[W_2(\varepsilon_j)] \\
    &= \sum_{j=1}^\infty f^j(\omega_1)\div[m^{-1}(\omega_1)\varepsilon_j] \\
    &= \sum_{j=1}^\infty \div[f^j(\omega_1)m^{-1}(\omega_1)\varepsilon_j] \\
    &= \div \left[ m^{-1}(\omega_1) \left( \sum_{j=1}^\infty f^j(\omega_1)\varepsilon_j \right) \right] \\
    &= \div h_2 = W_2(h_2).
  \end{align*}
\end{proof}

We extend $m$ to $\D^\infty(\Omega_1, H_1 \oplus H_2)$, again
denoted by $m$, by:
\begin{displaymath}
  \forall V \in \D^\infty(\Omega_1,H_1), m(\omega_1)V = V .
\end{displaymath}
Then, we extend $m$ to $\D^\infty(\Omega_1 \times \Omega_2, H_1 \oplus
H_2)$ in intself denoted again by $m$, with Theorem 2.2.i. This last
extension is $\D^\infty(\Omega_1 \times \Omega_2)$-linear because it
is $\D^\infty(\Omega_i)$-linear ($i=1,2$), so is linear for finite sum
of products like $\alpha(\omega_1)\beta(\omega_2), {\alpha(\omega_1)
  \in \D^\infty(\Omega_1)} ,\beta(\omega_2) \in \D^\infty(\Omega_2)$
and, by $\D^\infty$-density of these linear combinations in
$\D^\infty(\Omega_1 \times \Omega_2)$, is $\D^\infty(\Omega_1 \times
\Omega_2)$-linear.

Denote $H = H_1 \oplus H_2$. $(e_j)_{j\in \mathbb N_\ast}$ being an
Hilbertian basis of $H$, we define $m_R$, a linear operation from a
subset of $\D^\infty(\Omega_1 \times \Omega_2, H \otimes H)$ by
\\ $m_R(\sum_{j=1}^k e_j \otimes Y_j) = e_j \otimes m(Y_j)$, the $Y_j$
being $\D^\infty$-vector fields in $\D^\infty(\Omega, H)$. Then it is
easy to check that: if $X_j, j=1,\dots,k$ are constant vectors of $H$,
\begin{displaymath}
  m_R \left( \sum_{j=1}^k X_j \otimes Y_j \right) = \sum_{j=1}^k X_j \otimes m(Y_j) ;
\end{displaymath}
so the definition of $m_R$ does not depend on the choosen Hilbertian
basis. With Theorem 2.4 and Corollary 2.4, we extend $m_R$ in an
$\D^\infty$-continuous operator on $\D^\infty(\Omega, H_1 \oplus
H_2)$. We can also define an operator, $\div_R$, on $\sum_{j=1}^k X_j
\otimes Y_j$, the $X_j$ being constant vectors of $H$, by
\begin{displaymath}
  \div_R \left( \sum_{j=1}^k X_j \otimes Y_j \right) = \sum_{j=1}^k (\div Y_j)X_j.
\end{displaymath}
Again, thanks to the extension Theorem 2.4 and Corollary 2.4, we can
extend $\div_R$ in an $\D^\infty$-continuous operator from
$\D^\infty(\Omega, H \otimes H)$ in $\D^\infty(\Omega, H)$, denoted
again $\div_R$. Then:

\begin{lem} 
  \label{lemma:5.1}
  If $X \in \D^\infty(\Omega, H)$, $(e_i)_{i \in \mathbb N_\ast}$
  being an Hibertian basis of $H$, then
  \begin{displaymath}
    \div_R(\Grad X) + X = \Grad(\div X).
  \end{displaymath}
\end{lem}

\begin{proof}
  Let $X \in \D^\infty(\Omega, H)$ such that only $N$ components of
  $X$ are not $0$, and write $X = \sum_{j=1}^N X^j e_j$. Then,
  \begin{align*}
    \Grad (\div X) - X &= \Grad \left( \sum_{j=1}^N \langle \Grad X^j, e_j \rangle_H + \sum_{j=1}^N X^j W(e_j) \right) - X \\
    &= \sum_{i=1}^\infty \left( \sum_{j=1}^N \langle \Grad (\langle \Grad X^j, e_j \rangle_H), e_i \rangle_H e_i \right) + \sum_{j=1}^N (\Grad X^j) W(e_j) \\
    &= \lim_{\substack{k \up \infty \\ k > N}} \left[ \sum_{i=1}^k \sum_{j=1}^N \langle \Grad (\langle \Grad X^j, e_j \rangle_H), e_i \rangle_H e_i \right] + \sum_{j=1}^N (\Grad X^j) W(e_j) \\
    &= \div_R( \Grad X ),
  \end{align*}
  because 
  \begin{displaymath}
    \Grad X = \sum_{j=1}^N \Grad X^j \otimes e_j = \lim_{k \up\infty} \sum_{i=1}^k [e_i \otimes (\sum_{j=1}^N \langle \Grad X^j,e_i \rangle_H e_j)],
  \end{displaymath}
  and
  \begin{align*}
    \div_R (\Grad X) &= \lim_{k \up \infty} \sum_{i=1}^k \left[ e_i \div \left( \sum_{j=1}^N \langle \Grad X^j, e_i \rangle_H e_j \right) \right] \\
    & \!\!\!\!\!\!\!\!\!\!\!\! = \lim_{k \up \infty} \sum_{i=1}^k \sum_{j=1}^N e_i \left[ \langle \Grad X^j, e_i \rangle_H W(e_j) + \langle e_j, \Grad(\langle \Grad X^j, e_i \rangle_H) \rangle_H \right] \\
    & \!\!\!\!\!\!\!\!\!\!\!\! = \sum_{j=1}^N \Grad X^j W(e_j) + \lim_{k \up \infty} \sum_{i=1}^k \sum_{j=1}^N \langle e_j, \Grad(\langle \Grad X^j, e_i \rangle_H) \rangle_H e_i ,
  \end{align*}
  and as $X \in \D^\infty(\Omega, H)$, we have
  \begin{displaymath}
    \langle e_j, \Grad( \langle \Grad X^j, e_i \rangle_H ) \rangle_H = 
    \langle e_i, \Grad( \langle \Grad X^j, e_j \rangle_H ) \rangle_H .
  \end{displaymath}
\end{proof}

Let $V \in \D^\infty(\Omega_1 \times \Omega_2, H)$. We define an
$\mathbb R$-linear operator $\delta$ on $\D^\infty(\Omega_1 \times
\Omega_2, H)$ by:
\setcounter{equation}{0}
\begin{equation} \label{eqn:1}
  \delta V = \div_R [(m^{-1})_R (\Grad(mV) - m_R \Grad V)] + mV .
\end{equation}
$\delta$ is well defined and goes from $\D^\infty(\Omega_1 \times
\Omega_2, H)$ in itself.

\begin{lem}
  \label{lemma:5.2}
  If $V_1(\omega_1) \in H_1, V_2 (\omega_2) \in H_2$, then
  \begin{displaymath}
    \div_R[m_R(V_1 \otimes V_2)] = \tilde \theta_m(\div_R(V_1 \otimes V_2)),
  \end{displaymath}
  $\tilde \theta_m$ being the extension of $\theta_m$, as in
  Corollary 2.3.
\end{lem}

\begin{proof}
  Straightforward computation.
\end{proof}

We now return to the operator $\delta$ defined in (\ref{eqn:1}).
\begin{thm} 
  \label{thm:5.2}
  \begin{enumerate}[label=\roman*)]
  \item $\delta$ sends $\D^\infty(\Omega, H)$ in $\D^\infty(\Omega,
    H)$, continuously.
  \item $\delta [\Grad (fg)] = f \delta (\Grad g) + g \delta (\Grad
    f)$, for all $f,g \in \D^\infty(\Omega_1 \times \Omega_2)$.
  \end{enumerate}
\end{thm}

\begin{proof}
  i). All operators in $\delta$ are continuous from
  $\D^\infty(\Omega_1 \times \Omega_2, H)$ in $\D^\infty(\Omega_1
  \times \Omega_2, H \otimes H)$, or $\D^\infty(\Omega_1 \times
  \Omega_2, H \otimes H)$ in $\D^\infty(\Omega_1 \times \Omega_2, H)$.

  ii). Direct calculus on the following cases:
  $f_1(\omega_1)g_1(\omega_1)$, $f_1(\omega_1)g_2(\omega_2)$,
  $f_2(\omega_2)g_2(\omega_2)$, $f_i,g_i, i=1,2$ being functions of
  $\D^\infty(\Omega_i)$.
\end{proof}

\begin{thm} 
  \label{thm:5.3}
  $\theta_m$ can be extended in a bicontinuous bijection of
  $\D^\infty(\Omega_1 \times \Omega_2)$ in itself.
\end{thm}

\begin{proof}
  The only points left to prove, after Theorem \ref{thm:5.1}, are
  $\range \theta_m \subset \D^\infty(\Omega_1 \times \Omega_2)$ and
  the $\D^\infty$-continuity of $\theta_m$. If we denote $\tilde
  \theta_m$ the extension of $\theta_m$ to $\D^\infty(\Omega_1 \times
  \Omega_2, H)$ in $\L{\infty-0}(\Omega, H)$ as in Corollary 2.3, we
  have:
  \begin{displaymath}
    \tilde \theta [\delta (\Grad (fg))] = \theta_m(f) \delta(\Grad g) + \theta_m(g) \delta (\Grad f).
  \end{displaymath}
  As we also have:
  \begin{displaymath}
    \Grad \theta_m(fg) = \theta_m(f) \Grad \theta_m(g) + \theta_m(g) \Grad \theta_m(f).
  \end{displaymath}
  $\tilde \theta_m \circ \delta(\Grad )$ and $\Grad \theta_m$ are
  two $\theta$-derivations (cf.\ Definition 2.3) which coincide on
  $W(h_1)$, $h_1$ constant vector of $H_1$ and $W(h_2)$, $h_2$
  constant vector of $H_2$: with lemmas \ref{lemma:5.1} and
  \ref{lemma:5.2}, we have
  \begin{align*}
    \tilde \theta_m[\delta (\Grad W_2 (h_2))] &= \tilde \theta_m[\div_R((m_R)^{-1} \Grad m h_2)] + \tilde \theta_m(m h_2) \\
    &= \div_R(m_Rm_R^{-1}\Grad m h_2) + m h_2 \\
    &= \div_R(\Grad m h_2) + m h_2 \\
    &= \Grad (\div m h_2) - m h_2 + m h_2 \\
    &= \Grad \div (m h_2) \\
    &= \Grad \theta_m[W_2(h_2)] .
  \end{align*}
  Then $\tilde \theta_m \circ \delta(\Grad )$ and $\Grad \theta_m$
  coincide on all polynomials on Gaussian variables, and as they both
  are $\D^\infty$-continuous $\theta$-derivations, then on
  $\D^\infty$-functions.

  Now we proceed by induction: suppose $\theta_m$ sends continuously
  $\D^\infty(\Omega_1 \times \Omega_2)$ in $\D^\infty_r(\Omega_1
  \times \Omega_2)$ ; then $\tilde \theta_m$ sends continuously
  $\D^\infty(\Omega_1 \times \Omega_2, H)$ in $\D^\infty_r(\Omega_1
  \times \Omega_2, H)$. And for all $f \in \D^\infty(\Omega_1 \times
  \Omega_2)$, we will have
  \begin{displaymath}
    \tilde \theta_m[\delta \Grad f] = \Grad \theta_m f,
  \end{displaymath}
  which implies $\theta_m f \in \D^\infty_{r+1}(\Omega_1 \times
  \Omega_2)$. The $\D^\infty$-continuity of $\theta_m$ is obtained
  with the closed graph theorem.
\end{proof}

\begin{thm}[Reciprocal of Theorem \ref{thm:5.3}] 
  \label{thm:5.4}
  In the same setting than previously, let $m$ be an operator from
  $\Omega_1$ into the unitary operators on $H_2$ ; we can define a
  map $\theta_m$:
  \begin{align*}
    \theta_m[W_1(H_1)] &= W_1(h_1), h_1 \in H_1, \\
    \theta_m[W_2(h_2)] &= \div (mh_2), h_2 \in H_2.
  \end{align*}
  And suppose that $\theta_m$ can be extended in a diffeomorphism
  of $\D^\infty(\Omega_1 \times \Omega_2)$ in itself. Then $m$, and
  $m^{-1}$, are $\D^\infty(\Omega_1,H_2)$ multiplicators.
\end{thm}

\begin{proof}
  Straightforward computation shows that if if $X \in \D^\infty(\Omega_2,
  H_2)$, $X$ constant, then $mX = \Grad _2\theta_m[\div X]$ because
  $m$ depends only of varaibles in $\Omega_1$, and so relatively to
  the $\Omega_2$-variables, is constant.

  If $f \in \D^\infty(\Omega_1)$, we also have
  \begin{displaymath}
    \theta_m(f \div X) = f \theta_m(\div X),
  \end{displaymath}
  and again
  \begin{displaymath}
    \Grad _2\theta_m [f \div X] = f mX = m(f(\omega_1) X).
  \end{displaymath}
  Now $m(\omega_1)$ is $\L{\infty-0}(\Omega_1,H_2)$-continuous became
  $m$ is unitary ; so $m(\cdot)$ is closed in the $\D^\infty$-topology
  which is finer than the $\L{\infty-0}$-topology. So the finite
  linear sums like:
  \begin{displaymath}
    \sum_{i=1}^k f^i(\omega_1) \varepsilon_i,
  \end{displaymath}
  $(\varepsilon_i)_{i \in \mathbb N_\ast}$ an Hilbertian basis of $H_2$ and $fî(\omega_1) \in
  \D^\infty(\Omega_1)$, being a dense set in
  $\D^\infty(\Omega_1,H_2)$, $m$ is a multiplicator from
  $\D^\infty(\Omega_1,H_2)$ in itself.

  Same demonstration for $m^{-1}$.
\end{proof}

\begin{rem}
  \label{remark:5.3}
  A particular case of \ref{thm:5.2} is the following: let $t_0 \in
  {]0,1[}$, $\mathcal W_1$ the Wiener space built on $[0,t_0]$ and
  $\mathcal W_2$ the Wiener space built on $[t_0,1]$, and $\mathcal
  U_1(t,\omega_1)$ an unitary operator of $\mathbb R^n$ to $\mathbb
  R^n$, which is a multiplicator and such that $\forall t \in [t_0,
  1], \mathcal U_1(t, \omega_1) \in \mathcal F_{t_0}$. Let
  \begin{displaymath}
    \mathcal U(t, \omega_1) = \mathds 1_{[0,t_0[}(t) \mathrm{Id}_{\mathbb R^n} + \mathcal U_1(t,\omega_1) \mathds 1_{[t_0,1]}(t),
  \end{displaymath}
  and let
  \begin{displaymath}
    m(\omega_1)h_1(t) = \int_0^{t \wedge t_0} \dot h_1(s) \diff s, h_1 \in H_1 \quad (H_1 \text{ the Cameron-Martin space of $\mathcal W_1$})
  \end{displaymath}
  and
  \begin{displaymath}
    m(\omega_1)h_2(t) = \int_{t_0}^t \mathcal U^{-1}(s,\omega_1) \dot h_2(s) \diff s, h_2 \in H_2 \quad (H_2 \text{ the Cameron-Martin space of $\mathcal W_2$}).
  \end{displaymath}
  It is easy to see that such an operator $m$ is a multiplicator of
  $\D^\infty(\Omega_1,H_2)$ in $\D^\infty(\Omega_1,H_2)$ and that
  $m(\omega_1)$is an unitary operator on $H_2$. So as in Theorem
  \ref{thm:5.3}, if we denote
  \begin{align*}
    \theta_m[W(h_1)] &= W(h_1), \\
    \text{and } \theta_m[W(h_2)] &= \int_0^1 {}^t\dot h_2 \mathcal U^{-1} \itodiff B,
  \end{align*}
  $\theta_m$ can be extende in a $\D^\infty$-diffeomorphism of
  $\D^\infty(\mathcal W_1 \times \mathcal W_2)$ in itself.

  Conversely, if $\mathcal U_1$ is such as in this remark
  \ref{remark:5.3} and if the $\theta_m$ associated to $\mathcal U_1$
  is a $\D^\infty$-diffeomorphism, then $m$ is a multiplicator.

  \begin{proof}
    Use Theorem \ref{thm:5.4}.
  \end{proof}
\end{rem}

\begin{example}[Process $\mathcal U$, adapted, multiplicator, with
  values in $n \times n$-unitary matrices such that the associated
  $\theta_{\mathcal U}$ $(\theta_{\mathcal U}(W(h)) = \int_0^1
  {}^t\dot h \mathcal U^{-1} \itodiff B)$ is not a
  $\D^\infty$-diffeomorphism]
  Let $(X_t,Y_t)$ a standard brownian on $\mathbb R^2$ and let
  \begin{displaymath}
    \mathcal U(t,\omega) = 
    \begin{pmatrix}
      \cos \frac{X_t}{\sqrt t} & \sin \frac{X_t}{\sqrt t} \\
      - \sin \frac{X_t}{\sqrt t} & \cos \frac{X_t}{\sqrt t}
    \end{pmatrix}.
  \end{displaymath}
  $\mathcal U(t,\omega)$ is an unitary operator on $H_1 \oplus H_2$.

  To show that $\mathcal U(t,\omega)$ is a multiuplicator, we use the
  criterium 4.1: \\ $\forall r>1, \exists s > 1, \exists C(r,s),
  \forall f \in \D^2_\infty(\Omega, \mathbb R^2),$ the map $f \mapsto
  \mathcal U(t,w)f \ (\D^2_r \to \D^2_s)$ is bounded with a norm
  less or equal than $C(r,s)$ ; we note that $X_t/\sqrt t$ is in
  $\mathcal C^1$ (for a fixed $t$) so $\exists h_t \in H_1, X_t/\sqrt
  t = W(h_t)$. Then
  \begin{displaymath}
    \norm{h_t}_{H_1} = 1 \ \left( = \norm{ \frac{X_t}{\sqrt t} }_{L^2(\Omega_1)} \right) ; 
  \end{displaymath}
  then
  \begin{displaymath}
    \norm{ \Grad \left(\cos \frac{X_t}{\sqrt t} \right) }_{H_1} = \norm{ \left( \sin \frac{X_t}{\sqrt t} \right) h_t }_{H_1} \leq 1.
  \end{displaymath}
  In the same way:
  \begin{displaymath}
   \norm{ \Grad ^k\left( \cos \frac{X_t}{\sqrt t} \right) }_{\otimes^k H_1} \leq 1.
  \end{displaymath}
  Then straightforward computation shows that the map $f \mapsto \mathcal
  U(t,\omega)f$ verifies the above mentionned criterium 4.1. (To
  simplify the calculus, use lemma 4.1.i).

  Now let $A$ the process 
  \begin{displaymath}
    A = 
    \begin{pmatrix}
      0 & - a(t,\omega) \\
      a(t,\omega) & 0
    \end{pmatrix}
  \end{displaymath}
  where $a(\omega,t)$ is an adapted multiplicator. Then $\div A \Grad$
  exists as a derivation, and
  \begin{align*}
    D_aX_t &= - \int_0^t a(s,\omega) \itodiff Y_s, \\
    D_aY_t &= \int_0^t a(s, \omega) \itodiff X_s.
  \end{align*}
  Then 
  \begin{displaymath}
    \theta_{\proc U}
    \begin{pmatrix}
      X_t \\ Y_t
    \end{pmatrix} = \int_0^t
    \begin{pmatrix}
      \cos \frac{X_s}{\sqrt s} & -\sin \frac{X_s}{\sqrt s} \\
      \sin \frac{X_s}{\sqrt s} & \cos \frac{X_s}{\sqrt s}
    \end{pmatrix} .
    \begin{pmatrix}
      \diff X_s \\
      \diff Y_s
    \end{pmatrix} =
    \begin{pmatrix}
      \int_0^t \cos \frac{X_s}{\sqrt s}\itodiff X_s - \int_0^t \sin \frac{X_s}{\sqrt s}\itodiff Y_s \\
      \int_0^t \sin \frac{X_s}{\sqrt s}\itodiff X_s + \int_0^t \cos \frac{X_s}{\sqrt s}\itodiff Y_s
    \end{pmatrix}
  \end{displaymath}
  From that we deduce 
  \begin{align*}
    D_a(\theta_{\proc U}(X_t)) &= \int_0^t \frac 1 {\sqrt s} \sin
    \frac{X_s}{\sqrt s} \left( \int_0^s a(u, \omega) \itodiff Y_u \right)
    \itodiff X_s - \int_0^t \cos \frac{X_s}{\sqrt s} a(s, \omega) \itodiff Y_s \\
    &+ \int_0^t \frac 1 {\sqrt s} \cos \frac{X_s}{\sqrt s} \left(
      \int_0^s a(u, \omega) \itodiff Y_u \right) \itodiff Y_s - \int_0^t
    \sin \frac{X_s}{\sqrt s} a(s, \omega) \itodiff X_s ,
  \end{align*}
  \begin{align*}
    D_a(\theta_{\proc U}(Y_t)) &= - \int_0^t \frac 1 {\sqrt s} \cos
    \frac{X_s}{\sqrt s} \left( \int_0^s a(u, \omega) \itodiff Y_u \right)
    \itodiff X_s - \int_0^t \sin \frac{X_s}{\sqrt s} a(s, \omega)
    \itodiff Y_s \\
    &+ \int_0^t \frac 1 {\sqrt s} \sin \frac{X_s}{\sqrt s} \left(
      \int_0^s a(u, \omega) \itodiff Y_u \right) \itodiff Y_s + \int_0^t
    \cos \frac{X_s}{\sqrt s} a(s, \omega) \itodiff X_s.
  \end{align*}

  Now let suppose that $\theta_{\proc U}$ admits an inverse which
  is a $\D^\infty$-morphism ; we will show that
  \begin{lem} 
    \label{lemma:5.3}
    $\theta_{\proc U}^{-1}(D_a(\theta_{\proc U}))$ is a
    $\D^\infty$-continuous derivation of $\D^\infty(\Omega_1 \times
    \Omega_2)$, which has the form: $\div\tilde A\Grad$.
  \end{lem}
  \begin{proof}
    i). $\theta_{\proc U}(\mathcal F_t) \subset \mathcal F_t$: if $W(h) \in \mathcal F_t$, then $\dot h(s) = 0$ for $s> t$ and 
    \begin{displaymath}
      \theta [W(h)] = \int_0^1 {}^t\dot h \proc U^{-1} \itodiff B = \int _0^t {}^t\dot h \proc U^{-1} \itodiff B \in \mathcal F_t.
    \end{displaymath}

    ii). If $f \in \mathcal F_t^\perp$, then $\exists b (s, \omega) \in
    \D^\infty$ such that $f = \int_t^1 b(s, \omega) \itodiff B_s$ implies
    $\theta_{\proc U}(f) = \int_t^1 \theta b(s, \omega) {\proc
      U}^{-1} \itodiff B_s \in \mathcal F_t^\perp$. So $\theta_{\proc
      U}(f) \in \mathcal F_t^\perp$ implies $\theta_{\proc U}^\ast
    f \in \mathcal F_t$ if $f \in \mathcal F_t$ ($\theta_{\proc
      U}^{-1} = \theta_{\proc U}^\ast$ sends $\mathcal F_t$ in
    $\mathcal F_t$, $\theta_{\proc U}$ is an $\L{2}$-isometry). So
    $\theta_{\proc U}^{-1}D_a\theta_{\proc U}$ is adapted.

    iii). $\forall f \in \D^\infty$, we have
    \begin{displaymath}
      \int D_a(\theta_{\proc U} f) \mathbb P(\diff \omega) = 0, 
    \end{displaymath}
    so 
    \begin{displaymath}
      \int \theta_{\proc U}^{-1} D_a(\theta_{\proc U} f) \mathbb P(\diff \omega) = \int \theta_{\proc U}(1) D_a(\theta_{\proc U} f) \mathbb P(\diff \omega) = 0.
    \end{displaymath}

    iv). $\theta_{\proc U}^{-1}D_a\theta_{\proc U}$ is an
    operator, $\D^\infty$-continuous adapted, and with a null
    divergence: so, with Theorem 4.1, $\theta_{\proc
      U}^{-1}D_a\theta_{\proc U}$ can be written as $D_{\tilde a} =
    \div \tilde A \Grad$, with
    \begin{displaymath}
      \tilde A = 
      \begin{pmatrix}
        0 & \tilde a(t, \omega) \\
        -\tilde a(t, \omega) & 0
      \end{pmatrix}.
    \end{displaymath}
  \end{proof}

  Now we take an $\tilde a$ determinist, and try to find the
  corresponding $a$ such that $\theta_{\proc U}^{-1} D_a
  \theta_{\proc U} = D_{\tilde a}$. As we have suppose the
  existence of $\theta^{-1}_{\proc U}$, we have
  \begin{displaymath}
    \theta_{\proc U}^{-1}D_a\theta_{\proc U} = D_{\tilde a} \implies 
    D_a\theta_{\proc U} = \theta_{\proc U}D_{\tilde a} \implies
    D_a \theta_{\proc U} 
    \begin{pmatrix}
      X_t \\ Y_t
    \end{pmatrix} = 
    \theta_{\proc U} D_{\tilde a} 
    \begin{pmatrix}
      X_t \\ Y_t
    \end{pmatrix}.
  \end{displaymath}
  \begin{displaymath}
    D_{\tilde a} 
    \begin{pmatrix}
      X_t \\ Y_t
    \end{pmatrix} = 
    \begin{pmatrix}
      -\int_0^t \tilde a (s) \itodiff Y_s \\
      + \int_0^t \tilde a (s) \itodiff X_s
    \end{pmatrix},
  \end{displaymath}
  so
  \begin{align*}
    \theta_{\proc U} D_{\tilde a} 
    \begin{pmatrix}
      X_t \\ Y_t
    \end{pmatrix} &=
    \begin{pmatrix}
      \theta_{\proc U}\left( - \int_0^t \tilde a (s) \itodiff Y_s \right) \\
      \theta_{\proc U}\left( + \int_0^t \tilde a (s) \itodiff X_s \right)
    \end{pmatrix} \\
    &=
    \begin{pmatrix}
      - \int_0^t \tilde a (s) \theta_{\proc U} . \left( \diff Y_s \right) \\
      + \int_0^t \tilde a (s) \theta_{\proc U} . \left( \diff X_s \right)
    \end{pmatrix} \\
    &=
    \begin{pmatrix}
      - \int_0^t \tilde a ( \sin(X_s/\sqrt s) \itodiff X_s + \cos(X_s/\sqrt s)\itodiff Y_s ) \\
      + \int_0^t \tilde a ( \cos(X_s/\sqrt s) \itodiff X_s - \sin(X_s/\sqrt s)\itodiff Y_s ) \\
    \end{pmatrix} \\
    &=
    \begin{pmatrix}
      - \int_0^t \tilde a \sin(X_s/\sqrt s) \itodiff X_s - \int_0^t \tilde a \cos(X_s/\sqrt s)\itodiff Y_s \\
      + \int_0^t \tilde a \cos(X_s/\sqrt s) \itodiff X_s - \int_0^t \tilde a \sin(X_s/\sqrt s)\itodiff Y_s \\
    \end{pmatrix}.
  \end{align*}
  
  From $\theta_{\proc U} D_{\tilde a} (X_t, Y_t) = D_a
  \theta_{\proc U} (X_t, Y_t)$, we get, with
  \begin{displaymath}
    F(s,\omega) = -a(s, \omega) + \frac 1 s \int_0^s a(u, \omega) \itodiff Y_u,
  \end{displaymath}
  \begin{align*}
    \int_0^t \sin \frac{X_s}{\sqrt s} F(s,\omega) \itodiff X_s + \int_0^t \cos \frac{X_s}{\sqrt s} &F(s,\omega) \itodiff Y_s \\
    &= - \int_0^t \tilde a \sin \frac{X_s}{\sqrt s} \itodiff X_s - \int_0^t \tilde a \cos \frac{X_s}{\sqrt s} \itodiff Y_s
  \end{align*}
  and
  \begin{align*}
    - \int_0^t \cos \frac{X_s}{\sqrt s} F(s,\omega) \itodiff X_s + \int_0^t \sin \frac{X_s}{\sqrt s} &F(s,\omega) \itodiff Y_s \\
    &= + \int_0^t \tilde a \cos \frac{X_s}{\sqrt s} \itodiff X_s - \int_0^t \tilde a \sin \frac{X_s}{\sqrt s} \itodiff Y_s .
  \end{align*}
  From these two last equations:
  \begin{displaymath}
    \tilde a (t) = -a(t, \omega) + \frac 1 {\sqrt t} \int_0^ta(s, \omega) \itodiff Y_s.
  \end{displaymath}
  Let $a(\omega, t) = \sum_0^\infty a_n(\omega, t), a_n \in \mathcal
  C_n(\Omega)$ then 
  \begin{displaymath}
    -\sum_0^\infty a_n(t, \omega) + \sum_0^\infty \frac 1 {\sqrt t} \int_0^t a_n(s, \omega) \itodiff Y_s = \tilde a (t).
  \end{displaymath}
  Then $a_0(t, \omega) = - \tilde a (t)$, and if $b_n = \norm{a_n}_{\L
    2 (\Omega)}^2$, we get
  \begin{displaymath}
    b_n^2(t) = \frac 1 t \int_0^t b_{n-1}^2 \diff s.
  \end{displaymath}
  If for instance $\tilde a (t) = 1$, then $a(t, \omega) \notin \D^\infty$.
\end{example}

\medskip

\subsection{Some other conditions to obtain a $\mathbb{D}^\infty$-morphism}

Now we show that if $\sup_{t \in [0,1]} ||\Grad ^j \proc U||_{\otimes^j
  H} < +\infty$, $\theta_{\proc U}$ is a $\D^\infty$-morphism.
\begin{thm} 
  \label{thm:5.5} 
  If $\proc U(t,\omega)$ is an adapted process from $[0,1] \times
  \Omega$ with values in the $n \times n$-unitary matrices (on
  $\mathbb R^n$) and such that each $||\Grad ^j \proc U||_{\otimes^j
    H}, j \in \mathbb N_\ast$ is also uniformly (for $t\in [0,1]$)
  bounded, then $\theta_{\proc U}$ being the $\L{\infty-0}$-morphism
  associated to $\proc U$, is a $\D^\infty$-morphism.
\end{thm}

\begin{proof}
  For $h \in H$, we have $\theta_{\proc U}(W(h)) = \int_0^1
  {}^t\dot h \proc U^{-1} \itodiff B$. First, we show that if $f \in
  \mathcal C_k$, then there exists a polynomial $P_r(k)$ such that:
  $\norm{f}_{\D_r^2}^2 = P_r(k) \norm{f}_{\L 2}^2$. $L$ denoting as
  usual the O.U. operator, as $\D_r^p$-norm, we use
  \begin{displaymath}
    \norm{f}_{\D_r^p} = \left( \sum_{j=0}^r \norm{\Grad ^j f}_{\L p (\Omega, \otimes^j H)}^p \right)^{1/p}.
  \end{displaymath}
  From $L \Grad f - \Grad Lf = \Grad f$, we have $L(\Grad ^r f) -
  \Grad ^r(Lf) = r \Grad ^r f$ and
  \begin{align*}
    \langle \Grad ^r f, \Grad ^r f \rangle_{\otimes^r H} &=
    - \langle L (\Grad ^{r-1} f), \Grad ^{r-1} f \rangle_{\otimes^{r-1} H} \\
    &= (k-r+1) \langle \Grad ^r f, \Grad ^r f \rangle_{\otimes^r H},
  \end{align*}
  so
  \begin{displaymath}
    \norm{\Grad ^r f}_{\L 2}^2 = \frac{k!}{(n-k)!} \norm{f}_{\L 2}^2 .
  \end{displaymath}
  But $\norm{(I - L)^{r/2} f}_{\L 2} = (1+k)^r \norm{f}_{\L 2}^2$, so
  with $\norm{f}_{\D_r^2}^2 = \sum_{j=0}^r \norm{D^j f}_{\L 2}^2$, we
  get: $\norm{f}_{\D_r^2}^2 \simeq k^r$ when $k \to +\infty$, and
  \begin{equation}
    \label{eqn:2}
    \norm{f}_{\D_r^2}^2 = P_r(k) \norm{f}_{\L 2}^2, \quad P_r(k) \text{ polynomial}
  \end{equation}
  Now with $f = (f_1, \dots, f_n)$ and the hypothesis on $\proc
  U(t,\omega)$, Leibnitz formula implies, by induction, with $f^i \in
  \D^\infty$:
  \begin{equation}
    \label{eqn:3}
    \norm{ (\proc U^{-1})f }_{\D_r^2} \leq \norm{f}_{\D_r^2(\otimes^r H)} + K(r) \norm{f}_{\D_{r-1}^2(\otimes^{r-1} H)},
  \end{equation}
  $K(r)$ being a constant, $r$-depending.

  Now if $f = (f_1, \dots, f_n)$ and $f_i \in \mathcal C_k,
  i=1,\dots,n$, we have (Clark-Ocone): $f^i = \int_0^1 g_j^i \itodiff
  B^j$, with $g_j^i \in \mathcal C_{k-1}$. And 
  \begin{equation}
    \theta_{\proc U}
    \begin{pmatrix}
      f_1 \\ \vdots \\ f_n
    \end{pmatrix} = \int_0^1 \theta_{\proc U}(g_j^i)(\mathcal U^{-1})_\ell^j \itodiff B^\ell, \quad i=1,\dots,n.
  \end{equation}
  If $\vec g = (g_1, \dots, g_n)$, vector of $n$ functions in
  $\D^\infty$, we write
  \begin{displaymath}
    \norm{ D^r \vec g }_{\otimes^r H} \leq \sum_{\ell=1}^n \norm{ D^r g_\ell }_{\otimes^r H}.
  \end{displaymath}
  Then, if $f \in \D^\infty(\Omega)$, we can write:
  \begin{displaymath}
    f = E(f) + \int_0^1 g_\ell \itodiff B^\ell,
  \end{displaymath}
  $B^1, \dots, B^n$ being $n$ independants Brownians, and $g_\ell \in
  \D^\infty(\Omega)$. If $E(f) = 0$, we have
  \begin{equation} 
    \label{eqn:5}
    \norm{ f }_{\D_r^2(\Omega)} \leq \left( \int_0^1 \diff s \norm{\vec g}_{\D_r^2}^2 \right)^{1/2} + K_2(r) \left(\int_0^1 \diff s \norm{\vec g}_{\D_{r-1}^2}^2 \right)^{1/2},
  \end{equation}
  $K_2(r)$ being an $r$-depending constant.

  But $\theta_{\proc U}(f) = \int_0^1 \theta_{\proc
    U}(g_\ell)(\proc U^{-1})_j^\ell \itodiff B^j, j=1,\dots,n$. With
  (\ref{eqn:5}), we have:
  \begin{equation} 
    \label{eqn:6}
    \begin{split}
      \norm{ \theta_{\proc U} f }_{\D_r^2} \leq &\ \left( \int_0^1 \diff s \norm{\overrightarrow{\theta_{\proc U}(g_\ell)(\proc U^{-1})}_j^\ell}_{\D_r^2}\right)^{1/2} \\
      & + K_2(r) \left( \int_0^1 \diff s \norm{\overrightarrow{\theta_{\proc U}(g_\ell)(\proc U^{-1})}_j^\ell}_{\D_{r-1}^2} \right)^{1/2}.
    \end{split}
  \end{equation}
  Suppose $f \in \mathcal C_k$, then $g \in \mathcal C_{k-1}$ and
  denoting $C(m,r) = \norm{\theta_{\proc U} f}_{\D_r^2} / \norm
  f_{\D_r^2}$, with (\ref{eqn:3}), we have:
  \begin{displaymath}
    \norm{\theta_{\proc U} f}_{\D_r^2} \leq \left( \int_0^1 \diff s \norm{\overrightarrow{\theta_{\proc U}(g)}}_{\D_r^2}\right)^{1/2} + K_3(r) \left( \int_0^1 \diff s \norm{\overrightarrow{\theta_{\proc U}(g)}}_{\D_{r-1}^2} \right)^{1/2},
  \end{displaymath}
  so
  \begin{equation}
    \label{eqn:7}
    \begin{split}
      \norm{\theta_{\proc U} f}_{\D_r^2} \leq &\ C(k-1,r) \left(
        \int_0^1 \diff s \norm{\overrightarrow{\theta_{\proc
              U}(g)}}_{\D_r^2}\right)^{1/2} \\
      &+ K_3(r) C(k-1,r-1) \left(
        \int_0^1 \diff s \norm{\overrightarrow{\theta_{\proc
              U}(g)}}_{\D_{r-1}^2} \right)^{1/2}.
    \end{split}
  \end{equation}

  Using (\ref{eqn:2}) in (\ref{eqn:7}),
  \begin{displaymath}
    \begin{split}
      \norm{\theta_{\proc U} f}_{\D_r^2} \leq &\ C(k-1,r)
      \sqrt{\frac{P_r(k-1)}{P_r(k)}} \norm{f}_{\D_r^2} \\
      &+ K_3(r) C(k-1,r-1) \sqrt{\frac{P_{r-1}(k-1)}{P_{r-1}(k)}}
      \norm{f}_{\D_{r-1}^2}
    \end{split}
  \end{displaymath}
  which implies, as soon as $k$ is big enough ($k \geq k_0$):
  \begin{equation}
    \label{eqn:8}
    \norm{\theta_{\proc U} f}_{\D_r^2} \leq \left[ C(k-1,r) + K_4(r) C(n-1,r-1) \right] \norm{f}_{\D_r^2}.
  \end{equation}
  As $\norm{\theta_{\proc U}f}_{\D_r^2} / \norm f_{\D_r^2} =
  C(k,r)$, we deduce for $k \geq k_0$:
  \begin{displaymath}
    C(k,r) \leq C(k-1,r) + K_4(r) C(k-1,r-1).
  \end{displaymath}
  So, by induction, we see that for $k \geq k_0$, $C(k,r)$ has a
  polynomial growth ; then (\ref{eqn:8}) implies that
  $\theta_{\proc U} f \in \D_r^2$. Then with $f \in \D_\infty^2$
  now, we write $f = \sum_{k=1}^\infty f_k$; then
  \begin{displaymath}
    \norm{\theta_U f}_{\D_r^2} \leq \sum_{k = 1}^\infty \norm{\theta_{\proc U} f_k}_{\D_r^2}.
  \end{displaymath}
  Each $\norm{\theta_{\proc U} f_k}_{\D_r^2}$ has polynomial growth
  when $k \geq k_0$ ; but the sequence $(\norm{\theta_{\proc U}
    f_k})_{\D_r^2}$ is fast decreasing ; so $\norm{\theta_{\proc
      U}f}_{\D_r^2} < +\infty$, and $\theta_{\proc U} f \in
  \D_r^2$. Now by interpolation $\theta_{\proc U} \colon \D^\infty \to
  \L{\infty-0}$ and $\theta_{\proc U} \colon \D^\infty \to \D_r^2$ ;
  so $\theta_{\proc U} \colon \D^\infty \to \D^\infty$.
\end{proof}

\begin{defn} 
  \label{defn:5.1}
  Let $\circled H$ be the set
  \begin{displaymath}
    \begin{split}
      \{ \proc U \mid \proc U \text{ process } &\text{with values in } n \times n
      \text{ unitary matrices,} \\
      &\text{adapted, } [0,1]\times \Omega \text{-mesurable, in
      } \L{\infty-0}(\Omega) \}.
    \end{split}
  \end{displaymath}
  Then we denote by $\norm{\proc U(s,\omega)}_{\mathrm{op}}$ the
  operator norm of $\proc U(s,\omega)$ on $\mathbb R^n$, and if $\proc
  U_1, \proc U_2 \in \circled H$, we denote by
  \begin{displaymath}
    d(\proc U_1,\proc U_2) = \sup_{t \in [0,1]} \norm{\norm{\proc U_1(t,\cdot) - \proc U_2(t,\cdot)}_{\mathrm{op}}}_{\L{2}(\Omega)}.
  \end{displaymath}
  Then $d$ is a distance on $\circled H$, for wich $\circled H$ is
  complete.
\end{defn}

\begin{defn} 
  \label{defn:5.2}
  We denote by $\theta_{\proc U}$ the $\L{\infty-0}$-morphism
  generated by \\ $\theta_{\proc U}(W(h)) = \int_0^1 {}^t\dot h \proc
  U^{-1} \itodiff B$. Then a $n \times n$-matrix $\mathcal V$ will be
  said to be $k$-Lipschitzian if and only if
  \begin{displaymath}
    \forall \proc U_1, \proc U_2 \in \circled H, \ 
    \norm{\norm{\theta_{\proc U_1}(\mathcal V) - \theta_{\proc U_2}(\mathcal V)}_{\mathrm{op}(\mathbb R^m)}}_{\L 2 (\Omega)} \leq k d(\proc U_1, \proc U_2).
  \end{displaymath}
\end{defn}

\begin{thm} 
  \label{thm:5.6} 
  Let $\proc U \in \circled H$ such that there exists $k, 0<k<1$ so
  that for all $s \in [0,1]$, $\mathcal U(s, \cdot)$ is
  $k$-Lipschitzian. Then $\theta_{\proc U}$ is a bijection on
  $\L{\infty-0}(\Omega)$.
\end{thm}

\begin{proof}
  For every $\proc U \in \circled H$, $\proc V \mapsto
  \theta_{\proc V}(\proc U^{-1})$ is $k$-Lipschitzian because
  $\proc U \mapsto \proc U^{-1}$ is $1$-Lipschitzian. Then, the Picard
  theorem asserts that: there exists $\proc V_0 \in \circled H$ such
  that $\theta_{\proc V_0}(\proc U^{-1}) = \proc V_0$. And
  \begin{displaymath}
    \forall h \in H,\ 
    \theta_{\proc V_0} \left( \int_0^1 {}^t\dot h \proc U^{-1} \itodiff B \right) =
    \int_0^1 {}^t\dot h \theta_{\proc V_0}(\proc U^{-1}) \proc V_0^{-1} \itodiff B =
    \int_0^1 {}^t\dot h \proc V_0 \proc V_0^{-1} \itodiff B.
  \end{displaymath}
  So $\theta_{\proc V_0} \circ \theta_{\proc U} = \mathrm{Id}$
  which proves that $\theta_{\proc V_0}$ is surjective ; and as we
  know already that $\theta_{\proc V_0}$ is injective on
  $\L{\infty-0}$, $\theta_{\proc U}$ is a bijection on
  $\L{\infty-0}(\Omega)$.
\end{proof}

\begin{rem} 
  \label{rem:5.4} 
  The set of $k$-Lipschitzian processes is not limited to the
  determinist functions: any $W(h)$, with $\norm h_{\L 2} = k < 1$, is
  a $k$-Lipschitzian process (straightforward computation).
\end{rem}

Now we define the notion of $\D^\infty$-$\alpha$-Holderian processes,
which will allow us to study cases when the morphism $\theta$,
defined on $\mathcal C_1(\Omega)$ by $\theta[W(h)] = \int_0^1 {}^t
\dot h \proc U^{-1} \itodiff B$, can be extended in a continuous morphism
on $\D^\infty(\Omega)$.

\begin{defn}[Same as definition 2.3]
  \label{defn:5.3}
  A process $X \colon [0,1] \times \Omega \to \mathbb R^n$ is said to
  be $\D^\infty$-$\alpha$-Holderian if and only if $\forall t_1,t_2
  \in [0,1], \forall (p,r) \in {[1,+\infty[} \times \mathbb N_\ast,
  \\ \exists C(p,r)$ constant such that,
  \begin{displaymath}
    \norm{X(t_2,\omega) - X(t_1,\omega)}_{\D_r^p} \leq C(p,r) |t_2 - t_1|^\alpha .
  \end{displaymath}
\end{defn}

\begin{thm} 
  \label{thm:5.7}
  Let $\proc U$ be a $\D^\infty$-adapted process, with values in $n
  \times n$ unitary matrices, $\D^\infty$-$\alpha$-Holderian, with
  $\alpha > 1/2$. Then the operator $\theta$ defined on $\mathcal C_1$
  by:
  \begin{displaymath}
    \theta_{\proc U}[W(h)] = \int_0^1 {}^t\dot h \proc U^{-1} \itodiff B,
    \quad (h \in H, B \text{ an } n\text{-Brownian}),
  \end{displaymath}
  can be extended in a continuous morphism of $\D^\infty(\Omega)$ in
  itself.
\end{thm}

\begin{proof} 
  We know already that $\theta$ can be extended in a morphism of
  $\L{\infty-0}(\Omega)$ in $\L{\infty-0}(\Omega)$, because
  $\theta$ preserves laws. We will need the three following lemmas.
\end{proof}

\begin{lem} 
  \label{lemma:5.4}
  Let $E$ be a $n \times n$ antisymmetrical constant matrix, and $t
  \in [0,1]$. For $f \in \D^\infty(\Omega)$, we define $D_{E,t}f$ by:
  \begin{displaymath}
    D_{E,t} f = \div(\mathds 1_{[0,t[}(\cdot) E) \Grad  f .
  \end{displaymath}
  If $\Delta$ is a finite subdivision of $[0,1]$, $\Delta = \{0 = t_0
  < t_1 < \dots < t_n = 1\}$, denote $S_\Delta(f) = \sum_{i=0}^n
  |D_{E,t_{i+1}}(f) - D_{E,t_i}(f)|^2$.

  Then $ \forall (p,r) \in {[1, \infty[} \times \mathbb N_\ast,
  \exists C(p,r,f) \text{ constant}$, such that
  \begin{displaymath}
    \sup_\Delta \norm{S_\Delta(f)}_{\D_r^p} \leq C(p,r,f).
  \end{displaymath}
\end{lem}

\begin{proof}
  Let $\Delta = \{0=t_0<t_1< \dots <t_n=1\}$ a fixed finite
  subdivision and $\varepsilon_i = \pm 1, i = 1,\dots,n$. Denote 
  \begin{displaymath}
    \tilde S_\Delta (\{ \varepsilon_i \})(f) = \sum_{i=0}^n \varepsilon_i [D_{E,t_{i+1}}(f) - D_{E,t_i}(f)].
  \end{displaymath}
  Then
  \begin{displaymath}
    \tilde S_\Delta (\{ \varepsilon_i \})(f) = \div B_E(\{ \varepsilon_i \}) \Grad  f
  \end{displaymath}
  with $B_E(\{\varepsilon\}) = \sum_{i=0}^n \varepsilon_i \mathds
  1_{[t_i,t_{i+1}[}(\cdot) E$. The operators $B_E(\{\varepsilon_i\})$
  are operators on $\mathbb R^n$, and as such, are uniformly bounded,
  relatively to the set $\{\varepsilon_i\}$, when the subdivision is
  fixed, and relatively to the subdivisions $\Delta$, and are
  $\Delta,\{\varepsilon_i\}$-uniformly determinists.

  So the $\tilde S_\Delta (\{\varepsilon_i\})$ are linear operators on
  $\D^\infty(\Omega)$, $\D^\infty$-uniformly bounded relatively to the
  sets $\{\varepsilon_i\}$ and the subdivisions $\Delta$. But
  \begin{displaymath}
    \frac 1 {2^{\sharp (\{\varepsilon_i\})}} \sum_{\{\varepsilon_i\}} \abs{\tilde S_\Delta( \{\varepsilon_i\})(f)}^2 = S_\Delta(f),
  \end{displaymath}
  the sum being taken on all sets $\{\varepsilon_i\}$, once the
  subdivision $\Delta$ is fixed. $S_\Delta(f)$ belonging to the convex
  enveloppe of elements whose $\D_r^p$-norms do not depend either of
  the finite subdivision $\Delta$ or of the set $\{\varepsilon_i\}$,
  there exists a constant $C(p,r,f)$ such that $\sup_\Delta
  \norm{S_\Delta(f)}_{\D_r^p} \leq C(p,r,f)$.
\end{proof}

\begin{rem} \label{remark:5.4} Later we will suppose
  $\theta_{\proc U} \colon \D^\infty(\Omega) \to
  \D_r^\infty(\Omega)$ (Theorem \ref{thm:5.11}). Then the same
  demonstration as in Lemma \ref{lemma:5.4}, applied to $\forall
  \alpha, 1 \leq \alpha \leq r,$
  \begin{displaymath}
    S_\Delta [\Grad ^\alpha \theta_{\proc U}(D_{E,t}f)] = 
    \sum_{i=0}^n \norm{\Grad ^\alpha \theta_{\proc U}(D_{E,t_{i+1}}f) - \Grad ^\alpha \theta_{\proc U}(D_{E,t_i}f)}_{\otimes^\alpha H},
  \end{displaymath}
  proves that there exists, for every $p$, a constant $C(\alpha, p, f)$ such that 
  \begin{displaymath}
    \sup_\Delta \norm{S_\Delta(\Grad ^\alpha \theta_{\proc U}(D_{E,t}f))}_{\L p (\Omega, \otimes^\alpha H)} \leq C(\alpha, p, f).
  \end{displaymath}
  Now we denote: for a subdivision $\Delta$ of $\mathbb R$, $r \in
  \mathbb N_\ast$ and $f$ a process $\mathbb R \times \Omega \to
  \mathbb R$,
  \begin{displaymath}
    V_{\Delta,\mathbb R}(\Grad ^r f) = \sum_i \norm{\Grad ^r f(x_{i+1}, \cdot) - \Grad ^r f(x_i,\cdot)}_{\otimes^r H}^2
  \end{displaymath}
  and by $V_\Delta$, when the subdivision is on $[0,1]$.
\end{rem}

\begin{lem} \label{lemma:5.5} Let $f \colon [0,1] \times \Omega \to
  \mathbb R$ a process such that
  \begin{enumerate}[label=\roman*)]
  \item $\forall t \in [0,1], f(t, \cdot) \in \D^\infty(\Omega)$ and
    $f(0,\omega) = 0$ $\mathbb P$-almost surely.
  \item $\forall r \in \mathbb N_\ast, \forall p > 1, \exists
    C(p,r,f), \sup_\Delta \norm{V_\Delta \Grad ^r f}_{\L p \Omega} <
    C(p,r,f)$.
  \item $\forall r \in \mathbb N_\ast, \int_0^1 \norm{\Grad ^r
      f}_{\otimes^r H}^2 \diff t \in \L{\infty-0}(\Omega)$.
  \end{enumerate}
  Then, the extension of $f$, denoted $\tilde f$, which equals $0$ on
  ${]-\infty, 0] \cup [2,\infty[}$, and is an affine process $g$ on
  $[1,2]$ with $g(1,\omega) = f(1,\omega)$ and $g(t,\omega) = 0$
  $\mathbb P$-almost surely on ${[2,\infty[}$, we have:
  \begin{enumerate}[label=\Roman*)]
  \item
    \begin{displaymath}
      \int_{\mathbb R} \norm{\frac{\Grad ^r \tilde f(x+h,\cdot) - \Grad ^r \tilde f(x,\cdot)}{\sqrt h}}_{\otimes^r H}^2 \diff x
    \end{displaymath}
    is $\L{\infty-0}$-bounded, $h$-uniformly.
  \item $\norm{\Grad ^r \tilde f}_{B_{2,2}^{\varepsilon/2}} \in
    \L{\infty-0}(\Omega), \forall r \in \mathbb N_\ast$ and $0 <
    \varepsilon < 1$.
  \end{enumerate}
\end{lem}

\begin{proof}
  $I)$. 
  \begin{align*}
    \norm {\int_{-\infty}^{+\infty} \right. & \left. \norm{\frac{\Grad ^r \tilde f(x+h) - \Grad ^r \tilde f(x)}{\sqrt h}}_{\otimes^r H}^2 \diff x}_{\L p (\Omega)} \\
    &= \norm {\int_0^h \sum_{n \in \mathbb Z} \norm{\frac{\Grad ^r \tilde f(x+(n+1)h) - \Grad ^r \tilde f(x+nh)}{\sqrt h}}_{\otimes^r H}^2 \diff x}_{\L p (\Omega)} \\
    &\leq \frac 1 h \norm{\int_0^h C(p,r) \diff x}_{\L p (\Omega)} = C(p,r,f).
  \end{align*}
  $II)$. We have to show that:
  \begin{displaymath}
    \norm {\int_0^\infty \diff h \frac{\norm{\Grad ^r \tilde f(x+h) - \Grad ^r \tilde f(x)}_{\L 2 (\diff x, \otimes^r H)}^2}{h^{1+2\varepsilon/2}}}_{\L p (\Omega)} < C(p,r,f) .
  \end{displaymath}
  The left-hand side of the above inequality is bounded by:
  \begin{displaymath}
    \begin{split}
      \norm {\int_0^1 \frac{\diff h}{h^\varepsilon} \right. & \left. \norm{\frac{\Grad ^r \tilde f(x+h) - \Grad ^r \tilde f(x)}{\sqrt h}}_{\L 2 (\diff x, \otimes^r H)}^2}_{\L p (\Omega)} \\
      &+ \norm {\int_1^\infty \diff h \frac{\norm{\Grad ^r \tilde f(x+h) - \Grad ^r \tilde f(x)}_{\L 2 (\diff x, \otimes^r H)}^2}{h^{1+\varepsilon}}}_{\L p (\Omega)}.
    \end{split}
  \end{displaymath}
  The first integral above is bounded because $\diff h/h^\varepsilon$
  is a bounded measure on $[0,1]$, and with $I)$. The second integral
  is bounded by:
  \begin{displaymath}
    \norm {\int_1^\infty \diff h \, 4 \frac{\norm{\Grad ^r \tilde f(x,\omega)}_{\L 2 (\diff x, \otimes^r H)}^2}{h^{1+\varepsilon}}}_{\L p (\Omega)}
  \end{displaymath}
  and with $iii)$ we get the result.
\end{proof}

\begin{lem} \label{lemma:5.6} Let $f \colon [0,1] \times \Omega \to
  \mathbb R$ a $\D^\infty$-$\alpha$-Holderian process such that:
  \begin{enumerate}[label=\roman*)]
  \item $\alpha > 1/2$,
  \item $f(0, \cdot) = 0$, $\mathbb P$-almost surely.
  \end{enumerate}
  Then, $\tilde f$ being as in Lemma \ref{lemma:5.3}, we have:
  \begin{displaymath}
    \norm{\Grad ^r \tilde f}_{B_{2,2}^{1 - \varepsilon/2}} \in \L{\infty-0}, 
      \quad 0 < \varepsilon < \alpha - \frac 1 2 .
  \end{displaymath}
\end{lem}

\begin{proof}
  We must prove
  \begin{displaymath}
    \norm{ \int_0^\infty \diff h \frac{\norm{ \Grad ^r \tilde f(x+h) - \Grad ^r \tilde f(x) }_{\L 2 (\diff x, \otimes^r H)}^2}{h^{1+2(1-\varepsilon/2)}} }_{\L p (\Omega)} < C(p,r,f).
  \end{displaymath}
  This is bounded by
  \begin{align*}
    \norm{ \int_0^1 \frac{\diff h}{h^{1-\varepsilon}} \right. & \left. \frac{\norm{ \Grad ^r \tilde f(x+h) - \Grad ^r \tilde f(x) }_{\L 2 (\diff x, \otimes^r H)}^2}{h^2} }_{\L p (\Omega)} \\
    &+ \norm{ \int_1^\infty \diff h \frac{\norm{ \Grad ^r \tilde f(x+h) - \Grad ^r \tilde f(x) }_{\L 2 (\diff x, \otimes^r H)}^2}{h^{1+2(1-\varepsilon/2)}} }_{\L p (\Omega)}.
  \end{align*}

  $\Grad ^r \tilde f$ is also $\D^\infty(\Omega,
  H)$-$\alpha$-Holderian, so: the first integral is bounded if
  $2\alpha - 2 + 1 > 0$ so if $\alpha > 1/2$ ; the second integral is
  bounded by
  \begin{displaymath}
    \norm {\int_1^\infty \diff h \, 4 \frac{\norm{\Grad ^r \tilde f}_{\L 2 (\diff x, \otimes^r H)}^2}{h^{1+2(1-\varepsilon/2)}}}_{\L p (\Omega)} .
  \end{displaymath}
\end{proof}

Now we return to Theorem \ref{thm:5.7}. We have $\theta_{\proc
  U}[W(h)] = \int_0^1{}^t\dot h \proc U^{-1}\itodiff B$. So:
\begin{equation} \label{eqn:9}
  \Grad  \theta_{\proc U}[W(h)] = \int_0^1{}^t\dot h(\Grad \proc U^{-1} \, \proc U)(\proc U^{-1} \itodiff B) + (t \mapsto (\int_0^1{}^t\dot h \proc U^{-1} \diff s)),
\end{equation}
$(t \mapsto (\int_0^1{}^t\dot h \proc U^{-1} \diff s))$ being a
$\D^\infty(\Omega, H)$-vector field. We want to generalize
(\ref{eqn:9}) to a function $f \in \D^\infty(\Omega)$. The
generalization of $t \mapsto (\int_0^1{}^t\dot h \proc U^{-1} \diff
s)$ is straightforward: using Theorem 2.3, we generalize it by
\begin{displaymath}
  t \mapsto \int_0^1 {}^t\theta_{\proc U}(\Grad  f) \proc U^{-1} \diff s.
\end{displaymath}
For the generalization of the first integral in (\ref{eqn:9}), we
denote by $E_k^\ell$ the elementary antisymmetric matrix with all
items equal to zero, except the item $e_k^\ell = +1$ and $e_\ell^k =
-1$. We write:
\begin{displaymath}
  (\Grad  \proc U^{-1}) \proc U = \sum_{k,\ell=1}^n f_k^\ell E_\ell^k,
  \quad \text{with } f_k^\ell \in \D^\infty(\Omega, H).
\end{displaymath}
We extend $f_k^\ell$ by $0$ on ${]-\infty, 0]}$, and with the affine
function, $-f_k^\ell(1,\omega) t + 2 f_k^\ell(1,\omega)$ on
$t\in[1,2]$, and $0$ after $2$. This extension of $f_k^\ell$ is again
denoted $f_k^\ell$ and by using it, we have an extension of $(\Grad 
\proc U^{-1} \,\proc U)_\ell^k$, denoted again $(\Grad  \proc U^{-1}
\,\proc U)_\ell^k$. We denote again $\theta[D_{E_k^\ell, t}f]$, the
result of the same extension procedure applied to
$\theta_{\proc U}[D_{E_k^\ell, t}f]$.

And the generalisation of the first integral in (\ref{eqn:9}) is given
by:
\begin{equation}
  \label{eqn:10}
  \int_{\mathbb R} f_k^\ell \diffoperator (\theta_{\proc U}[D_{E_{\ell,t}^k}f]),
\end{equation}
this integral being a Bochner-Russo-Valois integral. More precisely,
we will prove that, if $\alpha > 1/2$,
\begin{displaymath}
  \norm{f_k^\ell}_{B_{2,2}^{1-\varepsilon/2}(H)} \in \L{\infty-0}(\Omega,H),
\end{displaymath}
and that
\begin{displaymath}
  \norm{\theta_{\proc U}(D_{E_k^\ell,t}f)}_{B_{2,2}^{\varepsilon/2}} \in \L{\infty-0}(\Omega),
\end{displaymath}
$B_{p,q}^\lambda$ being the Besov space with indexes $\lambda, p, q$ ;
$\norm{f_k^\ell}_{B_{2,2}^{1-\varepsilon/2}(H)} \in \L{\infty-0}$
because $f_k^\ell$ is $\D^\infty$-$\alpha$-Holderian ($\proc U$,
$\Grad  \proc U^{-1}$ are $\D^\infty$-$\alpha$-Holderians), so Lemma
\ref{lemma:5.6} applies to $f_k^\ell$. 

And $\norm{\theta_{\proc
    U}(D_{E_{\ell,t}^k}f)}_{B_{2,2}^{\varepsilon/2}} \in
\L{\infty-0}(\Omega)$ because Lemma \ref{lemma:5.5} applies to
$\theta_{\proc U}(D_{E_{\ell,t}^k}f)$:
\begin{displaymath}
  \norm{\theta_{\proc U}(D_{E_{\ell,t}^k}f)}_{\L p} = \norm{D_{E_{\ell,t}^k}f}_{\L p}.
\end{displaymath}
And $B_{2,2}^{\varepsilon/2}$ and $B_{2,2}^{1-\varepsilon/2}$ are
conjugate Besov spaces, so (\ref{eqn:10}) is legitimate.

Now $\Grad  \theta_{\proc U}(f)$ is in $\D^\infty$ when $f$ is a
polynomial of a finite number of Gaussian variables,
$P[W(h_1),\dots,W(h_r)]$. Then 
\begin{displaymath}
  \Grad  \theta_{\proc U}(P(W(h_1),\dots,W(h_r)) \in \D^\infty(\Omega, H). 
\end{displaymath}
If $f \in \D^\infty$, then $\Grad  \theta_{\proc U}(f)$ can be defined
as a distribution on $\D^\infty(\Omega,H)$ by: if $X \in
\D^\infty(\Omega,H)$,
\begin{displaymath}
  (\Grad  \theta_{\proc U}(f), X) = - \int \theta_{\proc U} {\div X} \, \mathbb P(\diff \omega).
\end{displaymath}
Then $f \mapsto \Grad  \theta_{\proc U}(f)$ is a weakly closed
operator, and with Hahn-Banach, its graph is strongly closed, so is a
closed operator.

Last, $\Grad  \theta_{\proc U}(f)$ is a $\theta$-derivation
(Definition 2.3).

\medskip 

Now we look for a $\theta$-derivation $\widecheck D$ such that:
\begin{enumerate}[label=\alph*)]
\item if $f = W(h)$, $\widecheck D[W(h)] = \Grad \theta_{\proc U}(W(h))$,
\item $\widecheck D$ is a $\theta_{\proc U}$-derivation,
\item $\widecheck D$ is continuous from $\D^\infty(\Omega)$ to
  $\D^\infty(\Omega,H)$.
\end{enumerate}
Using the generalisations of the two integrals in (\ref{eqn:9}), we
define: for $f \in \D^\infty(\Omega)$,
\begin{displaymath}
  \widecheck D f = \int_{\mathbb R} f_k^\ell \diffoperator (\theta_{\proc U}D_{E_{\ell,t}^k}f)
  + \left( t \mapsto \int_0^t {}^t\theta_{\proc U}(\Grad f) \proc U^{-1} \diff s \right) .
\end{displaymath}
Then: a) Straightforward computation, using $f_k^\ell E_\ell^k =
\Grad {\proc U^{-1}} \proc U$ proves a).

b) $\widecheck D$ is a $\theta_{\proc U}$-derivation thanks to the
presence of $\theta_{\proc U}$ in $\diffoperator (\theta_{\proc
  U}D_{E_{\ell,t}^k}f)$.

c) The Russo-Valois inequality: $C_0$ being a constant,
\begin{displaymath}
  \norm{ \int_{\mathbb R} f_k^\ell \diffoperator (\theta_{\proc U}D_{E_{\ell,t}^k}f) }_H \leq C_0 \times \norm{f_k^\ell}_{B_{2,2}^{1-\varepsilon/2}(H)} \times \norm{\theta_{\proc U}(D_{E_{\ell,t}^k}f)}_{B_{2,2}^{\varepsilon/2}} .
\end{displaymath}
As $\norm{f_k^\ell}_{B_{2,2}^{1-\varepsilon/2}(H)} \in
\L{\infty-0}(\Omega)$ and $\norm{\theta_{\proc
    U}(D_{E_{\ell,t}^k}f)}_{B_{2,2}^{\varepsilon/2}} \in
\L{\infty-0}(\Omega)$, we see that $\widecheck D$ is continuous from
$\D^\infty$ in $\L{\infty-0}(\Omega,H)$. Now we prove that: $\widecheck
D$ sends $\D^\infty$ in $\D^\infty(\Omega,H)$. We have:
\begin{equation} \label{eqn:11}
  \widecheck D f =  \int_{\mathbb R} f_k^\ell \diffoperator (\theta_{\proc U}D_{E_{\ell,t}^k}f)
  + \left( t \mapsto \int_0^t {}^t\theta_{\proc U}(\Grad f) \proc U^{-1} \diff s \right) .
\end{equation}
A $\Grad$ acts only on $\omega$, $\Grad$ and $\int_{\mathbb
  R}(\text{Russo-Valois})$ commute ; for the same reason: $\Grad$ and
$\diffoperator (\text{Russo-Valois})$ commute.

We know already that $\theta_{\proc U} \colon \D^\infty \to
\L{\infty-0}(\Omega)$. We suppose that $\theta_{\proc U} \colon
\D^\infty \to \D^\infty_r(\Omega)$ and proceed by induction. We apply
now $\Grad$ to the two sides of (\ref{eqn:11}): $\Grad$ on the vector
field
\begin{displaymath}
  t \mapsto \int_0^t {}^t\theta_{\proc U}(\Grad f) \proc U^{-1} \diff s
\end{displaymath}
is legitimate and in $\D^\infty_{r-1}(\Omega, H \otimes H)$. And $\Grad$
applied to the first integral of the right-hand side of (\ref{eqn:11})
gives two Russo-Valois integrals,
\begin{displaymath}
  \int_0^1(\Grad f_k^\ell) \diffoperator (\theta_{\proc U}D_{E_{\ell,t}^k}f)
  \quad \text{and} \quad
  \int_0^1 f_k^\ell \otimes \diffoperator [\Grad \theta_{\proc U}D_{E_{\ell,t}^k}f)].
\end{displaymath}
The first one is in $\D^\infty_r(\Omega,H \otimes H)$ and the second
one is legitimate thanks to the hypothesis $\theta_{\proc
  U}(D_{E_{\ell,t}^k}f) \in \D^\infty_r (\Omega)$. So we see that
$\Grad \widecheck D f \in \L{\infty-0}(\Omega, H \otimes H)$, so
$\widecheck D f \in \mathbb D^\infty_1(\Omega, H)$. We can repeat $r$
times this operation, and we get that: $\Grad^r \widecheck D f \in
\L{\infty-0}(\Omega,\bigotimes^{r+1}H)$ and is continuous from
$\D^\infty$ in $\L{\infty-0}(\Omega,\bigotimes^{r+1}H)$. Now
$\widecheck D$ and $\Grad \theta$ are two $\theta$-derivations which
coincide on polynomials built with Gaussian variables, $\widecheck D$
is continuous from $\D^\infty$ in $\D^\infty_r(\Omega,H)$ and $\Grad
\theta$ is strongly closed as an operator of $\D^\infty(\Omega)$ in
$\D^\infty(\Omega,H)$ ; then $\widecheck D$ and $\Grad \theta$
coincide, so $\Grad \theta(f) \in \D^\infty_r$ which implies
$\theta(f) \in \D^\infty_{r+1}(\Omega)$. The continuity of $\theta(f)$
from $\D^\infty(\Omega)$ in $\D^\infty(\Omega)$ is obvious.

\begin{rem} 
  \label{remark:5.6}
  If $f_k^\ell$ was an S.M., and not a $\D^\infty$-$\alpha$-Holderian
  process with $\alpha > 1/2$, the Russo-Valois inequality is not
  valid anymore. We will see later that such is the case for the
  $\D^\infty$-manifold $\mathbb P_{m_0}(V_n, g)$, which is the set of
  continuous paths in a compact Riemannian manifold, $(V_n,g)$,
  starting from $m_0$.
\end{rem}

\begin{thm} 
  \label{thm:5.8} 
  Let $\proc U$ be a $\D^\infty$-process, with values in $n \times n$
  unitary matrices, adapted and $\alpha$-Holderian, $0 < \alpha < 1$,
  and $\theta$ being the morphism from $\L{\infty-0}$ in
  $\L{\infty-0}$ generated by
  \begin{displaymath}
    \theta[W(h)] = \int_0^1 {}^t\dot h \proc U^{-1} \itodiff B.
  \end{displaymath}
  Suppose that $\theta^{-1}$ exists and is a $\D^\infty$-morphism
  of $\D^\infty(\Omega)$ in itself, then $\theta$ is a
  $\D^\infty$-isomorphism.
\end{thm}

\begin{rem} \label{remark:5.7} From Theorem \ref{thm:5.7}, we know
  that if $\alpha > 1/2$, Theorem \ref{thm:5.8} is automatically
  verified.
\end{rem}

\begin{proof}
  $\theta^{-1}$ sends $\mathcal F_t$ in $\mathcal F_t$:
  $\theta$ sends $\mathcal F_t^\perp$ in $\mathcal F_t^\perp$
  because
  \begin{displaymath}
    \theta\left[ \int_t^1 X \itodiff B \right] = \int_t^1 \theta(X) \proc U^{-1} \itodiff B \in \mathcal F_t^\perp,
  \end{displaymath}
  and $\theta^{-1} = \theta^\ast$ (adjoint of $\theta$). So
  $\forall f \in \D^\infty \cap \mathcal F_t, \forall g \in
  \L{\infty-0} \cap \mathcal F_t^\perp,$
  \begin{displaymath}
    \langle \theta^{-1}(f), g \rangle_{\L 2 (\Omega)}
    = \langle \theta^\ast(f), g \rangle_{\L 2 (\Omega)}
    = \langle f, \theta(g) \rangle_{\L 2 (\Omega)}
    = 0 .
  \end{displaymath}
  Now, $\Grad {\proc U}^{-1} \proc U$ is
  $\D^\infty$-$\alpha$-Holderian, as a product of two
  $\D^\infty$-$\alpha$-matrix processes and
  $\theta^{-1}(\Grad{\proc U}^{-1} \proc U)$ is also
  $\D^\infty$-$\alpha$-Holderian because $\theta^{-1}$ acts only on
  $\omega$. Then from $\theta[W(h)] = \int_0^1{}^t\dot h \proc
  U^{-1}\itodiff B$, and as $\theta^{-1}(\proc U^{-1} \itodiff B) = \diff
  B$, we have
  \begin{displaymath}
    \theta^{-1} \Grad(\theta[W(h)]) = \int_0^1 {}^t\dot h \theta^{-1}[\Grad\proc U^{-1} \proc U] \itodiff B + \left( t \mapsto \int_0^t {}^t\dot h\theta^{-1}(\proc U^{-1}) \diff s \right).
  \end{displaymath}

  We denote by $Z(f)$:
  \begin{displaymath}
    Z(f) = \div A \Grad f + t \mapsto \int_0^t ({}^t\Grad f)\theta^{-1}(\proc U^{-1}) \diff s
  \end{displaymath}
  where $A = \theta^{-1}(\Grad \proc U^{-1} \proc U)$ is a vector
  matrix, $(A)_i^j \in H$, and \\ $(A(\Grad f))_i = \sum_{j=1}^n
  \langle \Grad f, e_j \rangle(A)_i^j \in H$.  As $\theta^{-1}(\Grad
  \proc U^{-1} \proc U)$ is $\D^\infty$-$\alpha$-Holderian,
  $\theta^{-1}(\Grad \proc U^{-1} \proc U)$ is a $\D^\infty(\Omega,H)$
  multiplicator (Theorem 4.3), so $Z(f) \in
  \D^\infty(\Omega,H)$. Moreover, $Z(f)$ and $\theta^{-1}[\Grad
  \theta(f)]$ coincide when $f$ is a polynomial in Gaussian variables,
  because both are $\theta$-derivations, and direct computation show
  that
  \begin{displaymath}
    Z[W(h)] = \theta^{-1}[\Grad \theta(W(h))].
  \end{displaymath}
  So we extend $\theta^{-1}[\Grad \theta]$ as an operator on
  $f\in \D^\infty$, with $Z(f)$.

  Now we prove that $\theta$ sends $\D^\infty$ in $\D^\infty$ by
  induction ; we know already that $\theta \colon \D^\infty \to
  \L{\infty-0}$. Then assume that $\theta \colon \D^\infty \to
  \D^\infty_r$, and let $f\in \D^\infty$: as
  $\theta^{-1}[\Grad\theta(f)] = Z(f) \in \D^\infty(\Omega, H)$,
  \begin{displaymath}
    \theta \theta^{-1}[\Grad\theta(f)] = \theta(Z(f)) \in \D^\infty_r(\Omega, H)
  \end{displaymath}
  which implies $\Grad \theta(f) \in \D^\infty_r(\Omega,H)$ so
  $\theta(f) \in \D^\infty_{r+1}(\Omega)$.
\end{proof}

Let $\proc U$ be an adapted process, $\D^\infty$-$\alpha$-Holderian,
with values in the unitary $n \times n$ matrices on $\mathbb R^n$. The
map $\theta$ defined on $W(h), h\in H$, by
\begin{displaymath}
  \theta[W(h)] = \int_0^1{}^t\dot h \proc U^{-1}\itodiff B,
\end{displaymath}
can be extended in a morphism from $\L{\infty-0}(\Omega)$ in
$\L{\infty-0}(\Omega)$, because it preserve laws.

Let $A$ be an adapted process, valued in the space of $n \times n$-A.M.,
$A$ being moreover a multiplicator. Following Malliavin [\ \ ], we call
such a process an elementary tangent process. We define the operator
$T_{\proc U}(A)$ by:
\begin{displaymath}
  T_{\proc U}(A) = \proc U(D_A \proc U^{-1}) + A
\end{displaymath}
where $D_A = \div A \Grad$. Now we will prove:

\begin{thm} \label{thm:5.9} If $T_{\proc U}$ admits an inverse
  operator $(T_{\proc U})^{-1}$ from the space of elementary tangent
  processes in itself, and if $\theta \colon \L{\infty-0} \to
  \L{\infty-0}$ admits an inverse from $\L{\infty-0}$ in itself, then
  $\theta$ is a $\D^\infty$-diffeomorphism.
\end{thm}

Before proving this theorem, we need several lemmas.

\begin{lem} \label{lemma:5.7} The operator $T_{\proc U}$ takes its
  values in the $n \times n$-A.M., $T_{\proc U}(A)$ is an adapted
  process, and a multiplicator: $T_{\proc U}(A)$ is an elementary
  tangent process.
\end{lem}

\begin{proof}
  As $D_A$ is an adapted derivation, $T_{\proc U}$ is adapted. Then
  $\proc U^{-1}$ being $\D^\infty$-$\alpha$-Holderian, $D_A \proc
  U^{-1}$ is also $\D^\infty$-$\alpha$-Holderian ($D_A$ acts only on
  $\omega$) ; and a process which is $\D^\infty$-$\alpha$-Holderian is
  a multiplicator (Theorem 4.3).
\end{proof}

\begin{rem}
  As $\theta \colon \L{\infty-0} \to \L{\infty-0}$ preserves laws,
  $\theta \colon \L +_1(\Omega) \to \L +_1(\Omega)$, so
  $\theta^\ast \colon \L{\infty-0}(\Omega) \to
  \L{\infty-0}(\Omega)$. Then $\theta^\ast$ is a morphism
  ($\theta^\ast = \theta^{-1}$).
\end{rem}

\begin{lem} \label{lemma:5.8} Suppose $(T_{\proc U})^{-1}$ exists as
  an operator from the space of adapted multiplicators, $n \times
  n$-A.M., in itself. Then if $Y$ is $\D^\infty$-$\alpha$-Holderian,
  $(T_{\proc U})^{-1}Y$ is also $\D^\infty$-$\alpha$-Holderian.
\end{lem}

\begin{proof}
  Denote $A = (T_{\proc U})^{-1}Y$ ; then $T_{\proc U}A$ is
  $\D^\infty$-$\alpha$-Holderian by hypothesis ; so $\proc U D_A \proc
  U^{-1} + A$ is $\D^\infty$-$\alpha$-Holderian and so is $\proc
  U^{-1} D_A \proc U$, which implies $A = Y - \proc U D_A \proc
  U^{-1}$ is $\D^\infty$-$\alpha$-Holderian.
\end{proof}

\begin{lem} 
  \label{lemma:5.9} 
  In the same setting than in Lemma 5.8, $H'$ being an Hilbert space,
  the extension of $(T_{\proc U})^{-1}$ (Corollary 2.2) will send the
  space of the \\ $\D^\infty(\Omega,H')$-$\alpha$-Holderian elementary
  tangent process (with items in $H'$) in \\
  $\D^\infty$-$\alpha'$-Holderian processes, $\alpha' < \alpha$.
\end{lem}
\begin{unnumbered remark}
  We cannot apply directly Corollary 2.2, because the space
  of $\D^\infty$-$\alpha$-Holderian processes is not $\D^\infty$-closed.
\end{unnumbered remark}
\begin{proof}[Proof of Lemma \ref{lemma:5.9}]
  a) Let X be a completely $\D^\infty$-process with $\mathbb R$-valued
  matrix items (see Definition 2.2 for a completely
  $\D^\infty$-process). Then $E[X | \mathcal F_t]$ is an adapted
  process ; $E[X | \mathcal F_t] \ast \beta_{1-s}$ is a
  $\D^\infty$-Holderian process which implies $(T_{\proc U})^{-1}[E[X
  | \mathcal F_t] \ast \beta_{1-s}]$ is again $\D^\infty$-Holderian
  (Lemma 5.8) ; \\$\{ (T_{\proc U})^{-1}[E[X | \mathcal F_t] \ast
  \beta_{1-s}] \ast \beta_{s'} \}', s' > s$, is a completely
  $\D^\infty$-process, denoted $\Hat X$.

  Then $X \mapsto \Hat X$ is a transformation denoted $\circled H$
  which sends a complete $\D^\infty$-elementary tangent process in a
  completely $\D^\infty$-elementary tangent process. With Lemma
  2.1.ii, we have an extension map denoted $\widetilde{\circled H}$
  which sends the space of completely $\D^\infty$-elementary tangent
  process with matrix items in $H'$, in itself.

  b) Now let $Y$ an $\D^\infty$-$\alpha$-Holderian elementary tangent
  process with matrix items in $H'$. With Theorem 2.8, we know that $X
  = (Y \ast \beta_{s_0})', s_0 < \alpha$, is completely
  $\D^\infty$. So we denote by:
  \begin{displaymath}
    \widetilde{(T_{\proc U})^{-1}}(Y) = \widetilde{\circled H}\left((Y \ast \beta_{s_0})'\right)\ast\beta_{1-s_0}.
  \end{displaymath}
  Remind that according to Proposition 2.2.iv, the convolution by
  $\beta_{s_0}$ or $\beta_{1-s_0}$ leaves the adaptation property
  invariant. Then $\widetilde{(T_{\proc U})^{-1}}Y$ is an
  $\D^\infty$-$\alpha$-H\"olderian elementary process.

  c) Each matrix item of $Y$ can be written $Y_i^j = a_i^j h$ where
  $h$ is a constant vector of $H'$. Then 
  \begin{align*}
    \widetilde{(T_{\proc U})^{-1}} Y_i^j &= \left(\circled H (a_i^j \ast \beta_{s_0})'\right) \ast \beta_{1-s_0} h \\
    &= \{ [ (T_{\proc U})^{-1} (\mathbb E[(a_i^j \ast \beta_{s_0})' | \mathcal F_t] \ast \beta_{1-s_0} ] \ast \beta_{s_0} \}' \ast \beta_{1-s_0} h \\
    &= (T_{\proc U})^{-1} (\mathbb E[(a_i^j \ast \beta_{s_0})' | \mathcal F_t] \ast \beta_{1-s_0}) h \\
    &= (T_{\proc U})^{-1} (\mathbb E[(a_i^j \ast \beta_{s_0})' \ast \beta_{1-s_0} | \mathcal F_t] ) h \\
    &= (T_{\proc U})^{-1} (\mathbb E[a_i^j | \mathcal F_t])h \\
    &= ((T_{\proc U})^{-1} a_i^j) h \\
    &= (T_{\proc U})^{-1} Y_i^j .
  \end{align*}
\end{proof}

\begin{defn} 
  \label{defn:5.4} 
  Let $X \in \D^\infty(\Omega,H)$ and $Z$, an $n \times
  n$-A.M. process with its items belonging to $H$ (Cameron-Martin
  space). We define $D_Zf$, for $f \in \D^\infty(\Omega)$ by
  \begin{displaymath}
    D_Zf = \div_R(Z \otimes \Grad f),
  \end{displaymath}
  then $D_Zf \in \D^\infty(\Omega, H)$.
\end{defn}

(Remind $\div_R$ already defined in Theorem 5.1.) Then $(e_i)_{i \in
  \mathbb N_\ast}$ being an Hibertian basis of $H$, straightforward
computation shows:
\begin{displaymath}
  D_Zf = \sum_{i=1}^\infty (\div (\langle Z, e_i \rangle_H) \Grad f) e_i.
\end{displaymath}

\begin{lem} 
  \label{lemma:5.10} 
  $D_Z^\ast$ being the adjoint of $D_Z$, for all $V \in
  \D^\infty(\Omega,H)$, we have, with $f\in \D^\infty(\Omega)$:
  \begin{displaymath}
    D_Z^\ast(fV) = f D_Z^\ast V - \langle D_Zf,V \rangle_H.
  \end{displaymath}
\end{lem}

\begin{proof}
  $\forall g \in \D^\infty(\Omega),$
  \begin{align*}
    (D_Z^\ast(fV),g)
    &= \langle fV, D_Zg \rangle_{\L 2 (\Omega, H)} \\
    &= \sum_{i=1}^\infty \langle fV, (D_{Z_i} g) e_i \rangle_{\L 2 (\Omega, H)} \\
    &= \sum_{i=1}^\infty ((f D_{Z_i} g), \langle V, e_i \rangle_H)_{\L 2 (\Omega)} \\
    &= \sum_{i=1}^\infty \left[ (D_{Z_i} (fg), \langle V, e_i
      \rangle_H)_{\L 2 (\Omega)} - (g D_{Z_i} f, \langle V, e_i
      \rangle_H)_{\L 2 (\Omega)} \right] \\
    &= \sum_{i=1}^\infty \left[ (V, D_{Z_i} (fg) e_i)_{\L 2 (\Omega)}
      - (g , \langle V, D_{Z_i} f e_i \rangle_H)_{\L 2 (\Omega)}
    \right] \\
    &= (g,f D_Z^\ast V) - (g, \langle V, D_Zf \rangle_H).
  \end{align*}
\end{proof}

Now we go back to the proof of Theorem \ref{thm:5.9}.
\begin{proof}[Proof of Theorem \ref{thm:5.9}] 
  From $\theta [W(h)] = \int_0^1{}^t\dot h\proc U^{-1}\itodiff B$, we
  get: 
  \begin{equation}
    \label{eqn:12}
    \Grad \theta [W(h)] = \int_0^1{}^t\dot h(\Grad U^{-1})\proc U \proc U^{-1} \itodiff B + t \mapsto \int_0^1{}^t\dot h\proc U^{-1} \diff s .
  \end{equation}
  We write $Y = \Grad \proc U^{-1} \proc U$.

  Let $Z$ such that $T_{\proc U}(Z) = \proc U Y \proc U^{-1}$ so $Z =
  T_{\proc U}^{-1}(\proc U Y \proc U^{-1})$. With Lemma 5.9, $Z$ is a
  $\D^\infty$-$\alpha$-Holderian $n \times n$-A.M. matrix process with
  items in $H$ ; and we have:
  \begin{equation} 
    \label{eqn:13}
    Y \proc U^{-1} =  D_Z \proc U^{-1} + \proc U^{-1} Z
  \end{equation}
  and
  \begin{equation}
    \label{eqn:14}
    D_Z(\diff B) = \diffoperator (D_Z B) = Z \diff B .
  \end{equation}

  Using (\ref{eqn:13}) and (\ref{eqn:14}) in (\ref{eqn:12}), we get,
  with $\theta(h) = h$:
  \begin{equation}
    \label{eqn:15}
    \Grad \theta [W(h)] = D_Z[\theta(W(h))] + t \mapsto \int_0^1 {}^t(\theta(\Grad(W(h))))\proc U^{-1} \diff s, 
  \end{equation}
  but $\Grad \theta[W(h)] = D_Z[\theta(W(h))] + \proc
  U[\theta(\Grad W(h))]$. We would like to write, for $f\in
  \D^\infty(\Omega)$:
  \begin{equation}
    \label{eqn:16}
    \Grad \theta[f] = D_Z[\theta(f)] + \proc U[\theta(\Grad f)]
  \end{equation}
  but
  \begin{equation}
    \label{eqn:17}
    \theta[\Grad f] = \proc U^{-1}\Grad\theta(f) - \proc U^{-1}D_Z[\theta(f)].
  \end{equation}
  Due to the right hand side of (\ref{eqn:17}), we interpret
  (\ref{eqn:17}) as an equation between distributions.

  $(e_i)_{i\in\mathbb N_\ast}$ being Hilbertian basis of $H$, and $g
  \in \D^\infty(\Omega)$, we apply each member of (\ref{eqn:17}) to $g
  e_i$:
  \begin{displaymath}
    \begin{split}
      -( f e_i, \Grad \theta^\ast(g)) &- \langle f \div e_i,
      \theta^\ast(g) \rangle_{\L 2(\Omega)} \\
      &= (\proc U(g e_i), \Grad \theta(f)) - \langle \proc U^{-1} D_Z
      \theta(f), g e_i \rangle_{\L 2 (\Omega,H)}
    \end{split}
  \end{displaymath}
  which implies:
  \begin{displaymath}
    \begin{split}
      -( f e_i, \Grad \theta^\ast(g)) &- \langle f \div e_i,
      \theta^\ast(g) \rangle_{\L 2(\Omega)} \\
      &= -\langle f, \theta^\ast \div(\proc U(g e_i)) \rangle_{\L 2 (\Omega)} - (f,
      \theta^\ast D_Z^\ast \proc U (g e_i))
    \end{split}
  \end{displaymath}
  Using Lemma \ref{lemma:5.10}, we get:
  \begin{equation} 
    \label{eqn:18}
    \begin{split}
      -( f e_i, \Grad \theta^\ast(g)) &- \langle f \div e_i,
      \theta^\ast(g) \rangle_{\L 2(\Omega)} \\
      &= -\langle f, \theta^\ast \div(\proc U(g e_i)) \rangle_{\L 2
        (\Omega)} - (f, \theta^\ast [g D_Z^\ast (\proc U e_i)]) \\
      &\quad + \langle f, \theta^\ast(\langle D_Z g, \proc U e_i
      \rangle_H) \rangle_{\L 2 (\Omega)}
    \end{split}
  \end{equation}
  But:
  \begin{displaymath}
    \theta^\ast(\div(\proc U (g e_i))) = \theta^\ast(g) \theta^\ast(\div \proc U e_i) + \theta^\ast(\langle \Grad g, \proc U e_i \rangle_H)
  \end{displaymath}
  And as $\proc U$ is adapted and unitary:
  \begin{align*}
    \theta^\ast[\div \proc U e_i] &= \theta^\ast \left[ \int_0^1{}^t(\proc U e_i) \itodiff B \right] \\ 
    &= \theta^\ast \left[ \int_0^1({}^te_i)\proc U^{-1}\itodiff B \right] \\
    &= \theta^\ast \theta[W(e_i)] \\
    &= W(e_i).
  \end{align*}
  Using this in (\ref{eqn:18}), we get:
  \begin{equation} 
    \label{eqn:19}
    \begin{split}
      - ( f e_i, \Grad \theta^\ast(g) ) &=  - \langle \theta^\ast ( \langle \Grad g, \proc U e_i \rangle_H ), f \rangle_{\L 2 (\Omega)} - (f, \theta^\ast(g D_Z^\ast(\proc U e_i))) \\
      &\quad + \langle f, \theta^\ast ( \langle D_Zg, \proc U e_i
      \rangle_H ) \rangle_{\L 2 (\Omega)}.
    \end{split}
  \end{equation}
  From the formula $D_Z f = \sum_{i=1}^\infty (\div Z_i \Grad f) e_i =
  \sum_{i=1}^\infty(D_{Z_i} f) e_i$, in Definition \ref{defn:5.4}, we
  deduce by duality that if $X \in \D^\infty(\Omega, H)$, then
  $D_Z^\ast X \in \D^\infty(\Omega)$. Then $D_Z^\ast(\proc U e_i) \in
  \D^\infty(\Omega)$, so $\theta^\ast(g D_Z^\ast \proc U e_i)$ is
  legitimate, and is a function. Then (\ref{eqn:19}) becomes:
  \begin{displaymath}
      - ( e_i, \Grad \theta^\ast(g) ) = - \theta^\ast ( \langle
      \Grad g, \proc U e_i \rangle_H ) - \theta^\ast (g D_Z^\ast
      \proc U e_i) + \theta^\ast ( \langle D_Zg, \proc U e_i
      \rangle_H ) .
  \end{displaymath}
  But $\theta^\ast$ acts on contants vectors fields as the indentity, so :
  \begin{displaymath}
    \theta^\ast(\langle \Grad g, \proc U e_i \rangle_H) = \langle \theta^\ast(\proc U^{-1} \Grad g) , e_i \rangle_H,
  \end{displaymath}
  and
  \begin{displaymath}
    \theta^\ast(\langle D_Z g, \proc U e_i \rangle_H) = \langle \theta^\ast(\proc U^{-1} D_Z g) , e_i \rangle_H.
  \end{displaymath}
  So, we have :
  \begin{displaymath}
    -(e_i, \Grad \theta^\ast (g)) = - \langle e_i, \theta^\ast(\proc U^{-1} \Grad g) \rangle_H - \theta^\ast (g D_Z^\ast \proc U e_i) + \langle e_i, \theta^\ast(\proc U^{-1} D_Z g) \rangle_H.
  \end{displaymath}

  In this last equation, we choose $g = 1$. We get:
  $\theta^\ast(D_Z^\ast \proc U e_i) = 0$. We deduce
  \begin{displaymath}
    (e_i, \Grad \theta^\ast(g)) = \langle e_i, \theta^\ast(\proc U^{-1} \Grad g) \rangle_H - \langle e_i, \theta^\ast(\proc U^{-1} D_Z g) \rangle_H.
  \end{displaymath}
  In this last equation, $\theta^\ast(\proc U^{-1} \Grad g)$ and
  $\theta^\ast(\proc U^{-1} D_Z g)$ are $\L{\infty-0}$-functions
  and $\Grad(\theta^\ast g)$ a distribution ; so as distributions,
  we have:
  \begin{displaymath}
    \Grad \theta^\ast(g) = \theta^\ast (\proc U^{-1} \Grad g) - \theta^\ast(\proc U^{-1} D_Z g).
  \end{displaymath}
  Suppose that $\theta^\ast \colon \D^\infty(\Omega) \to
  \D_r^\infty(\Omega)$ ; then $\Grad \theta^\ast(g) \in
  \D_r^\infty$ which implies $\theta^\ast(g) \in
  \D_{r+1}^\infty$. As $\theta$ is the adjoint of $\theta^\ast$,
  $\theta \colon \D^\infty \to \D^\infty(\Omega)$.
\end{proof}

Now we prove the converse of Theorem \ref{thm:5.9}.

\begin{thm} 
  \label{thm:5.10} 
  If $\theta$, the $\L{\infty-0}$-morphism generated by $\theta[W(h)]
  = \int_0^1{}^t\dot h \proc U^{-1} \itodiff B$, is a
  $\D^\infty$-diffeomorphism of $\D^\infty(\Omega)$ in itself, the
  linear pseudo-tangent map admits an inverse, in the space of the
  elementary tangent processes.
\end{thm}

\begin{proof}
  $\theta[W(h)] = \int_0^1{}^t\dot h \proc U^{-1} \itodiff B$, so:
  \begin{align*}
    D_A\theta[W(h)] &= \int_0^1{}^t\dot h D_A \proc U^{-1} \itodiff B + \int_0^1{}^t\dot h \proc U^{-1} A \itodiff B \\
    &= \int_0^1{}^t\dot h(\proc U^{-1} T_{\proc U} A \proc U)\proc U^{-1} \itodiff B \\
    &= \theta \left[ \int_0^1{}^t\dot h \theta^{-1}(\proc
      U^{-1}T_{\proc U}A \proc U) \itodiff B \right].
  \end{align*}
  So 
  \begin{displaymath}
    \theta^{-1} D_A \theta(W(h)) = D_{\theta^{-1}(\proc U^{-1}T_{\proc U}A\proc U)}(W(h)).
  \end{displaymath}
  So the map $A \mapsto \theta^{-1}(\proc U^{-1}T_{\proc U}A\proc
  U)$ is inversible. Then the map $A \mapsto \proc U^{-1}T_{\proc
    U}A\proc U$ is inversible, so $A \mapsto T_{\proc U}A$ is
  inversible.
\end{proof}

Before the next theorem, we first remark that: if $f$ is a polynomial
in Gaussian variables, $f[W(h_1), \dots , W(h_n)]$, then
$\theta_{\proc U}(f) \in \D^\infty(\Omega)$ ; then if $z = it+s$,
and $s>0$, $(1-L)^{-(it+s)}$ is legitimate as
\begin{displaymath}
  \frac 1 {\Gamma(s)} \int_0^\infty \alpha^{s+it-1} \mathrm e^{-\alpha} P_\alpha(\theta_{\proc U} f)\diff \alpha \quad \text{(Mehler's formula)}.
\end{displaymath}
And if $r > s$, we write: $(1-L)^{r-s-it} = (1-L)^{-(s+it)} \circ
(1-L)^r(\theta_{\proc U} f)$.

\begin{lem} 
  \label{lemma:5.11} 
  If $\rho$ is an $\D^\infty$-$\alpha$-Holderian process, there exists
  $s$, $0 < s < 1$ and $\alpha+s > 1$, such that $\rho \ast
  \beta_{1-s} \in \mathcal C^1$ and $g = (\rho \ast \beta_{1-s})'$
  will be in $\L{\infty-0}([0,1] \times \Omega)$
\end{lem}

\begin{proof}
  Proposition 2.2.i.
\end{proof}

The next theorem of ``local inversibility'' is:
\begin{thm} \label{thm:5.11} If $T_{\proc U}$ is inversible from
  the space of adapted $\D^\infty$-multiplicators in itself, and if
  ${\proc U}$ is $\D^\infty$-$\alpha$-Holderian with $\alpha > 1/4$,
  then the morphism generated by $\theta$, $\theta[W(h)] =
  \int_0^1{}^t\dot h \proc U^{-1} \itodiff B$, is a $\D^\infty$-morphism
  of $\D^\infty$ in itself.
\end{thm}

As the proof of this theorem is more difficult and involves a
fractionnal induction, we give the followed plan of the proof:
\begin{enumerate}[label=\alph*)]
\item We first recall some notations in Lemma \ref{lemma:5.4}, and
  establish that: if $\rho$ is a $\D^\infty$-$\alpha$-Holderian
  function with $\alpha > 1/2$ then there exist a quantity $S(\rho
  E_{\ell,t}^k)$ such that
  \begin{displaymath}
    \int_0^1 \rho \diffoperator [\theta(D_{E_{\ell,t}^k} f)] = D_{S(\rho E_{\ell,t}^k)} \theta (f),
  \end{displaymath}
  the integral being a Russo-Valois integral.
\item We then suppose that: $\theta \colon \D^\infty(\Omega) \to
  \D^\infty_s(\Omega)$. We will prove, using the above formula in a),
  that if $f \in \D^\infty(\Omega)$, $D_{S(\rho E_\ell^k),t} \theta
  (f) \in \D^\infty_s(\Omega)$.

  For this, first we prove it for $s\in \mathbb N_\ast$ ; and then if
  $s \notin \mathbb N_\ast$, $s>0$, we will use the
  Phragmen-Lindel\"of method with an interpolation in the domain
  delimited by $E[s]$ and $E[s+1]$ to get this result.
\item There we will prove that $\rho$ being
  $\D^\infty$-$\alpha$-Holderian with $\alpha > 0$, then $D_{S(\rho
    E_\ell^k)}\theta (f) \in \D^\infty_{s-2}$.
\item Then another interpolation, using the Phragmen-Lindel\"of
  method, interpolation on $t$ this time, will proves Theorem
  \ref{thm:5.10}.
\end{enumerate}

\begin{proof}
  a) $E_\ell^k$ is the $n \times n$-elementary antisymmetrical matrix,
  \\ $D_{E_{\ell,t}^k}f = \div E_\ell^k \mathds 1_{[0,t]}(\cdot) \Grad f$,
  with $f\in \D^\infty(\Omega)$ and if $\rho(t,\omega)$ is an
  $\D^\infty$-$\alpha$-Holderian process, $H$-valued, the integral
  $\int_0^1 \rho \diffoperator [\theta(D_{E_{\ell,t}^k} f)]$ is to be
  understood as a Russo-Valois integral, with $E_{\ell,t}^k = \mathds
  1_{[0,t]}(\cdot) E_\ell^k$. $\rho$ is
  $\D^\infty$-$\alpha$-Holderian, $H$ valued ; so $\rho$ is
  $\D^\infty$-bounded. Using the decomposition: $\rho = g \ast
  \beta_s$ with $g \in \L p(\Omega \times [0,1])$ and $\beta_s \in \L
  1$, uneasy computation shows that there exists $\varepsilon, 0 <
  \varepsilon < \alpha - 1/2$ such that
  \begin{displaymath}
    \norm{\rho}_{B_{2,+\infty}^{1/2+\varepsilon}}(H) \in \L{\infty-0}(\Omega) .
  \end{displaymath}
  And
  \begin{displaymath}
    \forall s \in \mathbb N_\ast, \norm{\Grad^s \rho}_{B_{2,+\infty}^{1/2+\varepsilon}(\bigotimes^{s+1} H)} \in \L{\infty-0}(\Omega)
  \end{displaymath}
  because the operator $\Grad$ applies only on $\omega$, while the
  Besov affiliation of $\Grad^s \rho$ is due only to the $t$ variable.

  Now as we consider $f \in \D^\infty(\Omega)$, then $\theta
  (D_{E_{\ell,t}^k}f) \in \L{\infty-0}(\Omega)$ ; suppose $\theta
  \colon \D^\infty \to \D^\infty_s, s \in \mathbb N_\ast$. But the
  lemma \ref{lemma:5.5}, in which hypothesis are only i) and ii)
  brings as a result I), then $\Grad^s \theta(D_{E_{\ell, t}^k} f)$
  is an $1/2$-Holderian process ($s \in \mathbb N_\ast$) and so
  \begin{displaymath}
    \Grad^s(\theta D_{E_{\ell,t}^k} f) \in {B_{2,1}^{1/2-\varepsilon}} (\bigotimes^s H)
  \end{displaymath}
  and
  \begin{displaymath}
    \norm{ \Grad^s(\theta D_{E_{\ell,t}^k} f) }_{B_{2,1}^{1/2-\varepsilon} (\bigotimes^s H)} \in \L{\infty-0}(\Omega).
  \end{displaymath}
  So $\int_0^1 \rho \diffoperator (\theta D_{E_{\ell,t}^k} f)$ exists as a
  Russo-Valois integral ; same for $\Grad^s( \int_0^1 \rho \diff
  (\theta D_{E_{\ell,t}^k} f))$, with $s \in \mathbb N_\ast$, if $\rho
  \in \D^\infty(\Omega, H)$ and $\theta \colon \D^\infty \to
  \D_s^\infty(\Omega)$.

  Now $T_{\proc U}$ being inversible, if $C = \rho E_{\ell,t}^k = \rho
  E_\ell^k \mathds 1_{[0,t[}(\cdot)$, $T_{\proc U}^{-1}(\proc U C
  \proc U^{-1})$ is denoted $S(C)$. We have:
  \begin{equation}
    \label{eqn:20}
    \begin{split}
      D_{S(C)}[\theta (W(h))] &= D_{S(C)} \left[ \int_0^1{}^t\dot h \proc U^{-1} \itodiff B \right] \\
      &= \int_0^1{}^t\dot h D_{S(C)}\proc U^{-1} \itodiff B +
      \int_0^1{}^t\dot h \proc U^{-1} S(C) \itodiff B.
    \end{split}
  \end{equation}
  And from $T_{\proc U}(S(C)) = \proc U D_{S(C)} \proc U^{-1} + S(C) =
  \proc U C \proc U^{-1} = U \rho E_{\ell,t}^k \proc U^{-1}$, we get
  $D_{S(C)} \proc U^{-1} = C \proc U^{-1} - \proc U^{-1} S(C)$ ; so
  (\ref{eqn:20}) becomes:
  \begin{equation}
    \label{eqn:21}
    D_{S(C)}\theta [W(h)] = \int_0^1{}^t\dot h \rho E_{\ell,t}^k \proc U^{-1} \itodiff B
    = \int_0^t{}^t\dot h \rho E_\ell^k \proc U^{-1} \itodiff B.
  \end{equation}

  Now:
  \begin{displaymath}
    \theta(D_{E_{\ell,t}^k}(W(h)) = \theta \left[ \int_0^1{}^t\dot h E_{\ell,t}^k \itodiff B \right] = \theta \left[ \int_0^t {}^t\dot h E_\ell^k \itodiff B \right] = \int_0^t{}^t \dot h E_\ell^k \proc U^{-1} \itodiff B.
  \end{displaymath}
  And
  \begin{equation} \label{eqn:22} \int_{\mathbb R} \rho
    \diffoperator (\theta D_{E_{\ell,t}^k}(W(h))) = \int_0^t \rho {}^t\dot h
    E_\ell^k \proc U^{-1} \itodiff B .
  \end{equation}
  From (\ref{eqn:21}) and (\ref{eqn:22}), we see that
  \begin{equation}
    \label{eqn:23}
    \int_{\mathbb R} \rho \diff \left[ \theta D_{E_{\ell,t}^k} (W(h)) \right] = D_{S(\rho E_{\ell,t}^k)} \theta [W(h)].
  \end{equation}
  Each member of the above equation is a $\theta$-derivation, so
  (\ref{eqn:23}) becomes valid when a polynomial in Gaussian variables
  is substituted to $W(h)$. As the Russo-Valois integral is continuous
  if $f_n \mapsto f$ in $\L{\infty-0}(\Omega)$, and $D_{S(\rho
    E_{\ell,t}^k)} \theta(f_n)$ converges towards $D_{S(\rho
    E_{\ell,t}^k)} \theta(f)$ (as distributions), we get that
  (\ref{eqn:23}) is still valid for $f \in \D^\infty(\Omega)$, and
  that $D_{S(\rho E_{\ell,t}^k)} \theta(f)$ is a function.

  \medskip 
  
  b) We have supposed that $\theta \colon \D^\infty(\Omega) \to
  \D^\infty_s(\Omega)$, $s$ being an integer. To the left-hand
  side of (\ref{eqn:23}), we can apply $s$ times the operator $\Grad$
  and each
  \begin{displaymath}
    \Grad^a \left( \int_0^1 \rho \diffoperator [\theta D_{E_{\ell,t}^k} f] \right) \quad \text{with } a \in \{1, \dots, s\}
  \end{displaymath}
  is a Russo-Valois integral, in $\L{\infty-0}(\Omega)$. So if
  $\theta \colon \D^\infty(\Omega) \to \D_s^\infty(\Omega)$,
  \begin{displaymath}
    D_{S(\rho E_{\ell,t}^k)}\theta(f) \in \D^\infty_s(\Omega),
  \end{displaymath}
  $s$ being an integer. If $\theta$ sends $\D^\infty(\Omega)$ in
  $\D^\infty_s(\Omega)$, but with $s$ non-integer, we will prove that
  $D_{S(\rho E_{\ell,t}^k)} \theta(f) \in \D_s^\infty(\Omega)$
  using the Phragmen-Lindel\"of method on the strip of $\mathbb R^2$
  delimited by $0 \leq s - E[s] \leq 1$, denoted $\Delta$. 
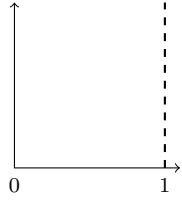
  \begin{center}
    \begin{tikzpicture}
      \draw[<->] (0,2.2) -- (0,0) node[below] {\scriptsize $0$} -- (2.2,0) ;
      \draw[thick, dashed] (2,0) node[below] {\scriptsize $1$} -- (2,2.2) ;
    \end{tikzpicture}
  \end{center}

  Let $f$ be a polynomial on Gaussian variables and consider the
  function of $z, 0 \leq \Re (z) \leq 1$, $\varphi$ being in
  $\D^\infty$, and $r = E[s]$ ; denote:
  \begin{displaymath}
    F(z) = \left\langle \mathrm e^{z^2} (1-L)^{\frac{r+z}{2}} \left[ \int_{\mathbb R} \rho (1-L)^{az + b} \diffoperator (\theta D_{E_{\ell,t}^k} f) \right], \varphi \right\rangle_{\L 2(\Omega)}
  \end{displaymath}
  We want that if $\Re (z) = 0$, $(1-L)^{az+b}
  \theta(D_{E_{\ell,t}^k} f) \in \D^\infty_r$ and if $\Re (z) = 1$,
  $(1-L)^{az+b} \theta(D_{E_{\ell,t}^k} f) \in
  \D^\infty_{r+1}$. These requirements imply: $a = -1/2$, $b =
  (s-r)/2$. $F(z)$ is holomorphic on $\Delta$ and is continuous on
  $\bar \Delta$. $\varphi$ being in $\D^\infty(\Omega)$,
  $(1-L)^{\frac{r+z}{2}} \varphi$ exists and is in
  $\D^\infty(\Omega)$. So $\abs{F(z)}$ is bounded on $\Delta$. Now
  $\forall p > 1, \forall q$ such that $1/p + 1/q = 1$,
  \begin{displaymath}
    \abs{F(i\lambda)} \leq \norm{\varphi}_{\L q (\Omega)} \norm{\mathrm e^{-\lambda^2}(1-L)^{\frac{r+i\lambda}{2}} \int_{\mathbb R} \rho \diffoperator ((1-L)^{ai\lambda+b} \theta D_{E_{\ell,t}^k} f) }_{\L p (\Omega)}.
  \end{displaymath}
  With the Sobolev logarithmic inequality, as
  $\theta(D_{E_{\ell,t}^k} f) \in \D^\infty_s$,
  \begin{displaymath}
    (1-L)^{\frac{-i\lambda}{2} + \frac{s-r}{2}} \theta(D_{E_{\ell,t}^k} f) \in \D^\infty_r,
  \end{displaymath}
  so
  \begin{displaymath}
    \int_{\mathbb R} \rho \diffoperator ((1-L)^{\frac{-i\lambda}{2} + \frac{s-r}{2}} \theta(D_{E_{\ell,t}^k} f)) \in \D^\infty_r(\Omega)
  \end{displaymath}
  and
  \begin{displaymath}
    (1-L)^{\frac{r + i \lambda}{2}} \int_{\mathbb R} \rho \diffoperator ((1-L)^{\frac{-i\lambda}{2} + \frac{s-r}{2}} \theta(D_{E_{\ell,t}^k} f)) \in \L{\infty-0}(\Omega).
  \end{displaymath}
  So $\abs{F(i\lambda)} \leq \norm{\varphi}_{\L q (\omega)} \times
  C_1(f)$, $C_1(f)$ constant, $f$-dependant. A similar computation
  shows that: $\abs{F(1 + i\lambda)} \leq \norm{\varphi}_{\L q
    (\omega)} \times C_2(f)$, $C_2(f)$ constant, $f$-dependant. So
  thanks to the Phragmen-Lindel\"of method, we have: \\ $\abs{F(s-r)}
  \leq \norm{\varphi}_{\L q (\Omega)} \times C_3(f)$, $C_3(f) \in
  \L{\infty-0}(\Omega)$. And
  \begin{displaymath}
    F(s-r) = (1-L)^{\frac 1 2} \int_{\mathbb R} \rho \diffoperator (\theta D_{E_{\ell,t}^k} f) \in \L{\infty-0}.
  \end{displaymath}
  So $\int_{\mathbb R} \rho \diffoperator (\theta D_{E_{\ell,t}^k} f) \in
  \D^\infty_s(\Omega)$. As $\int_{\mathbb R} \rho \diffoperator (\theta
  D_{E_{\ell,t}^k} f) = D_{S(\rho E_{\ell,t}^k)} \theta(f)$,
  $\rho$ being \\$\D^\infty$-$\alpha$-Holderian with $\alpha > 1/2$,
  if $\theta \colon \D^\infty \to \D^\infty_s$, then $D_{S(\rho
    E_{\ell,t}^k)} \theta(f)$ sends $\D^\infty$ to $\D^\infty_s$.
  
  \medskip

  c) Now we suppose again $s \in \mathbb N_\ast$ and $\theta \colon
  \D^\infty \to \D^\infty_s(\Omega)$. Then \\ $D_{S(\rho E_{\ell,t}^k)}
  \theta(f) = \div S(\rho E_{\ell,t}^k) \Grad \theta(f)$ which
  shows that $D_{S(\rho E_{\ell,t}^k)} \theta(f) \in
  \D^\infty_{s-2}(\Omega)$. If $\theta \colon \D^\infty(\Omega) \to
  \D^\infty_{s+1}(\Omega)$ then $D_{S(\rho E_{\ell,t}^k)} \theta(f)
  \in \D^\infty_{s-1}(\Omega)$. So by interpolation, we have: for all
  $s \in \mathbb R$, if $\theta \colon \D^\infty(\Omega) \to
  \D^\infty_s(\Omega)$, then $D_{S(\rho E_{\ell,t}^k)} \theta(f)
  \in \D^\infty_{s-2}(\Omega)$.

  \medskip

  d) Now $\tilde \rho$ being an $\D^\infty$-$\alpha$-Holderian
  function, with $1 > \alpha > 0$, we consider $\tilde \rho \ast
  \beta_{1-z}$, $0 \leq \Re (z) \leq 1$ ; then $\tilde \rho \ast
  \beta_{1-z}$ is a $\gamma$-$\D^\infty$-Holderian function with
  $\gamma > \alpha + \Re (z)$ and
  \begin{displaymath}
    \abs{ \tilde \rho \ast \beta_{1-z}(t+h) - \tilde \rho \ast \beta_{1-z}(t) } \leq C h^{\alpha + \Re(z)}(1 + \abs z).
  \end{displaymath}

  Now consider the function
  \begin{displaymath}
    \Tilde F(z) = \mathrm e^{z^2} \langle (1-L)^{az+b} D_{S(\tilde \rho \ast \beta_{1-z} E_{\ell,t}^k)} \theta(f) , \varphi \rangle_{\L 2 (\Omega)}
  \end{displaymath}
  with $f$ a polynomial in Gaussian variables and $\varphi \in
  \D^\infty(\Omega)$, $z$ such that $0 \leq \Re (z) \leq 1/2$. We will
  apply the Phragmen-Lindel\"of method to $\Tilde F(z)$, on the strip $0
  \leq \Re (z) \leq 1/2$. $\Tilde F(z)$ is holomorphic on $0< \Re(z) <
  1/2$, continuous on $0 \leq \Re (z) \leq 1/2$ and bounded on this
  domain thanks to the Sobolev logarithmic inequality.

  If $\Re (z) = 0$,
  \begin{displaymath}
    \Tilde F(i \lambda) = \mathrm e^{-\lambda^2} \langle (1-L)^{ia\lambda + b} D_{S(\tilde \rho \ast \beta_{1-i\lambda} E_{\ell,t}^k)} \theta(f), \varphi \rangle_{\L 2 (\Omega)},
  \end{displaymath}
  then $\tilde \rho \ast \beta_{1-i\lambda}$ is Holderian with yield
  strictly greather than $0$ and \\ $D_{S(\tilde \rho \ast
    \beta_{1-i\lambda} E_{\ell,t}^k)} \theta(f) \in \D^\infty_{s-2}$
  (remind that we supposed $\theta \colon \D^\infty \to
  \D^\infty_s$). As we want $\Tilde F(i\lambda) \in \L{\infty-0}$,
  this implies: $-2b + s - 2 = 0$ so $b = s/2 - 1$.

  If $\Re(z) = 1/2$, 
  \begin{displaymath}
    \Tilde F(\frac 1 2 + i \lambda) = \mathrm e^{\frac 1 4 - \lambda^2 + i \lambda} \langle (1-L)^{(\frac 1 2 + i \lambda)a + b} D_{S(\tilde \rho \ast \beta_{\frac 1 2 - i \lambda} E_{\ell,t}^k)} \theta(f), \varphi \rangle_{\L 2 (\Omega)},
  \end{displaymath}
  $\tilde \rho \ast \beta_{1/2 - i \lambda}$ is
  $\D^\infty$-$\alpha$-Holderian with yield strictly greather than
  $1/2$, and 
  \begin{displaymath}
    D_{S(\tilde \rho \ast \beta_{1/2 - i\lambda} E_{\ell,t}^k)} \theta(f) \in \D^\infty_s.
  \end{displaymath}
  As we want: $\Tilde F(1/2 + i \lambda) \in \L{\infty-0}$, we must
  have: $-(a+2b)+s = 0$ so $a=2$. Then $\abs{\Tilde F(z)}$ is bounded
  on the whole band $0 \leq \Re (z) \leq 1/2$.

  Now we remind equation (\ref{eqn:16}) in Theorem \ref{thm:5.9},
  which is still valid in our setting:
  \begin{displaymath}
    \Grad \theta[f] = D_{S(\tilde \rho \ast \beta_{1-z} E_{\ell,t}^k)} \theta(f) + \proc U (\theta (\Grad f)).
  \end{displaymath}
  Then for an induction to begin, we need $\Grad \theta(f) \in
  \D^\infty_{s-1+\delta}, \delta > 0$. For this, it is enough that
  $D_{S(\tilde \rho \ast \beta_{1-z} E_{\ell,t}^k)} \theta(f) \in
  \D^\infty_{s-1+\delta}, \delta > 0$. Now we know that: for all $z$
  such that $0 \leq \Re (z) \leq 1/2$, $\abs{\Tilde F(z)} \in
  \L{\infty-0}(\Omega)$. So we look for a $z \in [0,1/2]$ such that:
  \begin{displaymath}
    (1-L)^{2 \Re(z) + \frac s 2 - 1}D_{S(\tilde \rho \ast \beta_{1-z})} \theta(f) \in \L{\infty-0}
  \end{displaymath}
  with $D_{S(\tilde \rho \ast \beta_{1-z})} \theta(f) \in
  \D^\infty_{s-1+\delta}$. This implies $-2(2 \Re (z) + s/2 - 1) + s-1
  + \delta = 0$ wich implies $\Re (z) > 1/4$. Then $\Grad \theta(f)
  \in \D^\infty_{s-1+\delta}$ and the induction can begin, so
  $\theta(f) \in \D^\infty$. Now we show that as $\proc U$ is
  $\gamma$-Holderian with $\gamma > 1/4$, then the condition $\Re (z)
  > 1/4$ is fulfilled.

  Now if we choose $\rho = f_k^\ell$, with
  \begin{displaymath}
    \left( \Grad \proc U^{-1} \right) \proc U = \sum_{k,\ell = 1}^n f_k^\ell E_\ell^k,
  \end{displaymath}
  ($f_k^\ell$ being extended as $0$ on $\mathbb R_- \cup
  {[2,+\infty[}$, and affine on $[1,2]$). We have that each $f_k^\ell$
  is $\alpha$-Holderian, $\gamma > 1/4$. Then for all $(\ell,k)$,
  $f_k^\ell \ast \beta_z$, with $\Re(z) = 1/4$, is in $\mathcal
  C^1$. So $\tilde \varphi_k^\ell = (f_k^\ell \ast \beta_z)'$ exists
  (Proposition 2.2.i) and $\tilde \rho_k^\ell \ast \beta_{1-z} =
  f_k^\ell$, which is $\gamma$-Holderian with $\gamma > 1/4$
  (Proposition 2.2.iii).
\end{proof}

Now we will prove another theorem of inversibility and
$\D^\infty$-morphism, with another hypothesis on $(T_{\proc U})^{-1}$.

\begin{defn} 
  \label{defn:5.5}
  A family $(A_i)_{i \in \mathbb N_\ast}$ of elementary tangent
  processes will be said to be a multiplicative family if and only if
  it verifies the following two conditions: $(e_i)_{i \in \mathbb
    N_\ast}$ being an Hilbertian basis of $H$,
  \begin{enumerate}[label=\alph*)]
  \item if $X \in \D^\infty(\Omega, H)$, then $\sum_{i=1}^\infty e_i
    \otimes A_i X \in \D^\infty(\Omega, H \otimes H)$,
  \item $(X_i)_{i \in \mathbb N_\ast}$ being such that
    $\sum_{i=1}^\infty e_i \otimes X_i \in \D^\infty(\Omega, H \otimes
    H)$, then \\ $\sum_{i=1}^\infty A_iX_i \in \D^\infty(\Omega,H)$.
  \end{enumerate}
\end{defn}

\begin{lem} 
  \label{lemma:5.12} 
  The family $A_i = \proc U(t,\omega) e_i . \proc U^{-1}(t,w), i \in
  \mathbb N_\ast$ is a mutiplicative family.
\end{lem}

\begin{proof}
  Condition a) of Definition \ref{defn:5.5}. As $\proc U$ is a
  multiplicator, if $X \in \D^\infty(\Omega, H)$, $\proc U . X \in
  \D^\infty(\Omega, H)$. Then:
  \begin{displaymath}
    (\Grad \proc U) . X = \Grad (\proc U . X) - \proc U_R(\Grad X)
  \end{displaymath}
  and
  \begin{displaymath}
    \sum_{i=1}^\infty e_i \otimes A_i X = \proc U_R^{-1}(\Grad \proc U . X).
  \end{displaymath}

  Condition b) of Definition \ref{defn:5.5}. $B$ being an $n \times
  n$-matrix, let $X \in \D^{-\infty}(\Omega,H)$; we define $B \cdot X$
  by:
  \begin{displaymath}
    \forall Y \in \D^\infty(\Omega,H), \, (B \cdot X,Y) = (X,{}^tB Y).
  \end{displaymath}
  We will give meaning to $\Grad(B \cdot X)$ with $X \in
  \D^{-\infty}(\Omega, H)$ by:
  \begin{displaymath}
    \Grad(B \cdot X) = (\Grad B) \cdot X + B_R (\Grad X)
  \end{displaymath}
  with $(\Grad B) \cdot X$ defined on $Y \in \D^\infty(\Omega, H
  \otimes H)$ by
  \begin{displaymath}
    \begin{split}
      (\Grad B \cdot X, Y) &= \left( \Grad B \cdot X, \sum_{k,\ell} Y^{k\ell} e_k \otimes e_\ell \right) \\
      &= \left( X, \sum_{k=1}^\infty \left(e_k {}^tB \right) \left(
          \sum_{\ell=1}^\infty Y^{k\ell} e_\ell \right) \right)
    \end{split}
  \end{displaymath}
  and $(B \Grad X, Y) = (\Grad X, {}^tB Y)$.

  Following a similar proof as in condition a) in the same way we
  have: if $X \in \D^{-\infty}(\Omega, H)$, then $\sum_{i=1}^\infty
  e_i \otimes A_i X \in \D^{-\infty}(\Omega,H \otimes H)$. Then the
  map $\Phi \colon \D^{-\infty}(\Omega,H) \to \D^{-\infty}(\Omega,H
  \otimes H)$,
    \begin{displaymath}
      \Phi(X) = \sum_{i=1}^\infty e_i \otimes A_i X \in \D^{-\infty}(\Omega, H \otimes H),
    \end{displaymath}
    admits as dual map:
    \begin{align*}
      \Phi^\ast &\colon \D^\infty(\Omega,H \otimes H) \to
      \D^\infty(\Omega,H) \\
      \Phi^\ast &(Y^{k\ell} e_k \otimes e_\ell) = - \sum_{i=1}^\infty
      A_i (Y^{i\ell} e_\ell) \in \D^\infty(\Omega,H)
    \end{align*}
    Now let a family of $(X_i)_{i \in \mathbb N_\ast}$ such that:
    \begin{displaymath}
      \sum_{i=1}^\infty e_i \otimes X_i \in \D^\infty(\Omega, H \otimes H).
    \end{displaymath}
    Then $X_i = Y_i^\ell e_\ell$ is such that $Y = \sum_{i=1}^\infty
    Y_i^\ell e_i \otimes e_\ell \in \D^\infty(\Omega, H \otimes H)$
    and
    \begin{displaymath}
      \Phi^\ast (Y) = -\sum_{i=1}^\infty A_i (Y^{i \ell} e_\ell) = -\sum_{i=1}^\infty A_i X_i \in \D^\infty(\Omega,H).
    \end{displaymath}
\end{proof}

\begin{thm} 
  \label{thm:5.12}
  Let $\proc U$ an adapted process, multiplicator with values in $n
  \times n$-unitary matrices ; and let $\theta_{\proc U}$ the
  associated morphism from $\L{\infty-0}$ in itself. Let $A$ be an
  elementary tangent process, that is each $A$ is an $n \times
  n$-A.M., and is a multiplicator ; we define $T_{\proc U}$ by:
  \begin{displaymath}
    T_{\proc U}(A) = \proc U D_A \proc U^{-1} + A.
  \end{displaymath}
  $T$ is a ``pseudo linear tangent map'' of $\theta$.

  We suppose that $T_{\proc U}$ admits an inverse $(T_{\proc
    U})^{-1}$, which verifies the additional property: for each
  multiplicative family $(A_i)_{i \in \mathbb N_\ast}$, $(T_{\proc
    U})^{-1}(A_i)_{i \in \mathbb N_\ast}$ is a multiplicative
  family. Then if $(\theta_{\proc U})^{-1}$ exists from
  $\L{\infty-0}(\Omega)$ in itself, $(\theta_{\proc U})^{-1}$ is a
  $\D^\infty$-morphism.
\end{thm}

\begin{proof}
  We write $(T_{\proc U})^{-1}(A_i) = C_i$, with $A_i = \proc U
  (t,\omega)(e_i . \proc U^{-1})(t,\omega)$. \\ $(e_i)_{i \in \mathbb
    N_\ast}$ being an Hilbertian basis of $H$, and $h \in H$, then
  straightforward computation shows that:
  \begin{equation}
    \label{eqn:24}
    e_i . \theta[W(h)] - D_{C_i} \theta[W(h)] = \int_0^1 {}^t\dot h \proc U^{-1} e_i \diff s = \langle \proc U \theta(\Grad W(h)), e_i \rangle_H .
  \end{equation}
  Both members of (\ref{eqn:24}) are $\theta$-derivations, so
  (\ref{eqn:24}) is still valid for polynomials in Gaussian variables.

  Let $f$ such that $\theta(f) \in \D^\infty$ ; then $f \in
  \L{\infty-0} (\Omega)$ and if $f_n$ is a sequence of polynomials in
  Gaussian variables which converges $\L{\infty-0}$ towards $f$ then
  (\ref{eqn:24}) is still valid with $f$ as above, but (\ref{eqn:24})
  has to be rewritten:
  \begin{equation}
    \label{eqn:25}
    \sum_{i=1}^{\infty} D_{C_i}[\theta(f)] e_i = \Grad \theta(f) - \proc U [\theta \Grad f]
  \end{equation}
  and is an equation using distributions. If $\alpha \in
  \L{\infty-0}(\Omega)$, $\sum_{i=1}^\infty D_{C_i}(\alpha) e_i$ is a
  distribution because
  \begin{displaymath}
    \begin{split}
      \forall g \in \D^\infty(\Omega, H), \, \left( \sum_{i=1}^\infty
        D_{C_i}(\alpha) e_i, g \right) &= \left( \alpha,
        \sum_{i=1}^\infty D_{C_i}(\langle e_i, g \rangle_H) \right) \\
      &= \left( \alpha, \div \left( \sum_{i=1}^\infty C_i \Grad \langle e_i,
        g \rangle_H \right) \right)
    \end{split} 
  \end{displaymath}
  and now using Definition \ref{defn:5.5}.b, as $\theta(f) \in
  \D^\infty(\Omega)$, we have:
  \begin{displaymath}
    \sum_{i=1}^\infty D_{C_i}(\theta(f)) e_i = \div_R \left( \sum_{i=1}^\infty e_i \otimes C_i \Grad \theta(f) \right)
  \end{displaymath}
  which is in $\D^\infty(\Omega, H)$ thanks to Definition
  \ref{defn:5.5}.a; $\Grad \theta(f) \in \D^\infty(\Omega,H)$
  because $\theta(f) \in D^\infty$. Now we apply each item of
  (\ref{eqn:25}) to $\proc U \theta(g)$ where $g$ is a polynomial
  vector map: so we see that there exist a $\D^\infty(\Omega,H)$
  vector field $W$ such that
  \begin{equation} \label{eqn:26}
    \int_\Omega \langle W, \proc U \theta(g) \rangle_H = (\Grad f, g)
  \end{equation}

  Now we take a sequence, $g_n$, of polynomials vector maps,
  converging in $\D^\infty(\Omega,H)$ towards $g \in
  \D^\infty(\Omega,H)$. From (\ref{eqn:26}), we see that the
  $\L{\infty-0}$-limit $(\Grad f, g)$ exists and this for $f$ such
  that $\theta(f) \in \D^\infty$.

  The set $\{f \mid \theta(f) \in \D^\infty\}$ is
  $\L{\infty-0}$-dense in $\L{\infty-0}(\Omega)$ which implies $\Grad
  f \in \D^\infty(\Omega,H)$, so $f \in \D^\infty_1(\Omega)$. So
  (\ref{eqn:24}) is then an equation in which, each item is a
  function.

  Now we proceed by induction: we know, by hypothesis, that
  $\theta^{-1} \colon \D^\infty \to \L{\infty-0}$. Suppose that for
  all $f$ such that $\theta(f) \in \D^\infty$, then $f \in
  \D^\infty_r(\Omega)$. Then (\ref{eqn:25}) can be rewritten: 
  \begin{displaymath}
    \Grad f = \theta^{-1} \proc U^{-1} \left[ \Grad \theta(f) - \sum_{i=1}^\infty D_{C_i} \theta(f) e_i \right]    
  \end{displaymath}
  which implies $\Grad f \in \D^\infty_r(\Omega, H)$, so $f \in
  \D^\infty_{r+1}(\Omega)$ and $\theta^{-1}$ is a
  $\D^\infty$-morphism of $\D^\infty(\Omega)$ in itself.
\end{proof}

\begin{rem} 
  \label{remark:5.9} 
  If $\theta$ is a $\D^\infty$-diffeomorphism, $(T_{\proc U})^{-1}$
  will transform a multiplicative family $(A_i)_{i \in \mathbb
    N_\ast}$ in a multiplicative family $((T_{\proc U})^{-1} A_i)_{i
    \in \mathbb N_\ast}$.
\end{rem}

\begin{proof}
  \begin{displaymath}
    T_{\proc U} A_i = \proc U D_A \proc U^{-1} + A_i = -(D_{A_i} \proc U) \proc U^{-1} + A_i.
  \end{displaymath}

  For the condition a): we have to show that if $X \in
  \D^\infty(\Omega, H)$, 
  \begin{displaymath}
    \sum_{i=1}^\infty e_i \otimes (T_{\proc U} A_i) X \in \D^\infty(\Omega, H \otimes H).
  \end{displaymath}
  It is enough to show that $\sum_{i=1}^\infty e_i \otimes (D_{A_i}
  \proc U) \proc U^{-1} X \in \D^\infty(\Omega, H \otimes H)$. As
  \begin{displaymath}
    \sum_{i=1}^\infty e_i \otimes D_{A_i} f = \div_R \left( \sum_{i=1}^\infty e_i \otimes A_i \Grad f \right),
  \end{displaymath}
  we see that: $f \mapsto \sum_{i=1}^\infty D_{A_i} f e_i$ is a
  derivation, that sends $\D^\infty(\Omega)$ in
  $\D^\infty(\Omega,H)$. So $\sum_{i=1}^\infty e_i \otimes (D_{A_i}
  \proc U) \proc U^{-1} X$ is in $\D^\infty(\Omega, H \otimes H)$.

  Condition b): let $\Phi$ the map $\colon \D^\infty(\Omega, H \otimes H)
  \to \D^\infty(\Omega,H)$ that sends $\Phi (\sum_{i=1}^\infty e_i
  \otimes X_i) = \sum_{i=1}^\infty A_i X_i$. As $(A_i)_{i \in \mathbb
    N_\ast}$ is a multiplicative family, $\Phi$ is legitimate. Then, if
  $Z \in \D^\infty(\Omega, H)$: 
  \begin{displaymath}
    \Phi^\ast \colon \D^{-\infty}(\Omega,H) \to \D^{-\infty}(\Omega,H \otimes H)
  \end{displaymath}
  and
  \begin{displaymath}
    \Phi^\ast(Z) = - \sum_{i=1}^\infty e_i \otimes A_iZ.
  \end{displaymath}
  So the same treatment as in condition a) above proves that the
  condition b) is fulfilled.

  Now we know that $(T_{\proc U} A_i)_{i \in \mathbb N_\ast}$ is a
  multiplicative family. We will compute $(T_{\proc U})^{-1}$ to prove
  that $((T_{\proc U})^{-1} A_i)_{i \in \mathbb N_\ast}$ is a
  multiplicative family. From $\theta[W(h)] = \int_0^1{}^t{\dot
    h}{\proc U}^{-1}{\itodiff B}$ and $\theta^{-1}[W(h)] =
  \int_0^1{}^t{\dot h}{\proc V}^{-1}{\itodiff B}$, we deduce:
  \begin{displaymath}
    \proc V = \theta^{-1}(\proc U).
  \end{displaymath}
  Then 
  \begin{displaymath}
    ( \theta^{-1} \circ D_A \circ \theta )(W(h))
    = \int_0^1 {}^t{\dot h} \theta^{-1}(D_A \proc U^{-1}) \proc V \itodiff B
    + \int_0^1 {}^t{\dot h} \theta^{-1}(\proc U^{-1} A) \proc V \itodiff B
  \end{displaymath}
  which implies:
  \begin{displaymath}
    ( \theta^{-1} \circ D_A \circ \theta )(W(h))
    = D_{\theta^{-1}[\proc U^{-1} T_{\proc U}(A) \proc U]}(W(h)).
  \end{displaymath}
  As we have supposed $\theta$ to be a $\D^\infty$-diffeomorphism:
  \begin{displaymath}
    \forall f \in \D^\infty(\Omega),\, \theta^{-1} \circ D_A \circ \theta
    = D_{\theta^{-1}[\proc U^{-1} T_{\proc U}(A) \proc U]}.
  \end{displaymath}
  We denote $S_{\proc U}(A) = \proc U^{-1} T_{\proc U}(A) \proc
  U$, then:
  \begin{displaymath}
    \theta^{-1} \circ D_A \circ \theta = D_{\theta^{-1}(S_{\proc U}(A))}.
  \end{displaymath}
  We also have: $\theta \circ D_{A'} \circ \theta^{-1} =
  D_{\theta[S_{\theta(\proc U)}]} A'$. So the maps $A
  \mapsto \theta^{-1}(S_{\proc U}(A))$ and $A' \mapsto
  \theta(S_{\theta(\proc U)}(A'))$ are inverses.

  So we have:
  \begin{displaymath}
    \theta^{-1} \circ S_{\proc U} \circ \theta \circ S_{\theta(\proc U)} = \mathrm{Id}.
  \end{displaymath}
  So
  \begin{displaymath}
    S_{\proc U} \circ \theta \circ S_{\theta(\proc U)} \circ \theta^{-1} = \mathrm{Id},
  \end{displaymath}
  so $\theta \circ S_{\theta(\proc U)} \circ
  \theta^{-1}$ is the inverse of $S_{\proc U}$. Then $T_{\proc
    U}$ is inversible, and $((T_{\proc U})^{-1} A_i)_{i \in \mathbb
    N_\ast}$ is a multiplicative family.
\end{proof}

\medskip

Two examples showing that the transformation $\theta$:
$\theta[W(h)] = \int_0^1 {}^t{\dot h} \proc U^{-1} \itodiff B$,
extended to $\L{\infty-0} \to \L{\infty-0}$, does not admit an
inverse, even if $\proc U$ admits one.

\begin{counterexample}
  \label{cex:5.1}
  Let $(B_t^{(1)}, B_t^{(2)})$ an $\mathbb R^2$-Brownian, with
  $B_0^{(1)} = B_0^{(2)} = 0$, and denote $\Delta =
  \sqrt{(B_t^{(1)})^2+(B_t^{(2)})^2}$. And let
  \begin{displaymath}
    \proc U (t, \omega) =
    \begin{pmatrix}
      \frac {B_t^{(1)}} \Delta & \frac {B_t^{(2)}} \Delta \\
      - \frac {B_t^{(2)}} \Delta & \frac {B_t^{(1)}} \Delta
    \end{pmatrix}
    \quad \text{and} \quad
    \proc U(0,\omega) = \mathrm{Id}_{\mathbb R^2} .
  \end{displaymath}
  Then $\theta[W(h)] = \int_0^1 {}^t{\dot h} \proc U^{-1} \itodiff B$
  can be extended in an $\L{\infty-0}$-morphism, denoted again
  $\theta$. Then, if $R$ is a rotation with angle $\alpha$, in
  $\mathbb R^2$, with origin $0$, 
  \begin{displaymath}
    R[\theta(W(h))] = \int_0^1 R({}^t{\dot h} \proc U^{-1} \itodiff B) = \int_0^1 {}^t{\dot h} \proc U^{-1} \itodiff B = \theta[W(h)].
  \end{displaymath}
  Then by extension: $R(\theta(f)) = \theta(f),\, \forall f \in
  \L{\infty-0}(\Omega)$. So $\theta$ cannot be surjective because
  there exists elements of $\L{\infty-0}(\Omega)$ which are not
  invariant for rotations.
\end{counterexample}

\begin{counterexample}
  \label{cex:5.2}
  Let $W[\mathds 1_{[0,t]}(\cdot)] = \int_0^t \proc U_s \itodiff B_s$,
  with: \\ $\proc U_s = +1 \iff B_s \geq 0$ and $\proc U_s = -1 \iff B_s
  < 0$. Then
  \begin{displaymath}
    X_t = W[\mathds 1_{[0,t]}(\cdot)] = \int_0^t \left( \mathds 1_{\{B_s \geq 0\}} - \mathds 1_{\{B_s < 0\}} \right) \itodiff B_s = \int_0^t \left( \mathds 1_{\{B_s > 0\}} - \mathds 1_{\{B_s < 0\}} \right) \itodiff B_s.
  \end{displaymath}
  And let $\theta$ the map $\mathcal C_0([0,1], \mathbb R) \to
  \mathcal C_0([0,1], \mathbb R)$ defined by: $\theta(x) =
  -x$. Then
  \begin{displaymath}
    \theta \circ X_t = \int_0^t \left( - \mathds 1_{\{B_s > 0\}} + \mathds 1_{\{B_s < 0\}} \right) \raisebox{0.3ex}{.} (- \diff B_s) = X_t.
  \end{displaymath}
  So the $\sigma$-algebra generated by the $X_t, \sigma(X_t) = \{ X_t
  \mid t \in [0,1] \}$, is left invariant by $\theta$. But $\theta$ is
  not surjective because $\sigma(X_t) \subsetneq \mathcal F_1$, where
  $\sigma(X_t)$ is the $\sigma$-algebra generated by $\{ X_t \mid t
  \in [0,1] \}$.
\end{counterexample}

\section{\huge $\P_{m_0}(V_n, g)$ is a $\mathbb{D}^\infty$-stochastic manifold}

Let $V_n$ be an $n$-dimensional compact Riemannian manifold with the metric $g$, and let $\nabla$ be a connection on $V_n$, compatible with $g^{(1)}$; and $\Gamma_{jk}^i$ the Christoffel symbols.

\subsection{Introduction}
We recall the definition of a Brownian motion $p_t$, $V_n$-valued, starting from $p(0) = m_0 \in V_n$, and some of the properties of the stochastic parallel transport (SPT in short).

a) $\forall f \in C^{\infty}(V_n, g): f \circ p_t - f \circ p_0 - \frac{1}{2} \int_0^t \Delta f \circ p_s \cdot ds = M_f(t)$, $M_f(t)$ being a martingale.

b) The SPT is an intrinsic notion.

c) The scalar product is invariant by SPT: if $X_1$ and $X_2$ are two vectors in $T_{m_0} V_n$ and $X_i(t, \omega) \in T_{\omega(t)} V_n ~ (i = 1,2)$ are the SPT of $X_1$ and $X_2$ "along $\omega(t)$", then:
\begin{align*}
g_{\omega(t)} \left( X_1( t, \omega ), X_2(t, \omega) \right) = g_{m_0} (X_1, X_2)
\end{align*}

d) In a local chart of $(V_n, g)$, $X^k(\omega, t)$ being the $k^{th}$ coordinate of the SPT vector $X$, we have:
\setcounter{equation}{0}
\begin{align}
X^k (t, \omega) = - \int_0^t \Gamma_{ij}^k \left( p(s) \right) X^j \circ dp_s^j \label{eq6_1} 
\end{align}
\let\thefootnote\relax\footnote{(\ref{eq6_1}) with possibly a non-vanishing torsion.}

We will also denote by $X_{//}(t, \omega)$ 
the SPT of the vector $X \in T_{m_0} V_n$ "along the curve $(\omega(t))$", at time $t$.

e) Let $(\U, \theta)$ be a chart centered on $m \in V_n$, $\U$ being the domain of the chart and $\theta$ the coordinate map.
$u$ being an isomorphism of $\R^n$ in $T_{m_0} V_n$, and $e_\alpha$, $\alpha = 1, \dots, n$ the canonical basis unit vectors of $\R^n$, we denote $u_\alpha(t, \omega)$ the SPT of $u e_\alpha$, "along the curve $\omega(t)$", and by $Z_\alpha^k(t, \omega) = \left( \theta_\star u_\alpha(t, \omega) \right)^k$, the $k^{th}$ component of the vector $u_\alpha(t, \omega)$, when read in the chart $(\U, \theta)$. Then the matrix $Z_\alpha^k$ is invertible and if we write $d\tilde{B}_t^k = (Z^{-1})_\mu^k dM_f^\mu (t)$, $\tilde{B}_t$ is an $n$-dimensional Brownian motion, and we have:
\begin{align}
dp_t^k = \frac{1}{2} \Delta p^k ds + Z_\mu^k \cdot d\tilde{B}^\mu \label{eq6_2}
\end{align}

And direct calculus shows that:
\begin{align}
dp_t^k = Z_\mu^k \circ d\tilde{B}^\mu \label{eq6_3}
\end{align}
\\
\begin{defn}

\end{defn}

$H$ being the canonical $C.M.$ space, $u$ an isomorphism of $\R^n$ in $T_{m_0} V_n$, and $u_\alpha(t, \omega) = (u e_\alpha)_{//} (t, \omega)$ as above, we denote:
\begin{align*}
\tilde{H} = \left\{ v(t, \omega) = \sum_{\mu = 1}^n f^\mu(t) u_\mu(t, \omega) \bigg| f^\mu \in H \right\}
\end{align*}

$\tilde{H}$ is called the new Cameron-Martin space, in short: N.C.M.. A scalar product on $\tilde{H}$ is defined by:
\begin{align*}
\langle v_1(t, \omega), v_2(t, \omega) \rangle_{\tilde{H}} = \sum_{\mu = 1}^n \int_0^1 \dot{f}_1(s) \dot{f}_2(s) ds 
\end{align*}

With $\langle ~, \rangle_{\tilde{H}}$, $\tilde{H}$ is an Hilbert space.

Each $v(t, \omega) \in \tilde{H}$ is a process, valued in the fiber-tangent $TV_n$.

We recall the theorem of moment inequalities for martingales [8, p.110].

\begin{thm} \label{thm6_1}
If $\M$ is the set of continuous locally square integrable martingales, there exist universal constants $c_p$ and $C_p~ (1 < p < + \infty)$ such that $\forall M \in \M$, and $t \geq 0$:
\begin{align*}
c_p \E \left[ \max_{0 \leq s\leq t} | M_s |^{2p} \right] \leq \E \left[ \langle M, M \rangle_t^p \right] \leq C_p \E \left[ \max_{0 \leq s \leq t} |M_s|^{2p}\right]
\end{align*}

\end{thm}

\begin{cor}\label{cor6_1} The solution of the equation of the SPT, (\ref{eq6_1}), 

is $\dinf$-bounded.
\end{cor}

\begin{proof}
$(V_n, g)$ is a compact Riemannian manifold, and $0 \leq t \leq 1$.

The equation (\ref{eq6_1}) shows that $X^k$ is $\in L^{\infty - 0}(\Omega, H)$, but we do not know if $X^k$ admits as gradient a function. We can deduce from (\ref{eq6_1}), that $X^k$ admits as gradient, a distribution; but we do not know if this distribution is a function.

To overcome this difficulty, we proceed as such:

If gradient $X^k$ exists, as a function, it will verify the equation:

\begin{align}
\Grad X^k = - \int_0^t \left(\Grad \Gamma_{ij}^k \right) X^j \circ dp^i - \int_0^t \Gamma_{ij}^k \left( \Grad X^j \right) \circ dp^i - \int_0^t \Gamma_{ij}^k X^j \circ \Grad dp^i \label{eq6_4}
\end{align}

and $\Grad(dp^i)$ can be computed with $(\ref{eq6_2})$.

So we can look for a system of two unknown functions verifying equations (\ref{eq6_1}) and (\ref{eq6_4}): $V_n$ being a $C^\infty$ compact manifold, the coefficients of this system of two equations are all $C^\infty$-bounded; so the system ((\ref{eq6_1}), (\ref{eq6_4})) admits a unique solution which is $L^{\infty-0}$-bounded. Iterating this process shows that the solution of (\ref{eq6_1}) is $\dinf$-bounded.

\end{proof}

\begin{cor}\label{cor6_2}
If $u_1, \dots, u_k$ are $k$ SPT vectors, and $\T$ is a $C^\infty$ $k$-invariant tensor on $(V_n, g)$, then $\T(u_1, \dots, u_n)$ is $\dinf$-bounded.
\end{cor}

\begin{proof}
all derivatives of $\T$ are bounded on $(V_n, g)$ and 

$\sup_{t \in [0, 1]} \| u_\mu \|_{\drp(\Omega, H)}$, $\mu = 1, \dots, k$ are all bounded. Then for the Malliavin derivative of order $r$: 

$\sup_{0 \leq 0 \leq 1} \left\| \Grad^r \left( \T( u_1, \dots, u_n) \right) \right\|_{\overset{r}{\otimes} H}$ is $L^p$-bounded.

\end{proof}

\subsection{Construction of a $\mathbb{D}^\infty$-atlas on $\P_{m_0}(V_n, g)$}

Now we will construct on $\P_{m_0} (V_n , g)$ a $\dinf$-atlas:

Let two connections $\conx{1}$ and $\conx{2}$ compatible with the metric $g$; the $\dinf$-atlas $\mathcal{A}$ on $\P_{m_0}(V_n, g)$ consists of two Wiener spaces $\mathcal{W}_1$ and $\mathcal{W}_2$, and the corresponding It\={o} maps $I_1$ and $I_2$. The chart change maps are then:
\begin{align*}
J_1 = I_2^{-1} \circ I_1 ~, ~ J_2 = I_1^{-1} \circ I_2
\end{align*}

Now we limit ourselves to the case for which the trace of the tensor $\conx{1} - \conx{2}$ is zero, so that the Laplacian is invariant.

Otherwise $\mathcal{A}$ is still a $\dinf$-atlas on $\P_{m_0}(V_n, g)$, but the calculus is more complex because the Brownians associated to the two connections differ by a drift (a vector field on $V_n$); and the image of the probability on the first chart, by any of the $J$'s, differs from the probability on the second chart by a density.

Let $\theta$ be the morphism associated to the chart change map $J$; $\theta$ leaves invariant laws and filtrations, so it leaves invariant the quadratic variations and the martingale property. So there exists an $n \times n$ matrix $\V$ such that $\theta(dB_2) = \V dB_1$, which implies if $h \in H_2$: $\theta[W(h)] = \int {}^t \dot{h} \V dB_1$, ${}^t \dot{h}$ being the $n$-linear vector obtained by transposition of the $n$-column vector $\dot{h}$.

As $\theta$ keeps invariant the quadratic variation, we have ${}^t \V \V = Id$; and as $\theta$ leaves invariant the filtrations, $\V$ is also an adapted process.

\begin{lem}
$\V$ is a $\dinf$-process and a multiplicator.
\end{lem}

\begin{proof}
We denote by $E(t, \omega)$ and $F(t, \omega)$ the frames on $V_n$, obtained by SPT when using the two connections $\conx{1}$ and $\conx{2}$

From (\ref{eq6_1}), we know that the SPT vectors which form the basis $E(t, \omega)$, $F(t, \omega)$ are $\dinf$-semi-martingales with martingales parts, $\alpha$-H\"{o}lderian $(\alpha < \frac{1}{2})$, and bounded variations parts of class $C^1$.

Then from $\V E = F$, we get using Corollary \ref{cor6_1} that $\V$ is $\dinf$. As:

\begin{align}
d\V = dF . E^{-1} + F . dE^{-1} \label{eq6_5}
\end{align}

$\V$ is a semi-martingale, with the martingale part being $\alpha$-H\"{o}lderian and the bounded variation part being of class $C^1$. The iterated gradients of $\V$ will verify similar equations, and so will be also semi-martingales with martingales parts being $\alpha$-H\"{o}lderian, and bounded variations parts being of class $C^1$. So $\V$ is a $\dinf$-multiplicator.

Now we write $\U^{-1} = \V$.

\end{proof}

\begin{lem}
With the previous notations, if $\theta^{-1}(\U)$ is a multiplicator from $\dinf(\Omega, L^2([0,1], \R^n))$ in itself, and if $\theta^{-1}(\Grad \U^{-1} \cdot \U)$ is a multiplicator from $\dinf(\Omega, L^2([0,1], \R^n))$ to $\dinf(\Omega, H \otimes L^2([0,1], \R^n))$, then $\theta(\dinf) \subset \dinf(\Omega)$.
\end{lem}

Note: $\theta^{-1}(\Grad \U^{-1} \cdot \U)$ acts on $\dinf(\Omega, L^2([0,1], \R^n))$ by left-tensor 

matrix multiplication; if $\vec{X}_{ij}$ is the $(i, j)$ vector entry of the 

$n \times n$-matrix $\theta^{-1}\left[ \Grad \U^{-1} \cdot \U\right]$, and if $\alpha_k(t, \omega), k= 1, \dots, n$ is an item of $\dinf(\Omega, L^2([0,1], \R^n))$, we have:
\begin{align*}
\left[ \left( \theta^{-1}\left( \Grad \U^{-1} \cdot \U\right) \right) ( \alpha_k ) \right]_i = \sum_{l=1}^m \alpha_l \otimes \vec{X}_{il}
\end{align*}

\begin{proof}
first we remind the definition of the operator denoted $\Div_R$:

If $X_1, \dots, X_k$ are constant vectors of $H$, and if $Y_1, \dots, Y_k$ are $\dinf$-vector fields, by definition:
\begin{align*}
\Div_R \left( \sum_{i=1}^k X_i \otimes Y_i \right) = \sum_{i=1}^k \left( \Div Y_i \right) \cdot X_i \text{ (see Chap. 5, before Lemma 5, 1)}
\end{align*}

With Theorem 2, 4, and Corollary 2, 4, $\Div_R$ can be extended in a continuous linear operator from $\dinf(\Omega, H \otimes H)$ to $\dinf(\Omega, H)$.
\end{proof}

$\theta^{-1}$ being a continuous $\dinf$-morphism, for $h \in H$, we have:
\begin{align*}
\theta \left[ W(h) \right] = \int_0^1 {}^t \dot{h} ~ \U^{-1} . dB
\end{align*}

so:
\begin{align*}
\Grad \left[ \theta (W(h) )\right] &= \int_0^1 {}^t \dot{h} \Grad \U^{-1} . dB + \left( t \rightarrow \int_0^t \U^{-1} \dot{h} ~ds \right) \\
&= \int_0^1 {}^t \dot{h} ( \Grad \U^{-1} \U ) ~ \U^{-1} . dB + \left( t \rightarrow \int_0^t \U^{-1} \dot{h} ~ ds \right)
\end{align*}

And:
\begin{align}
\theta^{-1} \left[ \Grad \theta \left( W(h) \right) \right]
&= \int_0^1 {}^t \dot{h} \theta^{-1} \left[ \Grad \U^{-1} \cdot \U \right] \theta^{-1} ( \U^{-1} . dB) 
+ \left( t \rightarrow \int_0^t \theta^{-1}(\U^{-1}) \dot{h} ~ds \right) \notag\\
&= \int_0^1 {}^t \dot{h} \theta^{-1} \left[ \Grad \U^{-1} \cdot \U \right] . dB  
+ \left( t \rightarrow  \int_0^t \theta^{-1}(\U^{-1}) \dot{h} ~ds \right) \label{eq6_6}
\end{align}

$(e_i)_{i \in \N_\star}$ being an Hilbertian basis of $H$, we define:

\begin{align}
\theta^{-1} \left[ \Grad \theta(f) \right] 
&= \Div_R \left[ \sum_{l=1}^{\infty} e_l \otimes \langle e_l, \theta^{-1}( \Grad \U^{-1} \cdot \U ) \rangle_H \Grad f \right] + \left( t \rightarrow \int_0^t \theta^{-1}(\U^{-1}) \Grad f ~ds \right) \label{eq6_7}
\end{align}

From the r.h.s. of (\ref{eq6_7}), one can verify that the definition of $\theta^{-1} \left[ \Grad \theta(f) \right]$ is legitimate, and that it is a derivation on $\dinf(\Omega)$, by using $\U^\star . \U = Id$; and that if $f = W(h)$ $(h \in H)$, then (\ref{eq6_7}) is identical to (\ref{eq6_6}).

Moreover, $\theta^{-1} \left[ \Grad \theta(f) \right]$ is a $\dinf$-continuous derivation.

Now we proceed by induction: 

we know that $\theta : \dinf \rightarrow L^{\infty - 0}$. Suppose $\theta: \dinf \rightarrow \dinf_r, r \in \N_\star$.

The r.h.s. of (\ref{eq6_7}) implies that $\theta^{-1} \left[ \Grad \theta(f) \right] \in \dinf(\Omega, H)$, so:

 $\theta \circ \theta^{-1} \left[ \Grad \theta(f) \right] \in \dinf_r$, which implies: $\Grad \theta(f) \in \dinf_r$, so $\theta(f) \in \dinf_{r+1}$.
 
\begin{lem} \label{lem6_1}
If $\theta$ is an auto-diffeomorphism of $\dinf$:

i) the associated $\U$ to $\theta$ is a $\dinf$-multiplicator.

ii) $\theta^{-1}(\U)$ and $\theta ( \Grad \U^{-1} . \U )$ are multiplicators.
\end{lem}

\begin{proof}
$\U$ being associated to the diffeomorphism $\theta$, is $\dinf$-bounded. From Corollary 4, 1, we see that $ \frac{ B_i(t+h) - B_i(t) }{ \sqrt{h} }$ are $h$-uniformly multiplicators which implies tha the same is true for the processes; $\frac{1}{h} \int_t^{t+h} \U_{ij} dB^j$, because:

let $V$ be a $\dinf$-vector field; $\theta^{-1}(V)$ is also a $\dinf$-vector field; then $\left( t \rightarrow \frac{B_i(t+h) - B_i(t)}{\sqrt{h}} \cdot \theta^{-1} (V) \right)$ is a $\dinf$-vector field and so is: $\left( t \rightarrow \theta \left\{ \frac{B_i(t+h) - B_i(t)}{\sqrt{h}} \cdot \theta^{-1} (V) \right\} \right)$ 

which equals: $\left( t \rightarrow \theta \left[ \frac{ B_i(t+h) - B_i(t) }{ \sqrt{h} } \right] \cdot V \right)$

or equals: $\left( t \rightarrow \frac{1}{\sqrt{h}} \int_t^{t+h} \U_{ij}^{-1} . dB^j \right) . V $.

Then $\frac{1}{h} \left( \int_t^{t+h} \U_{ij}^{-1} . dB^j \right) \cdot \left( B_k(t+h) - B_k(t) \right)$
are $h$-uniformly 

multiplicators, and with It\={o}'s formula, denoting $M_i^{(1)} = \frac{1}{h} \int_{t}^{t+h} \U_{ij}^{-1} . dB^j$ and $M_k^{(2)} = B_k(t+h) - B_k(t)$, we get:

\begin{align*}
\frac{1}{h} \left( \int_t^{t+h} \U_{ij}^{-1} . dB^j \right) \cdot \left( B_k(t+h) - B_k(t) \right)
&= \int_t^{t+h} M_1 . dM_2 + \int_t^{t+h} M_2 . dM1 + \frac{1}{h} \int_t^{t+h} \U_{ik}^{-1} ds
\end{align*}

Direct calculus shows that: $\int_t^{t+h} M_i^{(1)} . dM_k^{(2)} + \int_t^{t+h} M_k^{(2)} . dM_i^{(1)}$ is $L^2$-bounded, $h$-uniformly. 

So an extracted sequence of
\begin{align*}
\left\{ \Phi_h^{ik} =  \left( \frac{1}{h} \int_t^{t+h} \U_{ij}^{-1} . dB^j \right) \cdot \left( B_k(t+h) - B_k(t) \right) - \frac{1}{h} \int_t^{t+h} \U_{ij}^{-1} ds ~\bigg/ ~ h \downarrow 0 \right\}
\end{align*}

converges $L^2$-weakly towards a limit.

But: $\frac{1}{h} \left( \int_t^{t+h} \U_{ij}^{-1} . dB^j \right) \cdot \left( B_k(t+h) - B_k(t) \right) \in \F_t^\perp \cap \F_{t+h}$ .

As the filtration is right-continuous, we have: $\lim_{h \downarrow 0} \F_t^\perp \cap \F_{t+h} = \{ 0 \}$

Then an extracted sequence of $\frac{1}{h} \left( \int_t^{t+h} \U_{ij}^{-1} . dB^j \right) \cdot \left( B_k(t+h) - B_k(t) \right)$ converges $L^2$-weakly towards $0$.

Combining these two extractions, we get a new sequence denoted again $\Phi_h^{ik}$ such that $\Phi_h^{ik}$ converges $L^2$-weakly towards $\U_{ik}^{-1}$ and such that:

$\frac{1}{h} \left( \int_t^{t+h} \U_{ij}^{-1} . dB^j \right) \cdot \left( B_k(t+h) - B_k(t) \right)$ converges $L^2$-weakly towards $0$.

Then a barycentric net $B_h^{ik}$ built with the $\Phi_h^{ik}$, will converge $L^2$ strongly towards $\U_{ik}^{-1}$.

With the same barycentric combination that was used to get $B_h^{ik}$ from the sequence $\left( \Phi_h^{ik} ~\big/~ h \downarrow 0 \right)$, but this time applied to the sequence 

$\left\{ \left( \frac{1}{h} \int_t^{t+h} \U_{ij}^{-1} . dB^j \right) \cdot \left( B^k_{t+h} - B^k_{t} \right) \big/ ~h \downarrow 0 \right\}$, we get a net of $h$-uniform multiplicators, denoted $M_h^{ik}$.

Then: $\forall X \in \dinf(\Omega, H)$, we have: 

$\forall (p, r)$ and $\forall (i, k) \in \N_\star \times \N_\star$: $\sup_{h} \| M_h^{ik} . X \|_{\drp(\Omega, H)}$ bounded and $M_h^{ik} X$ converges $L^2$-strongly towards $\U_{ik}^{-1} X$. 

Then, by interpolation, we have that $\U_{ik}^{-1}$ is a $\dinf(\Omega, H)$ multiplicator.

ii) $\U$ is a $\dinf$ multiplicator: so if $V$ is a $\dinf$-vector field, $\U . \theta(V)$ is also a $\dinf$-vector field; and then $\theta^{-1}(\U) . \theta^{-1}\theta(V)$ is a $\dinf$-vector field which implies that: $\theta^{-1}(\U)$ is a $\dinf$-multiplicator.

Similar proof for $\Grad \U^{-1} . \U$, with a vector field $V \in \dinf(\mathcal{W}, H)$, then $(\Grad \U^{-1} . \U ) \theta(V) \in \dinf(\mathcal{W}, H \otimes H)$.
\end{proof}

Now we have:

\begin{thm} \label{thm6_2}
The set\hspace{1mm} $\P_{m_0}(V_n, g)$ can be endowed with a $\dinf$-stochastic manifold structure.
\end{thm}

\begin{proof}
Let $\conx{1}$ and $\conx{2}$ be two connections on $(V_n, g)$, both compatible with $g$, $I_1$ and $I_2$ the respectively associated It\={o} maps, $J = I_2^{-1} \circ I_1$ the chart change map, and $\theta$ the morphism associated with $J$.

We suppose that $\conx{1}$ and $\conx{2}$ both verify the Driver condition, so the associated Laplacians $\overset{(1)}{\Delta}$ and $\overset{(2)}{\Delta}$ are equal (see following Lemma 6, 4).

Then if $\tilde{B}_1$ and $\tilde{B}_2$ are the associated Brownians, we have: 
\begin{align}
d\tilde{B}_2 = \U^{-1} d\tilde{B}_1 = \theta ( d\tilde{B}_1 ) \label{eq6_8}
\end{align}

$\tilde{B}_1$ and $\tilde{B}_2$ being as in (\ref{eq6_2}), and $\U$ being associated to $\theta$, such that: $\theta \left[ W(h) \right] = \int_0^1 {}^t \dot{h} \U^{-1} . d\tilde{B}_1$.

From (\ref{eq6_8}), we get: $\U(\omega_1) d\tilde{B}_2 = d\tilde{B}_1$, so: $\theta(\U) d\tilde{B}_2 = d\tilde{B}_1$, which implies that $\theta(\U)$ is a $\dinf$-multiplicator.

From $\U d\tilde{B}_2 = d\tilde{B}_1$, we deduce: 
\begin{align*}
\Grad \U^{-1} . d\tilde{B}_1 + \U^{-1} . \Grad (d\tilde{B}_1) = \Grad ( d\tilde{B}_2 ).
\end{align*}

This last SDE shows that $\Grad \U^{-1}$ is a $\dinf$-multiplicator, and then $\Grad \U^{-1} . \U$ is a $\dinf$-multiplicator.

Then from Lemma 4, 2, we get that $\theta$ is a $\dinf$-diffeomorphism and that $\P_{m_0}(V_n, g)$ with this chosen atlas is a $\dinf$-stochastic manifold.
\end{proof}

\begin{lem} \label{lem6_4}
If two connections $\conx{1}$ and $\conx{2}$ on the $n$-dimensional compact manifold $V_n$ are compatible with the metric $g$, and if both connections verify the Driver condition, then the Laplacians $\overset{(1)}{\Delta}$ and $\overset{(2)}{\Delta}$ are identical.
\end{lem}

\begin{proof} Denote $M(u, v) = \conx{1}_u v - \conx{2}_u v$.

Then $\conx{1}_u v - \conx{2}_u v = M(u,v) = \overset{(1)}{T}(u,v) - \overset{(2)}{T}(u,v)$, $\overset{(1)}{T}$ and $\overset{(2)}{T}$ being the torsions of the connections $\conx{1}$ and $\conx{2}$.

Then from $u . g(v, v) - u.g(v, v) = 0$, we have:
\begin{align}
g(M(u,v), v) = 0 \label{eq6_9}
\end{align}

From $v . g(u,v) - v.g(u,v) = 0$, we have:
\begin{align}
g(M(v, u), v) + g(u, M(v,v)) = 0 \label{eq6_10}
\end{align}

And because $\conx{1}$ and $\conx{2}$ verify the Driver condition, we have:
\begin{align}
g(M(u,v), v) - g(M(v, u), v) = 0 \label{eq6_11}
\end{align}

From (\ref{eq6_9}),  (\ref{eq6_10}),  (\ref{eq6_11}) we get: $g(u, M(v, v)) = 0$, so: $M(v,v)=0, ~\forall v$.

As $\overset{(1)}{\Delta} - \overset{(2)}{\Delta} = \sum_{i=1}^n \conx{1}_{e_i} e_i - \sum_{i=1}^n \conx{2}_{e_i} e_i$, we get: $\overset{(1)}{\Delta} = \overset{(2)}{\Delta}$.

\end{proof}

\section{\huge Derivations on $\P_{m_0}(V_n)$}

Let $(V_n, g)$ be a Riemannian $n$-dimensional compact manifold with connection $\nabla$, compatible with $g$, and $\P_{m_0}(V_n, g)$ be as usual the set of all continuous paths: $[0, 1] \rightarrow V_n$, starting from $m_0 \in V_n$.

We now want to prove that under the Driver condition, the $\dinf$-module generated by a specific type of derivations (being built using $C^\infty$ vector fields on $V_n$), is "dense" in the set of all $\dinf$-continuous derivations.

\subsection{If any $D_v$-type of derivation has an unique $\mathbb{D}^\infty$-derivation extension, the Driver condition is fullfilled}

We denote by $I$ the It\={o} application of the Wiener space $\mathcal{W}$ into $\P_{m_0}(V_n, g)$, and if $f \in C^\infty(V_n)$, by
\setcounter{equation}{0} 
\begin{align}
F_{f,t} (\omega) = \left( f \circ I \right)(\omega)(t) \label{eq7_1}
\end{align}
$\tilde{H}$ being the NCM as in Definition (5, 1), we define, with $v \in \tilde{H}$, an operator $\Dop_v$ by:
\begin{align}
\Dop_v \left( F_{f,t} \right)(\omega) = (v . f ) |_{I(\omega)(t)} \label{eq7_2}
\end{align}

We will show that if the Driver condition is satisfied, $\Dop_v$ can be extended in a $\dinf$-continuous derivation on $\dinf(\Omega)$, and conversely.

  The Driver condition being: if $T$ is the torsion of the manifold,
  \begin{equation*}
    \forall u, v \in \Gamma(V_n), \quad g(T(u,v), v) = 0.
  \end{equation*}
  We first show that $D_v$ is an adapted derivation, assuming it has a unique extension on 
  $\mathbb D^\infty$, denoted again $D_v$. For this we need:  
  
  \begin{lem}\label{lem1}
    \begin{equation*}
      \sigma\left[F_{f,s}\,/\, s \leq t, f \in C^\infty(V_n)\right] = \mathcal F_t
    \end{equation*}
  \end{lem}
  
  \begin{proof}
    The inclusion
    \begin{equation*}
      \sigma\left[F_{f,s}\,/\, s \leq t, f \in C^\infty(V_n)\right] \subset \mathcal F_t
    \end{equation*}
    is trivial. To prove the reverse inclusion, it is enough to prove that 
    \begin{equation*}
      \tilde B_t \in \sigma\left[F_{f,s}\,/\, s \leq t, f \in C^\infty(V_n)\right]
    \end{equation*}
    $\tilde B_t$ being defined as in Section 6, introduction a).
    \begin{equation*}
      \mathrm d\tilde B_t^k = \left(Z^{-1}\right)_\mu^k\,\mathrm dM_t^\mu 
    \end{equation*}
    We have with Section 6 notations:  
    \begin{align*}
      Z_\mu^k          & = \left( \theta_* u_\mu (\omega, t) \right)^k, \\
      \mathrm dZ_\mu^k & = -\Gamma_{i\ell}^k Z_\mu^\ell \circ \mathrm dp^i
    \end{align*}
    so $Z_\mu^k$ is the solution of a SDE with coefficients 
    in $\sigma\left[F_{f,s}  \,/\, s \leq t, f \in C^\infty(V_n)\right]$, so 
    $Z_\mu^k \in \sigma\left[F_{f,s}\,/\, s \leq t, f \in C^\infty(V_n)\right]$.

    Now from (6, 3) we have
    \begin{equation*}
      \mathrm d\tilde B_t^\ell = \left( Z^{-1}\right)_k^\ell \circ \mathrm dp_t^k
    \end{equation*}
    Then, as $\left(Z^{-1}\right)_\ell^k \in \sigma\left[F_{f,s}\,/\, s \leq t, f \in C^\infty(V_n)\right]$, 

    \begin{equation*}
      \tilde B_t \in \sigma\left[F_{f,s}\,/\, s \leq t, f \in C^\infty(V_n)\right].
    \end{equation*}
  \end{proof}
  
  \begin{thm}\label{thm1}
    Assuming there exists a unique extension of $D_v$, defined on its domain by (2), this extension is an
    adapted derivation.
  \end{thm}

  \begin{proof}
    The definition (2) of $D_v$ and \lemref{lem1} show that: $D_v F_{f, t} \in \mathcal F_t$. So the
    extension of $D_v$ to $\mathbb D^\infty$ being supposed $\mathbb D^\infty$-continuous, $D_v$ is
    an adapted derivation.
  \end{proof}

  \begin{thm}\label{thm2}
    The NSC for $D_v$ ($v \in \tilde H$) to have a unique $\mathbb D^\infty$-continuous, adapted
    extension on $\mathbb D^\infty(\mathcal W)$ with zero divergence is the Driver condition:
    if T is the torsion of $\nabla$, 
    \begin{equation}
      \forall u, v \in \Gamma(V_n), \quad g(T(u,v), u) = 0\tag{3}\label{l3}
    \end{equation}
  \end{thm}

  Before proving \thmref{thm2}, we need some lemmas.
  \begin{lem}\label{lem2}
    $Z_\mu^k = \left(Z(t,\omega)\right)_\mu^k$ is a $\mathbb D^\infty$-bounded process.
  \end{lem}

  \begin{proof}
    This is corollary 6, 1.
  \end{proof}

  To prove that \eqref{l3} is a necessary condition, we suppose now that $D_v$ can be extended in
  a $\mathbb D^\infty$-continuous unique adapted derivation on $\mathbb D^\infty(\mathcal W)$, again
  denoted $D_v$.

  \begin{lem}\label{lem3}
    If $v \in \tilde H$, $D_v(Z_\mu^k)$ and $\Gamma_{ij}^k Z_\mu^iv^j$ are $\mathbb D^\infty$-bounded
    semi-martingales.
  \end{lem}

  \begin{proof} ~

    \begin{enumerate}
      \item[i)]
        \begin{align*}
          \mathrm d Z_\mu^k & = -\Gamma_{ij}^kZ_\mu^j \circ \mathrm dp^i \\
                            & = -\Gamma_{ij}^kZ_\mu^j \cdot \mathrm dp^i
                                -\frac12 \left[\mathrm d(\Gamma_{ij}^kZ_\mu^j),
                                 \mathrm dp^i \right]
        \end{align*}
        The bracket gives a $\mathbb D^\infty$-bounded process $\times \mathrm dt$ denoted: 
        $\widehat Z_\mu^k\cdot\mathrm dt$ and $D_v\widehat Z_\mu^k$ has meaning because
        we have supposed
        that $D_v$ is a derivation on $\mathbb D^\infty(\mathcal W)$. Then:  
        \begin{equation*}
          D_v(\Gamma_{ij}^kZ_\mu^j.\mathrm dp^i) = \left( D_v \Gamma_{ij}^k\right)Z_\mu^j.\mathrm dp^i
                                                   + \Gamma_{ij}^k(D_vZ_\mu^j).\mathrm dp^i
                                                   + \Gamma_{ij}^kZ_\mu^jD_v(\mathrm dp^i)
        \end{equation*}
        But 
        \begin{equation*}
          D_v(\mathrm dp^i) = \mathrm d(D_vp^i) = dv^i\tag{4}\label{l4}
        \end{equation*}
        and $v^i$ is a $\mathbb D^\infty$-S.M.
           
      \item[ii)] $\Gamma_{ij}^k(p_t)$ is a S.M, and $\Gamma_{ij}^kZ_\mu^iv^j$ is the product of
        three S.M.
    \end{enumerate}
  \end{proof}

  \begin{proof}[Proof of the necessary condition]
      Now, from
        \begin{equation*}
          \mathrm dp^k = Z_\mu^k \circ \mathrm d\tilde B^\mu\tag{6, 3}\label{lVI-3}
        \end{equation*}
        we get:  
        \begin{equation*}
          D_v(\mathrm dp^k) = D_v(Z_\mu^k) \circ \mathrm d\tilde B^\mu
                              + Z_\mu^k\circ D_v(\mathrm d\tilde B^\mu)
                              \tag{5}\label{l5}
        \end{equation*}
        We suppose that $D_v(\mathrm d\tilde B^k)$ has the form
        \begin{equation*}
          \mathrm d\left(D_v\tilde B^k\right) = \dot h_1^k \mathrm dt 
                                                + A_\mu^k.\mathrm d\tilde B^\mu
                                                \tag{6}\label{l6}
        \end{equation*}
        where $t\mapsto\int_0^t\dot h_1\,\mathrm ds$ is a $\mathbb D^\infty$-vector field.
        Reporting \eqref{l6} and \eqref{l4} in \eqref{l5}:
        \begin{equation*}
          \mathrm dv^k(t, \omega) = D_v(Z_\mu^k) \circ \mathrm d\tilde B^\mu
                                    + Z_\mu^k\dot h_1^\mu\mathrm dt 
                                    + Z_\mu^k\circ A_\nu^\mu\cdot\mathrm d\tilde B^\nu
        \end{equation*}
        As $v^k$ is a SPT vector,
        \begin{equation*}
          -\Gamma_{ij}^kZ_\mu^iv^j\circ\mathrm d\tilde B^\mu = D_v(Z_\mu^k)\circ\mathrm d\tilde B^\mu
                                                               + Z_\mu^k\dot h_1^\mu\mathrm dt 
                                                               + Z_\lambda^k\circ A_\mu^\lambda\cdot\mathrm d\tilde B^\mu
                                                               \tag{7}\label{l7}
        \end{equation*}
        Identifying the It\^o integals in \eqref{l7}, we get:
        \begin{equation*}
          -\Gamma_{ij}^kv^jZ_\mu^i = D_v\left(Z_\mu^k\right) + Z_\lambda^kA_\mu^\lambda\tag{8}\label{l8}
        \end{equation*}
        From \eqref{l8} and \lemref{lem2}, then $A_\mu^\lambda$ is a $\mathbb D^\infty$-S.M.
        So we can rewrite $D_v(\mathrm d\tilde B^k)$ as:  
        \begin{equation*}
          \mathrm d\left( D_v\tilde B^k\right) = \dot h_2^k \mathrm dt 
                                                 + A_\mu^k\circ\mathrm d\tilde B^\mu
                                                 \tag{9}\label{l9}
        \end{equation*}
        where $t \mapsto \int_0^t \dot h_2\,\mathrm ds$ is a $\mathbb D^\infty$-vector
        field. Then as in definition 6.1 for $\tilde H$, we write 
        $v=\sum_{\mu=1}^nf^\mu(t)u_\mu(t,\omega)$:
        \begin{align*}
          D_v(\mathrm dp^k) &= \mathrm dv^k \\
                            &= \mathrm d\left(\sum_{\mu=1}^n f^\mu(t)Z_\mu^k\right) \\
                            &= \sum_{\mu=1}^n\dot f^\mu(t)Z_\mu^k\mathrm dt
                               + \sum_{\mu=1}^n f^\mu(t)\circ\mathrm dZ_\mu^k(t, \omega) \\
                            &= \sum_{\mu=1}^n\dot f^\mu(t)Z_\mu^k\mathrm dt
                               -f^\mu\Gamma_{ij}^kZ_\mu^j\circ\mathrm dp^i\tag{10}\label{l10}
        \end{align*}
        But
        \begin{align*}
          f^\mu Z_\mu^j & = f^\mu(\theta_*u_\mu(t,\omega))^j \\
                        & =\left(\theta_*\left(f^\mu u_\mu(t,\omega)\right)\right)^j  \\
                        & = v^j(t,\omega)
        \end{align*}
        so
        \begin{align*}
          D_v(\mathrm dp^k) & = \dot f^\mu(t)Z_\mu^k\mathrm dt 
                              -\Gamma_{ij}^kv^jZ_\mu^i\circ\mathrm d\tilde B^\mu\\
                            & =D_v(Z_\mu^k\circ\mathrm d\tilde B^\mu)
        \end{align*}
        with \eqref{l10}, we get then:
        \begin{align*}
          \dot f^\mu Z_\mu^k\mathrm dt
          - \Gamma_{ij}^kv^jZ_\mu^i\circ\mathrm d\tilde B^\mu
          & = D_v(Z_\mu^k)\circ\mathrm d\tilde B^\mu+Z_\mu^k\circ D_v(\mathrm d\tilde B^\mu) \\
          & = D_v(Z_\mu^k)\circ\mathrm d\tilde B^\mu + Z_\mu^k\dot h_2^\mu\mathrm dt+Z_\mu^kA_\lambda^\mu\circ\mathrm d\tilde B^\lambda
        \end{align*}
        Using \eqref{l8} for the identification of the terms with $\mathrm dt$, in the above
        equation, brings
        \begin{equation*}
          \boxed{h_2(t) = \int_0^t\dot f(s)\,\mathrm ds}\tag{11}\label{l11}
        \end{equation*}
        To determine $A_\mu^k$ we write
        \begin{equation*}
          \mathrm dA_\mu^k(t,\omega) = a_{1,\mu}^k\mathrm dt 
                                       + b_{\mu, \rho}^k.\mathrm d\tilde B^\rho
                                       \tag{12}\label{l12}
        \end{equation*}
        where $a_{1,\mu}^k$ and $b_{\mu,\rho}^k$ are the components of
        $(n+1)$ $n \times n$-matrices. 
  
        So we differentiate \eqref{l8}: 
        \begin{equation*}
          Z_\lambda^kA_\mu^k=-D_v(Z_\mu^k)-\Gamma_{ij}^kv^jZ_\mu^i
        \end{equation*}
        and report \eqref{l12} in \eqref{l8}. The only time that the expression 
        $b_{\mu,\rho}^k.\mathrm d\tilde B^\rho$ will appear, after differentiation
        of both members of \eqref{l8} will come on the right side; all other terms
        of \eqref{l8} after differentiation, will bring either terms in
        $\mathrm d\tilde B$ or $\mathrm dt$, for which the coefficients are
        $\mathbb D^\infty$-S.M. After identification of the terms in
        $\mathrm d\tilde B$, we see that $b_{\mu,\rho}^k$ is a
        $\mathbb D^\infty$-S.M.
  
        So we can rewrite \eqref{l12} as
        \begin{equation*}
          \mathrm dA_\mu^k(t,\omega) = a_{2,\mu}^k\mathrm dt
                                       + b_{\mu,\rho}^k\circ\mathrm d\tilde B^\rho
                                       \tag{13}\label{l13}
        \end{equation*}
        To make $a_{2,\mu}^k$ and $b_{\mu,\rho}^k$ explicit, we differentiate
        \eqref{l8} and report \eqref{l13} in $\mathrm d\left(Z_\lambda^kA_\mu^\lambda\right)$:  
        \begin{align*}
          \mathrm d\left(Z_\lambda^kA_\mu^\lambda\right) 
            & = A_\mu^\lambda\circ\mathrm dZ_\lambda^k 
                + Z_\lambda^k\circ\mathrm dA_\mu^\lambda \\
            & = -A_\mu^\lambda\Gamma_{ij}^kZ_\lambda^j\circ\mathrm dp^i
                + Z_\lambda^ka_{2,\mu}^\lambda\mathrm dt
                + Z_\lambda^kb_{\mu,\rho}^\lambda\circ\mathrm d\tilde B^\rho
                \tag{14}\label{l14}
        \end{align*}
        \begin{align*}
          \mathrm d\left(\Gamma_{ij}^kv^jZ_\mu^i\right) 
            & = v^jZ_\mu^i\circ\mathrm d\Gamma_{ij}^k
                + \Gamma_{ij}^kZ_\mu^i\circ\mathrm dv^j
                + \Gamma_{ij}^kv^j\circ\mathrm dZ_\mu^i \\
            & = v^jZ_\mu^i\left(\sum_{\rho=1}^n\frac{\partial\Gamma_{ij}^k}{\partial x^\rho}\circ\mathrm dp^\rho\right)
                + \Gamma_{ij}Z_\mu^i\dot f^\alpha Z_\alpha^j \mathrm dt \\
            & \qquad - \Gamma_{ij}^kZ_\mu^i\Gamma_{\rho\alpha}^jv^\alpha\circ\mathrm dp^\rho 
                - \Gamma_{ij}^kv^j\Gamma_{\rho\ell}^iZ_\mu^\ell\circ\mathrm dp^\rho\tag{15}\label{l15}
        \end{align*}
        \begin{align*}
          \mathrm d\left(D_v\left(Z_\mu^k\right)\right) 
            & = D_v(\mathrm dZ_\mu^k) 
              = -D_v\left[\Gamma_{ij}^kZ_\mu^j\circ\mathrm dp^i\right] \\
            & = -(D_v\Gamma_{ij}^k)Z_\mu^j\circ\mathrm dp^i-\Gamma_{ij}^k(D_vZ_\mu^j)\circ\mathrm dp^i\\
            & \qquad -\Gamma_{ij}^kZ_\mu^j\circ D_v(\mathrm dp^i)\tag{16}\label{l16}
        \end{align*}
        But
        \begin{equation*}
          D_v(\Gamma_{ij}^k) = \sum_{\rho=1}^n\frac{\partial\Gamma_{ij}^k}{\partial x^\rho}v^\rho
        \end{equation*}
        and \eqref{l8}:
        \begin{equation*}
          D_v(Z_\mu^j)=-\Gamma_{sr}^jv^rZ_\mu^s - Z_\lambda^jA_\mu^\lambda
        \end{equation*}
        so
        \begin{align*}
          \mathrm d\left(D_v(Z_\mu^k)\right) 
            &= - (D_v\Gamma_{\rho i}^k)Z_\mu^i\circ\mathrm dp^\rho 
               - \Gamma_{\rho i}^k(D_vZ_\mu^i)\circ\mathrm dp^\rho 
               - \Gamma_{\rho i}^kZ_\mu^i\circ D_v(\mathrm dp^\rho)\\
            &= - \sum_{r=1}^n\frac{\partial \Gamma_{\rho i}^k}{\partial x^r}v^rZ_\mu^i\circ\mathrm dp^\rho 
               + \Gamma_{\rho i}^k\left( \Gamma_{sr}^i v^rZ_\mu^s 
               + Z_\lambda^iA_\mu^\lambda\right)\circ\mathrm dp^\rho \\
            & \qquad-\Gamma_{ij}^kZ_\mu^j\left(\dot{f}^rZ_r^i\mathrm dt-\Gamma_{rs}^iv^s\circ\mathrm dp^r\right)
              \tag{17}\label{l17}
        \end{align*}
  
        Reporting \eqref{l14}, \eqref{l15}, \eqref{l17} in both differentiated sides of \eqref{l8} and identifying the
        terms in $\mathrm dt$, we get:
        \begin{equation*}
          \boxed{a_{2,\mu}^\nu = \dot f^\alpha T_{sr}^kZ_\mu^rZ_\alpha^s\left(Z^{-1}\right)_k^\nu}
          \tag{18}\label{l18}
        \end{equation*}
        So $a_{2,\mu}^\nu$ is intrinsically defined:
        \begin{equation*}
          a_{2,\mu}^\nu = g\left(T\left(\dot f^\alpha u_\alpha,u_\mu\right), u_\nu\right)
          \tag{18'}\label{l18p}
        \end{equation*}
        To evaluate $b_{\mu,\rho}^\lambda$ we report \eqref{l14}, \eqref{l15}, \eqref{l17}, in both differentiated sides of \eqref{l8}
        and we identify the terms in $\mathrm d\tilde B^\rho$ we get
        \begin{align*}
              - A_\mu^\lambda\Gamma_{ij}^kZ_\lambda^jZ_\rho^i\circ\mathrm d\tilde B^\rho 
               + & Z_\lambda^k b_{\mu,\rho}^\lambda\circ\mathrm d\tilde B^\rho \\
          = & - v^jZ_\mu^i\frac{\partial \Gamma_{ij}^k}{\partial x^r}Z_\rho^r\circ\mathrm d\tilde B^\rho 
              + \Gamma_{ij}^kZ_\mu^i\Gamma_{r\alpha}^jv^\alpha Z_\rho^r\circ\mathrm d\tilde B^\rho \\
            & + \Gamma_{ij}^kv^j\Gamma_{rs}^iZ_\mu^sZ_\rho^r\circ\mathrm d\tilde B^\rho 
              + \frac{\partial \Gamma_{ri}^k}{\partial x^s}v^sZ_\mu^iZ_\rho^r\circ\mathrm d\tilde B^\rho \\
            & - \Gamma_{rj}^k\left(\Gamma_{sn}^jv^nZ_\mu^s+Z_\lambda^jA_\mu^\lambda\right)Z_\rho^r\circ\mathrm d\tilde B^\rho\\
            & - \Gamma_{ij}^kZ_\mu^j\Gamma_{rs}^iv^sZ_\rho^r\circ\mathrm d\tilde B^\rho
        \end{align*}
        which after simplification and rewriting some indices:
        \begin{align*}
          v^sZ_\mu^iZ_\rho^r\left[\frac{\partial\Gamma_{ri}^k}{\partial x^s} 
            - \frac{\partial\Gamma_{is}^k}{\partial x^r}\right.
            & + \Gamma_{in}^k\Gamma_{rs}^n
            + \Gamma_{ns}^k\Gamma_{ri}^n\\
            & \left.- \Gamma_{rn}^k\Gamma_{is}^n  
            - \Gamma_{ni}^k\Gamma_{rs}^n\right] 
          = Z_\lambda^kb_{\mu,\rho}^\lambda
            \tag{19}\label{l19}
        \end{align*}
        We can rewrite \eqref{l19}:
        \begin{align*}
          v^sZ_\mu^iZ_\rho^r\left[
            \frac{\partial\Gamma_{ri}^k}{\partial x^s}
            - \frac{\partial\Gamma_{si}^k}{\partial x^r} \right.
            & + \Gamma_{ri}^n\Gamma_{sn}^k 
            + \Gamma_{ri}^nT_{ns}^k
            - \Gamma_{rn}^kT_{is}^n 
            - \Gamma_{rn}^k\Gamma_{si}^n \\
            & \left.+ \Gamma_{rs}^nT_{in}^k
            + \frac{\partial\Gamma_{si}^k}{\partial x^r} 
            - \frac{\partial\Gamma_{is}^k}{\partial x^r} 
          \right] 
          = Z_\lambda^kb_{\mu,\rho}^\lambda
        \end{align*}
        From 
        \begin{equation*}
          R_{sri}^k = \frac{\partial\Gamma_{ri}^k}{\partial x^s}
                      - \frac{\partial\Gamma_{si}^k}{\partial x^r} 
                      + \Gamma_{ri}^n\Gamma_{sn}^k 
                      - \Gamma_{rn}^k\Gamma_{si}^n
                      \tag{20}\label{l20}
        \end{equation*}
        and
        \begin{equation*}
          \nabla_rT_{si}^k = \partial_rT_{si}^k 
                             - \Gamma_{ri}^nT_{sn}^k 
                             - \Gamma_{rs}^nT_{ni}^k
                             + \Gamma_{rn}^k\Gamma_{si}^n
                               \tag{21}\label{l21}
        \end{equation*}
        we get:
        \begin{equation*}
          \boxed{Z_\lambda^kb_{\mu,\rho}^\lambda = \left(R_{sri}^k
                                                   + \nabla_rT_{si}^k\right)Z_\mu^iZ_\rho^rv^s}
                                                   \tag{22}\label{l22}
        \end{equation*}
        which can be written
        \begin{equation*}
          b_{\mu,\rho}^\alpha = g\left[ R(v, u_\rho,u_\mu), u_\alpha \right] 
                                + g\left( \left(\nabla_{u_\rho}T\right)(v, u_\mu), u_\alpha \right)
                                \tag{22'}\label{l22p}
        \end{equation*}
        Now we want to prove that $t\mapsto\int_0^t\dot h_1(s)\,\mathrm ds$
        is a $\mathbb D^\infty(\Omega, H)$-vector field. From \eqref{l6}, \eqref{l9} and
        \eqref{l12}, we have
        \begin{equation*}
          \boxed{\dot h_1^k = \dot h_2^k 
                              + \frac12 \sum_{\mu=1}^n b_{\mu,\mu}^k}
                              \tag{22''}\label{l22pp}
        \end{equation*}
        From \eqref{l22}, we have:
        \begin{equation*}
          b_{\mu,\mu}^k = \left(R_{jri}^\alpha
                          + \nabla_rT_{ji}^\alpha\right)Z_\mu^iZ_\mu^r(Z^{-1})_\alpha^kv^j
        \end{equation*}
        and $v=\sum_{\nu=1}^nf^\nu(t)u_\nu(t,\omega)$, so:  
        \begin{equation*}
          b_{\mu,\mu}^k = \left(R_{jri}^\alpha
                          + \nabla_rT_{ji}^\alpha\right)Z_\mu^iZ_\mu^r\left(Z^{-1}\right)_\alpha^kf^\nu(t)Z_\nu^j
        \end{equation*}
        As $f^\nu \in H$, $\sup_{t \in [0,1]} \left| f^\nu(t) \right|$ is 
        bounded, and each other element in the $b_{\mu,\mu}^k$'s formula
        is either a component of a SPT vector, or of a $C^\infty$ function
        of such components, on a compact manifold, so all these elements are
        $\mathbb D^\infty$-bounded, and $t\mapsto \int_0^tb_{\mu,\mu}^k\,\mathrm ds$
        is a $\mathbb D^\infty$-vector field.
        Now $h_1: t \mapsto \int_0^t\dot h_1(s)\,\mathrm ds$ is a 
        $\mathbb D^\infty$-vector field and $D_v - h_1$ is a $\mathbb D^\infty$-continuous
        derivation such that
        \begin{equation*}
          \mathrm d\left(D_v\tilde B^k - h_1^k\right) = A_\mu^k.\mathrm d\tilde B^\mu\tag{6}
        \end{equation*}
        But by Clark-Ocone, if $\alpha\in\mathbb D^\infty(\mathcal W)$ there exists
        $\tilde \alpha_\mu$ such that 
        \begin{equation*}
          \alpha = \text{constant} + \int_0^1 \tilde \alpha_\mu\cdot\mathrm d\tilde B^\mu
        \end{equation*}
        and 
        \begin{align*}
          (D_v-h_1)(\alpha) 
          & = \int_0^1\left((D_v-h1)\tilde \alpha_\mu\right)\cdot\mathrm d\tilde B^\mu 
              + \int_0^1\tilde \alpha_\mu(D_v-h_1)\cdot\mathrm d\tilde B^\mu \\
          & = \int_0^1(D_v-h_1)\tilde\alpha_\mu\cdot\mathrm d\tilde B^\mu 
              + \int_0^1\tilde\alpha_\mu A_\nu^\mu\cdot\mathrm d\tilde B^\nu
        \end{align*}
        so
        \begin{equation*}
          E\left[(D_v-h_1)\alpha\right] = 0 \Rightarrow \mathrm{div}\,(D_v-h_1) = 0
        \end{equation*}

        Now $D_v-h_1$ is a derivation, adapted and with a null divergence;  
        from Theorem IV, 1, we deduced that there exists a $n\times n$ antisymmetric
        matrix $\tilde A_\mu^k$ which as a process, is an adapted
        multiplicator such that: 
        \begin{equation*}
          D_v-h_1 = \mathrm{div}\,\tilde A\,\mathrm{grad}.
        \end{equation*}

        The fundamental isometry show that $\tilde A_\mu^k = A_\mu^k$ p.a.s.
        So $A_\mu^k$ is an antisymmetrical $n \times n$ matrix. From \eqref{l13}, we
        deduce that $a_{2,\mu}^k$ is antisymmetrical.
        Then $a_{2,\mu}^kg\left(u_k(t,\omega),u_\mu(t,\omega)\right) = 0$,
        $u_k(t,\omega)$ and $u_k(t,\omega)$ being the SPT of the vectors
        $ue_k$, $ue_\mu$ ($e_k$ and $e_\mu$ being canonical basis vectors
        of $\mathbb R^n$).

        So $g(a_{2,\mu}^ku_k, u_\mu) = 0$, which implies $g_{ij}(a_{2,\mu}^kZ_k^i,Z_\mu^j) = 0$.
        With \eqref{l18}, we have:  
        \begin{equation*}
          g_{ij}\left(T\left(u_\mu(\omega,t),\dot f^\alpha u_\alpha(\omega, t)\right)^i, Z_\mu^j\right) = 0
        \end{equation*}
        so
        \begin{equation*}
          g\left(T\left(u_\mu(\omega,t),\dot f^\alpha u_\alpha(\omega, t)\right), u_\mu(\omega, t)\right) = 0
        \end{equation*}
        and $g(T(X_1,X_2),X_1) = 0$, which is the Driver condition \eqref{l3}.
    \end{proof}

\subsection{Reapracally, if the Driver condition is fullfilled, then any $D_v$-type of derivation has an unique $\mathbb{D}^\infty$-derivation extension}	
	
  Now we want to show that the Driver condition \eqref{l3} is a sufficient condition.

  So given $v\in\tilde H$, $v = f^\nu(t)u_\nu(t,\omega)$ we have
  an operator $D_v$ acting on functions $F_{f,t}(\omega)$ such that
  \begin{equation*}
    \left(D_vF_{f,t}\right)(\omega)=(v(\omega(t))\cdot f)_{I(\omega)(t)}\tag{2}
  \end{equation*}
  and we define
  \begin{equation*}
    \tilde D_v = \mathrm{div}\,A_v\,\mathrm{grad} 
                 + \left(t \mapsto \int_0^t \dot h_1(s)\,\mathrm ds\right)
                 \tag{23}\label{l23}
  \end{equation*}
  where 
  \begin{equation*}
    (A_v)_\mu^k(\omega,t) = \int_0^ta_{1,\mu}^k(v)\,\mathrm ds 
                            + \int_0^tb_{\mu,\rho}^k\cdot\mathrm d\tilde B^\rho
                            \tag{24}\label{l24}
  \end{equation*}
  with
  \begin{equation*}
    b_{\mu,\rho}^k(v) = \left(R_{jri}^\alpha 
                        + \nabla_rT_{ji}^\alpha\right)Z_\mu^iZ_\rho^r(Z^{-1})_\alpha^kv^j
                        \tag{25}\label{l25}
  \end{equation*}
  and 
  \begin{equation*}
    a_{2,\mu}^k(v) = \dot f^\alpha T_{sr}^\nu Z_\mu^rZ_\alpha^s(Z^{-1})_\nu^k = g(T(\dot f_{u_\alpha}^\alpha,u_\mu),u_\nu)
                     \tag{26}\label{l26}
  \end{equation*}
  and
  \begin{equation*}
    \dot h_1^k = \dot f^k + \frac12\sum_{\mu=1}^nb_{\mu,\mu}^k 
               = \dot h_2 + \dot h_3 
               \tag{27}\label{l27}
  \end{equation*}
  with
  \begin{equation*}
    \dot h_3 = \frac12 \sum_{\mu=1}^nb_{\mu,\mu}^k
    \tag{27'}\label{l27p}
  \end{equation*}

  We have to prove that:
  \begin{itemize}
    \item $A_v$ is an antisymmetrical, adapted, matrix;  
    \item $A_v$ is a multiplicator;  
    \item the operator $D_v$ and $\mathrm{div}\,A\,\mathrm{grad} + \left( t \mapsto \int_0^t \dot h_1(s)\,\mathrm ds\right)$ coincide on $p^k$.
  \end{itemize}
  \eqref{l24} can be rewritten:  
  \begin{equation*}
    A(v)(t) = \int_0^ta_1(v)\,\mathrm ds + \int_0^tb(v)_\rho\cdot\mathrm d\tilde B^\rho
    \tag{24'}\label{l24p}
  \end{equation*}
  From the intrinsic formulations of $a_{1,\mu}^k$, $a_{2,\mu}^k$ and $b_\mu^k$,
  and with the Driver condiction \eqref{l3}, we get that $A_v$ in \eqref{l23} and \eqref{l24} is 
  indeed antisymmetric.

  We are going to show, for example, that $g\left((\nabla_{u_\rho}T)(v,u_\mu), u_\alpha\right)$ 
  is antisymmetric in $(\mu, \alpha)$. For this, it is enough to prove that if $X,Y,Z$ are vector
  fields on $(V_n,g)$, $g\left((\nabla_ZT)(X,Y),X\right)=0$. 

  From $g(T(X,Y),X) = 0$, we deduce $g(T(U,Y),V) = -g(T(V,Y),U)$. Then
  \begin{align*}
    g\left((\nabla_ZT)(X,Y),X\right) 
      & = g\left(\nabla_Z(T(X,Y)),X\right) 
        - g\left(T(\nabla_ZX,Y),X\right)
        - g\left(T(X,\nabla_ZY),X\right) \\
      & \text{The last term is zero by \eqref{l3} (Driver condition), and: } \\
      & = g\left(\nabla_Z(T(X,Y)),X\right)
        + g\left(T(X,Y),\nabla_ZX\right) \\
      & = Z\cdot g(T(X,Y),X) 
        = 0
  \end{align*}
  Now we have to show that:
  \begin{equation*}
    (A_v)_\mu^k(t,\omega)=\int_0^t a_{1,\mu}^k(v)\,\mathrm ds+\int_0^tb_{\mu,\rho}^k\cdot\mathrm d\tilde B^\rho
    \tag{24}
  \end{equation*}
  is a multiplicator.
  \begin{equation*}
    b_{\mu,\rho}^k = g\left(R(v,u_\rho,u_\mu),u_k\right)
                     + g\left((\nabla_{u\rho}T)(v,u_\mu),u_k\right)
                     \tag{25}
  \end{equation*}
  and $t\mapsto \int_0^t  b_{\mu,\rho}^k\cdot\mathrm d\tilde B^\rho$ is an It\^o stochastic
  integral of a $\mathbb D^\infty$-bounded process so is a $\frac12$-H\"olderian
  $\mathbb D^\infty$ process, so is a multiplicator. Then
  \begin{equation*}
    a_{1,\mu}^\nu = \dot f^\alpha g(T(u_\alpha,u_\mu),u_\nu) + \frac12 b_{\mu,\mu}^\nu
    \tag{26}
  \end{equation*}
  $\frac12b_{\mu,\mu}^\nu$ is $\mathbb D^\infty$-bounded so 
  $t\mapsto\frac12\int_0^t b_{\mu,\mu}^\nu\,\mathrm ds$ is a $\mathbb D^\infty(\omega, H)$ vector field, so the
 process $t\mapsto\frac12\int_0^t b_{\mu,\mu}^\nu\,\mathrm{ds}$ is a $\mathbb D^\infty$-multiplicator.

  Then
  \begin{align*}
    \mathrm{grad}^j\,& \left\{\int_0^t\dot f^\alpha(s)g(T(u_\alpha,u_\mu),u_\nu)\,\mathrm ds\right\} \\
      & = \int_0^t \dot f^\alpha(s)\mathrm{grad}^j \left\{g(T(u_\alpha,u_\mu),u_\nu)\right\}\,\mathrm ds \\
      & \leq\left(\int_0^1\left|\dot f^\alpha(s)\right|^2\,\mathrm ds\right)^{\frac12}
            \left(\int_0^1\left\|\mathrm{grad}^j\left\{g(T(u_\alpha,u_\mu),u_\nu)\right\}\right\|_{\otimes^j H}\,\mathrm ds\right)^{\frac12}
  \end{align*}
  so with criterion IV, 2, we see that
  \begin{equation*}
    t\mapsto\int_0^t\dot f^\alpha(s)g(T(u_\alpha,u_\mu),u_\nu)\,\mathrm ds
  \end{equation*}
  is a multiplicator. 

  Last: straightforward computation shows that 
  $\left(t\mapsto\int_0^t\frac12\sum_{\mu=1}^nb_{\mu,\mu}^k\,\mathrm ds\right)$ is a 
  $\mathbb D^\infty$-vector field.
  Now we will show that $(D_v-v)p^\ell = 0$, with $v\in \tilde H$ and 
  $v = \sum_{\mu=1}^n f^\mu(t)u_\mu(t,\omega)$. With the same notations as
  before, and with a Stratonovitch integration by parts, we have:
  \begin{align*}
    v(t,\omega)\cdot p^\rho 
      & = v^\rho(t,\omega) 
        = \int_0^t\dot f^\mu(s)u_\mu^\rho(s,\omega)\,\mathrm ds
          + \int_0^tf^\mu(s)\circ\mathrm dZ_\mu^\rho\\
    (D_v-v)p^\ell 
      & = D_v\left(\int_0^tZ_\mu^\ell\circ\mathrm d\tilde B^\mu\right)
        + \int_0^t f^\mu(s)\Gamma_{jk}^\ell Z_\mu^kZ_\rho^j\circ\mathrm d\tilde B^\rho 
        - \int_0^t\dot f(s)^\mu Z_\mu^\ell\,\mathrm ds
  \end{align*}
  With \eqref{l9}:
  \begin{align*}
    & = \int_0^t D_vZ_\rho^\ell\circ\mathrm d\tilde B^\rho 
      + \int_0^t\dot h_2^\mu Z_\mu^\ell\,\mathrm ds 
      + \int_0^t Z_\mu^\ell A_\rho^\mu\circ\mathrm d\tilde B^\rho \\
    & \qquad + \int_0^tf^\mu\Gamma_{jk}^\ell Z_\mu^kZ_\rho^j\circ\mathrm d\tilde B^\rho 
      - \int_0^t \dot f^\mu(s)Z_\mu^\ell\,\mathrm ds\qquad\text{but $\dot h_2^\mu = \dot f^\mu$} \\
    & = \int_0^t \widehat R_\rho^\ell\circ\mathrm d\tilde B^\rho
    \tag{27}
  \end{align*}
  with
  \begin{equation*}
    \widehat R_\rho^\ell = D_vZ_\rho^\ell+Z_\mu^\ell A_\rho^\mu+f^\mu\Gamma_{jk}^\ell Z_\mu^kZ_\rho^j.
  \end{equation*}
  So
  \begin{equation*}
    \widehat R_\rho^\ell = D_vZ_\rho^\ell 
                           + Z_\mu^\ell A_\rho^\mu 
                           + \Gamma_{jk}^\ell v^kZ_\rho^j
                           \tag{28}\label{l28}
  \end{equation*}
  Using
  \begin{equation*}
    Z_\rho^\ell = -\int_0^t \Gamma_{\mu\nu}^\ell Z_\rho^\nu Z_\lambda^\mu\circ\mathrm d\tilde B^\lambda
  \end{equation*}
  and \eqref{l13}: 
  \begin{equation*}
    \mathrm dA_\mu^k = a_{2,\mu}^k\mathrm dt + b_{\mu,\rho}^k\circ\mathrm d\tilde B^\rho
  \end{equation*} 
  we have
  \begin{align*}
    \widehat R_\rho^\ell = 
      & -D_v\left(\int_0^t \Gamma_{\mu\nu}^\ell Z_\rho^\nu Z_\alpha^\mu\circ\mathrm d\tilde B^\alpha\right) \\
      & - \int_0^t A_\rho^\mu\Gamma_{ij}^\ell Z_\mu^jZ_\alpha^i\circ\mathrm d\tilde B^\alpha 
        + \int_0^tZ_\mu^\ell a_{2, \rho}^{\mu}\,\mathrm ds 
        + \int_0^tZ_\mu^\ell b_{\rho,\alpha}^\mu\circ\mathrm d\tilde B^\alpha \\
      & + \int_0^t \frac{\partial \Gamma_{jk}^\ell}{\partial x^\beta}v^{k}Z_\rho^jZ_\alpha^\beta\circ\mathrm d\tilde B^\alpha 
        - \int_0^t \Gamma_{jk}^\ell v^k\Gamma_{\mu\nu}^jZ_\rho^\nu Z_\alpha^\mu\circ\mathrm d\tilde B^\alpha \\
      & + \int_0^t \Gamma_{jk}^\ell Z_\rho^j\dot f^\mu Z_\mu^k\,\mathrm ds
        - \int_0^t\Gamma_{jk}^\ell Z_\rho^j\Gamma_{\lambda\beta}^kv^\beta Z_\alpha^\lambda\circ\mathrm d\tilde B^\alpha
        \tag{29}\label{l29}
  \end{align*}
  using \eqref{l9}:
  \begin{equation*}
    \mathrm d(D_v(\tilde B^\alpha)) = \dot h_2^\alpha \mathrm dt 
                                      + A_\mu^\alpha \circ \mathrm d\tilde B^\mu
  \end{equation*}
  we can rewrite:  
  \begin{align*}
    D_v\left(\int_0^t\Gamma_{\mu\nu}^\ell Z_\rho^\nu Z_\alpha^\mu\circ\mathrm d\tilde B^\alpha\right) = 
      & \int_0^t \left((D_v -v)\cdot\Gamma_{\mu\nu}^\ell\right)Z_\rho^\nu Z_\alpha^\mu\circ\mathrm d\tilde B^\alpha \\
      & + \int_0^t(v\cdot\Gamma_{\mu\nu}^\ell)Z_\rho^\nu Z_\alpha^\mu\circ\mathrm d\tilde B^\alpha \\
      & + \int_0^t\Gamma_{\mu\nu}^\ell D_v\left(Z_\rho^\nu Z_\alpha^\mu\circ\mathrm d\tilde B^\alpha\right) \\
    = & \int_0^t \frac{\partial \Gamma_{\mu\nu}^\ell}{\partial x^\beta}\left((D_v-v)\cdot p^\beta\right)\cdot Z_\rho^\nu Z_\alpha^\mu \circ\mathrm d\tilde B^\alpha \\
      & +\int_0^t\frac{\partial\Gamma_{\mu\nu}^\ell}{\partial x^\beta}Z_\rho^\nu Z_\alpha^\mu v^\beta\circ\mathrm d\tilde B^\alpha \\
      & + \int_0^t \Gamma_{\mu\nu}^\ell (D_vZ_\rho^\nu)Z_\alpha^\mu\circ\mathrm d\tilde B^\alpha \\
      & + \int_0^t \Gamma_{\mu\nu}^\ell Z_\rho^\nu(D_vZ_\alpha^\mu)\circ\mathrm d\tilde B^\alpha \\
      & + \int_0^t \Gamma_{\mu\nu}^\ell Z_\rho^\nu Z_\alpha^\mu \dot f^\alpha\mathrm ds \\
      & + \int_0^t \Gamma_{\mu\nu}^\ell Z_\rho^\nu Z_\lambda^\mu A_\alpha^\lambda\circ\mathrm d\tilde B^\alpha
      \tag{30}\label{l30}
  \end{align*}
  Reporting \eqref{l30} in \eqref{l29}
  \begin{align*}
    \widehat R_\rho^\ell = 
      & - \int_0^t\frac{\partial\Gamma_{\mu\nu}^\ell}{\partial x^\beta}\left((D_v-v)\cdot p^\beta\right)Z_\rho^\nu Z_\alpha^\mu\circ\mathrm d\tilde B^\alpha
        - \int_0^t\frac{\partial\Gamma_{\mu\nu}^\ell}{\partial x^\beta}Z_\rho^\nu Z_\alpha^\mu v^\beta\circ\mathrm d\tilde B^\alpha \\
      & - \int_0^t \Gamma_{\mu\nu}^\ell(D_vZ_\rho^\nu)Z_\alpha^\mu\circ\mathrm d\tilde B^\alpha
        - \int_0^t\Gamma_{\mu\nu}^\ell Z_\rho^\nu(D_vZ_\alpha^\mu)\circ\mathrm d\tilde B^\alpha \\
      & - \int_0^t \Gamma_{\mu\nu}^\ell Z_\rho^\nu Z_\alpha^\mu\dot f^\alpha\,\mathrm ds 
        - \int_0^t\Gamma_{\mu\nu}^\ell Z_\rho^\nu Z_\lambda^\mu A_\alpha^\lambda\circ\mathrm d\tilde B^\alpha \\
      & - \int_0^t A_\rho^\mu\Gamma_{ij}^\ell Z_\mu^j Z_\alpha^i\circ\mathrm d\tilde B^\alpha 
        + \int_0^tZ_\mu^\ell a_{2,\rho}^\mu\,\mathrm ds 
        + \int_0^tZ_\mu^\ell b_{\rho,\alpha}^\mu\circ\mathrm d\tilde B^\alpha \\
      & + \int_0^t \frac{\partial\Gamma_{jk}^\ell}{\partial x^\beta}v^kZ_\rho^jZ_\alpha^\beta\circ\mathrm d\tilde B^\alpha 
        - \int_0^t\Gamma_{jk}^\ell v^k\Gamma_{\mu\nu}^jZ_\rho^\nu Z_\alpha^\mu\circ\mathrm d\tilde B^\alpha \\
      & + \int_0^t \Gamma_{jk}^\ell Z_\rho^j\dot f^\mu Z_\mu^k\,\mathrm ds 
        - \int_0^t\Gamma_{jk}^\ell Z_\rho^j\Gamma_{\lambda\beta}^kv^\beta Z_\alpha^\lambda\circ\mathrm d\tilde B^\alpha
        \tag{31}\label{l31}
  \end{align*}
  In this latest equation: the terms Nr.\ 5 + Nr.\ 8 + Nr.\ 12 = 0 (see \eqref{l18}).
  \begin{itemize}
    \item The terms
      \begin{align*}
        \text{Nr.\ 3 + Nr.\ 7} & =-\int_0^t\Gamma_{ij}^\ell Z_\alpha^i (D_vZ_\rho^j+A_\rho^\mu Z_\mu^j)\circ\mathrm d\tilde B^\alpha \\
          & = -\int_0^t\Gamma_{ij}^\ell Z_\alpha^i(\widehat R_\rho^j-\Gamma_{rk}^jv^kZ_\rho^r)\circ\mathrm d\tilde B^\alpha\qquad\text{(with \eqref{l28})}
      \end{align*}
    \item Same for Nr.\ 4 + Nr.\ 6:
      \begin{align*}
        - \int_0^t\Gamma_{\mu\nu}^\ell Z_\rho^\nu(D_vZ_\alpha^\mu+Z_\lambda^\mu A_\alpha^\lambda)\circ\mathrm d\tilde B^\alpha 
        = -\int_0^t\Gamma_{\mu\nu}^\ell Z_\rho^\nu(\widehat R_\alpha^\mu-\Gamma_{jk}^\mu v^kZ_\alpha^j)\circ\mathrm d\tilde B^\alpha
      \end{align*}
    \item Nr.\ 9:
      \begin{align*}
        + \int_0^tZ_\mu^\ell b_{\rho,\alpha}^\mu\circ\mathrm d\tilde B^\alpha 
        = \int_0^t\left(R_{sri}^\ell+\nabla_rT_{si}^\ell\right)Z_\rho^iZ_\alpha^rv^s\circ\mathrm d\tilde B^\alpha
      \end{align*}
    \item Nr.\ 2 + Nr.\ 10:
    \begin{align*}
      - \int_0^t\frac{\partial\Gamma_{ri}^\ell}{\partial x^s}Z_\rho^iZ_\alpha^rv^s\circ\mathrm d\tilde B^\alpha 
      + \int_0^t\frac{\partial\Gamma_{is}^\ell}{\partial x^r}v^sZ_\rho^iZ_\alpha^r\circ\mathrm d\tilde B^\alpha
    \end{align*}
  \end{itemize}
  The sum of all derivatives of the Christoffel symbols of the terms Nr.\ 9 + Nr.\ 2 + Nr.\ 10 is zero.

  Now we collect all the terms with double products of Christoffel symbols: they come 
  from terms Nr.\ 11, 13, from Nr.\ 3 + Nr.\ 7, Nr.\ 4 + Nr.\ 6, and from the unused
  parts of $R_{sri}^\ell + \nabla_rT_{si}^\ell$ in Nr.\ 9: we get ($\lambda$: summation index)
  \begin{align*}
    \int_0^t+Z_\rho^\nu Z_\alpha^\mu v^k (
      & - \Gamma_{jk}^\ell\Gamma_{\mu\nu}^j
        - \Gamma_{\nu\lambda}^\ell\Gamma_{\mu k}^\lambda 
        + \Gamma_{\mu\nu}^\lambda\Gamma_{k\lambda}^\ell 
        - \Gamma_{\mu\lambda}^\ell\Gamma_{k\nu}^\lambda \\
      & -\Gamma_{\mu\nu}^\lambda T_{k\lambda}^\ell
        - \Gamma_{\mu k}^\lambda T_{\lambda\nu}^\ell 
        + \Gamma_{\mu\lambda}^\ell T_{k\nu}^\lambda 
        + \Gamma_{\mu j}^\ell\Gamma_{\nu k}^j 
        + \Gamma_{\lambda\nu}^\ell\Gamma_{\mu k}^\lambda)\circ\mathrm d\tilde B^\alpha
  \end{align*}
  which after reduction, equals 0.

  So \eqref{l31} becomes:
  \begin{equation*}
    \widehat R_\rho^\ell = 
      - \int_0^t\left\{\frac{\partial\Gamma_{\mu\nu}^\ell}{\partial x^\beta}\left((D_v-v)\cdot p^\beta\right)Z_\rho^\nu Z_\alpha^\mu 
      + \Gamma_{\mu\nu}^\ell\left( Z_\alpha^\mu\widehat R_\rho^\nu+Z_\rho^\nu\widehat R_\alpha^\mu \right) \right\}\circ\mathrm d\tilde B^\alpha
      \tag{32}\label{l32}
  \end{equation*}
  We also have: 
  \begin{equation*}
    (D_v-v)\cdot p^\ell = \int_0^t\widehat R_\alpha^\ell\circ\mathrm d\tilde B^\alpha
    \tag{27}
  \end{equation*}
  \eqref{l32} and \eqref{l27} constitute a linear system of SDE, for which the unknown variables are 
  $\widehat R_\rho^\ell$, $(D_v-v)\cdot p^\ell$ and with null initial conditions of 
  $\widehat R_\rho^\ell$ and $(D_v-v)\cdot p^\ell$. So $(D_v-v)\cdot p^\ell = 0$.

\subsection{Calculus of the $D_v$-derivation of a $k$-covariant tensor, on $(V_n, g)$}
  Let $\mathscr C$ be a $k$-covariant tensor of $(V_n,g)$ and
  $v\in\tilde H$: $v = f^\mu (t)u_\mu(t,\omega)$ where $u_\mu(t,\omega)$ is the SPT of
  $ue_\mu$, at instant $t$ ``along $\omega(t)$''.

  Let $x_i(t,\omega)$, $i = 1, \dotsc, k$, SPT of vectors $x_i\in T_{m_0}V_n$, we want
  to compute $D_v\mathscr C\left(x_1(t,\omega),\dotsc,x_k(t,\omega)\right)$,
  \begin{equation*}
    D_v\left[\mathscr C\left( x_1(t,\omega),\dotsc,x_k(t,\omega)\right)\right] 
    = D_v\left(\mathscr C_{i_1\dots i_k} x_1^{i_1}\dots x_k^{i_k} \right),
    \quad i_j = 1, \dotsc, n
  \end{equation*}
  To simplify the notations, we keep only one index $i_j$ and make the calculus only with this $i_j$:  
  \begin{align*}
    D_v\left[\mathscr C(x_{i_j})\right] 
      & = (v\cdot\mathscr C_{i_j})(x_j^{i_j})
        + \mathscr C_{i_j}\left[D_vx_j^{i_j}\right]\\
      & = (v\cdot\mathscr C_{i_j})(x_j^{i_j}) 
        + \mathscr C_{i_j}\left[ -\Gamma_{\alpha_j,\beta_j}^{i_j}v^{\beta_j}x_j^{\alpha_j}-x_{\lambda_j}^{i_j}A_j^{\lambda_j}\right]\qquad\text{with \eqref{l8}}
  \end{align*}
  $\alpha_j$, $\beta_j$ running from $1$ to $n$ and $\lambda_j$ from $1$ to $k$.
  \begin{align*}
    & = (v\cdot \mathscr C_{i_j})(x_j^{i_j})
      - \mathscr C_{i_j}\left(\Gamma_{\alpha_j,\beta_j}^{i_j}x_j^{\beta_j}v^{\alpha_j}\right) 
      - \mathscr C_{i_j}\left(T_{\alpha_j,\beta_j}^{i_j}v^{\beta_j}x_j^{\alpha_j}\right) 
      - \mathscr C_{i_j}\left(x_{\lambda_j}^{i_j}A_j^{\lambda_j}\right) \\
    & = \left(\nabla_v\mathscr C\right)(x_j) 
      - \mathscr C_{i_j}\left(T^{i_j}(x_j,v)
      + x_{\lambda_j}^{i_j}A_j^{\lambda_j}\right)
  \end{align*}
  So
  \begin{equation*}
    D_v\left[\mathscr C\left[x_1,\dotsc,x_k\right]\right]
      = \left(\nabla_v\mathscr C\right)(x_1,\dotsc,x_k) 
      - \sum_{j=1}^k\mathscr C\left(x_1,\dotsc,x_{j-1},T(x_j,v)+A_j^{\lambda_j}x_{\lambda_j}, \dotsc, x_k\right)
  \end{equation*}
  We now apply this result to a bilinear symmetrical form on $(V_n,g)$, denoted $q$,
  and compatible with the  operator $\nabla$ of the connection on $(V_n,g)$:  
  \begin{align*}
    0 & = D_vq(x_1(t,\omega),x_2(t,\omega)) \\
      & = (\nabla_vq)(x_1,x_2)
        - q(x_1,Ax_2+T(x_2,v))
        - q(Ax_1+T(x_1,v),x_2)\\
      & = (\nabla_vq)(x_1,x_2)
        - q(x_1,Ax_2)
        - q(x_1,T(x_2,v))
        - q(Ax_1,x_2)
        - q(T(x_1,v),x_2)\\
      & = q(T(x_1,v),x_2)
        - q(x_1,T(x_2,v))
  \end{align*}
  because $\nabla_vq=0$ and the antisymmetry of $A$. Which implies $q(T(x,y),x)=0$.

  We recall that $\tilde H$, the NCM (New Cameron-Martin space) is the set
  \begin{equation*}
    \left\{\sum_{\mu=1}^nf^\mu(t)u_\mu(t,\omega)\,/\,f^\mu\in H\text{ and } u_\mu(t,\omega)=ue_\mu\right\}
  \end{equation*}
  $u_\mu(t,\omega)$ is the SPT of $ue_\mu$, evaluated at instant $t\in[0,1]$ 
  ``along $\omega$''; $(e_\mu)_{\mu=1,\dotsc,n}$ is the canonical basis
  of $\mathbb R^n$ and $u$ is the isomorphism between $\mathbb R^n$ and $T_{m_0}V_n$. 

  The scalar product on $\tilde H$ is
  \begin{equation*}
    \left\langle \sum_{\mu=1}^nf^\mu(\cdot)u_\mu,\sum_{\nu=1}^ng^\mu(\cdot)u_\nu\right\rangle_{\tilde H}
    = \sum_{\mu=1}^n \int_0^t\dot f^\mu(t)\dot g^\mu(t)\,\mathrm dt
  \end{equation*}
  with this scalar product, $\tilde H$ is complete.

  There is a correspondence between $\tilde H$ and $H$: $\forall v\in\tilde H$, 
  $v = \sum f^\mu u_\mu(t,\omega)$, we associate the element of $H$:
  $(f^\mu(t))_{\mu=1,\dotsc,n}$. This correspondence is an isometry and the image
  of v is denoted $h(v)$: $h(v) = (f^\mu(t))_{\mu=1,\dotsc,n}$. As
  a basis $\widecheck B$ of $\tilde H$, we choose the following elements of $\tilde H$:  
  \begin{equation*}
    \widecheck B = \begin{cases}
      \varepsilon_{\ell,j} = t\mapsto\left(\sqrt{2}\int_0^t\cos 2\pi\ell s\cdot\mathrm ds\right)u_j(t,\omega)\\
      \varepsilon_{k,j}'   = t\mapsto\left(\sqrt{2}\int_0^t\sin 2\pi ks\cdot\mathrm ds\right)u_j(t,\omega)\\
      \varepsilon_j''      = t\mapsto\left(\int_0^t 1\cdot\mathrm ds\right)u_j(t,\omega)
    \end{cases}\tag{33}\label{l33}
  \end{equation*}
  $\widecheck B$ is a basis of $\tilde H$. We denote by $\varepsilon_\rho$ 
  (or $\varepsilon_i$) the generic basis vector of $\widecheck B$. Now we
  will define an operator denoted again $\mathrm{grad}$, but associated to the NCM.
  $\varepsilon$ being a vector of $\widecheck B$, we recall that
  \begin{equation*}
    D_\varepsilon = \mathrm{div}\,A(\varepsilon)\mathrm{grad} 
                    + \left(t\mapsto\int_0^t\dot h_1(\varepsilon)(s)\,\mathrm ds\right)
                    \tag{23}
  \end{equation*}
  with 
  \begin{equation*}
    (A(\varepsilon))_\mu^k(t,\omega) 
      = \int_0^t a_{1,\mu}^k(\varepsilon)\,\mathrm ds 
      + \int_0^tb_{\mu,\rho}(\varepsilon)\cdot\mathrm d\tilde B^\rho
      \tag{ 24}
  \end{equation*}
  and
  \begin{align*}
    b_{\mu,\rho}^k(\varepsilon) 
      & = \left(R_{jri}^\alpha 
        + \nabla_rT_{ji}^\alpha\right)Z_\mu^iZ_\rho^r(Z^{-1})_\alpha^k\varepsilon^j
        \tag{25} \\
    a_{2,\mu}^k(\varepsilon) 
      &= \dot h^\alpha(\varepsilon)(t)T_{sr}^\nu Z_\mu^rZ_\alpha^s(Z^{-1})_\nu^k
      \tag{26}
  \end{align*}
  
\subsection{Definition of the new-gradient on $\mathbb{D}^\infty$.}	  
  
  \begin{defn}
    $(v_i)_{i\in\mathbb N_*}$ being a basis of $\tilde H$, we define, for $f\in\mathbb D^\infty$
    \begin{equation*}
      {\mathrm{grad}}\,f = \sum_{i = 1}^\infty\left(D_{v_i}f\right)v_i
      \tag{34}\label{l34}
    \end{equation*}
  \end{defn}
  We have to show that this definition is legitimate and that $\mathrm{grad}\,f$ 
  is a derivation (trivial).   We will show that first,
  with the basis $\widecheck B$, the defining series in \eqref{l33} is 
  $\mathbb D^\infty$-convergent. For this it is enough to prove:
  \begin{lem}\label{lem4}
    If $f,g\in\mathbb D^\infty$, then $\sum_{\rho=1}^\infty D_{\varepsilon_\rho}f\cdot D_{\varepsilon_\rho}g$ 
    is $\mathbb D^\infty$-convergent. Nota: $\varepsilon_\rho$ is either $\varepsilon_{\ell,j}$, $\varepsilon_{k,j}'$ or 
    $\varepsilon_j''$, basis vectors of $\widecheck B$.
  \end{lem}
  To prove \lemref{lem4}, it is enough to prove that $(D_{\varepsilon_\rho}f)_{\rho\in\mathbb N_*}$ 
  is a $\mathbb D^\infty$ vector field. From \eqref{l23}, we have:
  \begin{equation*}
    D_{\varepsilon_\rho}f = \mathrm {div}\, A(\varepsilon_\rho)\mathrm{grad}\,f 
                          + \left( t\mapsto \int_0^t\dot h_1(\varepsilon_\rho)(s)\,\mathrm ds \right)
  \end{equation*}
  \begin{align*}
    \dot h_1(\varepsilon_\rho) = 
      & ~ \sqrt{2}\cos(2\pi \rho t) 
        + \frac12 \sum_{\mu=1}^nb_{\mu,\mu}\qquad \text{(33)}\\
          \text{or } 
      & ~ \sqrt{2}\sin(2\pi\rho t)
        + \frac12\sum_{\mu=1}^nb_{\mu,\mu}\qquad\text{(27)}
  \end{align*}
  with
  \begin{equation*}
    b_{\mu,\mu}(\varepsilon) = \left(R_{jri}^\alpha + \nabla_rT_{ji}^\alpha\right)Z_\mu^iZ_\mu^r(Z^{-1})_\alpha^j h(\varepsilon_\rho)
  \end{equation*}
  ($\dot h(\varepsilon_\rho) \neq \dot h_1(\varepsilon_\rho)$). 
  Now, using $\left|\int_0^t\cos 2\pi\ell s\,\mathrm ds\right| \leq \frac{C_0}{\ell}$ 
  and $\left|\int_0^t\sin 2\pi k s\,\mathrm ds\right| \leq \frac{C_0}{k}$, $C_0$
  being a constant, we see that 
  \begin{equation*}
    \left(t\mapsto\int_0^t\dot h_1(\varepsilon_\rho)(s)\,\mathrm ds\right)_{\rho \in \mathbb N_*}
  \end{equation*}
  is a $\mathbb D^\infty$-vector field.

  To be able to do the same trick with $\mathrm{div}\,A(\varepsilon_\rho)\mathrm{grad}\,f$, 
  we know that:
  \begin{equation*}
    A(\varepsilon_\rho)_\mu^k = \int_0^ta_{1,\mu}^k(\varepsilon_\rho)\,\mathrm ds
                              + \int_0^tb_{\mu,\alpha}^k(\varepsilon_\rho)\cdot\mathrm d\tilde B^\alpha
                              \tag{24}
  \end{equation*}
  with
  \begin{equation*}
    b_{\mu,\alpha}^k(\varepsilon_\rho) = \left(R_{jri} +\nabla_rT_{ji}^m\right)Z_\mu^iZ_\alpha^r(Z^{-1})_m^k\varepsilon_\rho 
  \end{equation*}
  and
  \begin{equation*}
    a_{1,\mu}^k(\varepsilon_\rho) = \dot h(\varepsilon_\rho)^\alpha T_{sr}^\nu Z_\mu^rZ_\alpha^s(Z^{-1})_\nu^k + \frac12 b_{\mu,\mu}^k
  \end{equation*}
  (Eq. 25 and 26)
  
  The only item for which the trick is not directly possible is $a_{2,\mu}^k(\varepsilon_\rho)$. 
  So we make a Stratonovitch integration by parts and we get:  
  \begin{align*}
    \int_0^ta_{2,\mu}^k(\varepsilon_\rho)\,\mathrm ds 
      & = \int_0^t\dot h(\varepsilon_\rho)^\alpha T_{sr}^\nu Z_\mu^rZ_\alpha^s(Z^{-1})_\nu^k\,\mathrm ds \\
      & = \left(\int_0^t h(\varepsilon_\rho)^\alpha\,\mathrm ds\right)
        \times \int_0^t T_{sr}^\nu Z_\mu^rZ_\alpha^s(Z^{-1})_\nu^k\circ\mathrm d\tilde B \\
      & \qquad -\int_0^t h(\varepsilon_\rho)^\alpha\circ\mathrm d\left(T_{sr}^\nu Z_\mu^rZ_\alpha^s(Z^{-1})_\nu^k\right)
      \tag{35}\label{l35}
  \end{align*}
  As $a_{1,\mu}^k$, $T_{sr}^\nu$, $Z_\mu^r$, $Z_\alpha^s$, and $(Z^{-1})_\nu^k$ are 
  $\mathbb D^\infty$-semi-martingales, we can make this integration by parts.
  Then each Stratonovitch integral in \eqref{l35} is $\alpha$-H\"olderian, 
  $\alpha < \frac12$, $\rho$-uniformly, and is multiplied by $\frac{1}{\rho}$,
  which is due to the presence of $h(\varepsilon_\rho)^\alpha$.

  The Lebesgue integral in the r.h.s. of \eqref{l35} is also a Stratonovitch 
  integral multiplied by $\frac{1}{\rho}$ which appears because of
  $\int_0^t h(\varepsilon_\rho)^\alpha(s)\,\mathrm ds$. 

  So $\left(D_{\varepsilon_\rho}f\right)_{\rho\in\mathbb N_*}$ is a $\mathbb D^\infty$-vector field.
  We also have to prove that the definition of the new $\mathrm{grad}$ does not
  depend on the basis $(v_j)_{j\in\mathbb N_*}$. For this, we
  prove first that the map $\tilde H \ni v \mapsto D_vf \in \mathbb D^\infty$ is continuous;  
  \begin{equation*}
    D_vf = \mathrm{div}\,A_v\mathrm{grad}\,f 
         + t\mapsto \int_0^t\dot h_1(v)(s)\,\mathrm ds
         \tag{23}
  \end{equation*}
  From \eqref{l27}, we see that $v \mapsto \left(t\mapsto \int_0^t \dot h_1(v)(s)\,\mathrm ds\right)$
  is continuous, when seen as a vector field. The map $v\mapsto\mathrm{div}\,A_v\mathrm{grad}\,f$
  is continuous because with \eqref{l24}, \eqref{l25}, \eqref{l26}, we see from the equations defining
  $b(v)$ and $a(v)$, so $A(v)$, that $A(v)$ is uniformly multiplicator. 

  Then if $(v_i)_{i\in\mathbb N_*}$ is a basis of $\tilde H$, and which $w\in\tilde H$, we have
  \begin{align*}
    \left\langle\sum_{i=1}^LD_{v_i}f\cdot v_i, w\right\rangle_{\tilde H} 
      & =\sum_{i=1}^L \left(D_{v_i}f\right)\langle v_i, w\rangle_{\tilde H} \\
      & = D_{\sum_{i=1}^L\langle v_i,w\rangle_{\tilde H}v_i}f
  \end{align*}
  so, when $L\uparrow\infty$, we get
  \begin{equation*}
    \left\langle{\mathrm{grad}}\,f,w\right\rangle_{\tilde H} = D_w f
  \end{equation*}
  \begin{rem}
    If $\overrightarrow{\mathrm{grad}}\,f$ denotes the gradient of a function $f\in C^\infty(V_n)$
    \begin{equation*}
      \overrightarrow{\mathrm{grad}}\,f 
      = \sum_{i=1}^n \frac{\partial f}{\partial x^i}\frac{\partial}{\partial x^i}
    \end{equation*}
    then $F_{f,t}$ being $F_{f,t}(\omega) = f(I(\omega)(t)$, $I$ being the It\^o map,
    we have
    \begin{equation*}
      \left\langle v,\mathrm{grad}\,F_{f,t}\right\rangle_{\tilde H}(\omega) 
        = \left\langle v,\overrightarrow{\mathrm{grad}}\,f\right\rangle_{V_n}(I(\omega)(t))
    \end{equation*}
  \end{rem}
  
\section{\LARGE Standard Quadratic form on $\P_{m_0}(V_n, g)$}
  
This section is dedicated to the study of the standard quadratic form on the $\dinf$-stochastic manifold $\P_{m_0}(V_n, g)$.

We recall the theorem of Dini-Lipshitz:

If $f$ is a real function which is piecewise of class $C^1$ on a bounded interval, its Fourier series converges uniformly where $f$ is $C^1$, and on the finite number of discontinuity points, $f$ converges towards the half sum of $f$'s right and left limits. 

This theorem is still valid if $f$ is piecewise $\alpha$-H\"{o}lderian $(0 < \alpha < 1)$, and the proof is similar.

We suppose from now on, that the metric $g$ satisfies the Driver condition: $V_1$ and $V_2$ being $C^\infty$ vector fields on $V_n$, $T$ the torsion, $g \left( T(V_1, V_2), V_1 \right) = 0$.

\subsection{Definition of the New Cameron-Martin space and of the standard bilinear form}	

We recall some properties of the NCM (new Cameron-Martin)

$\left\{ e_\mu \big/ \mu = 1, \dots, n \right\}$ being the canonical basis of $\R^n$, NCM denoted $\tilde{H}$ is:
\begin{align*}
\tilde{H} = \left\{ \sum_{\mu=1}^{n} f^\mu(t) u_\mu (t, \omega) \bigg/ u_\mu \text{ being the SPT of } ue_\mu \text{ and } f^\mu(t) \in H \right\}
\end{align*}
 
Scalar product on $\tilde{H}$:\\
If $v_i=\sum^n_{\mu=1} f_i^\mu(t) u_\mu (t,w) \in \tilde{H}, i=1,2:$\\
\[<v_1,v_2>_{\tilde{H}}=\int_0^1 \dot{f}_1^\mu(t)\dot{f}_2^\mu(t)dt\]
$-$ $\tilde{H}$ is complete.\\
$-$ A basis of $\tilde{H}$ is:\\
$\{ v_i=\sum^n_{\mu=1} f_i^\mu(t) u_\mu (t,w)/ \quad^{t}(f_i^1,\ldots,f_i^n), i\in \mathbb{N}_*$ being a basis of $H$, \\ $(u_\mu(t,w)/\mu=1,\ldots,n)$ being the SPT of a basis in $T_{m_0}V_n\}$
\begin{defn}\label{def:VIII,1} \hfill \\
i) If $\delta\in \textup{Der}(\Omega), f\in\mathbb{D}^\infty(\Omega): df(\delta)=\delta(f),$ then $df\in \textup{Der}^*(\Omega)$.\\
ii) If $(v_i)_{i\in\mathbb{N}_*}$ is a basis of $\tilde{H}$, we define $\delta_f \in \textup{Der}$ by: 
\[\delta_f(g)=\sum_{i=1}^\infty D_{v_i}f \cdot D_{v_i}g.\]
$\delta_f \in \textup{Der}(\Omega)$ because $\delta_f(g)=<\textup{grad }f,\textup{grad }g>_{\tilde{H}}$.\\
iii) If $\alpha\in \textup{Der}^*(\Omega)$: $\delta_\alpha(f)=\alpha(\delta_f)$, then $\delta_\alpha \in\textup{Der}(\Omega)$.
\end{defn}
\begin{rem} 
If the set $(f_k)_{k\in K}$ is bounded in $\mathbb{D}^\infty(\Omega)$, the set $(\delta_{f_k})_{k\in K}$ is bounded in $\textup{Der}(\Omega)$.
\end{rem}
\begin{defn} 
If $\alpha,\beta\in \textup{Der}^*(\Omega)$, we define the standard bilinear form on $\textup{Der}^*$ by:
\[ q(\alpha,\beta)=\sum^\infty_{i=1}\alpha(D_{v_i})\beta(D_{v_i}) \]
\end{defn}
We have to show that this definition does not depend of the chosen basis, that this series is $\mathbb{D}^\infty$-convergent, and that $q$ is $\mathbb{D}^\infty$-continuous relatively to each of its argument.

We first prove that the series defining $q(\alpha,\beta)$ is $\mathbb{D}^\infty$-convergent. \\
a) If $\alpha=df, \beta=dg, f$ and $g\in\mathbb{D}^\infty(\Omega)$,
\[ q(\alpha,\beta)=\sum^\infty_{i=1}df(D_{v_i})\cdot dg(D_{v_i})= <\textup{grad }f,\textup{grad }g>_{\tilde{H}} \]
b) If $\beta=dg$ and $\alpha\in \textup{Der}^*$
\[ q(\alpha,\beta)=\alpha (\sum^\infty_{i=1}D_{v_i}g\cdot D_{v_i})\]
As $\sum^\infty_{i=1}D_{v_i}g\cdot D_{v_i}=\delta_g\in \textup{Der}(\Omega),$ and as $\alpha$ is continuous, $q(\alpha, \beta)$ is legitimate.\\ 
c) Let $\alpha \in \textup{Der}^*$, with theorem 3,2, there exists a bounded net $(\alpha_k)_{k\in K}$ in $\textup{Der}^*$ which converges towards $\alpha$, and:
\[ \alpha_k=\sum_{i_k \in I_k} h_{i_k} dg_{i_k}, I_k \textup{ finite}, h_{i_k}\in \mathbb{D}^\infty(\Omega) \textup{ and } g_{i_k}\in \mathbb{D}^\infty(\Omega)\]
Then direct computation shows that:
\[ \delta_{\alpha_k}=\sum_{i_k\in I_k} h_{i_k} \sum_{j=1}^\infty D_{v_j}g_{i_k}\cdot D_{v_j}\] 
Moreover the set $(\delta_{\alpha_k})_{k\in K}$ is a bounded set in Der$(\Omega)$, $(v_j)_{j\in \mathbb{N}_*}$ being a base of $\tilde{H}$, then:
\[ q(\alpha_k,\alpha_k)=\sum_{j \in \mathbb{N}_*}\vert\alpha_k (D_{v_j})\vert^2=\alpha_k(\delta_{\alpha_k}) \]
because we have:
\begin{equation*}
\begin{split}
q(\alpha_k,\alpha_k) &= q(\alpha_k, \sum_{i_k\in I_k} h_{i_k}dg_{i_k})=\sum_{i_k\in I_k} h_{i_k} q(\alpha_k, dg_{i_k})\\
&=\sum_{i_k\in I_k} h_{i_k} \sum_{j=1}^\infty \alpha_k(D_{v_j})dg_{i_k}(D_{v_j})\\
&=\sum_{i_k\in I_k} h_{i_k} \sum_{j=1}^\infty \alpha_k(D_{v_j})D_{v_j}(g_{i_k})\\
&=\alpha_k(\delta_{\alpha_k})
\end{split}
\end{equation*}

The sets $(\alpha_k)_{k\in K}$ and $(\delta_{\alpha_k})_{k\in K}$ are bounded in $\textup{Der}^*$ and Der, so the finite sums $\sum^L_{i=1 }\vert \alpha_k(D_{v_i}) \vert^2$ are $\mathbb{D}^\infty$-bounded uniformly in $L$ and in $k$, so $q(\alpha,\alpha)$ exists and is $\in L^{\infty-0}(\Omega)$. But the set $\lbrace \alpha_k (\delta_{\alpha_k})/ k\in \mathbb{N}_* \rbrace$ is $\mathbb{D}^\infty$-bounded, then by interpolation, $q(\alpha,\alpha)\in \mathbb{D}^\infty$.

The net $q(\alpha_k,\alpha_k)-q(\alpha,\alpha)$ is $L^{p^\prime}$-convergent for each $p^\prime>1$ and is $\mathbb{D}^p_r$-bounded, $p>1$, $r\in \mathbb{N}_*$. By interpolation this net converges towards 0 in $\mathbb{D}^\infty(\Omega)$, so $q(\alpha,\alpha)\in \mathbb{D}^\infty(\Omega)$.

Now we want to prove the continuity of $q(\alpha,\beta)$ relatively to each of its arguments. If $\alpha\in \textup{Der}^*$ and $\beta \in \textup{Der}^*$, we know that there exists a net $(\alpha_k)_{k\in K}$ as in theorem 3,2 such that: $\alpha_k \xrightarrow{\textup{Der}^*} \alpha$, then we have seen that
\setcounter{equation}{0}
\begin{eqnarray}\label{eq:VIII,1}
\lim_k q(\alpha_k,\beta) = q(\alpha,\beta) = \lim_k \alpha_k (\delta_{\beta}) = \beta (\delta_{\alpha})
\end{eqnarray}
Then the continuity of $q$ relatively to each of its arguments is trivial.

The independence of the the definition of $q$ relatively to the chosen basis is given by:
\[ q(\alpha,\beta)=\beta(\delta_\alpha) \qquad \textup{(definition \ref{def:VIII,1}) and }\eqref{eq:VIII,1}\]

\begin{rem}
$\theta$ being the bijection between $\textup{Der}^*(\Omega)$ and $D_0(\Omega)$, obtained through the standard bilinear form, we have: $\forall \alpha \in \textup{Der}^*(\Omega)$, $(v_i)_{i\in \mathbb{N}_*}$ being an Hilbertian basis of $\tilde{H}$:
\[ \theta(\alpha)=\sum^\infty_{i=1} \alpha(v_i) v_i \] 
\end{rem}

\subsection{Non degenerescence of the standard quadratic form}	

We will now prove that the standard bilinear form is non degenerate. It is enough to prove that if $\alpha \in \textup{Der}^*\setminus\lbrace 0\rbrace$, there exists $v\in\tilde{H}$ such that $\alpha(D_v)\neq 0$, $D_v$ being the derivation associated to $v$ as (7.23).

We first, will prove that the "derivation adherence" of a particular $\mathbb{D}^\infty$-module $\mathfrak{D}$, contains all $\mathbb{D}^\infty$-vector fields. The derivative adherence means that it includes all limits of bounded nets of $\mathfrak{D}$, which converges as  derivations:
\[ (\delta_i)_{i\in I}\rightarrow \delta \Leftrightarrow \forall f\in \mathbb{D}^\infty (\Omega), \quad \delta_i f \xrightarrow{\mathbb{D}\infty}\delta f. \]

We consider the following operator: $B$ is a constant antisymmetrical $n\times n$ matrix, if $h\in H$, we denote $(\hat{A}h)(t)=\int_0^t B\dot{h}(s)ds$.
$\hat{A}$ is a bounded operator on $H$. As basis of $\tilde{H}$, we take (7.33) and denote the generic item of this basis $\check{B}$ by $\varepsilon_j$. 

We recall that to $v\in \tilde{H}$, we can associate a derivation 
\[D_v=\textup{div } A(v)\textup{ grad} + t\rightarrow \int_0^t\dot{h}(v)ds\quad (7.23)\]

Now, as basis of $\tilde{H}$, we take $\check{B}$ (7.33). $P_N$ being the projection on the subspace of $H$ generated by the: $h(\varepsilon_{l,j})$, $h(\varepsilon^\prime_{l,j})$ and $h(\varepsilon^{\prime\prime}_{j})$ with $l=1,\ldots,N$ and $j=1,\ldots,n$, denote $\hat{A}_N=P_N\hat{A}P_N$.

We recall that if $v(t,\omega)=\sum_{\mu=1}^n f^\mu(t) u_\mu(t,\omega)$ with $u_\mu(t,\omega)$ being the SPT of $ue_\mu$ and $(f^\mu)_{\mu=1,\ldots,n}\in H$, then $v\in\tilde{H}$ and $h(v)^\mu=f^\mu(t)$.

Taking in account the special form of the basis vectors of $\check{B}$, we deduce $[\hat{A},P_N]=0$. We denote by $h(\varepsilon_i)$ the generic item obtained from $\varepsilon_i \in \check{B}$, with the bijection $\tilde{H}\rightarrow H$: 
\[ \tilde{H}\ni v=h(t)u(t,\omega)\rightarrow h(v)(t)=h(t)\in H \]

We now will study the limit, as a derivation, of:
\begin{equation}\label{eq:2}
\sum_{i,j \leqslant N} a_{ij} \lbrace W (h(\varepsilon_i) ) D_{\varepsilon_j}-W (h(\varepsilon_j) ) D_{\varepsilon_i} \rbrace
\end{equation}
where $a_{ij}=<h(\varepsilon_i),\hat{A}h(\varepsilon_j)>_H$.

With (7.23) and (7.27), and denoting $Z(\varepsilon_i)=\frac{1}{2}\sum^n_{\mu=1} b_{\mu\mu}(\varepsilon_i)$, we get from \eqref{eq:2}:
\begin{equation*}
\begin{split}
(\ref{eq:2}) 
&= \sum_{i,j \leqslant N} a_{ij}\lbrace W ( h(\varepsilon_i) )h(\varepsilon_j)-W (h(\varepsilon_j) )h(\varepsilon_i) \rbrace  \hspace{44mm} \textup{(T1)}\\
&+ \sum_{i,j \leqslant N}a_{ij}\lbrace W ( h(\varepsilon_i) )\int_0^t Z(\varepsilon_j)ds-W (h(\varepsilon_j) )\int_0^t Z(\varepsilon_i)ds \rbrace \hspace{23.3mm} \textup{(T2)}\\
&- \sum_{i,j \leqslant N}a_{ij}\lbrace <h(\varepsilon_i),A(\varepsilon_j)\textup{ grad}(\cdot)>_H-<h(\varepsilon_j),A(\varepsilon_i)\textup{ grad}(\cdot)>_H\rbrace \hspace{5.8mm} \textup{(T3)}\\
&+ \sum_{i,j \leqslant N} a_{ij}\lbrace \textup{div } \left( W ( h(\varepsilon_i) )A(\varepsilon_j)-W (h(\varepsilon_j) )A(\varepsilon_i)\right)  \rbrace  \textup{ grad}(\cdot)\hspace{20mm} \textup{(T4)}
\end{split}
\end{equation*}

We want to prove that each of the $Ti, i=1,2,3,4,$ converges when $N\rightarrow \infty$, as derivations, towards a $\mathbb{D}^\infty$-continuous derivation.\\

For (T1):
\[(T1)=2 \textup{ div } \hat{A}_N \textup{ grad}(\cdot)\]

As bounded operators, we have: $\hat{A}_N \rightarrow \hat{A}$, so: 
\[ \forall X\in\mathbb{D}^\infty (\Omega,H): \quad \| (\hat{A}_N-\hat{A})X \|_{D_r^p }\rightarrow 0\]
div being a $\mathbb{D}^\infty$-continuous operator, $\forall f\in \mathbb{D}^\infty (\Omega)$:
\[ \lim_{N\rightarrow\infty} \textup{div } \hat{A}_N \textup{ grad }f = \textup{div } \hat{A} \textup{ grad }f \quad (\textup{in } \mathbb{D}^\infty(\Omega))\]

For (T2):

We can rewrite (T2) as the vector field:
\[ t\rightarrow 2\int_0^t Z\left[ W\left( \sum^N_{j=1} \hat{A}_N h(\varepsilon_j)\right) \cdot \varepsilon_j\right]  ds\]

Now $\varepsilon_j(s,\omega)=\sum^n_{\alpha=1} h^\alpha(\varepsilon_j)(s)u_\alpha(s,\omega)$, where as usual: $u_\alpha(s,\omega)=$ SPT of $ue_\alpha$, at time $s$ "along $\omega$", $(e_\alpha)_{\alpha=1,\ldots,n}$ canonical unit vectors of $\mathbb{R}^n$.

We denote by $(k^\alpha_s)_{\alpha=1,\ldots,n}$ this element of $H: \rho\in[0,1]$
\[ k^\alpha_s(\rho)=(\rho\wedge s) e_\alpha,\quad \alpha=1,\ldots,n\] 
Then 
\begin{equation}\label{eq:VIII,3}
h^\alpha(\varepsilon_j)(s)=\int^s_0 \dot{h}^\alpha(\varepsilon_j)(r)dr=<k^\alpha_s, h(\varepsilon_j)>_H
\end{equation}
Then (T2) becomes, with Einstein summation convention:
\begin{equation*}
\begin{split}
(\textup{T2})&=2\int^t_0 Z\left[ W \left( \sum_{j=1}^N\hat{A}_N h(\varepsilon_j) <k^\alpha_s, h(\varepsilon_j)>_H \right) u_\alpha(s,\omega)\right]ds\\
&= 2\int^t_0 Z\left[ W \left( \hat{A}\left(  \sum^N_{j=1}<k^\alpha_s, h(\varepsilon_j)>_H h(\varepsilon_j)\right) \right) u_\alpha(s,\omega)\right]ds
\end{split}
\end{equation*}
But $Z$ is $\mathbb{R}$-linear, so:
\[ (\textup{T2})=2\int^t_0 W \left( \hat{A} \left( \sum^N_{j=1}<k^\alpha_s, h(\varepsilon_j)>_H h(\varepsilon_j)\right) \right) Z(u_\alpha(s,\omega))ds\]
With corollary 6.1, we know that $\sup_{\alpha\in\lbrace 1,\ldots,n\rbrace} Z(u_\alpha(s,\omega))$ is $\mathbb{D}^\infty$-bounded.
Now 
\begin{align}\label{eq:4}
&\left\Arrowvert W\left( \hat{A} \sum^N_{j=1}<k^\alpha_s, h(\varepsilon_j)>_H h(\varepsilon_j)\right) \right\Arrowvert_{L^2(\Omega)} \nonumber \\
&=\left\Arrowvert\hat{A} \left( \sum^N_{j=1}<k^\alpha_s, h(\varepsilon_j)>_H h(\varepsilon_j)\right) \right\Arrowvert_H
\leqslant\lVert \hat{A}\|\cdot\| k^\alpha_s\|_H
\end{align}
and $\| k^\alpha_s \|^2_H=\int_0^1 \mathds{1}_{[0,s]}(u)du.$\\

As on $\mathcal{C}_1$ all $\mathbb{D}_r^p$-norms are equivalent, the l.h.s. of 
(\ref{eq:4}) is $\mathbb{D}^\infty$-bounded, uniformly relatively to $s$ and $N$. So (T2) converges, as multiplicators, towards: $2\int_0^t W[\hat{A} k^\alpha_s]Z(u_\alpha)ds$.\\

For (T3): (T3) can be rewritten as:
\[ (\textup{T3})=-2<\sum_{1\leqslant i,j \leqslant N} a_{ij} A(\varepsilon_j)\cdot h(\varepsilon_i), \textup{grad}(\cdot)>_H\]
To prove that this sequence of derivations is $\mathbb{D}^\infty$-converging towards a  $\mathbb{D}^\infty$-continuous derivation, the convergence being a derivation convergence, it is enough to prove that: $2\sum_{1\leqslant i,j \leqslant N} a_{ij} A(\varepsilon_j)h(\varepsilon_i)$ converges, as vector fields, towards a $\mathbb{D}^\infty$-vector field.\\

With (7.24) we can write, with shorter notations:
\[ A(\varepsilon_j)(t)=\int^t_0 \dot{h}^\alpha(\varepsilon_j)(s)\cdot \gamma_{1,\alpha}(s)ds+\int^t_0 h^\alpha(\varepsilon_j)(s) Z_{1,\alpha}(s)\circ d\tilde{B} \]
where $\gamma_{1,\alpha}(s)$ and $Z_{1,\alpha}(s)$ are:
\[(\gamma_{1,\alpha}(s))_\mu^\nu =g(T(u_\alpha,u_\mu),u_\nu) \qquad (7.18^\prime)\]
and 
\[(Z_{1,\alpha}(s))^\nu_{\mu,\rho}=g(R(u_\alpha,u_\rho,u_\mu),u_\nu)+g((\bigtriangledown_{u_\rho}T)(u_\alpha,u_\mu),u_\nu) \qquad (7.22^\prime).\]
$Z_1$ and $\gamma_1$ are $\mathbb{D}^\infty$-semi-martingales, $\frac{1}{2}$-$\mathbb{D}^\infty$-Holderian processes and $\mathbb{R}$-multilinear for the variables $u_\alpha, u_\rho, u_\mu, u_\nu$.

We use the Stratonovich integration by parts on 
\[\int_0^t h_\alpha (\varepsilon_j)\times Z_1(u_\alpha,u_\rho,u_\mu,u_\nu)\circ d\tilde{B}^\rho\]
and we get: 
\begin{equation}\label{eq:VIII,5}
A(\varepsilon_j)(t)=\int_0^t \dot{h}^\alpha (\varepsilon_j)(s)\gamma_{2,\alpha}(s)ds+h^\alpha(\varepsilon_j)(t)Z_{2,\alpha}(t)
\end{equation}
where again $\gamma_{2,\alpha}$ and $Z_{2,\alpha}$ are $\mathbb{D}^\infty$-semi-martingales, $\frac{1}{2}$-$\mathbb{D}^\infty$-Holderian processes and $\mathbb{R}$-multilinear for their variables $u_\alpha, u_\rho, u_\mu, u_\nu$,\\
\begin{equation*}\label{eq:VIII,5'}
\hspace{36mm} Z_{2,\alpha}=\int_0^t Z_1(u_\alpha,u_\rho,u_\mu,u_l)\circ d\tilde{B}^l \hspace{36mm}  (5') 
\end{equation*}
and
\begin{equation*}\label{eq:VIII,5''}
\hspace{31mm}\gamma_{2,\alpha}=\gamma_{1,\alpha}-\int_0^s Z_1(u_\alpha,u_\rho,u_\mu,u_l)\circ d\tilde{B}^\rho \hspace{31mm}(5'') 
\end{equation*}
So $\gamma_{2,\alpha}(s)$ can be written $\gamma_2(u_\alpha,u_\mu,u_l)$.

Then $A(\varepsilon_j)$ acting on the vector $h(\varepsilon_i)$ is given by:
\begin{align}\label{eq:VIII,6}
(A(\varepsilon_j)\cdot h(\varepsilon_i))(t)&=\int_0^t Z_{2,\alpha}(s)\cdot h^\alpha(\varepsilon_j)(s)\cdot \dot{h}(\varepsilon_i)(s)ds \nonumber\\
&+\int_0^t \left( \int_0^s \gamma_{2,\alpha}(r) \dot{h}^\alpha (\varepsilon_j)(r)dr\right) \cdot \dot{h}(\varepsilon_i)ds
\end{align}

We define a vector field $w_s(u_\mu,u_l)$ by:
\begin{equation*}\label{eq:VIII,6'}
\hspace{11.2mm} (w_s(u_\mu,u_l))_\alpha=t \rightarrow \int_0^{s\wedge t} \gamma_2 (u_\alpha,u_\mu, u_l)(r)dr=\int_0^{s\wedge t} \gamma_{2,\alpha}(r)dr \hspace{11.2mm} (6') 
\end{equation*}
Then
\begin{equation}\label{eq:VIII,7}
<w_s(u_\mu,u_l),h(\varepsilon_j)>_H=\int_0^s \gamma_{2,\alpha}(r)\dot{h}^\alpha(\varepsilon_j)(r)dr
\end{equation}
And (\ref{eq:VIII,6}) becomes:
\[ (A(\varepsilon_j)\cdot h(\varepsilon_i))(t)=\int_0^t Z_{2,\alpha}(s) h^\alpha(\varepsilon_j)\cdot \dot{h}(\varepsilon_i)(s)ds +\int_0^t <w_s,h(\varepsilon_j)>_H \dot{h}(\varepsilon_i)ds \]
Now:
\begin{align*}
\sum_{i,j\leqslant N} a_{ij}(A(\varepsilon_j)h(\varepsilon_i))(t)&= \int_0^t Z_{2,\alpha}(s) \left( \sum_{i,j\leqslant N}a_{ij}h^\alpha(\varepsilon_j)\right) \cdot \dot{h}(\varepsilon_i)ds \\ &+\int_0^t<w_s,\sum_{i,j\leqslant N}a_{ij}h(\varepsilon_j)>_H \dot{h}(\varepsilon_i)ds
\end{align*}
Using: $\sum_{i,j\leqslant N} a_{ij}h(\varepsilon_j)=-\hat{A}_N h(\varepsilon_i)$ and (\ref{eq:VIII,3}): $<k^\alpha_s, h(\varepsilon_j)>_H=h^\alpha(\varepsilon_j)$, we get:
\begin{align*}\label{eq:VIII,7'}
\sum_{i,j\leqslant N} a_{ij}(A(\varepsilon_j)\cdot h(\varepsilon_i))(t)
&= -\int_0^t Z_{2,\alpha}(s) \sum_{i\leqslant N}<k_s^\alpha,\hat{A}_N(h(\varepsilon_i))>_H\dot{h}(\varepsilon_i)ds \nonumber \\ 
&-\int_0^t \sum_{i\leqslant N}<w_s,\hat{A}_N (h(\varepsilon_i))>_H \dot{h}(\varepsilon_i)ds \nonumber \\
&=\int_0^t\sum_{i\leqslant N}Z_{2,\alpha}(s)<\hat{A}_N k_s^\alpha,h(\varepsilon_i)>_H\dot{h}(\varepsilon_i)ds \nonumber \\
&+\int_0^t\sum_{i\leqslant N}<\hat{A}_N (w_s),h(\varepsilon_i)>_H\dot{h}(\varepsilon_i)ds \nonumber \\
&=\int_0^t Z_{2,\alpha}(s)\sum_{i\leqslant N}<\hat{A}_N k_s^\alpha,h(\varepsilon_i)>_H\dot{h}(\varepsilon_i)ds \nonumber \\
&+\int_0^t \sum_{i\leqslant N}<\hat{A} w_s,h(\varepsilon_i)>_H \dot{h}(\varepsilon_i)ds\hspace{25mm} (7')
\end{align*}

In the first integral of ($7'$), the series, $\sum_{i\leqslant N}<\hat{A} k_s^\alpha,h(\varepsilon_i)>_H\dot{h}(\varepsilon_i)(r)$ is independent of $\omega$, and depends only of $s$. The theorem of Dini-Lipschitz (Theorem 2.11) applied to this series shows that it converges uniformly on each compact of $[0,1]$ which does not include discontinuity points, relatively to $r$, and on the discontinuity points, the series converges towards the half sum of the left and right limits. Here the only discontinuity point is $s=r$, so on this point the series converges towards

\[\frac{1}{2}(\dot{\wideparen{\hat{A} }}k^\alpha_{s-} + \dot{\wideparen{\hat{A}}} k_{s+}^\alpha)=\frac{1}{2}Be_\alpha\]
As $s\in [0,1]$, we see that the integrand in the first integral in ($7'$) is $\mathbb{D}^\infty$-bounded, uniformly in $s$ ($Z_{2,\alpha}$ is 1/2-$\mathbb{D}^\infty$-Holderian). So the first integral in ($7'$), converges $\mathbb{D}^\infty$ towards $\sum_{\alpha=1}^n\int_0^t Z_{2,\alpha}(s)Be_\alpha ds$; if we apply the O.U. operator to this first integral, we get:
\[ \sum_{\alpha=1}^n\int_0^t ((1-L)^{r/2}Z_{2,\alpha})\sum_{i\leqslant N}<\hat{A}k_s^\alpha, h(\varepsilon_i)>_H\dot{h}(\varepsilon_i)(s)ds\]

The same reasoning shows that this new sequence will $\mathbb{D}^\infty$-converge towards $\sum_{\alpha=1}^n\int_0^t ((1-L)^{r/2}Z_{2,\alpha})Be_\alpha ds$. So the first integral of ($7'$) $\mathbb{D}^\infty$-converges as vectors fields towards $\sum_{\alpha=1}^n\int_0^t Z_{2,\alpha}Be_\alpha ds$.

For the second integral in ($7'$), the same reasoning on the series 
\[\sum_{i\leqslant N}<\hat{A} w_s,h(\varepsilon_i)>_H \dot{h}(\varepsilon_i)(r)\]
but considering the Dini-Lipschitz convergence in the Frechet space $\mathbb{D}^\infty (\Omega)$, shows that this series converges on the unique discontinuity point ($r=s$), and that this $\mathbb{D}^\infty$-convergence towards the half value of the jump in $r=s$, is uniformly , in $s$, bounded.

Then the second integral of ($7'$) will $\mathbb{D}^\infty(\Omega,H)$-converges towards the vectors field:
\[ t\rightarrow \sum_{\alpha=1}^n \int_0^t \frac{1}{2}B\gamma_{2,\alpha} e_\alpha ds\]

Now, for (T4), we will need the following lemma.
\begin{lem}
The vector field $s\rightarrow w_s$ is $\frac{1}{2}$-$\mathbb{D}^\infty$-Holderian.
\end{lem}
\pf. Recall ($6'$): 
\[\frac{1}{\sqrt{\varepsilon}}(1-L)^{r/2}(w_{t+\varepsilon}-w_t)=\frac{1}{\sqrt{\varepsilon}}\int^{t+\varepsilon}_t (1-L)^{r/2} \gamma_2(s,\omega)ds\]
and 
\begin{align*}
\|\frac{1}{\sqrt{\varepsilon}}(1-L)^{r/2}(w_{t+\varepsilon}-w_t) \|^p_{L^p(\Omega,H)}&=\int \mathbb{P}(d\omega)\left[ \int^1_0 \vert \frac{1}{\sqrt{\varepsilon}}(1-L)^{r/2}\gamma_2 \vert^2 \mathds{1}^{(s)}_{[t,t+\varepsilon]}ds \right]^{p/2}\\
&\leqslant\int^{t+\varepsilon}_t ds \int \mathbb{P}(d\omega) \vert \frac{1}{\sqrt{\varepsilon}}(1-L)^{r/2}\gamma_2 \vert^p
\end{align*}

Last, we study (T4):
\[ (T4)= \sum_{i,j\leqslant N} a_{ij} \textup{div }\left[ W(h(\varepsilon_i))A(\varepsilon_j)-W(h(\varepsilon_j))A(\varepsilon_i) \right]\textup{ grad} \]

We want to prove that when $N\rightarrow \infty$, (T4) converges, as a sequence of derivation, towards a $\mathbb{D}^\infty$-continuous derivation, div $Q$ grad, when $Q$ is a $\mathbb{D}^\infty$-multiplicator. For this, it is enough to prove that the sequence:
\begin{align}\label{eq:VIII,8}
S_N&=\sum_{i,j\leqslant N} a_{ij} \left[ W(h(\varepsilon_i))A(\varepsilon_j)-W(h(\varepsilon_j))A(\varepsilon_i) \right] \nonumber \\
&=2\sum_{j\leqslant N}W\left[ \hat{A}_N(h(\varepsilon_j)) \right]A(\varepsilon_j)
\end{align}
converges as multiplicators, towards a $\mathbb{D}^\infty$-multiplicator $Q$. 

Using the expression of $A(\varepsilon_j)$ in \eqref{eq:VIII,5}, with \eqref{eq:VIII,8}:
\begin{align}\label{eq:VIII,9}
S_N&=2\sum_{j\leqslant N}W\left[ \hat{A}_N(h(\varepsilon_j)) \right]\cdot h^\alpha(\varepsilon_j)\cdot Z_{2,\alpha}(t) \nonumber \\
&+2\sum_{j\leqslant N}W\left[ \hat{A}_N(h(\varepsilon_j)) \right]\cdot \int^t_0 \dot{h}^\alpha(\varepsilon_j)(s)\gamma_{2,\alpha}(s)ds
\end{align}

We denote the two items in \eqref{eq:VIII,9} by (T4,1) and (T4,2). So $S_N=$(T4,1)+(T4,2). The case of (T4,1):
\begin{align*}
(T4,1)&=2Z_{2,\alpha}(t)\cdot \sum_{j\leqslant N}W\left[ \hat{A}_N(h(\varepsilon_j)) \right] h^\alpha(\varepsilon_j)(t), \quad \alpha=1,\ldots,n \nonumber \\
&=2Z_{2,\alpha}(t) W \left[\hat{A}_N (\sum_{j\leqslant N} h^\alpha(\varepsilon_j)(t) \cdot h(\varepsilon_j))\right] \nonumber \\
&=2Z_{2,\alpha}(t) W \left[\hat{A}(\sum_{j\leqslant N} h^\alpha(\varepsilon_j)(t) \cdot h(\varepsilon_j))\right]
\end{align*}

The series $\sum_{j=1}^\infty h^\alpha(\varepsilon_j)(t) \cdot h(\varepsilon_j)$ converges in $H$ because the basis $\check{B}$ is such that: $\forall j\in \mathbb{N}_*, \forall\alpha\in \lbrace 1,\ldots,n \rbrace, \vert h^\alpha(\varepsilon_j) \vert \leqslant \frac{Const}{j}$, uniformly relatively to $t\in [0,1]$.

This implies that $W \left[\hat{A}(\sum_{j=1}^\infty h^\alpha(\varepsilon_j)(t) \cdot h(\varepsilon_j))\right]$ is $L^2$-convergent, $t$-uniformly, towards $W[\hat{A}(k^\alpha_t)]$
\[  h^\alpha(\varepsilon_j)(t)=< h(\varepsilon_j),k_t^\alpha>_H \qquad (\textup{see} (3))\]

So as $\forall N: W \left[\hat{A}(\sum_{j\leqslant N} h^\alpha(\varepsilon_j)(t) \cdot h(\varepsilon_j))\right]\in \mathcal{C}^1$, and with corollary 4.1, we see that the sequence (T4,1) converges, as multiplicators, towards $W[\hat{A}(k^\alpha_t)]$, $t$-uniformly. And $Z_{2,\alpha}(t)$ is also a multiplicator because it is an $\frac{1}{2}$-Holderian process, so the limit of (T4,1), limit as multiplicators, is the $\mathbb{D}^\infty$-multiplicator:
\[ 2Z_{2,\alpha}(t)\cdot W\left[ \hat{A}(k^\alpha_t) \right] \]

Now for (T4,2), we rewrite it:
\begin{align}\label{eq:VIII,10}
(T4,2)&=2\sum_{j\leqslant N} \lbrace W(\hat{A}_N (h(\varepsilon_j))-\E\left[ W(\hat{A}_N h(\varepsilon_j))|\mathcal{F}_t\right]  \rbrace \int^t_0 \dot{h}^\alpha(\varepsilon_j)(s)\gamma_{2,\alpha}(s)ds \nonumber \\
&+ 2\sum_{j\leqslant N}\E\left[ W(\hat{A}_N h(\varepsilon_j))|\mathcal{F}_t\right]\int^t_0 \dot{h}^\alpha(\varepsilon_j)(s)\gamma_{2,\alpha}(s)ds
\end{align}
These two items in \eqref{eq:VIII,10} are labeled as (T4,2,a) and (T4,2,b).

We first study the limit of (T4,2,a):
\begin{align*}
(T4,2,a)&=2\sum_{j\leqslant N} \lbrace W(\hat{A}_N (h(\varepsilon_j)))-\E\left[ W(\hat{A}_N h(\varepsilon_j))|\mathcal{F}_t\right]  \rbrace \int^t_0 \dot{h}^\alpha(\varepsilon_j)\gamma_{2,\alpha}ds \\
&= 2\sum_{j\leqslant N} \lbrace \textup{div }\left[ \hat{A}_N (h(\varepsilon_j))-P_t(\hat{A}_N (h(\varepsilon_j))) \right]  \rbrace \cdot \int^t_0 \dot{h}^\alpha(\varepsilon_j)(s)\gamma_{2,\alpha}(s)ds
\end{align*}
where $P_t$ is the projection operator on $H$, defined by:
\[ (P_t Y)(u)=\int^{u\wedge t}_0 X(u)du \textup{ with } Y(s)=\int^s_0 X(u)du, Y\in H\]
($P_t$ is the projector of vectors of $H$, on vectors which are zero after $t$).

With (7): $<w_s,h(\varepsilon_j)>_H= \int_0^s \gamma_{2,\alpha}(r)\dot{h}^\alpha(\varepsilon_j)(r)dr$, so we have:
\begin{align}\label{eq:VIII,11}
(T4,2,a)&= 2\textup{div } \left[ \sum_{j\leqslant N}\left( \int^t_0 \dot{h}^\alpha(\varepsilon_j)\gamma_{2,\alpha}ds\right)   \left( \hat{A}_N (h(\varepsilon_j))-P_t\hat{A}_N (h(\varepsilon_j))\right) \right] \nonumber \\
&= 2\textup{div } \left[\hat{A}_N \left( \sum_{j\leqslant N}<w_t,h(\varepsilon_j)>_H h(\varepsilon_j) \right) \right] \nonumber \\
&- 2\textup{div }\left[P_t \hat{A}_N \left( \sum_{j\leqslant N}<w_t,h(\varepsilon_j)>_H h(\varepsilon_j) \right) \right]
\end{align}

We first study the limit of: \[2\textup{div } \left[\hat{A}_N \left( \sum_{j\leqslant N}<w_t,h(\varepsilon_j)>_H h(\varepsilon_j) \right) \right]=2\textup{div }(P_N\hat{A}w_t)\] (Recall $[\hat{A},P_N]=0$).

Now , with lemma 8.1, $P_N\hat{A}w_t$ is $\frac{1}{2}$-$\mathbb{D}^\infty$-Holderian, in $t$, $N$-uniformly, so, as div is $\mathbb{D}^\infty$-continuous, div($P_N\hat{A}w_t$) is $\frac{1}{2}$-$\mathbb{D}^\infty$-Holderian, so is a multiplicator, $N$-uniformly. Using the closed graph theorem, the sequence  2div ($P_N\hat{A}w_t$) converges as multiplicators, towards the multiplicator 2div $\hat{A}w_t$.

Now we study the limit of the second item in \eqref{eq:VIII,11}, 
\[ 2\textup{div } \left[P_t \hat{A}_N \left( \sum_{j\leqslant N}<w_t,h(\varepsilon_j)>_H h(\varepsilon_j) \right) \right] =2\textup{div } \left[P_t P_N\hat{A}w_t\right];\]
to study this limit, we consider with $\lambda\in[0,1]$: $P_\lambda P_N\hat{A}w_t=I_{\lambda,N,t}$. $I_{\lambda,N,t}$ is $\frac{1}{2}$-$\mathbb{D}^\infty$-Holderian, in $t$, $(\lambda, N)$-uniformly, but is also $\frac{1}{2}$-$\mathbb{D}^\infty$-Holderian, in $\lambda$, $(N,t)$-uniformly, because:
\begin{align*}
\|(1-L)^{r/2}(P_{\lambda+\varepsilon}-P_\lambda)(P_N\hat{A}w_t) \|_H
&\leqslant\int_0^1 \mathds{1}_{[\lambda,\lambda+\varepsilon]} (s)\| (1-L)^{r/2} P_N\hat{A}w_t (s,w)\|_H ds\\
&\leqslant \varepsilon^{1/2} \left( \int_0^1 \| P_N\hat{A} (1-L)^{r/2}w_t \|^2_H ds\right)^{1/2} 
\end{align*}
Then 
\begin{align*}
\|(P_{\lambda+\varepsilon}-P_\lambda)(P_N\hat{A}w_t) \|_{\mathbb{D}^p_r(\Omega,H)}
&\leqslant \varepsilon^{1/2}\|P_N\hat{A}w_t \|_{\mathbb{D}^p_r(\Omega,H)}\\
&\leqslant \varepsilon^{1/2}\|\hat{A}w_t \|_{\mathbb{D}^p_r(\Omega,H)}
\end{align*}
and $\hat{A}w_t$ is $\mathbb{D}^\infty$-bounded because $\gamma_2$ is $\mathbb{D}^\infty$-bounded. We have then the $\frac{1}{2}$-$\mathbb{D}^\infty$-Holderianity of $I_{\lambda,N,t}$ when $\lambda=t$. Then as for the first item of (T4,2,a), $\textup{div } P_t P_N\hat{A}w_t$ converges, as a sequence of multiplicators towards the multiplicator $\textup{div } P_t \hat{A}w_t$. And $\textup{div } (P_t \hat{A}w_t)=\textup{div } \hat{A}w_t$.

The last limit to study is (T4,2,b), in \eqref{eq:VIII,10}:
\[ (T4,2,b)=2\sum_{j\leqslant N}\E\left[ W(\hat{A}_N h(\varepsilon_j))|\mathcal{F}_t\right]\int^t_0 \dot{h}^\alpha(\varepsilon_j)(s)\gamma_{2,\alpha}(s)ds \]

With the Ito formula, we have:
\begin{align}\label{eq:VIII,12}
(T4,2,b)&=2\sum_{j\leqslant N} \int^t_0 \left( \int_0^s \dot{h}^\alpha(\varepsilon_j)\gamma_{2,\alpha}dr \right) d\left(  \E\left[ W(\hat{A}_N h(\varepsilon_j))|\mathcal{F}_s\right]\right) \nonumber \\
&+ 2\sum_{j\leqslant N} \int^t_0 \E\left[ W(\hat{A}_N h(\varepsilon_j))|\mathcal{F}_s  \right]\cdot  \dot{h}^\alpha(\varepsilon_j)\gamma_{2,\alpha}(s)ds
\end{align}

We compute $d\left(  \E\left[ W(\hat{A}_N h(\varepsilon_j))|\mathcal{F}_s\right]\right)$:
We remind: $\tilde{B}$ is the Brownian defined in (6.2). We have: 
\[ \hat{A}_N h(\varepsilon_j)(t)=\int_0^t B(\dot{\wideparen{P_N h(\varepsilon_j)}})ds \]
\[ W \left[ \hat{A}_N h(\varepsilon_j))\right]=\int_0^1 {}^t(B(\dot{\wideparen{P_N h(\varepsilon_j)}}))\cdot d\tilde{B} \]
and
\[\E\left[ W(\hat{A}_N h(\varepsilon_j))|\mathcal{F}_s\right]=\int_0^s  {}^t(B(\dot{\wideparen{P_N h(\varepsilon_j)}}))\cdot d\tilde{B}\]
so
\[d\left( \E\left[ W(\hat{A}_N h(\varepsilon_j))|\mathcal{F}_s\right]\right) ={}^t(B(\dot{\wideparen{P_N h(\varepsilon_j)}}))d\tilde{B}_s\]

Then \eqref{eq:VIII,12} becomes, with $<w_s,h(\varepsilon_j)>_H=\int_0^s \dot{h}_\alpha(\varepsilon_j)\gamma_{2,\alpha}dr$:
\begin{align}\label{eq:VIII,13}
(T4,2,b)
&=2\sum_{j\leqslant N} \int^t_0 <w_s,h(\varepsilon_j)>_H {}^t(B(\dot{\wideparen{P_N h(\varepsilon_j)}}))\cdot d\tilde{B}_s \nonumber \\
&+ 2\sum_{j\leqslant N} \int^t_0 \E\left[ W(\hat{A}_N h(\varepsilon_j))|\mathcal{F}_s  \right]\cdot  \dot{h}^\alpha(\varepsilon_j)\gamma_{2,\alpha}(s)ds
\end{align}
We denote these two integrals in \eqref{eq:VIII,13} by $I^{(1)}_N$ and $I^{(2)}_N$.
Then:
\begin{align*}\label{eq:VIII,13'}
I^{(1)}_N
&=\int^t_0 {}^t \left\lbrace  B \left( \sum_{j\leqslant N} <w_s,h(\varepsilon_j)>_H \dot{\wideparen{P_N h(\varepsilon_j)}}(s)\right) \right\rbrace  d\tilde{B}_s \nonumber \\
\hspace{22mm}&= \int_0^t {}^t \left\lbrace  B \left( \sum_{j\leqslant N} <w_s,h(\varepsilon_j)>_H \dot{h}(\varepsilon_j)\right) \right\rbrace  d\tilde{B}_s \hspace{22mm} (13') 
\end{align*}

With the Dini-Lipschitz theorem (Theorem 2.11), we know that the series ${}^t B \left( \sum_{j\leqslant N} <w_s,h(\varepsilon_j)>_H \dot{h}(\varepsilon_j)(s)\right)$ converges in $\mathbb{D}^\infty(\Omega)$, towards $\frac{1}{2}{}^tB\gamma_2$ and all items $ \sum_{j\leqslant N} <w_s,h(\varepsilon_j)>_H \dot{h}(\varepsilon_j)(s)$ are adapted and $s$-uniformly $\mathbb{D}^\infty$-bounded. So all integrals as in ($13'$) are $N$-uniformly, $\frac{1}{2}$-$\mathbb{D}^\infty$-Holderian, and the $L^2$-convergence of this series (obtained with the fundamental isometry) proves that the $I^{(1)}_N$ converges, as miltiplicators, towards $\frac{1}{2} \int_0^t {}^tB\gamma_2 d\tilde{B}=\frac{1}{2}\textup{div }\hat{A}w_t$.

At last, we study $I^{(2)}_N$:
\[ I^{(2)}_N=2 \sum_{j\leqslant N} \int^t_0 \E\left[ W(\hat{A}_N h(\varepsilon_j))|\mathcal{F}_s  \right]\cdot\dot{h}^\alpha(\varepsilon_j)(s)\gamma_{2,\alpha}(s)ds \]
We know that $\gamma_{2,\alpha}(s)$ is a  $\mathbb{D}^\infty$-S.M, so we can write:
\[ \gamma_{2,\alpha}(t)=\gamma_{2,\alpha}(0)+\int_0^t \gamma_{3,\alpha}ds+\int_0^t \gamma_{4,\alpha}\cdot d\tilde{B}_s \] 
So $\frac{1}{2} I^{(2)}_N$ becomes, with a Stratonovitch integration by parts:
\begin{align*}
\frac{1}{2} I^{(2)}_N
&= \sum_{j\leqslant N} \E\left[ W(\hat{A}_N h(\varepsilon_j))|\mathcal{F}_t  \right]\times 
h^\alpha(\varepsilon_j)(t)\gamma_{2,\alpha}(t)  \\
&- \sum_{j\leqslant N} \int^t_0 \E\left[ W(\hat{A}_N h(\varepsilon_j))|\mathcal{F}_s  \right]h^\alpha(\varepsilon_j)(s)\gamma_{3,\alpha}(s)ds \\
&- \sum_{j\leqslant N} \int^t_0 \E\left[ W(\hat{A}_N h(\varepsilon_j))|\mathcal{F}_s  \right]\cdot h^\alpha(\varepsilon_j)(s)\circ\gamma_{4,\alpha}d\tilde{B}_s\\
&- \sum_{j\leqslant N} \int^t_0 h^\alpha(\varepsilon_j)(s)\gamma_{2,\alpha}(s)\circ  d\left( \E\left[ W(\hat{A}_N h(\varepsilon_j))|\mathcal{F}_s \right] \right)
\end{align*}
With \eqref{eq:VIII,13}, we have: 
\begin{align}\label{eq:VIII,14}
\frac{1}{2} I^{(2)}_N
&= \sum_{j\leqslant N} \E\left[ W(\hat{A}_N h(\varepsilon_j))|\mathcal{F}_t  \right]\cdot h^\alpha(\varepsilon_j)(t)\gamma_{2,\alpha}(t)  \nonumber \\
&- \sum_{j\leqslant N} \int^t_0 \E\left[ W(\hat{A}_N h(\varepsilon_j))|\mathcal{F}_s  \right] h^\alpha(\varepsilon_j)(s)\gamma_{3,\alpha}(s)ds \nonumber \\
&- \sum_{j\leqslant N} \int^t_0 \E\left[ W(\hat{A}_N h(\varepsilon_j))|\mathcal{F}_s  \right] h^\alpha(\varepsilon_j)(s)\circ\gamma_{4,\alpha}d\tilde{B}_s \nonumber \\
&- \sum_{j\leqslant N} \int^t_0 h^\alpha(\varepsilon_j)(s)\gamma_{2,\alpha}(s)\circ {}^t(B(\dot{\wideparen{P_N h(\varepsilon_j)}}))d\tilde{B}_s 
\end{align}
The first item of the l.h.s. of \eqref{eq:VIII,14} is: 
\begin{align*}
&\sum_{j\leqslant N} \E\left[ W(\hat{A}_N h(\varepsilon_j))|\mathcal{F}_t  \right]\cdot h^\alpha(\varepsilon_j)(t)\gamma_{2,\alpha}(t)\\
&=\E\left[ W\left( \hat{A}_N \left( \sum_{j\leqslant N} h^\alpha(\varepsilon_j)(t)h(\varepsilon_j)\right) \right)|\mathcal{F}_t  \right]\cdot\gamma_{2,\alpha}(t)\\
&=\gamma_{2,\alpha}(t)\cdot \textup{div }\left[ P_t\left( \hat{A}\left(  \sum_{j\leqslant N} h^\alpha(\varepsilon_j)(t)h(\varepsilon_j) \right) \right)  \right]
\end{align*}
As $\vert h^\alpha(\varepsilon_j) \vert\leqslant \frac{C_0}{j}$, $C_0$ being a constant,   $P_t\left( \hat{A}\left(  \sum_{j\leqslant N} h^\alpha(\varepsilon_j)(t)h(\varepsilon_j) \right) \right)$ is $\frac{1}{2}$-$\mathbb{D}^\infty$-Holderian, $N$-uniformly, and converges as multiplicator towards div($P_t \hat{A} k_t^\alpha$), so the first item of the l.h.s of \eqref{eq:VIII,14} converges towards 
\[ \textup{div}(P_t \hat{A} k_t^\alpha)\cdot \gamma_{2,\alpha}(t)= \textup{div}(\hat{A} k_t^\alpha)\cdot \gamma_{2,\alpha}(t)\]

The second item of the l.h.s. of \eqref{eq:VIII,14} is:
\begin{align*}
&\sum_{j\leqslant N} \int^t_0 \E\left[ W(\hat{A}_N h(\varepsilon_j))|\mathcal{F}_s  \right] h^\alpha(\varepsilon_j)(s)\gamma_{3,\alpha}ds \\
&=\int^t_0 \E\left[ W\left( \hat{A} \left( \sum_{j\leqslant N}h^\alpha(\varepsilon_j)(s)h(\varepsilon_j)\right) \right)|\mathcal{F}_s  \right] \gamma_{3,\alpha}(s)ds
\end{align*}
As $W\left[ \hat{A} \left( \sum_{j\leqslant N}h^\alpha(\varepsilon_j)(s)h(\varepsilon_j)\right) \right] \in \mathcal{C}^1$, and converges towards $W(\hat{A}k^\alpha_s)$ uniformly relatively to $s$, the same is true for $\E\left[ W\left( \hat{A} \left( \sum_{j\leqslant N}h^\alpha(\varepsilon_j)(s)h(\varepsilon_j)\right) \right)|\mathcal{F}_s  \right]$ and so the sequence $\E\left[ W\left( \hat{A} \left( \sum_{j\leqslant N}h^\alpha(\varepsilon_j)(s)h(\varepsilon_j)\right) \right)|\mathcal{F}_s  \right]$ converges towards $\E\left[ W(\hat{A}k^\alpha_s)|\mathcal{F}_s  \right]$ as multiplicators. Then the sequence 
\[\E\left[ W\left( \hat{A} \left( \sum_{j\leqslant N}h^\alpha(\varepsilon_j)(s)h(\varepsilon_j)\right) \right)|\mathcal{F}_s  \right] \gamma_{3,\alpha}(s)\] 
converges as vector fields, towards the vector field, 
\[ \E[ W(\hat{A}k^\alpha_s)|\mathcal{F}_s] \gamma_{3,\alpha}(s)\]
And then 
\[ \int^t_0 \E\left[ W\left( \hat{A} \left( \sum_{j\leqslant N}h^\alpha(\varepsilon_j)(s)h(\varepsilon_j)\right) \right)|\mathcal{F}_s  \right] \gamma_{3,\alpha}(s)ds \]
converges as multiplicators towards $\int_0^t \E\left[ W(\hat{A}k^\alpha_s)|\mathcal{F}_s  \right] \gamma_{3,\alpha}ds$.

The third item of the l.h.s. of \eqref{eq:VIII,14} is:
\begin{align}\label{eq:VIII,15}
&\sum_{j\leqslant N} \int^t_0 \E\left[ W(\hat{A}_N h(\varepsilon_j))|\mathcal{F}_s  \right]\cdot h^\alpha(\varepsilon_j)(s)\circ\gamma_{4,\alpha}d\tilde{B}_s \nonumber \\
&= \int^t_0 \E\left[ W\left( \hat{A} \left( \sum_{j\leqslant N}h^\alpha(\varepsilon_j)(s)h(\varepsilon_j)\right) \right)|\mathcal{F}_s  \right] \cdot\gamma_{4,\alpha}d\tilde{B}_s \nonumber \\
&+\frac{1}{2} \int^t_0 \left[ d\left( \E\left[ W\left( \hat{A} \left( \sum_{j\leqslant N}h^\alpha(\varepsilon_j)(s)h(\varepsilon_j)\right) \right)|\mathcal{F}_s  \right]\right),  \gamma_{4,\alpha}\tilde{B}_s  \right]
\end{align}
Direct computation shows that:
\begin{align*}
&d\left( \E\left[ W\left( \hat{A} \left( \sum_{j\leqslant N}h^\alpha(\varepsilon_j)(s)h(\varepsilon_j)\right) \right)|\mathcal{F}_s  \right]\right) \\
&= {}^t\!\left\lbrace B\left( \sum_{j\leqslant N}\dot{h}^\alpha(\varepsilon_j)(s)h(\varepsilon_j) \right)  \right\rbrace d\tilde{B}_s
\end{align*}
So \eqref{eq:VIII,15} becomes:
\begin{align}\label{eq:VIII,16}
\eqref{eq:VIII,15} 
&=\int^t_0 \E\left[ W\left( \hat{A} \left( \sum_{j\leqslant N}h^\alpha(\varepsilon_j)(s)h(\varepsilon_j)\right) \right)|\mathcal{F}_s  \right] \cdot \gamma_{4,\alpha}(s)d\tilde{B}_s \nonumber \\
&+\frac{1}{2} \int^t_0 {}^t\!\left\lbrace \hat{A}\left( \sum_{j\leqslant N}<k_s^\alpha,h(\varepsilon_j)>_H h(\varepsilon_j) \right)  \right\rbrace \gamma_{4,\alpha}(s)ds 
\end{align}
In \eqref{eq:VIII,16}, the quantity $ \E\left[ W\left( \hat{A} \left( \sum_{j\leqslant N}h^\alpha(\varepsilon_j)(s)h(\varepsilon_j)\right) \right)|\mathcal{F}_s  \right]$
converges $\mathbb{D}^\infty(\Omega)$ towards $\E\left[ W(\hat{A}k^\alpha_s)|\mathcal{F}_s  \right]$ and as $\gamma_{4,\alpha}$ is $\mathbb{D}^\infty$-bounded, the first integral in \eqref{eq:VIII,16}, an Ito integral, is $\frac{1}{2}$-$\mathbb{D}^\infty$-Holderian so this sequence of Ito integrals is uniformly relatively to $N$, $\frac{1}{2}$-$\mathbb{D}^\infty$-Holderian, and converges as multiplicators towards the multiplicator: 
\[ \int_0^t \E\left[ W(\hat{A}k^\alpha_s)|\mathcal{F}_s  \right] \gamma_{4,\alpha}(s)d\tilde{B}_s \]
For the second integral in \eqref{eq:VIII,16}, a Lebesgue integral, the sequence 
\[{}^t\!\left\lbrace \hat{A}\left( \sum_{j\leqslant N}<k_s^\alpha,h(\varepsilon_j)>_H h(\varepsilon_j) \right)  \right\rbrace\] 
is a sequence of determinist vectors of $H$ and so converges towards 
\[{}^t\!\left\lbrace \hat{A}\left( \sum_{j=1}^\infty<k_s^\alpha,h(\varepsilon_j)>_H h(\varepsilon_j) \right)  \right\rbrace\]

As $\gamma_{4,\alpha}$ is $\mathbb{D}^\infty$-bounded and $\frac{1}{2}$-$\mathbb{D}^\infty$-Holderian the sequence of vector fields ${}^t\!\left\lbrace \hat{A}\left( \sum_{j\leqslant N}<k_s^\alpha,h(\varepsilon_j)>_H h(\varepsilon_j) \right)  \right\rbrace \gamma_{4,\alpha}$ is convergent, as vector fields, towards ${}^t\hat{A}k_s^\alpha \cdot\gamma_{4,\alpha}$, then the sequence of integrals 
\[\int^t_0 {}^t\!\left\lbrace \hat{A}\left( \sum_{j\leqslant N}<k_s^\alpha,h(\varepsilon_j)>_H h(\varepsilon_j) \right)  \right\rbrace \cdot\gamma_{4,\alpha}(s)ds\] 
converges as multiplicators towards $\int^t_0{}^t\hat{A}k_s^\alpha\cdot\gamma_{4,\alpha}(s)ds$. So the third item of the l.h.s. of \eqref{eq:VIII,14} converges as multiplicators towards $ \int_0^t \E\left[ W(\hat{A}k^\alpha_s)|\mathcal{F}_s  \right] \circ\gamma_{4,\alpha}d\tilde{B}_s .$

For the last item of the l.h.s. of \eqref{eq:VIII,14}:
\begin{align}\label{eq:VIII,17}
&\sum_{j\leqslant N} \int^t_0 h^\alpha(\varepsilon_j)(s)\gamma_{2,\alpha}(s)\circ {}^t\lbrace B(\dot{\wideparen{P_N h(\varepsilon_j)}})\rbrace d\tilde{B}_s \nonumber \\
&=\int_0^t \gamma_{2,\alpha} {}^t\!\left\lbrace B\left( \sum_{j\leqslant N}h^\alpha(\varepsilon_j)(s)\cdot\dot{h}(\varepsilon_j)(s) \right)  \right\rbrace \cdot d\tilde{B}_s \nonumber \\
&+ \frac{1}{2} \sum_{j\leqslant N} \int^t_0 \left[ d\left(   h^\alpha(\varepsilon_j)\gamma_{2,\alpha}(s) \right),B(\dot{\wideparen{P_N h(\varepsilon_j)}})d\tilde{B}_s   \right] \nonumber \\
&=  \int^t_0 \gamma_{2,\alpha}(s) {}^t\!\left\lbrace B\left( \sum_{j\leqslant N}<k_s^\alpha,h(\varepsilon_j)>_H \dot{h}(\varepsilon_j)(s) \right)  \right\rbrace d\tilde{B}_s \nonumber \\
&+\frac{1}{2}\int^t_0 \gamma_{4,\alpha}(s) B\left( \sum_{j\leqslant N}<k_s^\alpha,h(\varepsilon_j)>_H \dot{h}(\varepsilon_j)(s) \right)  ds
\end{align}
The first integral in \eqref{eq:VIII,17}: As $\sum_{j\leqslant N}<k_s^\alpha,h(\varepsilon_j)>_H \dot{h}(\varepsilon_j)(s)$ converges towards $\frac{1}{2}\dot{k}_s^\alpha$, and this sequence is $N$-uniformly bounded, the convergence of 
\[ \gamma_{2,\alpha}(s) {}^t\!\left\lbrace B\left( \sum_{j\leqslant N}<k_s^\alpha,h(\varepsilon_j)>_H \dot{h}(\varepsilon_j)(s) \right)  \right\rbrace \]
towards $\frac{1}{2}\gamma_{2,\alpha}(s,w){}^t(B\dot{k}_s^\alpha)$ is a $\mathbb{D}^\infty$-bounded convergence. Then the sequence of Ito integrals in \eqref{eq:VIII,17}, all $N$-uniformly $\frac{1}{2}$-$\mathbb{D}^\infty$-Holderian, converges as multiplicators towards the $\mathbb{D}^\infty$-multiplicator: 
\[\frac{1}{2}\int_0^t\gamma_{2,\alpha}(s){}^t(B\dot{k}_s^\alpha)d\tilde{B}_s\]
For the sequence of the Lebesque integrals in \eqref{eq:VIII,17}, we see that 
\[\gamma_{4,\alpha} B\left( \sum_{j\leqslant N}<k_s^\alpha,h(\varepsilon_j)>_H \dot{h}(\varepsilon_j)(s) \right)\] 
converges $\mathbb{D}^\infty$ towards 
\[\frac{1}{2}\gamma_{4,\alpha}(B\dot{k}_s^\alpha)=\frac{1}{2}\gamma_{4,\alpha}Be_\alpha\]
So the convergence of the Lebesque integrals
\[ \int_0^t \gamma_{4,\alpha}(s) B\left( \sum_{j\leqslant N}<k_s^\alpha,h(\varepsilon_j)>_H \dot{h}(\varepsilon_j)(s) \right) ds \]
is a multiplicator convergence towards $\frac{1}{2}\int_0^t \gamma_{4,\alpha}(B\dot{k}_s^\alpha)ds$. So the limit of the fourth term in \eqref{eq:VIII,14} can be written as:
\begin{align*}
&\int_0^t\gamma_{2,\alpha}(s){}^t(B\dot{k}_s^\alpha)d\tilde{B}_s+\frac{1}{2}\int_0^t \gamma_{4,\alpha}(B\dot{k}_s^\alpha)ds \\
&=\sum_{\alpha=1}^n \int_0^t \gamma_{2,\alpha}(s) {}^t(Be_\alpha)\circ d\tilde{B}_s
\end{align*}

Now we recapitulate all limits obtained:\\
\begin{align*}
&\lim_N(T1)=2\textup{div }\hat{A}  \textup{ grad} \hspace{89mm}\\
&\lim_N(T2)=2\int_0^t W(\hat{A} k^\alpha_s)Z(u_\alpha)(s)ds\\
&\lim_N(T3)=-2\sum_{\alpha=1}^n\int_0^t Z_{2,\alpha}(s)\cdot\frac{1}{2} Be_\alpha ds-2 \sum_{\alpha=1}^n \int_0^t \frac{1}{2}B\gamma_{2,\alpha}(s,w) e_\alpha ds\\
&\lim_N(T4,1)= 2\textup{div } \left[ Z_{2,\alpha}(t) W\left[ \hat{A}k^\alpha_t \right]\right] \textup{ grad}\\
&\lim_N(T4,2,a)= 2\textup{div }  \hat{A}w_t-2\textup{div } \hat{A}w_t=0\\
&\lim_N(T4,2,b)=  2\lim_N I^{(1)}_N+2\lim_N I^{(2)}_N\\
&\lim_N I^{(1)}_N=\textup{div }[(\textup{div }\hat{A}w_t)\textup{ grad}]\\
\end{align*}
\vspace{-13mm}
\begin{align*}
\lim_N I^{(2)}_N&=\textup{div }(\textup{div }(\hat{A}k_t^\alpha)\cdot \gamma_{2,\alpha}(t))\textup{ grad} \hspace{61mm}\\
&+\textup{div }\left[ \int_0^t \E\left[ W(\hat{A}k^\alpha_s)|\mathcal{F}_s  \right] \gamma_{3,\alpha}(s)ds\right] \textup{ grad}\\
&+\textup{div }\left[ \int_0^t \E\left[ W(\hat{A}k^\alpha_s)|\mathcal{F}_s  \right] \circ\gamma_{4,\alpha}d\tilde{B}_s \right] \textup{ grad}\\
&+\textup{div }\sum_{\alpha=1}^n \left[ \int_0^t \gamma_{2,\alpha}(s)\cdot {}^t(Be_\alpha)\circ d\tilde{B}_s\right] \textup{ grad}
\end{align*}
Instead of the Fourier basis $\check{B}$ (7.33) on $[0,1]$, we could have chosen the same type of Fourier basis but on $[t_0,1]$ $(0<t_0<1)$. Then if the matrix $B$ is multiplied by a function $f(w)\in \mathcal{F}_{t_0}$, as all coefficients in (T1),(T2),(T3),(T4) are adapted, the similar limits obtained in the case of this Fourier basis on $[t_0,1]$ and with $fB\cdot \mathds{1}_{[t_0,t]}$ will have the same form.\\
And this remains true if 
\begin{equation}\label{eq:VIII,18}
B=\sum_{i=0}^n \mathds{1}_{[t_i,t_{i+1}[} f_i(w)B_i
\end{equation}
with $f_i(w)\in\mathcal{F}_{t_i}$, $t_0=0$ and $t_n=1$.

From Theorem 4.7, each antisymmetrical matrix, adapted, and multiplicator, is a limit (in the multiplicator way) of step-functions as \eqref{eq:VIII,18}, so we see that for such antisymmetrical $B$, adapted and multiplicator, the limits (T1),(T2),(T3),(T4) have the same form than previously computed, but with $ W(\hat{A}k^\alpha_s)$ being $\int_0^s {}^t(Be_\alpha)\cdot d\tilde{B}$ and $(\textup{div }\hat{A}w_t)$ being $\int_0^t {}^t(B\gamma_{2,\alpha} e_\alpha)\cdot d\tilde{B}$.

Now, given a vector field $u\in H\cap C^2([0,1],\mathbb{R}^n)$, we want to find an $\mathbb{D}^\infty$-antisymmetrical matrix $B$, adapted, and $v\in \tilde{H}$ such that:\\
$D_v=(T_1+T_2+T_3+T_4)(B)+u$ which leads to two equations, using (7.23).

As vector fields, we must have:
\begin{equation}\label{eq:VIII,19}
(T_2+T_3)(B)+\int^t_0 \dot{u}(s)ds=\int_0^t \dot{h}_s(v)(s,w)ds
\end{equation}
As derivations, we must have:
\begin{equation}\label{eq:VIII,20}
(T_1+T_4)(B)=\textup{div }A(v)\textup{ grad}
\end{equation}

Using the formulas for $\lim_N T_1$, $\lim_N T_2$, $\lim_N T_3$, $\lim_N T_4$, we get with \eqref{eq:VIII,19}:
\begin{equation}\label{eq:VIII,21}
-\left( \sum_{\alpha=1}^n Z_{2,\alpha}(t)Be_\alpha+ \sum_{\alpha=1}^n B\gamma_{2,\alpha}e_\alpha\right)+\dot{u}(t)=\dot{h}_s(v)(t). 
\end{equation}
And for \eqref{eq:VIII,20}, we get:
\begin{align}\label{eq:VIII,22}
&B+\sum_{\alpha=1}^n Z_{2,\alpha}(t)\int_0^t {}^t(Be_\alpha\cdot d\tilde{B}+ \int_0^t {}^t(B\gamma_2)\cdot d\tilde{B} \nonumber \\
&+\sum_{\alpha=1}^n \left( \int_0^t {}^t(Be_\alpha)\cdot d\tilde{B} \right) \gamma_{2,\alpha}(t)+\sum_{\alpha=1}^n \left( \int_0^t \gamma_{3,\alpha}(s) \left( \int_0^s {}^t(Be_\alpha)\cdot d\tilde{B} \right)\right)ds \nonumber \\
&+\int_0^t \sum_{\alpha=1}^n \left( \int_0^s {}^t(Be_\alpha)\cdot d\tilde{B} \right) \gamma_{4,\alpha}(s)\cdot d\tilde{B}_s+\frac{1}{2}\sum_{\alpha=1}^n  \int_0^t \gamma_{4,\alpha} {}^t(Be_\alpha)ds  \nonumber \\
&+ \sum_{\alpha=1}^n  \int_0^t \gamma_{2,\alpha}(s) \cdot{}^t(Be_\alpha)\cdot d\tilde{B}+\frac{1}{2}\sum_{\alpha=1}^n  \int_0^t \gamma_{4,\alpha}(s) ({}^t Be_\alpha)ds \nonumber \\
&=\frac{1}{2}A(v)
\end{align}
Now we use the Stratonovich integration by parts and we have, with ($7.5'$):
\begin{align}\label{eq:VIII,23}
&Z_{2,\alpha}(t)\int_0^t {}^t(Be_\alpha)\cdot d\tilde{B}=\int_0^t Z_{2,\alpha}(s)\circ{}^t(Be_\alpha) d\tilde{B}_s \nonumber \\
&+\int_0^t \left( \int_0^s {}^t(Be_\alpha)d\tilde{B} \right)\circ dZ_{s,\alpha} \nonumber \\
&=\int_0^t  Z_{2,\alpha}(s){}^t(Be_\alpha)\cdot d\tilde{B}+\int_0^t \left( \int_0^s ({}^t Be_\alpha)\cdot d\tilde{B} \right) Z_{1,\alpha} \cdot d\tilde{B} \nonumber \\
&+ \int_0^t <Z_{1,\alpha},Be_\alpha>_{\mathbb{R}^n}ds
\end{align} 
We treat $\sum_{\alpha=1}^n \left( \int_0^t ({}^t Be_\alpha)\cdot d\tilde{B} \right) \cdot \gamma_{2,\alpha}(t)$ the same way and we get:
\begin{align}\label{eq:VIII,24}
\sum_{\alpha=1}^n \left( \int_0^t {}^t(Be_\alpha)\cdot d\tilde{B} \right) \gamma_{2,\alpha}(t)&=\int_0^t \left( \int_0^s {}^t(Be_\alpha)\cdot d\tilde{B} \right)\circ d\gamma_{2,\alpha}+\int_0^t \gamma_{2,\alpha}(s)\circ {}^t(Be_\alpha)d\tilde{B} \nonumber \\
&=\int_0^t \left( \int_0^s {}^t(Be_\alpha)\cdot d\tilde{B} \right) \cdot d\gamma_{2,\alpha}+ \frac{1}{2}\int_0^t \left[ d\gamma_{2,\alpha}, {}^t(Be_\alpha)d\tilde{B}\right] \nonumber \\
&+\int_0^t \gamma_{2,\alpha}(s)\cdot ({}^tBe_\alpha)d\tilde{B}+\frac{1}{2}  \int_0^t \left[ d\gamma_{2,\alpha}, {}^t(Be_\alpha)d\tilde{B}\right]
\end{align}
With ($5'$): 
\[ \gamma_{2,\alpha}(t)=\gamma_{1,\alpha}(t)-\int_0^t Z_{1,\alpha}\cdot d\tilde{B}-\frac{1}{2}\int_0^t\left[ dZ_{1,\alpha},d\tilde{B} \right]  \]
so
\[\left[ d\gamma_{2,\alpha}, {}^t(Be_\alpha)d\tilde{B}\right]= d\gamma_{1,\alpha}-<Z_{1,\alpha},{}^t(Be_\alpha)>_{\mathbb{R}^n}ds\]
\eqref{eq:VIII,24} becomes:
\begin{align}\label{eq:VIII,25}
\eqref{eq:VIII,24}&=\int_0^t \left( \int_0^s {}^t(Be_\alpha)d\tilde{B} \right) d\gamma_{2,\alpha}+\int_0^t \gamma_{2,\alpha}(s) ({}^tBe_\alpha)\cdot d\tilde{B} \nonumber \\
&+\int_0^t \mathfrak{L}_\alpha ({}^t (Be_\alpha))ds-\int_0^t<Z_{1,\alpha},{}^t(Be_\alpha)>_{\mathbb{R}^n}ds
\end{align}
In \eqref{eq:VIII,25}, $\mathfrak{L}_\alpha ({}^t (Be_\alpha))$ is a linear equation on ${}^t (Be_\alpha)$ with coefficients which are $L^\infty (\Omega\times [0,1])$ bounded.

Now we transfer \eqref{eq:VIII,23} and \eqref{eq:VIII,25} in \eqref{eq:VIII,22} and if we denote 
\begin{equation}\label{eq:VIII,26}
X(t)=\int_0^t {}^t(Be_\alpha)\cdot d\tilde{B}
\end{equation}
and denote this new equation by ($22'$), we get a system of three equations, \eqref{eq:VIII,21}, ($22'$), \eqref{eq:VIII,26} with four unknown variables $B$, $\dot{h}_s(v)$, $A(v)$, $X$. Now we will take as unknown variable $h(v)$: using the canonical isometry between $H$ and $\tilde{H}$, and with (7.27), ($7.27'$), (7.25) and (7.26), we see that $A(v)$ is a linear equation on the unknown variable $h(v)$, and after transfer of (7.27) and ($7.27'$) in \eqref{eq:VIII,21}, we get an equation which as unknown variables, has only $h(v)$ and $B$; this equation is numbered (27).

Now we make the same transfer in ($22'$), using (7.24), (7.25) and (7.26) to eliminate $A$ in ($22'$), $A$ being linearly dependent of $v$, so of $h(v)$ with the canonical isometry. This last equation is denoted ($22''$). So finally we get three equations, ($21'$), ($22''$) and \eqref{eq:VIII,26} with three unknown variables $h(v)$, $B$, $X$. These three equations make a system of three linear equations, linear in the variables $h(v)$, $B$, $X$, as direct inspection of these equations shows.

In these three equations ($21'$), ($22''$) and (26), there are two coefficients which are not $L_\infty (\Omega\times[0,1])$ bounded: $\gamma_{2,\alpha}$ and $Z_{2,\alpha}$, because they are formed with $\int_0^t Z_{1,\alpha}\circ d\tilde{B}$ (see ($5'$) and ($5''$)).

To be able to use the classical theorem on solutions of a system of linear SDE, we will truncate this system with a sequence of stopping times, so we need the following lemma:
\begin{lem}
$\forall \lambda\in\mathbb{R}, \E\left[ e^{\lambda\sup_{t\in[0,1]}\vert Z_2(t) \vert} \right]<+\infty$.
\end{lem}
\pf. We denote: $\tilde{Z}_{2,\alpha}(t)=\int_0^t Z_{1,\alpha}\cdot d\tilde{B}$ and $Z_{1,\alpha}\in L^\infty(\Omega\times[0,1])$. With the Ito formula:
\[ e^{\lambda \tilde{Z}_{2,\alpha}(t)-\frac{1}{2}\lambda^2\int_0^t Z_{1,\alpha}^2ds} =1+ \lambda \int_0^t e^{\lambda \tilde{Z}_{2,\alpha}-\frac{\lambda^2}{2}\int_0^s Z_{1,\alpha}^2ds}\cdot Z_1 d\tilde{B}\] 
So $e^{\lambda \tilde{Z}_{2,\alpha}(t)-\frac{1}{2}\lambda^2\int_0^t Z_{1,\alpha}(s)^2ds}$ is a local martingale. Let $(\tau_k)_{k\in N^*}$ a sequence of stopping times converging towards $+\infty$, then:
\[ e^{\lambda \tilde{Z}_{2,\alpha}(t\wedge \tau_k)-\frac{1}{2}\lambda^2\int_0^{t\wedge \tau_k} Z_{1,\alpha}^2ds} = \lambda \int_0^{t\wedge\tau_k} e^{\lambda \tilde{Z}_{2,\alpha}-\frac{\lambda^2}{2}\int_0^s Z_{1,\alpha}^2ds}\cdot Z_1 d\tilde{B}+1\]
And
\[ \E\left[  e^{\lambda \tilde{Z}_{2,\alpha}(t\wedge \tau_k)-\frac{\lambda^2}{2}\int_0^{t\wedge \tau_k} Z_{1}^2ds}\right]=1\]
So $\E\left[  e^{\lambda \tilde{Z}_{2,\alpha}(t\wedge \tau_k)}\right]\leqslant C_0(\lambda)$, $C_0$ constant is dependent of $k$. Then, with the Fatou Lemma: $\E\left[  e^{\lambda \tilde{Z}_{2,\alpha}(t)}\right]\leqslant C_0(\lambda)$. But: $\forall p>1, \E\left[  e^{p \lambda \tilde{Z}_{2,\alpha}(t)}\right]\leqslant C_0(p)$. So the local martingale is a martingale. Moreover,
\[\sup_te^{\lambda \tilde{Z}_{2,\alpha}(t)}+\sup_te^{-\lambda \tilde{Z}_{2,\alpha}(t)}\geqslant e^{\sup_t\lambda\vert \tilde{Z}_{2,\alpha}(t)\vert}\]
So as $\int_0^t [dZ_1,d\tilde{B}]$ is $L_\infty (\Omega\times[0,1])$, we have $\E\left[  e^{\lambda \sup_{t\in [0,1]}\vert Z_{2,\alpha}(t)\vert}\right]<+\infty$.

Now we denote by $\tau_k: \tau_k=\inf_t \vert Z_{2,\alpha}(t) \vert \geqslant k$, so $\vert Z_{2,\alpha}(t\wedge \tau_k) \vert \leqslant k$. Then the solution $S_k$ of the localized SDE is unique and verifies: \\
$\| S_k(t\wedge \tau_k) \|_{L^p}\leqslant C(p)e^{\beta k t}$, $\beta$ constant. So $\| S(t) \mathds{1}_{[\tau_k,\tau_{k+1}]}\|_{L^p}\leqslant C(p)e^{\beta k t}$. As in $[\tau_k,\tau_{k+1}], \sup_t Z_{2,\alpha}(t)$ is $\in[k,k+1]$, we have $\E\left[ \mathds{1}_{[\tau_k,\tau_{k+1}]} e^{\lambda \sup_t\vert Z_{2,\alpha}(t)\vert}\right]<C_1(\lambda)$ and $\E\left[ \mathds{1}_{[\tau_k,\tau_{k+1}]}\right]<C_1(\lambda)e^{-\lambda k}$. 

Then if $1<p'<p$, with Holder:
\[\| S(t)\mathds{1}_{[\tau_k,\tau_{k+1}]}\mathds{1}_{[\tau_k,\tau_{k+1}]} \|_{L^{p^\prime}}=\|S(t)\mathds{1}_{[\tau_k,\tau_{k+1}]} \|_{L^{p^\prime}}\leqslant \left( C_2e^{k\beta} \right)^{r_1}\left( C_1(\lambda)e^{-k\lambda} \right)^{r_2} \]
with $r_1,r_2>1$.

We can choose $\lambda$ so that the r.h.s. of the last equation is like $e^{-k}\times C_3(\lambda)$. Then the series defining $S(t)$ is $L^{p^\prime}(\Omega)$, $\forall p^\prime$, so $S(t)\in L^{\infty-0}$, uniformly in $t$. 

We can repeat this process for the gradient of $S$ and then $S(t)\in \mathbb{D}^\infty(\Omega)$, uniformly relatively to $t$, and so is $\frac{1}{2}$-$\mathbb{D}^\infty$-Holderian, so is a multiplicator.

Now let suppose that if $q$ is the quadratic form and that there exists $\alpha\in \textup{Der}^*$ such that $q(\alpha,\alpha)=0$. Then $\forall j\in \mathbb{N}_*$, if  $(v_j)_{j\in \mathbb{N}_*}$ is an Hilbertian basis of $\tilde{H}$: $\alpha(D_{V_j})=0$. Now $\mathcal{D}$ being the set of the combinations of $D_{V_j}$ which was used to obtain the limits: $\lim_N T_1$, $\lim_N T_2$, $\lim_N T_3$, $\lim_N T_4$, we have $\alpha(\lim_N T_1)=\alpha(\lim_N T_2)=\alpha(\lim_N T_3)=\alpha(\lim_N T_4)=0$. So when $B$ is built with step functions like Theorem 4.7, if we denote by $\delta_B$ the derivation associated to $B$, we have $\alpha(\delta_B)=0$. As $\alpha(D_v)=0$, we get from the system of linear SDE, $\alpha(u)=0$, $u$ being read as the derivation associated to the vector field $u\in C^2([0,1],\mathbf{R}^n)\cap H$.

Then with Theorem 2.7, we get $\alpha(\textup{Der})=0$, so $q$ is non-degenerate.

\section{\huge Some Tools on $\P_{m_0}(V_n, g)$}

Now we are going to study some properties of some mathematical tools on a $\P_{m_0}(V_n)$-stochastic manifold, and draw an incomplete list of opened questions.

\subsection{Some"renormalisation" theorem}

  \begin{thm}
    Let $T \in \left(\bigotimes^p \mathrm{Der}\right)^*$, 
    If $(\varepsilon_i)_{i\in\mathbb N_*}$ is a basis of $H$, then
    \begin{equation*}
      \sum_{\varepsilon_{i_1},\dotsc,\varepsilon_{i_p}}T(\varepsilon_{i_1},\dotsc,\varepsilon_{i_p})^2
      \in\mathbb D^\infty
    \end{equation*}
    the sum being on all $p$-uples that can be extracted from the basis $(\varepsilon_i)_{i\in\mathbb N_*}$.
  \end{thm}

  To prove this theorem, we need two lemmas.
  \begin{lem}
    A continuous $\mathbb R$-bilinear form on $\mathrm{Der}(\Omega)$, $\mathbb R$-valued, 
    which is continuous for each of its arguments, is bounded on each part of $\mathrm{Der}(\Omega)$.
  \end{lem}
  \begin{proof}
    Denote by $B(p,r,p',r')$ the set of all $\mathbb R$-linear continuous maps of 
    $\mathbb D_r^p(\Omega)$ in $\mathbb D_{r'}^{p'}(\Omega)$ and let $B(p_n,r_n,p_n',r_n')$, 
    $n \in \mathbb N_*$ a sequence of such sets: the projective limit of this sequence is 
    denoted $B(s,s')$ with $s = (p_n,r_n)_{n\in\mathbb N_*}$, $s' = (p_n',r_n')_{n\in\mathbb N_*}$;
    $B(s,s')$ is a Fr\'echet space.

    We denote by $\varinjlim_{s,s'}B(s,s')$ the inductive limit of the $B(s,s')$; we have 
    $\varinjlim_{(s,s')}B(s,s')=\mathrm{Der}$ and if $D$ is a bounded part of $\mathrm{Der}$,
    $\exists B(s_0,s_0')$ with $D\subset B(s_0,s_0')$ and $D$ is a bounded subset of
    $B(s_0,s_0')$ relatively to the Fr\'echet structure of $B(s_0,s_0')$.

    Let $D$ be a bounded part of $\mathrm{Der}$, and $q$ an $\mathbb R$-bilinear form on 
    $\mathrm{Der}$, and $v \in D$ fixed. We want to show that there exists $C_0(v)$ constant
    such that $\forall u\in\mathrm{Der},\quad |q(u,v)|\leq C_0(v)$.

    Suppose that $\exists(u_n)_{n\in\mathbb N_*}\in D$ such that $|q(u_n,v)|\to\infty$. Then 
    there exist $(\alpha_n)_{n\in\mathbb N_*}$, $\alpha_n\in\mathbb R$, $\alpha_n > 0$, such that
    $\alpha_n\to 0$ and $\alpha_nq(u_n,v)\nrightarrow 0$.
    But $\alpha_nu_n\to 0$ in $D$ ($D$ bounded), so $q$ being continuous, $\lim_{n\to\infty}q(\alpha_nu_n,v)=0$. 

    As $D\subset\mathrm{Der}\cap B(s_0,s_0')$ and as the topology on $D$ is the restriction of
    the Fr\'echet topology on $B(s_0,s_0')$ we can apply the Banach-Steinhaus theorem:
    \begin{equation*}
      \exists \text{ a constant }C, \text{ such that }\forall u\in D, \forall v \in D, \quad |q(u,v)|\leq C
    \end{equation*}
  \end{proof}

  \begin{lem}
    Let $q$ be an $\mathbb R$-bilinear form on $\mathrm{Der}(\Omega)$, positive, bounded
    on each bounded part of $\mathrm{Der}$ and symmetric: that is
    $\forall f \in \mathbb D^\infty, \forall u, \forall v\in\mathrm{Der}$, 
    $q(fu,v) = q(u,fv)$; $(e_i)_{i\in\mathbb N_*}$ being an OTHN basis of $H$, then 
    \begin{equation*}
      \sum_{i=1}^\infty q(e_i,e_i) < +\infty.
    \end{equation*}
  \end{lem}
  \begin{proof}
    Let $\varepsilon(j) = \pm 1$, $j\in \{1, \dotsc, n\}$. Denote
    \begin{equation*}
      D_\varepsilon = \sum_{j=1}^n \varepsilon(j)\left(W(h_j)k_j - W(k_j)h_j\right)
    \end{equation*}
    $(h_j)_{j\in\mathbb N_*}$, $(k_j)_{j\in\mathbb N_*}$ being two OTHN bases of H,
    we have $D_\varepsilon = \mathrm{div}\,A_n\mathrm{grad}$ where
    \begin{equation*}
      A_n = 
        \begin{pmatrix}
          \varepsilon(1)\begin{pmatrix}0&-1\\1&0\end{pmatrix} & ~ & ~ & \text{\Huge0}\\
          ~ & \varepsilon(2)\begin{pmatrix}0&-1\\1&0\end{pmatrix} & ~ & ~\\
          ~ & ~ & \ddots & ~ \\
          \text{\Huge0} & ~ & ~ & \varepsilon(n)\begin{pmatrix}0&-1\\1&0\end{pmatrix}
        \end{pmatrix}
    \end{equation*}
    So on a bounded part of $\mathbb D^\infty$, $D_\varepsilon$ is bounded, $\varepsilon$-uniformly.

    Then $\frac{1}{2^n}\sum_\varepsilon q(D_\varepsilon,D_\varepsilon)$ is also $n$-uniformly
    bounded on the subset of $\mathrm{Der}$ constitued by the $D_\varepsilon$. 

    This bound does not depend on the chosen basis of H because if we change this basis,
    the new basis is obtained from the initial basis by a unitary transformation.

    Then we have:
    \begin{align*}
      \frac{1}{2^n}\sum_\varepsilon q(D_\varepsilon,D_\varepsilon) =
        & \frac{1}{2^n}\sum_\varepsilon\sum_{j=1}^n\sum_{\ell=1}^n\varepsilon(j)\varepsilon(\ell)q\left(W(h_j)k_j,W(h_\ell)k_\ell\right) \\
        & -\frac{1}{2^n}2\sum_\varepsilon\sum_{j=1}^n\sum_{\ell=1}^n\varepsilon(j)\varepsilon(\ell)q\left(W(h_j)k_j,W(k_\ell)h_\ell\right) \\
        & + \frac{1}{2^n}\sum_\varepsilon\sum_{j=1}^n\sum_{\ell=1}^n\varepsilon(j)\varepsilon(\ell)q\left(W(k_j)h_j,W(k_\ell)h_\ell\right) \\
      = & \sum_{j=1}^nq\left(W(h_j)k_j,W(h_j)k_j\right)-2\sum_{j=1}^nq\left(W(h_j)k_j,W(k_j)h_j\right) \\
        & + \sum_{j=1}^nq\left(W(k_j)h_j, W(k_j)h_j\right)
        \tag{1}\label{l1}
    \end{align*}
    So the r.h.s. member of \eqref{l1} is $n$-uniformly bounded by a constant $C_0$.
    We fix n, and choose for $k$: $k_\ell = h_{n+\ell-1+j}$ and rewrite the r.h.s. of 
    \eqref{l1} with these new values, and average it on $\ell = 1, \dotsc, r$
    \begin{align*}
      & \frac1r\sum_{\ell=1}^r\sum_{j=1}^nq\left[W(h_j)h_{n+\ell-1+j}, W(h_j)h_{n+\ell-1+j}\right] \\
      & - \frac2r\sum_{\ell=1}^r\sum_{j=1}^nq\left[W(h_j)h_{n+\ell-1+j}, W(h_{n+\ell-1+j})h_j\right] \\
      & + \frac1r\sum_{\ell=1}^r\sum_{j=1}^nq\left[W(h_{n+\ell-1+j})h_j, W(h_{n+\ell-1+j})h_j\right] < C_0 \\
    \end{align*}
    So:
    \begin{align*}
      & \frac1r\sum_{\ell=1}^r\sum_{j=1}^nq\left[W(h_j)h_{n+\ell-1+j}, W(h_j)h_{n+\ell-1+j}\right] \\
      & - 2 \sum_{j=1}^nq\left[\frac1r\sum_{\ell=1}^rW(h_{n+\ell-1+j})h_{n+\ell-1+j}, W(h_j)h_j\right] \\
      & + \frac1r\sum_{j=1}^nq\left[\frac1r\sum_{\ell=1}^rW(h_{n+\ell-1+j})^2h_j,h_j\right] < C_0
    \end{align*}
    The first item of the above equation is positive so:
    \begin{align*}
      & - 2 \sum_{j=1}^nq\left[\frac1r\sum_{\ell=1}^rW(h_{n+\ell-1+j})h_{n+\ell-1+j}, W(h_j)h_j\right] \\
      & + \sum_{j=1}^nq\left[\frac1r\sum_{\ell=1}^rW(h_{n+\ell-1+j})^2h_j,h_j\right] < C_0\tag{2}\label{l2}
    \end{align*}
    The last item of the l.h.s. of \eqref{l2} can be rewritten as
    \begin{equation*}
      \sum_{j=1}^nq\left[\left(\frac1r\sum_{\ell=1}^rW(h_{n+\ell-1+j})^2 - 1\right)h_j,h_j\right] 
        + \sum_{j=1}^nq(h_j,h_j)
        \tag{3}\label{l3}
    \end{equation*}
    We denote by $a_{r,j} = \frac1r\left\{\sum_{\ell=1}^rW(h_{n+\ell-1+j})^2-1\right\}$. 
    Then $a_{r,j}\to 0$ in $L^2(\Omega$).
    Then \eqref{l3} can be rewritten as
    \begin{equation*}
      \sum_{j=1}^n\left\|a_{r,j}\right\|_{L^2(\Omega)}q\left[\frac{1}{\left\|a_{r,j}\right\|_{L^2(\Omega)}}
        \times\left(\frac1r\sum_{\ell=1}^rW(h_{n+\ell-1+j})^2-1\right)h_j,h_j\right] 
        + \sum_{j=1}^nq(h_j,h_j)
        \tag{4}\label{l4}
    \end{equation*}
    $\frac{1}{\left\|a_{r,j}\right\|_{L^2(\Omega)}}\times\frac1r\left(\sum_{\ell=1}^rW(h_{n+\ell-1+j})^2-1\right)h_j$
    is a set of derivations $(r\in\mathbb N_*)$ which is bounded in $\mathrm{Der}$ 
    ($n$ being previously fixed); so 
    \begin{equation*}
      q\left[\frac{1}{\left\|a_{r,j}\right\|_{L^2(\Omega)}}
        \times\left(\frac1r\sum_{\ell=1}^nW(h_{n+\ell-1+j})^2-1\right)h_j,h_j\right]
    \end{equation*}
    is bounded and as $\left\|a_{r,j}\right\|_{L^2(\Omega)} \to 0$ for each $j = 1, \dotsc, n$,
    \eqref{l4} is reduced to $\sum_{j=1}^nq(h_j,h_j)$.

    In \eqref{l2}, the only item left to compute is
    \begin{equation*}
      \sum_{j=1}^nq\left(\frac1r\sum_{\ell=1}^rW(h_{n+\ell-1+j})h_{n+\ell-1+j}, W(h_j)h_j\right)
    \end{equation*}
    The $L^2(\Omega, H)$-norm of $\frac1r\sum_{\ell=1}^rW(h_{n+\ell-1+j})h_{n+\ell-1+j}$ converges
    towards 0, so we can apply the same method as above and 
    \begin{equation*}
      \sum_{j=1}^nq\left(\frac1r\sum_{\ell=1}^rW(h_{n+\ell-1+j})h_{n+\ell-1+j}, W(h_j)h_j\right)
    \end{equation*}
    converges towards 0.

    So at the end we have $\sum_{j=1}^nq(h_j,h_j)<C_1$, $C_1$ constant, which implies 
    $\sum_{j=1}^\infty q(h_j,h_j)<+\infty$ 
  \end{proof}
  
  \begin{proof}[Proof of the theorem]
    $T$ is a $p$-linear form, $\mathbb D^\infty$-linear, continous for each of its
    arguments. We proceed with an induction: suppose the property
    is true for $T\in\left(\bigotimes^{p-1}\mathrm{Der}\right)^*$. 
    Fix $u_{i_1}, \dotsc, u_{i_{p-1}}$ vectors of a basis of $H$ and $u$ another
    vector of this basis; and $\psi\in L^{1+0} = \left(L^{\infty-0}\right)^*$, $\psi \geq 0$. Then
    \begin{equation*}
      \sum_{\left(u_{i_1},\dotsc,u_{i_{p-1}}\right)}
        \int\mathbb P(\mathrm d\omega)\psi\left[(1-L)^{\frac{r}{2}}T(u, u_{i_1},\dotsc,u_{i_{p-1}})\right]^2
    \end{equation*}
    is correctly defined (induction hypothesis) and $\geq 0$, and
    \begin{equation*}
      u \mapsto\sum_{\left(u_{i_1},\dotsc,u_{i_{p-1}}\right)}
                 \int\mathbb P(\mathrm d\omega)\psi\left[(1-L)^{\frac{r}{2}}T(u,u_{i_1},\dotsc,u_{i_{p-1}})\right]^2
    \end{equation*}
    defines an $\mathbb R$-valued quadratic form, positive, which satisfies the symmetry
    property, and bounded on any bounded part of $\mathrm{Der}$. We can apply to this
    quadratic form the result of Lemma 9,2 to get the result.
  \end{proof}
 
  Now we will study the link of the operator divergence, when defined on $\mathrm{Der}(\Omega)$
  (Definition 2,4, Remarks 2,1 and 2,2) with the NCM $\tilde H$, and the new gradient
  (Definition 7, 1). 

  Recall that $\mathrm{div}: \mathrm{Der}(\Omega)\to\mathbb D^{-\infty}(\Omega)$ is such that:
  if $\delta\in\mathrm{Der}(\Omega)$, $\forall \varphi\in\mathbb D^\infty(\Omega)$, 
  \begin{equation*}
    \left(\mathrm{div}\,\delta,\varphi\right) = -\int \delta\varphi\mathbb P(\mathrm d\omega)
  \end{equation*}
  $\mathrm{div}$ is a continuous operator.

  Now, given a Hilbertian basis of $\tilde H$, $(v_i)_{i\in\mathbb N_*}$ and $\alpha_i(\omega)$,
  $i \in\mathbb N_*$ being the components of a vector field in $H$, 
  $\sum_{i=1}^\infty\alpha_iv_i\in\mathbb D^\infty(\Omega,\tilde H)$; 
  we denote as usual $D_{v_i}$ the derivation associated to $v_i\in\tilde H$. Then:
  \begin{thm}
    \begin{equation*}
      \mathrm{div}\left(\sum_{i=1}^\infty\alpha_iD_{v_i}\right)
    \end{equation*}
    is well defined and $\in\mathbb D^\infty(\Omega)$
  \end{thm}
  \begin{proof}
    \begin{equation*}
      \mathrm{div}\left[\sum_{i=1}^\infty\alpha_iD_{v_i}\right] 
        = \mathrm{div}\left[\sum_{i=1}^\infty\alpha_i(D_{v_i}-h(v_i))\right] 
        + \mathrm{div}\,\sum_{i=1}^\infty\alpha_ih(v_i)
    \end{equation*}
    where $D_{v_i} = h_1 + \mathrm{div}\,A(v_i)\mathrm{grad}$ (Eq 7, 23), $h_1$ being the vector
    field of $H$ such as $h_1 = h_2 + h_3$ (Eq 7, 27), $h_2(v_i)=h(v_i)$ (Eq 7, 11).

    Now $\sum_{i=1}^\infty\alpha_ih(v_i)\in\mathbb D^\infty(\Omega,H)$, so 
    $\mathrm{div}\left(\sum_{i=1}^\infty\alpha_ih(v_i)\right) \in\mathbb D^\infty(\Omega) $. 

    It remains to show that 
    $\lim_{N\uparrow\infty}\mathrm{div}\left(\sum_{i=1}^N\alpha_i\left(D_{v_i}-h(v_i)\right)\right)\in\mathbb D^\infty(\Omega)$. But
    \begin{equation*}
      \mathrm{div}\left(\sum_{i=1}^N\alpha_i\left(D_{v_i}-h(v_i)\right)\right) 
        = \sum_{i=1}^N\alpha_i\mathrm{div}(D_{v_i}-h(v_i))
        + \sum_{i=1}^N(D_{v_i}-h(v_i))\cdot\alpha_i
    \end{equation*}
    According to (Eq 7, 23), and Remark 2, 2, iii, 
    \begin{equation*}
      \mathrm{div}(D_{v_i}-h(v_i)) = 0
    \end{equation*}
    We therefore just need to prove that 
    $\lim_{N\uparrow\infty}\sum_{i=1}^N(D_{v_i}-h(v_i))\cdot\alpha_i\in\mathbb D^\infty(\Omega)$, 
    and with (Eq 26, 27),
    \begin{equation*}
      D_{v_i}-h(v_i) = h_3(v_i) + \mathrm{div}\,A(v_i)\mathrm{grad}
    \end{equation*}
    We need the following lemma:
    \begin{lem}
      The set $(b_{ij})_{i,j\in\mathbb N_*}$ being the components of a vector field in
      $\mathbb D^\infty(\Omega,H\otimes H)$, we denote by $T$ an operator defined only
      on the finite sums like $\sum_{i=1}^mb_{ij}h(v_i)\otimes h(v_j)$ by
      \begin{equation*}
        T\left(\sum_{i=1}^mb_{ij}h(v_i)\otimes h(v_j)\right) 
          = \sum_{i,j=1}^mb_{ij}(\omega)A(v_i)\cdot h(v_j)
      \end{equation*}
      Then there is a unique extension of $T$, which is a multiplicator from 
      \begin{equation*}
        \mathbb D^\infty(\Omega, H\otimes H)
      \end{equation*}
      to $\mathbb D^\infty(\Omega, H)$.
    \end{lem}
    
    \begin{proof}
      Let $d\in\mathbb D^{-\infty}(\Omega, H)$. Then straightforward computation shows
      that, if $T^*$ exists,
      \begin{equation*}
        \left(T^*d, \sum_{i,j=1}^mb_{ij}(\omega)h(v_i)\otimes h(v_j)\right) 
          = -\left(\sum_{i=1}^mA(v_i)(d)\otimes h(v_i), \sum_{j,k=1}^mb_{kj}h(v_k)\otimes h(v_j)\right)
      \end{equation*}
      So $T^*d = + \sum_{i=1}^m A^*(v_i)(d)\otimes h(v_i)$.

      Each $A^*(v_i)\otimes h(v_i)$ can be considered as a vector matrix $\vec{A}_i$
      with entries $\left(\vec{A}_i^*\right)_k^\ell = -(a(v_i)_\ell^kh(v_i)$.

      As shown previously in the proof of Lemma 7,4, $ia(v_i)_\ell^k$ are $i$-uniformly
      $\mathbb D^\infty$-H\"olderian, and as $|h(v_i)|$ is bounded by $\frac{\text{constant}}{i}$,
      the sum
      \begin{equation*}
        \sum_{i=1}^\infty A^*(v_i)\otimes h(v_i)
      \end{equation*}
      is $\mathbb D^\infty$-H\"olderian, so $\sum_{i=1}^\infty A^*(v_i)\otimes h(v_i)$ is a
      multiplicator from $\mathbb D^{-\infty}(\Omega,H)$ to 
      $\mathbb D^{-\infty}(\Omega,H\otimes H)$.
    \end{proof}
    Now we go back to $\sum_{i=1}^\infty\mathrm{div}\,A(v_i)\mathrm{grad}\,\alpha_i$, 
    $(\alpha_i)_{i\in\mathbb N_*}$ being by hypothesis the coordinates of a $D^\infty$-vector field
    in $H$, $\left(\mathrm{grad}\,\alpha_i\right)_{i\in\mathbb N_*}\in\mathbb D^\infty(\Omega,H\otimes H)$.

    Then $\sum_{i=1}^\infty\mathrm{div}\,A(v_i)(\mathrm{grad}\,\alpha_i)\in\mathbb D^\infty(\Omega)$ 
    because $\sum_{i=1}^\infty A(v_i)(\mathrm{grad}\,\alpha_i)$ can be written as 
    $\sum_{i,j=1}^\infty b_{ij}h(v_i)\otimes h(v_j)$.

    The last sum for which the convergence is to be proven is $\sum_{i=1}^\infty h_3(v_i)\cdot\alpha_i$.
    We can write
    \begin{equation*}
      \sum_{i=1}^\infty h_3(v_i)\cdot\alpha_i 
        = \left\langle \sum_{i=1}^\infty h_3(v_i)\otimes h(v_i),\sum_{j=1}^\infty\mathrm{grad}\,\alpha_j\otimes h(v_j)\right\rangle_{H\otimes H}
    \end{equation*}
    and as $\left\|h_3(v_i)\right\|_{\mathbb D_r^p(\Omega)}\leq \frac{C_2(p,r)}{i}$, $C_2(p,r)$ being a constant
    $\sum_{i=1}^\infty h_3(v_i)\otimes h(v_i)$ is $\mathbb D^\infty(\Omega, H\otimes H)$-convergent,
    and so is $\sum_{j=1}^\infty\mathrm{grad}\,\alpha_j\otimes h(v_j)$; and the scalar 
    product of two vector fields which are $\mathbb D^\infty(\Omega, H\otimes H)$ is in
    $\mathbb D^\infty(\Omega)$.
  \end{proof}

  Now we will study the new O.U. operator built with the new $\mathrm{grad}$ (Def 7, 1) 
  and the new $\mathrm{div}$ (Theorem 9, 2): $\mathrm{div}\,\mathrm{grad}$.

  \begin{lem}
    If $V_n$ is a compact Riemannian manifold, then there exist $C^\infty(V_n)$ functions
    $\varphi_i$, $\psi_i$, $\alpha_i$, $i = 1, \dotsc, M$, such that for any $C^\infty$-vector
    field $U$ on $V_n$, we have:
    \begin{equation*}
      U = \sum_{i=1}^M\psi_i\left\langle U,\mathrm{grad}\,\varphi_i\right\rangle_{V_n}\mathrm{grad}\,\alpha_i
      \tag{5}\label{l5}
    \end{equation*}
  \end{lem}
  \begin{proof}
    Given a chart on $V_n$ with coordinates $x_i$, $\left(\frac{\partial}{\partial x_i}\right)$, $i=1, \dotsc, n$, are the
    canonical basis vectors on each point of this chart; with the Gram-Schmidt process, we
    get an orthonormal basis, denoted $(v_j)_{j = 1, \dotsc, n}$.

    Then $\sum_{j=1}^n\left\langle U,v_j\right\rangle_{V_n}v_j = U$.

    Then using a finite partition of unity on $V_n$, we have the desired result.
  \end{proof}

  We recall that if $f\in C^\infty(V_n)$, we denote by $F_{f,t}(\omega) = f\circ I(\omega)(t)$,
  $I$ being the It\^o map from $\mathcal W(\mathbb R^n)$ into $\mathbb P_{m_0}(V_n)$. We
  denote by $k_{j,t}$ the vector field of $\tilde H$ defined by
  \begin{equation*}
    k_{j,t} (s)= (s\wedge t)(ue_j)(s,\omega)
  \end{equation*}
  $ue_j(s,\omega)$ being the SPT of $ue_\mu$ at time $s$ ``along $\omega$'', and $j = 1, \dotsc, n$.

  \begin{defn}
    With $t_0\in [0,1]$, we call a $\mathcal F_{t_0}$-vector field a vector field which
    when written as a process $\varphi(t,\omega)$ is such that
    \begin{enumerate}
      \item[i)]$\dot\varphi(t,\omega) = 0\qquad\forall t\geq t_0$ a.s.;
      \item[ii)]$\varphi(t,\omega)\in\mathcal F_{t_0}\qquad\forall t\in [0,1]$
    \end{enumerate}
  \end{defn}

  \begin{rem}
    The scalar product in $\tilde H$ of two $\mathcal F_{t_0}$-vector fields is again 
    in $\mathcal F_{t_0}$.
  \end{rem}

  \begin{thm}
    The new O.U. operator, $\mathrm{div}\,{\mathrm{grad}}$, verifies: 
    if $f\in\mathcal F_t$, $\mathrm{div}\,{\mathrm{grad}}\,f\notin\mathcal F_t$
    in general.
  \end{thm}
  \begin{proof}
    We know that $k_{j,t}\in\mathcal F_t$, with $j=1,\dotsc,n$. With Lemma 9, 4, we can write
    \begin{equation*}
      \sum_{i=1}^M\psi_i\left[I(\omega)(t)\right]\left\langle\overrightarrow{\mathrm{grad}}\,\varphi_i,k_t\right\rangle_{V_n}\left(I(\omega)(t)\right)\left(\overrightarrow{\mathrm{grad}}\,\alpha_i\right)\left(I(\omega)(t)\right) = k_t
    \end{equation*}
    where $\overrightarrow{\mathrm{grad}}$ is the usual gradient on the manifold $V_n$.
    With Remark 7, 1, we get
    \begin{equation*}
      \sum_{i=1}^M\psi_i\left(I(\omega)(t)\right)\left\langle k_t,\mathrm{grad}\,F_{\varphi_i,t}\right\rangle_{\tilde H}(\omega)\cdot\left(\overrightarrow{\mathrm{grad}}\,\alpha_i\right)\left(I(\omega)(t)\right)
        = k_t\left(I(\omega)(t)\right)
    \end{equation*}
    We make the scalar product on $V_n$ of both members of this last equation, with a 
    determinist vector field $V$ and get:
    \begin{equation*}
      \sum_{i=1}^M\psi_i\left\langle k_t,\mathrm{grad}\,F_{\varphi_i,t}\right\rangle_{\tilde H} \left\langle V,\overrightarrow{\mathrm{grad}}\,\alpha_i\right\rangle_{V_n} 
        = \left\langle k_t, V\right\rangle_{V_n}
        \tag{6}\label{l6}
    \end{equation*}
    But
    \begin{equation*}
      \left\langle k_t, V\right\rangle_{\tilde H} = V(t) = \frac1t\left\langle k_t,V\right\rangle_{V_n}
    \end{equation*}
    so \eqref{l6} becomes
    \begin{equation*}
      \sum_{i=1}^M\psi_i\left\langle k_t, \mathrm{grad}\,F_{\varphi_i,t}\right\rangle_{\tilde H}\left\langle V,\overrightarrow{\mathrm{grad}}\,\alpha_i\right\rangle_{V_n} 
        = t\left\langle k_t, V\right\rangle_{\tilde H}
        \tag{7}\label{l7}
    \end{equation*}
    As $V \in \tilde H$, we have 
    \begin{equation*}
      \left\langle V,\overrightarrow{\mathrm{grad}}\,\alpha_i\right\rangle_{V_n} 
        = \left\langle V,{\mathrm{grad}}\,\alpha_i\right\rangle_{\tilde H}
    \end{equation*}
    and \eqref{l7} becomes:
    \begin{equation*}
      \sum_{i=1}^M\psi_i\left\langle k_t, \mathrm{grad}\,F_{\varphi_i,t}\right\rangle_{\tilde H}\left\langle \mathrm{grad}\,F_{\alpha_i,t},V\right\rangle_{\tilde H} 
        = t\left\langle k_t, V\right\rangle_{\tilde H}
    \end{equation*}

  From which we deduce
  \begin{equation*}
    tk_t = \sum_{i=1}^M\psi_i\left\langle k_t, \mathrm{grad}\,F_{\varphi_i,t}\right\rangle_{\tilde H}\mathrm{grad}\,F_{\alpha_i,t}
    \tag{8}\label{l8}
  \end{equation*}
  which proves that $k_t\in\tilde H$.

  Now we show that $\mathrm{div}\,k_t\notin\mathcal F_t$: $k_t\in\tilde H$ so $k_t$ 
  as an operator is $D_{k_t}$, with $D_{k_t} = h_1(k_t) + \mathrm{div}\,A(k_t)\mathrm{grad}$, $h(k_t)$
  being a vector field and $A(k_t)$ a multiplicator.

  Then $\mathrm{div}\,D_{k_t} = \mathrm{div}(h_1(k_t)) + \mathrm{div}(\mathrm{div}\,A(k_t)\mathrm{grad})$.
  With Remark 2, 2, iii, $\mathrm{div}(\mathrm{div}\,A(k_t)\mathrm{grad})=0$ so 
  \begin{equation*}
    \mathrm{div}\,D_{k_t} = \mathrm{div}[h_2(k_t)] + \mathrm{div}[h_3(k_t)]\tag{Eq 7, 27}
  \end{equation*}
  As
  \begin{equation*}
    h_3(k_t)(s)=\frac12\int_0^s\sum_{\mu=1}^nb_{\mu,\mu}(k_t)(r,\omega)\,\mathrm dr
    \tag{Eq 7, 25}
  \end{equation*}
  we see that generally, $\mathrm{div}\,h_3\notin\mathcal F_t$.

  Now we apply $\mathrm{div}$ to both members of \eqref{l8},
  \begin{align*}
    \mathrm{div}(tk_t) = 
      & ~\mathrm{div}\left(\sum_{i=1}^M\psi_i\left\langle k_t, \mathrm{grad}\,F_{\varphi_i,t}\right\rangle_{\tilde H}\mathrm{grad}\,F_{\alpha_i,t}\right)\\
    = & \sum_{i=1}^M\psi_i\left\langle k_t,\mathrm{grad}\,F_{\varphi_i,t}\right\rangle_{\tilde H}\cdot\mathrm{div}\,\mathrm{grad}\,F_{\alpha_i,t} \\
      & + \sum_{i=1}^M\left\langle\mathrm{grad}\,F_{\alpha_i,t}, \mathrm{grad}\left(\psi_i\left\langle k_t, \mathrm{grad}\,F_{\varphi_i,t}\right\rangle_{\tilde H}\right)\right\rangle_{\tilde H}
  \end{align*}
  but as $F_{\alpha_i,t}\in\mathcal F_t$, $\mathrm{grad}\,F_{\alpha_i,t}$ is a 
  $\mathcal F_t$-vector field, $\psi_i\in\mathcal F_t$ and $F_{\varphi_i,t}\in\mathcal F_t$
  so $\mathrm{grad}\left[\psi_i \left\langle k_t, \mathrm{grad}\,F_{\varphi_i,t}\right\rangle_{\tilde H} \right]$
  is a $\mathcal F_t$-vector field, and with Remark 9, 1,
  \begin{equation*}
    \sum_{i=1}^M\left\langle\mathrm{grad}\,F_{\alpha_i,t}, \mathrm{grad}\left(\psi_i\left\langle k_t, \mathrm{grad}\,F_{\varphi_i,t}\right\rangle_{\tilde H}\right)\right\rangle_{\tilde H}\in\mathcal F_t
  \end{equation*}
  So at least one term $\psi_i\left\langle k_t, \mathrm{grad}\,F_{\varphi_i,t}\right\rangle_{\tilde H}\cdot \mathrm{div}\,\mathrm{grad}\,F_{\alpha_i,t}$ is not in $\mathcal F_t$.
  But $\psi_i\left\langle k_t, \mathrm{grad}\,F_{\varphi_i,t}\right\rangle_{\tilde H}\in\mathcal F_t$;
  so there exists an $i$ such that 
  \begin{equation*}
    \mathrm{div}\,\mathrm{grad}\,F_{\alpha_i,t}\notin\mathcal F_t
  \end{equation*}
  while
  \begin{equation*}
    F_{\alpha_i,t} = \alpha_i\left[I(\omega)(t)\right]\in\mathcal F_t
  \end{equation*}
\end{proof}

  The same method can be used to show that if $A$ is an adapted multiplicator, 
  then $\mathrm{div}\,A\mathrm{grad}$ will not,
  in general, sends $\mathcal F_t$ in $\mathcal F_t$.

  \begin{thm}
    Using again the notation $D_0(\Omega)$ for the subset of $\mathrm{Der}(\Omega)$ 
    that is in bijection with $\mathrm{Der}^*(\Omega)$, we have: $D_0\subset\tilde H$
    and $D_0$ is dense in $\tilde H$, the density being conceived as usual, as simple
    convergence on $\mathbb D^\infty$-bounded subsets of $\mathbb D^\infty(\Omega)$.
  \end{thm}
  \begin{proof}
    The (Eq 9, 8) show that $k_t\in D_0(\Omega)$. And the $\mathbb D^\infty$-linear 
    sums of $k_t$ will be dense in $\tilde H$.
  \end{proof}
  
  \begin{thm}
    Let $q_1$ a bilinear form, positive definite on\\
    $\mathrm{Der}^*(\Omega)\times\mathrm{Der}^*(\Omega)$, with values in $\mathbb D^\infty(\Omega)$,
    $D_0^{(1)}(\Omega)$ being the subset of $\mathrm{Der}(\Omega)$ which is in bijective correspondence
    to $\mathrm{Der}^*(\Omega)$, thanks to $q_1$. This bijection being denoted $\theta_1$, let $\alpha$ a map
    \begin{equation*}
      \alpha: D_0^{(1)}(\Omega)\times D_0^{(1)}(\Omega)\to \mathbb D^\infty(\Omega)
    \end{equation*}
    such that
    \begin{enumerate}
      \item[i)] $\alpha$ admits an extension on $\mathrm{Der}(\Omega)\times\mathrm{Der}(\Omega)\to\mathbb D^\infty(\Omega)$
      \item[ii)] $q_1+\alpha\circ\theta_1=q_2$ is a bilinear form, positive definite and continuous
      from $\mathrm{Der}^*(\Omega)\times\mathrm{Der}^*(\Omega)$ to $\mathbb D^\infty(\Omega)$. 
    \end{enumerate}
    Then the extension of $\alpha$ is unique, and $D_0^{(2)}\subset D_0^{(1)}(\Omega)$.
  \end{thm}
  \begin{proof} ~ 

    \begin{enumerate}
      \item[i)] Let $\tilde \alpha_1$ and $\tilde \alpha_2$ be two extensions of $\alpha$ which
        are equal on $D_0^{(1)} \times D_0^{(1)}$. Suppose that there exists $v_0 \in D_0^{(1)}(\Omega)$
        such that $\gamma(v)=(\tilde \alpha_1 - \tilde\alpha_2)(v_0,v)$; then 
        $\gamma\in\mathrm{Der}^*(\Omega)$ and $\gamma\left(D_0^{(1)}(\Omega)\right)=0$.
        But $\theta(\gamma)\in D_0^{(1)}(\Omega)$. Then $\gamma(\theta(\gamma))=0 = q_1(\gamma,\gamma)$
        so $\gamma=0$. Same proof for $\gamma'(\omega) = (\tilde\alpha_1-\tilde\alpha_2)(v,w)$
        with $v\in\mathrm{Der}(\Omega)$ fixed.
        
      \item[ii)] Let $\beta$ and $\beta'\in\mathrm{Der}^*(\Omega)$, 
        \begin{equation*}
          q_1(\beta,\beta')+\alpha(\theta_1(\beta),\theta_1(\beta')) = q_2(\beta,\beta')
        \end{equation*}
        Then $\beta(\theta_1(\beta'))+\alpha(\theta_1(\beta),\theta_1(\beta'))=\beta(\theta_2(\beta'))$,
        $\theta_1$ and $\theta_2$ are the bijections of $\mathrm{Der}^*(\Omega)$ on $D_0^{(i)}(\Omega)$,
        $i = 1,2$, obtained with $q_1$ and $q_2$.

        Let $\mu(\beta'')=\alpha(\theta_1(\beta),\beta'')$. Then $\mu\in\mathrm{Der}^*(\Omega)$ and 
        \begin{equation*}
          \beta(\theta_1(\beta'))
            + \mu(\theta_1(\beta'))
            = \beta(\theta_2(\beta'))
            = \beta'(\theta_2(\beta))
        \end{equation*}
        as $q_2$ is symmetrical. So $(\beta+\mu)(\theta_1(\beta'))=\beta'(\theta_2(\beta))$ which implies 
        $\beta'(\theta_1(\beta+\mu))=\beta'(\theta_2(\beta))$. As this is valid 
        $\forall \beta'\in\mathrm{Der}^*(\Omega)$, 
        \begin{equation*}
          \theta_1(\beta+\mu)=\theta_2(\beta)
        \end{equation*}
        which implies 
        $D_0^{(2)}(\Omega)\subset D_0^{(1)}(\Omega)$.

        So on $D_0^{(2)}(\Omega)$, it is possible to make some variational calculus.
    \end{enumerate}
  \end{proof}

 \subsection{Another "renormalisation" theorem}

    Let $V_n$ an $n$-dimensional manifold, $q$ a metric on $V_n$ ($q$ is bilinear, symmetrical and
    positive definite), $\mu_q$ the canonical measure on $V_n$ associated to $q$, $\nabla$ the 
    Levi-Civita connection related to $q$, $\phi$ a $C^\infty$ density on $V_n$, $\phi >0$, and
    $m_0\in V_n$.
    
    We recall the following formulas:
    
    If $f,g$ are $C^\infty$ functions on $V_n$, and $u,v\in\Gamma(V_n)$ ($u$ and $v$ are vector fields),
    we have:
    \begin{align*}
      u\cdot f &=q\left(\mathrm{grad}\,f,u\right)\\
      \left(\mathrm{Hess}\,f\right)(u,v)&=u\cdot(v\cdot f)-\left(\nabla_uv\right)\cdot f\\
                                        &=\left(\mathrm{Hess}\,f\right)(v,u)
    \end{align*}
   
    If $(e_i)_{i=1,\dotsc,n}$ is an orthonormal basis on $T_mV_n$, $m$ being in a small neighbourhood
    of $m_0\in V_n$:
    \begin{align*}
      \Delta f&=\sum_{i=1}^n\left\{e_i\cdot(e_i\cdot f)-\left(\nabla_{e_i}e_i\right)\cdot f\right\} \\
              &=\sum_{i=1}^nq\left(\nabla_{e_i}\mathrm{grad}\,f,e_i\right)\\\\
      \mathrm{div}\,u&=\sum_{i=1}^nq\left(\nabla_{e_i}u,e_i\right)
    \end{align*}
    
    Then we define:
    \begin{equation*}
      \Delta u = \sum_{i=1}^n\left\{\nabla_{e_i}\left(\nabla_{e_i}u\right)-\nabla_{\nabla_{e_i}e_i}u\right\}
    \end{equation*}
    and the O.U. operator $L$ by
    \begin{equation*}
      \forall f,\forall g\in C^\infty(V_n),\qquad\int gL(f)\cdot\phi\,\mathrm d\mu_q
      =\int q(\mathrm{grad}\,f,\mathrm{grad}\,g)\phi\cdot\mathrm d\mu_q
    \end{equation*}
  From the last equation, we deduce:
  \begin{equation*}
    L(f) = -\Delta f-q(\mathrm{grad}\,\log\varphi,\mathrm{grad}\,f)
  \end{equation*}
  Then we define the operator $\mathrm{div}$ by
  \begin{equation*}
    \int g(\mathrm{div}_\varphi f)\varphi\mathrm d\mu_q 
      = -\int q(\mathrm{grad}\,f, \mathrm{grad}\,g)\varphi\,\mathrm d\mu_q
  \end{equation*}
  and
  \begin{equation*}
    L(u)=-\Delta u-\nabla_{\mathrm{grad}\,\log\varphi}u
  \end{equation*}
  From the symmetry of the Hessian, we have
  \begin{equation*}
    q(v,\nabla_u\mathrm{grad}\,f) = q(u,\nabla_v\mathrm{grad}\,f)
  \end{equation*}
  and
  \begin{equation*}
    \sum_{i=1}^nq(\nabla_{e_i}\mathrm{grad}\,f,\nabla_ue_i)
      =\sum_{i=1}^n\sum_{j=1}^nq(\nabla_{e_i}\mathrm{grad}\,f,e_j)q(e_j,\nabla_ue_i)=0
  \end{equation*}
  because $q(e_i,e_j)=0$ and the connection is the Levi-Civita connection.

  Now we are going to compute $q(\mathrm{grad}\,L(f),u)$:
  \begin{align*}
    q(\mathrm{grad}\,L(f),u)
      & = -q(\mathrm{grad}\,\Delta(f),u)-q(\mathrm{grad}\,q(\mathrm{grad}\,\log\varphi,\mathrm{grad}\,f),u) \\
      & = -u\cdot\sum_{i=1}^nq(\nabla_{e_i}\mathrm{grad}\,f,e_i)
          -u\cdot q(\mathrm{grad}\,\log\varphi,\mathrm{grad}\,f) \\
      & = -\sum_{i=1}^nq(\nabla_u(\nabla_{e_i}\mathrm{grad}\,f),e_i)
          -\sum_{i=1}^nq(\nabla_{e_i}\mathrm{grad}\,f,\nabla_ue_i) \\
      & \qquad -u\cdot q(\mathrm{grad}\,\log\varphi,\mathrm{grad}\,f) \\
      & = -\sum_{i=1}^nq(\nabla_u(\nabla_{e_i}\mathrm{grad}\,f),e_i) 
          -u\cdot q(\mathrm{grad}\,\log\varphi,\mathrm{grad}\,f) \\
      & = R(u,\mathrm{grad}\,f) 
          -\sum_{i=1}^nq(\nabla_{e_i}(\nabla_u\mathrm{grad}\,f),e_i) \\
      & \qquad -\sum_{i=1}^nq(\nabla_{[u,e_i]}\mathrm{grad}\,f,e_i) 
         -u\cdot q(\mathrm{grad}\,\log\varphi,\mathrm{grad}\,f)
        \tag{8'}\label{l8p}
  \end{align*}
  Using the Hessian symmetry:
  \begin{align*}
    q(\mathrm{grad}\,L(f),u)
      & = R(u,\mathrm{grad}\,f)
        -\sum_{i=1}^nq(\nabla_{e_i}\nabla_u\mathrm{grad}\,f,e_i)\\
      & \qquad  -\sum_{i=1}^nq(\nabla_{e_i}\mathrm{grad}\,f,[u,e_i])
        -u\cdot q(\mathrm{grad}\,\log\varphi,\mathrm{grad}\,f)
  \end{align*}
  Now we compute $\sum_{i=1}^nq(\nabla_{e_i}\mathrm{grad}\,f,[u,e_i])$:
  \begin{align*}
    \sum_{i=1}^nq(\nabla_{e_i}\mathrm{grad}\,f,[u,e_i])
      & = \sum_{i=1}^nq(\nabla_{e_i}\mathrm{grad}\,f,\nabla_ue_i)
          -\sum_{i=1}^nq(\nabla_{e_i}\mathrm{grad}\,f,\nabla_{e_i}u) \\
      & = -\sum_{i=1}^nq(\nabla_{e_i}\mathrm{grad}\,f,\nabla_{e_i}u)
  \end{align*}
  So
  \begin{align*}
    q(\mathrm{grad}\,L(f),u)
      & = R(u,\mathrm{grad}\,f)
        - \sum_{i=1}^nq(\nabla_{e_i}\nabla_u\mathrm{grad}\,f,e_i)\\
      & \qquad + \sum_{i=1}^nq(\nabla_{e_i}\mathrm{grad}\,f,\nabla_{e_i}u)
        - u\cdot q(\mathrm{grad}\,\log\varphi,\mathrm{grad}\,f)
  \end{align*}
  Now we compute $\sum_{i=1}^nq(\nabla_{e_i}\nabla_u\mathrm{grad}\,f,e_i)$:
  \begin{align*}
    \sum_{i=1}^nq(\nabla_{e_i}\nabla_u\mathrm{grad}\,f,e_i)
      & = \sum_{i=1}^ne_i\cdot q(\nabla_u\mathrm{grad}\,f,e_i) 
        - \sum_{i=1}^nq(\nabla_u\mathrm{grad}\,f,\nabla_{e_i}e_i) \\
      & = \sum_{i=1}^ne_i\cdot q(\nabla_{e_i}\mathrm{grad}\,f,u)
        - \sum_{i=1}^nq(\nabla_u\mathrm{grad}\,f,\nabla_{e_i}e_i) \\
      & = \sum_{i=1}^nq(\nabla_{e_i}\nabla_{e_i}\mathrm{grad}\,f,u)
        - \sum_{i=1}^nq(\nabla_u\mathrm{grad}\,f,\nabla_{e_i}e_i)\\
      & \qquad + \sum_{i=1}^nq(\nabla_{e_i}\mathrm{grad}\,f,\nabla_{e_i}u)
  \end{align*}
  Using this equality in \eqref{l8p} we have
  \begin{align*}
    q(\mathrm{grad}\,L(f),u)
      & = R(u,\mathrm{grad}\,f)
        - q(\Delta\mathrm{grad}\,f,u)
        - u\cdot q(\mathrm{grad}\,\log\varphi,\mathrm{grad}\,f)\\
      & = R(u,\mathrm{grad}\,f)
        - q(\Delta\mathrm{grad}\,f,u)
        - q(\nabla_u\mathrm{grad}\,\log\varphi,\mathrm{grad}\,f) \\
      & \qquad  - q(\mathrm{grad}\,\log\varphi,\nabla_u\mathrm{grad}\,f)
  \end{align*}
  We denote $Z(u)=\nabla_{u}\mathrm{grad}\,\log\varphi$ and 
  \begin{equation*}
    \tilde R(u,v)=\mathrm{Ricc}(u,v)-q(Z(u),v)
  \end{equation*}
  Then
  \begin{equation*}
    q(\mathrm{grad}\,L(f),u) = \tilde R(u,\mathrm{grad}\,f)+q(L(\mathrm{grad}\,f),u)\tag{9}\label{l9}
  \end{equation*}
  Now we compute $L(u\cdot f)$:
  \begin{equation*}
    L(u\cdot f) 
      = q(L(u),\mathrm{grad}\,f)
      + q(u,L(\mathrm{grad}\,f)
      - 2\sum_{i=1}^nq(\nabla_{e_i}u,\nabla_{e_i}\mathrm{grad}\,f))
      \tag{10}\label{l10}
  \end{equation*}
  and
  \begin{align*}
    \sum_{i=1}^nq(\nabla_{e_i}u,\nabla_{e_i}\mathrm{grad}\,f)
      & = \sum_{i=1}^n\left\{e_i\cdot q(u,\nabla_{e_i}\mathrm{grad}\,f) 
        - q(u,\nabla_{e_i}\nabla_{e_i}\mathrm{grad}\,f)\right\} \\
      & = \sum_{i=1}^ne_i\cdot q(e_i,\nabla_u\mathrm{grad}\,f)
        - \sum_{i=1}^nq(u,\nabla_{e_i}\nabla_{e_i}\mathrm{grad}\,ff) \\
      & = \sum_{i=1}^n\left\{ 
        q(\nabla_{e_i}\nabla_u\mathrm{grad}\,f,e_i)
        + q(\nabla_u\mathrm{grad}\,f,\nabla_{e_i}e_i)
        - q(u,\nabla_{e_i}\nabla_{e_i}\mathrm{grad}\,f) 
        \right\} \\
      & = \sum_{i=1}^nq(\nabla_{e_i}\nabla_u\mathrm{grad}\,f,e_i)
        - q(u,\Delta(\mathrm{grad}\,f))\\
      & = \sum_{i=1}^nq(\nabla_{e_i}\nabla_u\mathrm{grad}\,f,e_i)
        + q(u,L(\mathrm{grad}\,f))
        + q(\nabla_u\mathrm{grad}\,f,\mathrm{grad}\,\log\varphi) \\
      & = q(u,L(\mathrm{grad}\,f) 
        + \frac{1}{\varphi}\mathrm{div}\left(\varphi\left\{\nabla_u\mathrm{grad}\,f\right\}\right)
        \tag{11}\label{l11}
  \end{align*}
  Using \eqref{l9}, \eqref{l10} in \eqref{l11}, we get:
  \begin{align*}
    L(u\cdot f) 
      & = q(L(u),\mathrm{grad}\,f) 
        + \tilde R(u,\mathrm{grad}\,f) 
        - q(\mathrm{grad}\,L(f),u) 
        - \frac{2}{\varphi}\mathrm{div}\,\left\{\varphi\cdot\nabla_u\mathrm{grad}\,f\right\} \\
      & = q(L(u),\mathrm{grad}\,f) 
        + \tilde R(u,\mathrm{grad}\,f) 
        - q(\mathrm{grad}\,L(f),u) 
        - 2\mathrm{div}_\varphi(\nabla_u\mathrm{grad}\,f)
  \end{align*}
  
  Finally:
  \begin{equation*}
    q(L(u)+\tilde R(u),\mathrm{grad}\,f) 
      = L(u\cdot f)+q(\mathrm{grad}\,L(f),u) 
        + 2\mathrm{div}_\varphi(\nabla_u\mathrm{grad}\,f)
        \tag{12}\label{l12}
  \end{equation*}

  Now we will extrapolate the previous formula for a $\mathbb D^\infty$-stochastic manifold, 
  with an atlas having only one global chart, and such that the operator $\mathrm{div}$ sends
  $\mathcal D_0$ in $\mathbb D^\infty$.

  Then, if $u$ and $v$ are elements of $\mathcal D_0$:
  \begin{equation*}
    \mathrm{div}([u,v]) = (u\cdot \mathrm{div}\,v)-v\cdot(\mathrm{div}\,u)\in\mathbb D^\infty 
  \end{equation*}
  because $u,v\in\mathcal D_0$ and $\mathrm{div}\,u, \mathrm{div}\,v\in\mathbb D^\infty$.

  \begin{rem}
    Without the hypothesis  
    \begin{equation*}
      \mathrm{div}: \mathcal D_0 \to \mathbb D^\infty
    \end{equation*}
    we only have 
    \begin{equation*}
      \mathrm{div}\,u\in\mathbb D^{-\infty}.
    \end{equation*}
    But with this hypothesis on $\mathrm{div}$, we have
    if $u,v\in\mathcal D_0$: $\mathrm{div}\,\nabla_uv\in\mathbb D^\infty$, because
    \begin{equation*}
      \frac12\mathrm{div}(\nabla_{u+v}u+v)-\frac12\mathrm{div}\,\nabla_uu-\frac12\mathrm{div}\,\nabla_vv
    \end{equation*}
    and we know that (Lemma 3, 4) $\nabla_uu\in\mathcal D_0$ if $u\in\mathcal D_0$. 
  \end{rem}

  Now $\frac{2}{\varphi}\mathrm{div}\left\{\varphi\nabla_u\mathrm{grad}\,f\right\}$, the last
  item in \eqref{l12}, can be extrapolated by $\mathrm{div}\,\nabla_u\mathrm{grad}\,f$,
  and as $f\in\mathbb D^\infty$, $\mathrm{grad}\,f\in\mathcal D_0$, and 
  $\mathrm{div}\left(\nabla_u\mathrm{grad}\,f\right)\in\mathbb D^\infty$.

  In \eqref{l12} we have two more items, $L(u\cdot f)$ and $q(\mathrm{grad}\,L(f),u)$. As $u\in\mathcal D_0$, 
  and $\mathrm{grad}\,L(f)\in\mathcal D_0$, $q(\mathrm{grad}\,L(f),u)\in\mathbb D^\infty$ and 
  $u\cdot f\in\mathbb D^\infty$, $L(u\cdot f)\in\mathbb D^\infty$.

  So although quantities like $\tilde R$ and $L(u)$ cannot, generally, be defined on such a stochastic
  manifold, the quantity $\tilde R(u,\mathrm{grad}\,f)+L(u)$ can be given a meaning by
  \begin{equation*}
    q\left(\tilde R(u,\mathrm{grad}\,f)+L(u),\mathrm{grad}\,f\right)
      = L(u\cdot f)+q(\mathrm{grad}\,L(f),u)
        +\mathrm{div}\left(\nabla_u\mathrm{grad}\,f\right)
        \tag{13}\label{l13}
  \end{equation*}
  This extrapolation is legitimate because in the case of the Wiener space ``with $N$ dimensions'',
  \begin{align*}
    \varphi                                & = e^{-\frac12\sum_{i=1}^N W(e_i)^2} \\
    \nabla_u\mathrm{grad}\,f               & = \sum_{i=1}^N(u\cdot e_i\cdot f)e_i \\
    \mathrm{div}\,\nabla_u\mathrm{grad}\,f &=-\sum_{i=1}\left\{(u\cdot e_i f)W(e_i)+e_i\cdot(u\cdot e_if)\right\}
  \end{align*}
  In these two last formulas, $\nabla_u$ and $\mathrm{div}$ are operators on the Wiener space.

  And 
  \begin{equation*}
    \mathrm{div}_\varphi\left[\nabla_u\mathrm{grad}\,f\right]
      = q(\mathrm{grad}\,\log\varphi,\nabla_u\mathrm{grad}\,f)
      + \mathrm{div}(\nabla_u\mathrm{grad}\,f)
  \end{equation*}
  so when $\varphi = e^{-\frac12\sum_{i=1}^N W(e_i)^2}$,
  \begin{align*}
    \mathrm{grad}\,\log\varphi                                         & = -\sum_{i=1}^NW(e_i)e_i \\
     q\left(\mathrm{grad}\,\log\varphi,\nabla_u\mathrm{grad}\,f\right) & =-\sum_{i=1}W(e_i)\cdot u\cdot (e_if) \\
     \nabla_u\mathrm{grad}\,f                                          & =\sum_{i=1}^N(u\cdot(e_i\cdot f))\cdot e_i \\
     \mathrm{div}\left(\nabla_u\mathrm{grad}\,f\right)             
       & = \sum_{i,j=1}^Nq\left(\nabla_{e_j}\left((u\cdot(e_i\cdot f))e_i\right), e_j\right) \\
       & = \sum_{i=1}^Ne_i\cdot u\cdot(e_i\cdot f)
  \end{align*}
  So the formula \eqref{l13} is valid when the density is $\varphi = e^{-\frac12\sum_{i=1}^N W(e_i)^2}$ 
  and it does not depend on N; so it is valid when $N=+\infty$; in this case, the Wiener space comes with
  the standard quadratic form, $\tilde R = \mathrm{Id}$; so $L(u)$ then has meaning, even when
  $u\in\mathrm{Der}$, because we can give meaning to $\nabla_u\mathrm{grad}\,f$
  with $u\in\mathrm{Der}$, thanks to the extension theorem (tensor product by $H$).

  \begin{rem}
    If $\mathcal U(t,\omega)$ is a unitary process, adapted and multiplicator, but not continuous, it
    is still possible to approximate $\mathcal U$, in the multiplicator sense, with a sequence
    $\mathcal U_n$ of step-processes, adapted and $n$-uniformly multiplicators.

    The proof will use the Egorov theorem.
  \end{rem}
  
\subsection{Some open questions}
  \begin{enumerate}
    \item $A$ being an adapted process such that it sends $\mathbb D^\infty$-adapted vector 
    fields in $\mathbb D^\infty$-adapted vector fields, is $A$ a $\mathbb D^\infty$-multiplicator?
      
    \item If $\mathcal U$, unitary operator, sends $\mathbb D^\infty$-vector fields in 
    $\mathbb D^\infty$-vector fields, and admits an inverse, is $\mathcal U^{-1}$ a multiplicator?

    \item In the Wiener space case, is any $\mathbb D^\infty$-derivation the sum of a vector field
    and an operator which can be written as $\mathrm{div}\,A\mathrm{grad}$?
    We know that if the derivation is adapted with a null divergence, this is true.

    \item Does the $\mathbb D^\infty$-dual of $\mathbb D^\infty$-vector fields consist only of
    $\mathbb D^\infty$-vector fields?
  
    \item Is it possible to generalize the results about operators of the $\mathcal U\mathrm dB$
    form with multiples times Brownians?
  
    \item Given a Gaussian space $\Omega$, and a bilinear form $q$ on $\mathrm{Der}^*$, it is
    possible to obtain an O.U. operator associated with $q$, and then a $\mathbb D^\infty(q)$ 
    space. Under which conditions will we have $\mathbb D^\infty(q)\subseteq\mathbb D^\infty(\Omega)$?

    \item Given the map $\theta$ generated by a $\mathcal U\mathrm dB$ type of map, $\mathcal U$
    being an adapted multiplicator, is $\theta$ a $\mathbb D^\infty$-morphism? We know already that
    generally, it is not a $\mathbb D^\infty$-isomorphism.

    \item In the same setting than in (7), if $\mathcal U$ moreover is continuous relatively to
    $t$, $\mathbb P$-a.s., then is $\theta$ inversible in $L^{\infty-0}$?

    \item Given a diffeomorphism of $\mathbb D^\infty$ to $\mathbb D^\infty$, does it imply the
    existence of a $\mathbb D^\infty$-density?

    \item Given a derivation on $\mathbb D^\infty$, is its divergence an item of $\mathbb D^\infty$?
    If (3) is true, then (10) is true.

    \item If $\theta$ is a diffeomorphism generated by a $\mathcal U\mathrm dB$ type of map, of
    the Wiener in itself, does it induce a diffeomorphism of Wiener$^{\mathbb N}$ in itself?

    \item If we have a $p$-linear form $\varphi$ on $(\bigotimes \mathrm{Der})^p$ and a 
    $\mathbb D^\infty$-bilinear positive form $q$ on $\mathrm{Der}\times \mathrm{Der}$, after
    having chosen a basis of the Cameron-Martin space $H$, $(e_i)_{i\in\mathbb N_*}$, such
    that $q(e_i,e_j) = \delta_{ij}$, what are the NSC to have 
      \begin{equation*}
        \sum_{\left(e_{i_1},\dotsc,e_{i_p}\right)}\varphi(e_{i_1},\dotsc,e_{i_p})^2 \in \mathbb D^\infty\quad?
    \end{equation*}

    \item Is it possible to give meaning to the O.U. operator when it acts on a 
    $\mathbb D^\infty$-derivation $\delta$ such that $(\text{O.U})(\delta)\in\mathrm{Der}$? 
    Same question with O.U. acting on $\mathrm{Der}^*$.

    \item If a $\mathbb D^\infty$-derivation has the form $\mathrm{div}\,A\mathrm{grad}$, is the
    choice for $A$ unique? We know that it is true when the derivation is adapted, with zero divergence.

    \item If the matrix process defined by the antisymmetrical matrix $A$ is not adapted, is the
    multiplicator condition on $A$ necessary, so that $\mathrm{div}\,A\mathrm{grad}$ is a derivation
    on $\mathbb D^\infty$?

    \item Let $A$ be a multiplicator on the $\mathbb D^\infty$-module $\mathbb D^\infty(\Omega,H)$,
    $A: \mathbb D^\infty(\Omega,H)\to\mathbb D^\infty(\Omega,H)$. Can $A$ be written as a finite
    $\mathbb D^\infty$-linear arrangement of unitary multiplicators?

    \item Is it possible to generalize to the $n$-$\mathbb D^\infty$-linear forms on $\mathrm{Der}$,
    the approximation theorem 3, 2?

    \item Given an infinite sequence of morphisms $\varphi_n$, from a Wiener $\mathcal W$ to a Wiener $\mathcal W$
     under which
    conditions can we get an induced $\mathbb D^\infty$-morphism from $\mathcal W^{\mathbb N}$ to 
    $\mathcal W^{\mathbb N}$?

    \item Given $r\in \mathbb N_*$, does a derivation $\delta: \mathbb D^\infty \to \mathbb D^\infty$
    exist, such that there exists $f\in\mathbb D_r^\infty$ with $\delta f\in L^{\infty-0}$ and 
    $\delta f\notin\mathbb D_1^\infty$?

    \item If we define as a multiplicator, an operator $A$ that sends continuously 
    $\mathbb D^\infty(\Omega, H)$ to $\mathbb D^\infty(\Omega, H)$,
    then as in the adapted case, is it enough for $A$ to send continuously $\mathbb D_\infty^2(\Omega, H)$
    to $\mathbb D_\infty^2(\Omega, H)$ to be a multiplicator?

    \item $V_n$ being a Riemannian manifold, we define a random connection on $V_n$, the randomness
    according to $\omega\in\mathbb P(m_0,V_n)$. Then which results, that were sound for
    $\mathbb P(m_0,V_n)$, remain valid?

    \item Let us consider a map from a probabilised space to a subset of finite dimensional manifolds. 
    Under which conditions can the graph of this map be endowed with a $\mathbb D^\infty$-stochastic atlas?

    \item Under which conditions on the bilinear positive form $q$ can we generalize the results that
    were obtained for the standard bilinear form?

    \item In the case where $V_n$ is not a compact Riemannian manifold, let $\tau_B$ be the exit stopping
    time of the Brownian on $V_n$, distinct from $\infty$ when the manifold $V_n$ is not Brownian complete.
    Does a sequence of stopping times exist, $\tau_j$, $j\in\mathbb N_*$, $\tau_j\uparrow\tau_B$, such that
    on each stochastic interval $[\![0,\tau_j]\!]$, there exists a $\mathbb D^\infty$-manifold structure on
    $\mathbb P_{m_0}(V_n,g)$ and that the restriction to $[\![0,\tau_j]\!]$ of the $\mathbb D^\infty$-stochastic
    process on $[\![0,\tau_{j+1}]\!]$ is the $\mathbb D^\infty$-stochastic structure of $[\![0,\tau_j]\!]$?
\end{enumerate}

\end{document}